\documentclass{article}

\usepackage{microtype}
\usepackage{graphicx}
\usepackage{subcaption}
\usepackage{booktabs} 
\usepackage{dblfloatfix}

\usepackage[accepted]{icml2026}

\usepackage{amsmath}
\usepackage{amssymb}
\usepackage{mathtools}
\usepackage{amsthm}
\usepackage{thmtools}
\newcommand{\bs}{\boldsymbol} 

\theoremstyle{plain}
\newtheorem{theorem}{Theorem}[section]
\newtheorem{proposition}[theorem]{Proposition}
\newtheorem{lemma}[theorem]{Lemma}

\newtheorem{remark}[theorem]{Remark}
\theoremstyle{definition}
\newtheorem{definition}[theorem]{Definition}
\newtheorem{assumption}[theorem]{Assumption}

\usepackage{hyperref}

\usepackage[capitalize,noabbrev]{cleveref}

\usepackage[textsize=tiny]{todonotes}

\usepackage{listings}
\makeatletter
\newcommand{\STATEx}{\item[]}
\makeatother
\renewcommand{\algorithmiccomment}[1]{\hfill$\triangleright$ \textit{#1}}

\usepackage{multirow}

\icmltitlerunning{Asymmetric conformal prediction with penalized kernel sum-of-squares}

\begin{document}

\twocolumn[
\icmltitle{Asymmetric conformal prediction with penalized kernel sum-of-squares}



\icmlsetsymbol{equal}{*}

\begin{icmlauthorlist}
\icmlauthor{Louis Allain}{crest,safran}
\icmlauthor{Sébastien Da Veiga}{crest}
\icmlauthor{Brian Staber}{safran}
\end{icmlauthorlist}

\icmlaffiliation{crest}{Univ Rennes, Ensai, CNRS, CREST - UMR 9194, F-35000 Rennes, France}
\icmlaffiliation{safran}{Safran Tech, Digital Sciences \& Technologies, 78114 Magny-Les-Hameaux, France}

\icmlcorrespondingauthor{Louis Allain}{louis.allain@safrangroup.com}

\icmlkeywords{Machine Learning, ICML}

\vskip 0.3in
]



\printAffiliationsAndNotice{}  

\begin{abstract}
Conformal prediction (CP) is a distribution-free method to construct reliable prediction intervals that has gained significant attention in recent years. Despite its success and various proposed extensions, a significant practical feature which has been overlooked in previous research is the potential skewed nature of the noise, or of the residuals when the predictive model exhibits bias. In this work, we leverage recent developments in CP to propose a new asymmetric procedure that bridges the gap between skewed and non-skewed noise distributions, while still maintaining adaptivity of the prediction intervals. 
We introduce a new statistical learning problem to construct adaptive and asymmetric prediction bands, with a unique feature based on a penalty which promotes symmetry: when its intensity varies, the intervals smoothly change from symmetric to asymmetric ones. This learning problem is based on reproducing kernel Hilbert spaces and the recently introduced kernel sum-of-squares framework. First, we establish representer theorems to make our problem tractable in practice, and derive dual formulations which are essential for scalability to larger datasets. Second, the intensity of the penalty is chosen using a novel data-driven method which automatically identifies the symmetric nature of the noise. We show that consenting to some asymmetry can let the learned prediction bands better adapt to small sample regimes or biased predictive models.
\end{abstract}

\section{Introduction}
\label{sec:introduction}

Quantifying the prediction uncertainty of machine learning models has become a major concern for machine learning adoption in high stakes industries such as healthcare, aeronautics, financial forecasting and autonomous driving, where predictions help make important decisions.
In practice, such prediction intervals should provide at least marginal coverage guarantees that hold in finite sample and without making any distributional assumptions on the data. Conformal prediction (see, e.g., \citet{gammerman1998learningbytransduction, papadopoulos2002inductiveconfidencemachineregression, shafer2007tutorialconformalprediction, angelopoulos2022gentleintroductionconformalprediction}) has emerged recently as a simple yet very powerful tool to provide such uncertainty quantification in the form of prediction intervals, with guaranteed marginal coverage in finite sample, while being distribution-free. CP also distinguished itself for its simplicity as it is no more complex than computing a quantile.

However in real-word applications, additional features are now commonly sought for. First adaptivity, which ensures that prediction bands are wider when the model lacks confidence or if the variability in the data is high, and narrower when both the model is confident and the variability is low. Quickly identified as an important bottleneck, adaptivity has been the subject of intensive research and advances in recent years \citep{lei2012distributionfreepredictionbands,romano2019conformalizedquantileregression,hore2024conformalpredictionlocalweights,gibbs2025CPwithconditionalguarantees, allain2025scalableandadaptivepredictionbandsusingkerenlsumofsquares}. Second, the noise distribution may be skewed: this implies that symmetric intervals (as produced by standard split CP) would exhibit under- or over-coverage below or above the predictive model, even if they are calibrated to achieve marginal coverage. The same situation also arises if the noise is symmetric but the predictive model is biased. Contrary to adaptivity, accounting for asymmetric noise is a problem that has gone slightly unnoticed in recent research, and is a common obstacle in real-word applications \citep{pouplin24a}, see Appendix \ref{sec:noise distribution} for an in-depth discussion. If the noise is by essence asymmetric, symmetric prediction bands were shown to perform poorly in terms of coverage \citep{linussonSignedErrorConformalRegression2014}. 

Mainly two alternatives have been introduced to account for asymmetry in CP. Updating the calibration step was first proposed by \citet{linussonSignedErrorConformalRegression2014}, with coverage guarantees outside the interval. This calibration procedure has the appealing property of being applicable directly to most CP methods (see e.g., \citet{romano2019conformalizedquantileregression,barberPredictiveInferenceJackknife2021,pion2024gaussianprocessinterpolationwithconformalprediction}). Concurrently, new score functions adapted to asymmetry were proposed:  Conformalized Quantile Regression (CQR) \citep{romano2019conformalizedquantileregression}, which builds upon quantile regression, and Distributional Conformal Prediction (DCP) \citep{chernozhukovDistributionalConformalPrediction2021} based on the estimation of the conditional density function. Unfortunately, if the noise is symmetric, relying on asymmetric intervals may be detrimental since they are built without using all available information and may be ultimately wider. Adapting to the noise empirical distribution is thus essential and more robust than assuming a specific symmetry structure of the residuals. This flexibility is crucial, because the empirical distribution often diverges from the true noise distribution due to sampling artifacts: small sample settings can induce artificial asymmetry, while bias in the predictive model estimation can skew the residuals \citep{cheung2026biasawareconformalpredictionmetricbased}. Fully asymmetric bands may overfit this bias, while symmetric bands remain overly conservative. Penalizing asymmetry would allow the method to recover symmetry when appropriate, without enforcing it a priori. To adapt to these empirical imperfection, there is thus the need of a new paradigm that can transition between symmetric and asymmetric bands in a purely data-driven way.

Very recently, \citet{allain2025scalableandadaptivepredictionbandsusingkerenlsumofsquares} suggested to learn a score function to tackle adaptivity with symmetric prediction bands. Focusing on a normalized score function, they rely on reproducing kernel Hilbert spaces (RKHS) and kernel-sum-of-squares (kSoS) \citep{marteauferey2020nonparametricmodelsnonnegativefunctions} to handle the positivity constraint of the normalization function. Interestingly, they also discuss an extension of their work to learn an asymmetric score function, but only as a numerical illustration. Following their initial findings, we build upon their convenient learning framework to design a method that can handle symmetric and asymmetric bands seamlessly.

\paragraph{Contributions.}
We start by presenting CP in \Cref{sec:conformal prediction} and recent methods to account for asymmetric noise. In \Cref{sec:contributions}, we first formalize the underlying theory to learn an asymmetric score function in the kSoS framework, by providing a representer theorem and detailing its dual formulation to construct scalable asymmetric prediction bands. Second, we introduce two penalized versions to bridge the gap between asymmetric and symmetric prediction bands, for which we also prove a representer theorem and a dual formulation to enable faster computation on large datasets. These new problems can provide intermediate settings where information is shared between upper and lower prediction bands. Finally, we propose a new data-driven strategy to tune the penalty intensity and other critical hyperparameters, such as the kernel lengthscales. This strategy is based on an adaptivity criterion, for which we also provide warm-start strategies to speed up the hyperparameter search. This allows to automatically detect the amount of asymmetry needed for any dataset. In \Cref{sec:experiments}, we conduct extensive experiments to compare our methods to usual conformal prediction methods and asymmetric focused ones. In particular, we illustrate that the proposed method dynamically adapts to asymmetry induced by data defects, such as limited sample sizes. Crucially, it counterbalances estimation errors in the predictive model, that would otherwise degrade the performance of strictly symmetric models.

\section{Conformal prediction and asymmetry}
\label{sec:conformal prediction}


\paragraph{Split conformal prediction.} The full CP setting was introduced by \citet{gammerman1998learningbytransduction}, but we focus here on the split variant \citep{papadopoulos2002inductiveconfidencemachineregression}. We suppose we have a training dataset \(\mathcal{D}_N = \{\left(X_i, Y_i\right)\}_{i=1}^{N}\) from a pair \(\left(X, Y\right)\sim P_{XY}\) where \(X\in\mathcal{X}\subset\mathbb{R}^d\) and \(Y\in\mathcal{Y}\subset\mathbb{R}\). This dataset is split in two parts: a \emph{pre-training} dataset \(\mathcal{D}_n = \{\left(X_i, Y_i\right)\}_{i=1}^{n}\) and a \emph{calibration} one \(\mathcal{D}_m = \{\left(X_i, Y_i\right)\}_{i=1}^{m}\) with $N=n+m$.

The pre-training dataset \(\mathcal{D}_n\) is used to train a predictive model \(\widehat{m}_n(\cdot)\), which can be any machine learning algorithm. Then, the performance of the model is evaluated through so-called \emph{scores} on the hold-out calibration dataset \(\mathcal{D}_m\): the most common score in the literature is defined as the absolute errors \(S(X_i,Y_i):=S_i = \lvert Y_i-\widehat{m}_n(X_i) \rvert\) for \(i \in \mathcal{D}_m\). These scores are used to compute the quantile \(\widehat{q}_{\alpha}\) of the set \(\{S_i\}_{i \in \mathcal{D}_m}\) with an adjusted level \(\lceil(1-\alpha)(m+1)\rceil/m\), where \(\alpha\) is the desired error rate. Finally, for a new observation \(X_{N+1}\), the split CP prediction bands are \(\widehat{C}_{\mathcal{D}_{N}}(X_{N+1}) = \left[\widehat{m}_n(X_{N+1}) \pm \widehat{q}_{\alpha}\right]\), which satisfy the marginal coverage guarantee
\begin{align}
    \label{eq:marginal_coverage}
    \mathbb{P}\left(Y_{N+1} \in \widehat{C}_{\mathcal{D}_{N}}(X_{N+1})\right) \geq 1 - \alpha
\end{align}
for any $N$ if \(\left(X_1, Y_1\right),\ldots,\left(X_N, Y_N\right),\left(X_{N+1}, Y_{N+1}\right)\) are exchangeable. Importantly, these prediction bands are symmetric around the point prediction $\widehat{m}_n(X_{N+1})$ and are not adaptive, i.e. they do not depend on \(X_{N+1}\). 

\paragraph{Accounting for asymmetry.}
Asymmetry in the prediction bands can be accommodated at different stages within CP. 
Modifying the calibration step was first proposed by \citet{linussonSignedErrorConformalRegression2014}: they consider the signed scores \(S_i = Y_i-\widehat{m}_n(X_i)\) and compute the lower and upper quantiles \(\widehat{q}_{\alpha_{\mathrm{low}}}\) and \(\widehat{q}_{\alpha_{\mathrm{up}}}\) of the set \(\{S_i\}_{i \in \mathcal{D}_m}\) at adjusted levels \(\lceil\alpha_{\mathrm{low}}(m+1)\rceil/m\) and \(\lceil(1-\alpha_{\mathrm{up}})(m+1)\rceil/m\).  For \(\alpha_{\mathrm{low}} + \alpha_{\mathrm{up}}=\alpha \), the new prediction bands \(\widehat{C}_{\mathcal{D}_{N}}(X_{N+1}) = \left[\widehat{m}_n(X_{N+1}) - \widehat{q}_{\alpha_{\mathrm{low}}}, \widehat{m}_n(X_{N+1}) + \widehat{q}_{\alpha_{\mathrm{up}}}\right]\) satisfy the marginal coverage as in \Cref{eq:marginal_coverage}. But this calibration scheme further guarantees a lower and upper coverage with respective probability \(1-\alpha_{\mathrm{low}}\) and \(1-\alpha_{\mathrm{up}}\), and it can easily be adopted in many conformal prediction method (see e.g., \citet{romano2019conformalizedquantileregression, barberPredictiveInferenceJackknife2021, hanSplitLocalizedConformal2023, pion2024gaussianprocessinterpolationwithconformalprediction}). Unfortunately, those supplementary guarantees may come at the cost of inflating the width of the prediction bands \citep{romano2019conformalizedquantileregression}.

The second way to deal with asymmetric noise in CP is to modify the score function itself. The popular CQR \citep{romano2019conformalizedquantileregression} relies on quantile regression: instead of using an interval built around an estimate \(\widehat{m}_n(\cdot)\) of the regression function, they rely on estimates \(\widehat{q}^{\alpha_{\mathrm{low}}}_n(\cdot)\) and \(\widehat{q}^{\alpha_{\mathrm{up}}}_n(\cdot)\) of the conditional quantiles, and build the interval \(
\widehat{C}_{\mathcal{D}_{N}}(X_{N+1}) = \left[\widehat{q}^{\alpha_{\mathrm{low}}}_n(X_{N+1}) - \widehat{q}_{\alpha},\widehat{q}^{\alpha_{\mathrm{up}}}_n(X_{N+1}) + \widehat{q}_{\alpha}\right] \) with the score function  \(S(X,Y)=\max\left(\widehat{q}^{\alpha_{\mathrm{low}}}_n(X)-Y,Y-\widehat{q}^{\alpha_{\mathrm{up}}}_n(X) \right)\) and 
where \(\widehat{q}_{\alpha}\) is the adjusted quantile of the set \(\{\max\left(\widehat{q}^{\alpha_{\mathrm{low}}}_n(X_i)-Y_i,Y_i-\widehat{q}^{\alpha_{\mathrm{up}}}_n(X_i) \right)\}_{i\in \mathcal{D}_m}\). By design, CQR builds asymmetric prediction bands with a symmetric calibration procedure. Although appealing, CQR suffers from two well known practical limitations: (a) in high stakes problems, decision makers usually prefer a point estimate with an interval around that point and (b) quantile regression in small data regime and/or in high dimensional problems can be quite challenging. In a parallel line of work, DCP \citep{chernozhukovDistributionalConformalPrediction2021} considers \(\widehat{F}_{Y\vert X}\) an estimate of the conditional CDF with scores \(S_i = \lvert \widehat{F}_{Y_{i}\vert X_{i}} - 1/2\rvert\), but with the same limitations as CQR, see also \citet{sesia2021conformalpredictionusingconditionalhistograms}.

\paragraph{Learning a score function for adaptivity.}
Recently, several authors proposed to \emph{learn} the score function in order to target adaptivity. The core idea is to parameterize the score with unknown functions, which are estimated on the pre-training set. For example, \citet{xie2024boostedconformalpredictionintervals} define a task-specific loss (e.g. conditional coverage or minimum width), and consider a score function given by \(S(X,Y)=\max\left(\mu_1(X)-Y,Y-\mu_2(X)\right)/\sigma(X)\). This score is parameterized by three unknown functions \((\mu_1,\mu_2,\sigma)\), which are iteratively optimized with a boosting algorithm. In a similar vein, \citet{allain2025scalableandadaptivepredictionbandsusingkerenlsumofsquares} focus on a score function \(S(X,Y)=(Y-m(X))^{2}/f(X)\), where \(f\) is a positive function parameterized using kernel sum-of-squares. Their learning problem is defined through several ingredients: first, minimization of intervals mean width and second, $100\%$ coverage constraints on the pre-training set to uncover the band shape, which is later adjusted with the CP calibration step. For medium-size datasets, their approach showed better adaptivity than traditional competitors. However, their method is only limited to symmetric intervals, even if they incidentally suggest a possible generalization to asymmetric bands but do not provide theory, optimization, tuning, or empirical validation. This is the starting point of our proposal, which we tackle in the next section.

\section{Regularized kernel SoS for asymmetric prediction bands}
\label{sec:contributions}

Let us consider two RKHSs \(\mathcal{H}_{\mathrm{low}}\) and  \(\mathcal{H}_{\mathrm{up}}\) with respective kernels \(k_{\mathrm{low}}\), \(k_{\mathrm{up}}\), kernel lengthscales \(\theta_{\mathrm{low}}\), \(\theta_{\mathrm{up}}\) and feature maps \(\phi_{\mathrm{low}}\), \(\phi_{\mathrm{up}}\). For \(\mathcal{A_{\mathrm{low}}}\in\mathcal{S}_{+}\left(\mathcal{H}_{\mathrm{low}}\right)\) and \(\mathcal{A_{\mathrm{up}}}\in \mathcal{S}_{+}\left(\mathcal{H}_{\mathrm{up}}\right)\) two positive semi-definite (PSD) operators from \(\mathcal{H}_{\mathrm{low}}\) (resp. \(\mathcal{H}_{\mathrm{up}}\)) to \(\mathcal{H}_{\mathrm{low}}\) (resp. \(\mathcal{H}_{\mathrm{up}}\)), we define two non-negative functions \(f_{\mathrm{low}}(X) = \langle\phi_{\mathrm{low}}(X), \mathcal{A_{\mathrm{low}}}\phi_{\mathrm{low}}(X)\rangle_{\mathcal{H}_{\mathrm{low}}}\) and \(f_{\mathrm{up}}(X) = \langle\phi_{\mathrm{up}}(X), \mathcal{A_{\mathrm{up}}}\phi_{\mathrm{up}}(X)\rangle_{\mathcal{H}_{\mathrm{up}}}\), called \emph{kernel sum-of-squares}. For brevity, we will subsequently use the notation \((\cdot)\) when objects can be assessed for both \(_\mathrm{low}\) and \(_\mathrm{up}\). These two functions, thanks to their non-negativity property, are key components of our proposed new asymmetric score:
\begin{align}
\label{eq:asymmetric score function}
    S(X, Y) &= \max\bigl( \widehat{m}_n(X) - f_{\mathrm{low}}(X) - Y,\\ \nonumber
    &Y - \widehat{m}_n(X) - f_{\mathrm{up}}(X) \bigr).
\end{align}
This is a variant of the CQR score function, anchored at a predictive model \(\widehat{m}_n(X)\), where \(\widehat{q}^{\alpha_{\mathrm{low}}}_n(X)\) and 
\(\widehat{q}^{\alpha_{\mathrm{up}}}_n(X)\) are replaced by \(\widehat{m}_n(X) - f_{\mathrm{low}}(X)\) and \(\widehat{m}_n(X) + f_{\mathrm{up}}(X)\), respectively. Our framework is agnostic to the choice of point predictor \(\widehat{m}_n(X)\), allowing the use of diverse off-the-shelf models best suited to the underlying data distribution. From there, we propose to estimate the functions  \(f_{\mathrm{low}}(X)\) and \(f_{\mathrm{up}}(X)\) defining the prediction bands by solving the following learning problem:
\begin{align}
    \underset{
        \substack{
            \mathcal{A}_{\mathrm{low}} \in \mathcal{S}_{+}\left(\mathcal{H}_{\mathrm{low}}\right)\\
            \mathcal{A}_{\mathrm{up}} \in \mathcal{S}_{+}\left(\mathcal{H}_{\mathrm{up}}\right)
        }
    }{\inf}
    \quad& \frac{b}{n}\sum_{i=1}^{n} \left(f_{\mathrm{low}}(X_{i}) + f_{\mathrm{up}}(X_{i})\right)\label{eq:infinite asymmetric problem} \\
     &+\Omega_{\mathrm{low}}(\mathcal{A}_{\mathrm{low}}) + \Omega_{\mathrm{up}}(\mathcal{A_{\mathrm{up}}})\nonumber\\
    \mathrm{s.t.} \quad &\widehat{m}_n(X_i) - Y_i - f_{\mathrm{low}}(X_{i})\leq 0, \;i \in \left[n\right] \nonumber\\
    \quad& Y_i - \widehat{m}_n(X_i) - f_{\mathrm{up}}(X_{i})\leq 0, \;i \in \left[n\right]\nonumber
\end{align}
where \(\Omega_{(\cdot)}(\mathcal{A}) = \lambda_{(\cdot)1}\lVert\mathcal{A}\rVert_{\star}+\lambda_{(\cdot)2}\lVert\mathcal{A}\rVert_{F}^{2}\) is a regularization function with nuclear and Frobenius norms. Such penalty controls the complexity of functions \(f_{(\cdot)}\) and the bands adaptivity, while the first term in the objective function promotes tighter intervals. Importantly, the constraints impose $100\%$ coverage on the pre-training set, which helps learn an adaptive shape and makes the problem convex. 

Once the non-negative functions are estimated, we apply the usual split CP procedure using the asymmetric score defined in \Cref{eq:asymmetric score function}, with final calibrated prediction intervals
\begin{align}
    \label{eq:asymmetric prediction intervals}
    \widehat{C}_{\mathcal{D}_{N}}(X) &= [\widehat{m}_n(X) - \widehat{f}_{\mathrm{low}}(X) - \widehat{q}_{\alpha},\nonumber \\
    &\widehat{m}_n(X) + \widehat{f}_{\mathrm{up}}(X) + \widehat{q}_{\alpha}],
\end{align}
where \(\widehat{q}_{\alpha}\) is the adjusted quantile of the set \(\{S_i\}_{i \in \mathcal{D}_m}\). Here the predictive model \(\widehat{m}_n\) and the functions are estimated sequentially on the pre-training dataset \(\mathcal{D}_n\). 

Problem \ref{eq:infinite asymmetric problem} is infinite dimensional, but by noticing it is separable we can derive a useful representer theorem. Before stating it, let us introduce additional notations.  \(\mathbf{K}_{(\cdot)}\) and \(\mathbf{k}_{(\cdot)}(X) = \left(k_{(\cdot)}(X_{1}, X), \ldots, k_{(\cdot)}(X_{n}, X)\right)^{\top}\) denote the kernel matrix and vector associated to kernel \(k_{(\cdot)}\).

\begin{figure*}[!htbp]
    \centering
    \includegraphics[]{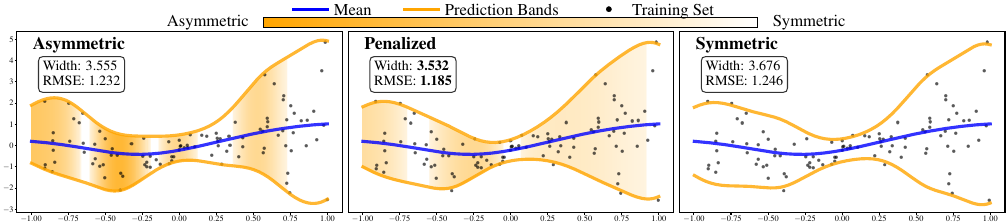}
    \caption{Penalized kSoS with varying penalty (dataset \(1\), symmetric noise). Left: asymmetric prediction bands produce tighter bands. Right: symmetric prediction bands tend to be overly conservative. An intermediate penalty value in the middle achieves tighter bands without being overly conservative and is closer to the oracle. Asymmetry of prediction bands is magnified with orange color.}
    \label{fig:impact of penalty on the shape of the bands}
\end{figure*}

We further consider the Cholesky decomposition of the kernel matrix \(\mathbf{K}_{(\cdot)} = \mathbf{V}^{\top}_{(\cdot)} \,\mathbf{V}_{(\cdot)} \) and \(r_{\mathrm{low}}\left(X_i, Y_i\right) = \widehat{m}_n(X_i) - Y_i\), \(r_{\mathrm{up}}\left(X_i, Y_i\right)=Y_i - \widehat{m}_n(X_i)\) the residuals. Finally, for a symmetric matrix \(\mathbf{A}\) with eigendecomposition \(\mathbf{A} = \mathbf{U} \mathbf{D} \mathbf{U}^{\top}\), its positive part is defined as \(\left[\mathbf{A}\right]_{+} = \mathbf{U} \max(0,\mathbf{D}) \mathbf{U}^{\top}\) and we write \(\forall x \in \mathbb{R}, \;\mathrm{Diag}\left({(\cdot)_x}\right) := \mathrm{Diag}\left({(\cdot)}\right) + \frac{x}{n}\mathbf{I}_{n}\).

\begin{theorem}[Representer theorem]
\label{thm:representer theorem for asymmetric problem}
    Let \((b,\lambda_{(\cdot){1}})\in\mathbb{R}_{+}^{2}\) and \(\lambda_{(\cdot){2}}>0\).
    Then \Cref{eq:infinite asymmetric problem} admits a unique solution \((\tilde{f}_{\mathbf{A}_{\mathrm{low}}^{\star}}, \tilde{f}_{\mathbf{A}_{\mathrm{up}}^{\star}})\) of the form \(\tilde{f}_{\mathbf{A}_{(\cdot)}^{\star}}(X) = \bs{\Phi}_{(\cdot)}(X)^{\top}\mathbf{A}_{(\cdot)}^{\star}\bs{\Phi}_{(\cdot)}(X)\) for some matrix \(\mathbf{A}_{(\cdot)}^{\star}\in \mathbb{S}_{+}^{n}\) solution of the semi-definite programming (SDP) problem
    \begin{align}
    \underset{\mathbf{A}_{(\cdot)}\in \mathbb{S}_{+}^{n}}{\inf} \quad& \frac{b}{n}\sum_{i=1}^{n} \tilde{f}_{\mathbf{A}_{(\cdot)}}(X_{i}) + \lambda_{(\cdot)1}\lVert \mathbf{A}_{(\cdot)}\rVert_{\star} + \lambda_{(\cdot)2}\lVert \mathbf{A}_{(\cdot)}\rVert_{F}^{2}\nonumber\\
        \mathrm{s.t.} \quad& r_{(\cdot)}\left(X_i, Y_i\right)-\tilde{f}_{\mathbf{A}_{(\cdot)}}(X_{i}) \leq 0, \;i \in \left[n\right].\label{eq:sdp for asymmetric problem}
\end{align}
\end{theorem}

The proof can be found in \Cref{sec:apdx:asymmetric problem with different kernels}.
In practice, the associated SDP problem can be solved efficiently up to \(200\) samples using off-the-shelves solvers \citep{odonoghue2016conicoptimizationoperatorsplittinghomogeneousselfdualembedding}. Crucially, to scale up to larger training sizes, we obtain a dual formulation for both functions.
\begin{proposition}[Dual formulation]
\label{prop:dual formulation for asymmetric problem}
    Let \((b,\lambda_{(\cdot){1}})\in\mathbb{R}_{+}^{2}\) and \(\lambda_{(\cdot){2}}>0\). \Cref{eq:sdp for asymmetric problem} admits the dual formulation
    \begin{align*}
        \underset{
        \substack{
            \bs{\Gamma_{(\cdot)}} \in \mathbb{R}_+^{n} 
        }
        }{\sup}
        \bs{\Gamma_{(\cdot)}}\mathbf{r}_{(\cdot)}^{\top}(\mathbf{X}, Y)  - \Omega^{\star}_{+,(\cdot)}(\mathbf{V}_{(\cdot)}\mathrm{Diag}({\bs{\Gamma}_{(\cdot)}}_{-b})\mathbf{V}_{(\cdot)}^{\top})
    \end{align*}
    where \(\Omega_{+,(\cdot)}^{\star}(\mathbf{B}) = \frac{1}{4\lambda_{(\cdot)2}}\lVert \left[\mathbf{B}-\lambda_{(\cdot)1}\mathbf{I}_{n}\right]_{+}\rVert_{F}^{2}\).
    Moreover, if \(\widehat{\bs{\Gamma}}_{(\cdot)}\) is a solution of the dual formulation, a solution of \Cref{eq:sdp for asymmetric problem} can be retrieved as 
    \begin{align*}
        \widehat{\mathbf{A}}_{(\cdot)} = \frac{1}{2\lambda_{(\cdot)2}}\left[\mathbf{V}_{(\cdot)}\mathrm{Diag}(\widehat{\bs{\Gamma}}_{(\cdot){-b}})\mathbf{V}_{(\cdot)}^{\top}-\lambda_{(\cdot){1}}\mathbf{I}_{n}\right]_{+}.
    \end{align*}
\end{proposition}

As with CQR, the prediction bands obtained with \Cref{eq:asymmetric prediction intervals} are asymmetric and flexible: lower and upper bands can have different kernel functions.

\subsection{Symmetric penalization}
\label{sec:symmetric penalization}

To go further, the nature of the noise is often unknown and assuming asymmetry can be detrimental since upper and lower bands are estimated separately. Taking inspiration from the literature on penalization for supervised learning, our central idea is to add a penalty \(\mathcal{P}(f_{\mathcal{A_{\mathrm{low}}}}, f_{\mathcal{A_{\mathrm{up}}}})\) in the objective function of \Cref{eq:infinite asymmetric problem}, so that the larger the penalty, the more symmetric the prediction bands:
\begin{align}
    \underset{
        \substack{
            \mathcal{A}_{\mathrm{low}} \in \mathcal{S}_{+}\left(\mathcal{H}_{\mathrm{low}}\right) \\
            \mathcal{A}_{\mathrm{up}} \in \mathcal{S}_{+}\left(\mathcal{H}_{\mathrm{up}}\right)
        }
    }{\inf}
    \quad& \frac{b}{n}\sum_{i=1}^{n} \left(f_{\mathcal{A_{\mathrm{low}}}}(X_{i}) + f_{\mathcal{A_{\mathrm{up}}}}(X_{i})\right)\label{eq:infinite asymmetric problem general penalty} \\[-3pt]
    &+\Omega_{\mathrm{low}}(\mathcal{A}_{\mathrm{low}}) + \Omega_{\mathrm{up}}(\mathcal{A_{\mathrm{up}}})\nonumber \\
    &+\mathcal{P}(f_{\mathcal{A_{\mathrm{low}}}}, f_{\mathcal{A_{\mathrm{up}}}})\nonumber\\
    \mathrm{s.t.} \quad &\widehat{m}_n(X_i) - Y_i - f_{\mathcal{A_{\mathrm{low}}}}(X_{i})\leq 0, \;i \in \left[n\right] \nonumber\\
    \quad& Y_i - \widehat{m}_n(X_i) - f_{\mathcal{A_{\mathrm{up}}}}(X_{i})\leq 0, \;i \in \left[n\right].\nonumber
\end{align}
Introducing such a continuum between asymmetric and symmetric prediction bands can be advantageous in practice: it allows to mitigate the impact of small samples and compensate biased predictive models. As an illustration, in \Cref{fig:impact of penalty on the shape of the bands} we show that even for a test case with true symmetric noise, penalized asymmetric prediction bands can lead to better local coverage. Importantly, these new problems bridge the gap between our asymmetric Problem (\ref{eq:sdp for asymmetric problem}) and the symmetric one proposed by \citet{allain2025scalableandadaptivepredictionbandsusingkerenlsumofsquares}. In the following, we discuss two different penalties that can achieve this behavior.

\paragraph{Operator penalty.}
First, let us consider that lower and upper bounds are associated to the same RKHS \(\mathcal{H}_{\mathrm{low}}=\mathcal{H}_{\mathrm{up}}=\mathcal{H}\); the functions are thus only parameterized by their operators, defined on the same space. In this case, a natural way to incorporate a symmetric penalty in our problem is to enforce equality of the operators by considering a regularization on their difference:
\begin{align}
\label{eq:operator penalty}
    \mathcal{P}(f_{\mathcal{A_{\mathrm{low}}}}, f_{\mathcal{A_{\mathrm{up}}}}) = &\lambda_{\mathrm{pen}1}\lVert\mathcal{A}_\mathrm{low}-\mathcal{A}_\mathrm{up}\rVert_{\star}\\&+\lambda_{\mathrm{pen}2}\lVert\mathcal{A}_\mathrm{low}-\mathcal{A}_\mathrm{up}\rVert_{F}^{2}.\nonumber
\end{align}
Intuitively, as \(\lambda_{\mathrm{pen}1}, \lambda_{\mathrm{pen}2} \rightarrow \infty\), the operators will tend to be equal and thus the lower and upper bands will coincide on their whole domain of definition (see \Cref{prop:error bounds}). This will yield symmetric prediction bands, although different from the ones initially proposed by \citet{allain2025scalableandadaptivepredictionbandsusingkerenlsumofsquares} (see Appendix \ref{sec:comparison with symmetric and asymmetric prediction bands} for a detailed comparison).

\paragraph{Penalization on the training set.}
\label{sec:penalization on the training set}
Another way to impose symmetry is to control the difference of the lower and upper bands on the training points only:
\begin{equation}
\label{eq:training set penalty}
    \mathcal{P}(f_{\mathcal{A}_{\mathrm{low}}}, f_{\mathcal{A}_{\mathrm{up}}}) 
    = \lambda_{\mathrm{pen}} \sum_{i=1}^{n}\bigl(f_{\mathcal{A}_{\mathrm{low}}}(X_i) - f_{\mathcal{A}_{\mathrm{up}}}(X_i)\bigr)^2.
\end{equation}
As \(\lambda_{\mathrm{pen}} \rightarrow \infty\), the lower and upper functions will coincide on all training points. Contrary to the previous operator penalty, this no longer imposes equality on their domain of definition. But we show in \Cref{prop:error bounds} that, under mild assumptions, their difference can be upper bounded.

\paragraph{Representer theorems and dual formulations.}

To make the penalized problems tractable, we prove the following representer theorems, see Sections \ref{sec:asymmetric problem with penalty 2} and \ref{sec:asymmetric problem with penalty 1} for the proofs and the resulting SDP problems.
\begin{theorem}[Representer theorems with penalty]
\label{thm:representer theorem for penalty}
    Let \((b,\lambda_{\mathrm{low}{1}}, \lambda_{\mathrm{up}{1}}, \lambda_{\mathrm{pen}{1}})\in\mathbb{R}_{+}^{4}\) and \(\lambda_{\mathrm{low}{2}}, \lambda_{\mathrm{up}{2}}, \lambda_{\mathrm{pen}{2}}>0\). Then Problem (\ref{eq:infinite asymmetric problem general penalty}) with Penalty (\ref{eq:operator penalty}) or Penalty (\ref{eq:training set penalty}) admits a unique solution \((\tilde{f}_{\mathbf{A}_{\mathrm{low}}^{\star}}(X), \tilde{f}_{\mathbf{A}_{\mathrm{up}}^{\star}}(X)) = \bigl(\bs{\Phi}(X)^{\top}\mathbf{A}_{\mathrm{low}}^{\star}\bs{\Phi}(X),\ \bs{\Phi}(X)^{\top}\mathbf{A}_{\mathrm{up}}^{\star}\bs{\Phi}(X)\bigr)\) for some matrices \(\bigl(\mathbf{A}_{\mathrm{low}}^{\star}, \mathbf{A}_{\mathrm{up}}^{\star}\bigr)\in \mathbb{S}_{+}^{n}\times \mathbb{S}_{+}^{n}\).
\end{theorem}
For scalability, we also derive dual formulations which allow to consider datasets up to thousand samples.

\begin{proposition}[Dual formulations with penalty]
\label{prop:dual formulation penalty}
Let \((b,\lambda_{\mathrm{low}{1}}, \lambda_{\mathrm{up}{1}}, \lambda_{\mathrm{pen}{1}})\in\mathbb{R}_{+}^{4}\) and \(\lambda_{\mathrm{low}{2}}, \lambda_{\mathrm{up}{2}}, \lambda_{\mathrm{pen}{2}}>0\).
Problem (\ref{eq:infinite asymmetric problem general penalty}) with Penalty (\ref{eq:operator penalty}) admits the dual formulation
\begin{align*}
        \underset{
        \substack{
            (\bs{\Gamma}_{\mathrm{low}},\bs{\Gamma}_{\mathrm{up}}) \in \mathbb{R}_+^{2n}\\
            \mathbf{W} \in \mathbb{S}^{n}
        }
        }{\sup}
        &(\bs{\Gamma}_{\mathrm{up}}-\bs{\Gamma}_{\mathrm{low}})\mathbf{r}^{\top} 
        - \Omega_{\mathrm{pen}}^{\star}(\mathbf{W})\\
        -&\Omega^{\star}_{+,\mathrm{low}}(\mathbf{V}_{\mathrm{low}}\mathrm{Diag}({\bs{\Gamma}_{\mathrm{low}}}_{-\mathbf{b}})\mathbf{V}_{\mathrm{low}}^{\top} - \mathbf{W})\nonumber\\
        -&\Omega^{\star}_{+,\mathrm{up}}(\mathbf{V}_{\mathrm{up}}\mathrm{Diag}({\bs{\Gamma}_{\mathrm{up}}}_{-\mathbf{b}})\mathbf{V}_{\mathrm{up}}^{\top} + \mathbf{W})\nonumber
    \end{align*}
    where $\mathbf{r}$ is the vector of residuals $r_i=Y_i-\widehat{m}_n(X_i)$, \(\Omega_{+,(\cdot)}^{\star}(\mathbf{B}) = \frac{1}{4\lambda_{(\cdot)2}}\lVert \left[\mathbf{B}-\lambda_{(\cdot)1}\mathbf{I}_{n}\right]_{+}\rVert_{F}^{2}\), $\Omega_{\mathrm{pen}}^{\star}(\mathbf{B})=(1/4\lambda_{\mathrm{pen}2})\sum_{i=1}^{n}\max(0, \lvert \lambda_{i}(\mathbf{B}) \rvert-\lambda_{\mathrm{pen}1})^{2}$, while Problem (\ref{eq:infinite asymmetric problem general penalty}) with Penalty (\ref{eq:training set penalty}) admits the dual formulation
    \begin{align*}
        \underset{
        \substack{
            (\bs{\Gamma}_{\mathrm{low}},\bs{\Gamma}_{\mathrm{up}}) \in \mathbb{R}_+^{2n}\\
            \bs{\alpha}_0 \in \mathbb{R}^{n}
        }
        }{\sup}
        &(\bs{\Gamma}_{\mathrm{up}}-\bs{\Gamma}_{\mathrm{low}})\mathbf{r}^{\top} - \frac{1}{4\lambda_{\mathrm{pen}}} \bs{\alpha}_0\bs{\alpha}_0^{\top}\\
        &-\Omega^{\star}_{+,\mathrm{low}}(\mathbf{V}_{\mathrm{low}}\mathrm{Diag}(({\bs{\Gamma}_{\mathrm{low}}}+\bs{\alpha}_0)_{-\mathbf{b}})\mathbf{V}_{\mathrm{low}}^{\top})\nonumber\\
        &-\Omega^{\star}_{+,\mathrm{up}}(\mathbf{V}_{\mathrm{up}}\mathrm{Diag}(({\bs{\Gamma}_{\mathrm{up}}}-\bs{\alpha}_0)_{-\mathbf{b}})\mathbf{V}_{\mathrm{up}}^{\top}). \nonumber
    \end{align*}
\end{proposition}
Contrary to the asymmetric case, observe that the dual formulation for the operator penalty involves an optimization problem with \(\mathcal{O}(n^2)\) unknowns. Interestingly, the training set penalty however scales linearly, and is thus more suited to larger datasets. The only downside is that controlling equality of lower and upper bands everywhere requires additional assumptions, as elaborated in the next paragraph.

Moreover, both dual formulations naturally accommodate low-rank kernel approximations, enabling them to handle datasets with up to $n=10000$ samples, see Section~\ref{sec:low rank approximation}.

\paragraph{Errors bounds.}
We now give theoretical insights on these two novel penalties. We place ourselves in the ideal setting \(\mathcal{H}_{\mathrm{low}}=\mathcal{H}_{\mathrm{up}}=\mathcal{H}\) where we can reach strict equality of lower and upper bands, and introduce assumptions on \(\mathcal{H}\) and \(\mathcal{X}\), the domain of definition of the inputs $X$.

\begin{assumption}
    \label{assumption:paper:rkhs_1}
    For a bounded open set $\mathcal{X}\in\mathbb{R}^d$, the RKHS $\mathcal{H}$ of functions on $\mathcal{X}$ with norm $\Vert \cdot \Vert_{\mathcal{H}}$ satisfies $f\vert\mathcal{X}\in\mathcal{H}$, $\forall f \in C^{\infty}(\mathbb{R}^d)$. Moreover $\forall u,v\in\mathcal{H}$, $u\cdot v \in\mathcal{H}$ and $\exists M\geq 1$ such that $\Vert u\cdot v \Vert_{\mathcal{H}} \leq \mathrm{M} \Vert u \Vert_{\mathcal{H}} \Vert v \Vert_{\mathcal{H}}$.
\end{assumption}
\begin{assumption}
    \label{assumption:paper:rkhs_2}
    For a bounded open set $\mathcal{X}\in\mathbb{R}^d$, the kernel $k$ of $\mathcal{H}$ satisfies
    $\max_{\vert \alpha\vert=1} \sup_{x,y\in\mathcal{X}} \vert \partial^{\alpha}_x \partial^{\alpha}_y k(x,y)\vert \leq \mathrm{D}^2 < \infty$
    for some $\mathrm{D}\geq 1$.
\end{assumption}
\begin{assumption}
\label{assumption:paper:geometric of the domain}
    $\mathcal{X} = \cup_{x\in S}\mathrm{B}_{r}(x)$, where $S$ is a bounded subset of $\mathbb{R}^d$ and $\mathrm{B}_{r}(x)$ is the ball of center $x$ and radius $r$.
\end{assumption}

Assumptions \ref{assumption:paper:rkhs_1} and \ref{assumption:paper:rkhs_2} are mild assumptions that hold for classic kernels such as the Matérn \(5/2\) one. As noted by \citet{rudi2025finding}, Assumption \ref{assumption:paper:geometric of the domain} can be relaxed to $\mathcal{X}$ having Lipschitz continuous boundaries, which will typically hold for most datasets in practice. We can now state our result which provides an upper bound on \(\Delta_f \coloneq \sup_{x\in\mathcal{X}}\, \lvert \tilde{f}_{\mathbf{A}_{\mathrm{low}}}(x) - \tilde{f}_{\mathbf{A}_{\mathrm{up}}}(x) \rvert\), the difference between lower and upper bands,  see \Cref{sec:apdx:error bounds} for the proof.

\begin{figure*}[!htbp]
    \centering
    \includegraphics{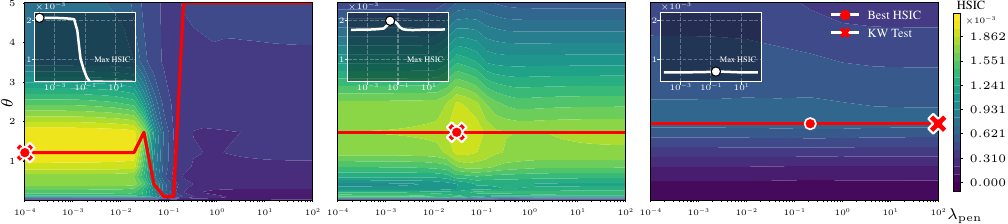}
    \caption{\(\mathrm{HSIC}\) contour plots for asymmetric noise distribution (left) and symmetric ones (middle, right). 
    Left: lower \(\lambda_{\mathrm{pen}}\) values are clearly favored by HSIC.
    Middle: allowing for some asymmetry produces more adaptive bands than asymmetric or symmetric ones. 
    Right: all \(\lambda_{\mathrm{pen}}\) values achieve similar adaptivity. According to the KW test, the symmetric model is preferred over the highest \(\mathrm{HSIC}\) model.}
    \label{fig:contour plots for multiple noise distributions}
\end{figure*}

\begin{proposition}[Error bounds]
\label{prop:error bounds}
Let $\mathcal{H}$ be a RKHS with associated kernel $k$ which satisfies Assumption \ref{assumption:paper:rkhs_1}. If $k$ is bounded such that $\Vert k\Vert_{\infty}:= \sup_{x\in\mathcal{X}} \sqrt{k(X,X)} < \infty$, then for any two PSD matrices $\mathbf{A}_{\mathrm{low}}, \mathbf{A}_{\mathrm{up}} \in \mathbb{S}^{n}_{+}$, we have
\begin{equation*}
    \Delta_f \leq \mathrm{M} \Vert k\Vert_{\infty} \,\lVert \mathbf{A}_{\mathrm{low}} - \mathbf{A}_{\mathrm{up}}\rVert_{\star}.
\end{equation*}
If furthermore \(\mathcal{X}\) satisfies Assumption \ref{assumption:paper:geometric of the domain} and \(\mathcal{H}\) satisfies Assumption \ref{assumption:paper:rkhs_2}, then for any finite subset \(\widehat{X}=\{X_i\}_{i=1}^n\) of \(\mathcal{X}\) and any two PSD matrices $\mathbf{A}_{\mathrm{low}}, \mathbf{A}_{\mathrm{up}} \in \mathbb{S}^{n}_{+}$, we have
\begin{align*}
    \Delta_f &\leq 2\,C_{\tilde{f}_{\mathbf{A}_{\mathrm{low}}},\tilde{f}_{\mathbf{A}_{\mathrm{up}}}} \rho_{\widehat{X}, \mathcal{X}}\\ 
    &+ \sqrt{\sum_{i=1}^{n}(\tilde{f}_{\mathbf{A}_{\mathrm{low}}}(X_i) - \tilde{f}_{\mathbf{A}_{\mathrm{up}}}(X_i))^{2}}
\end{align*}
where $C_{\tilde{f}_{\mathbf{A}_{\mathrm{low}}},\tilde{f}_{\mathbf{A}_{\mathrm{up}}}}=2d\mathrm{D}\mathrm{M}\lVert \mathbf{A}_{\mathrm{low}}-\mathbf{A}_{\mathrm{up}} \rVert_{\star}$ with \(\mathrm{D}\) and \(\mathrm{M}\) constants depending on kernel \(k\). \(\rho_{\widehat{X}, \mathcal{X}}\) is the so-called fill-in distance defined by \(\rho_{\widehat{X}, \mathcal{X}} = \sup_{x\in\mathcal{X}}\min_{X_i\in\widehat{X}}\lVert x-X_i\rVert\).
\end{proposition}

The first part of \Cref{prop:error bounds} gives an upper bound which involves the penalty on the operators, while the second part relates to the penalty on the training set.

\subsection{Hyperparameter tuning}
\label{sec:hyperparameter tuning}

In our experiments, we observe that \(\lambda_{(\cdot)1}\) and \(\lambda_{(\cdot)2}\) have small impact on the estimated bands, which confirms the findings of \citet{allain2025scalableandadaptivepredictionbandsusingkerenlsumofsquares}: we thus propose to fix their value at \(\lambda_{(\cdot)1} = \lambda_{(\cdot)2} = 1\). On the other hand, 
\(b\), \(\lambda_{\mathrm{pen}}\) and kernel lengthscales \((\theta_{\mathrm{low}}, \theta_{\mathrm{up}})\) play an important role on the shape of the prediction bands. Actually, the width penalty \(b\) is compensated by \(\theta_{\mathrm{low}}\) and \(\theta_{\mathrm{up}}\) (see Appendix \ref{sec:implementation details}), such that it is sufficient to fix \(b\) and optimize the lengthscales. We thus focus now exclusively on \((\theta_{\mathrm{low}}, \theta_{\mathrm{up}}, \lambda_{\mathrm{pen}})\).

Given that adaptivity is a critical property, it can be used as a principled objective for hyperparameter selection.
Perfectly adaptive intervals would satisfy the conditional coverage \(\mathbb{P}(Y_{N+1} \in \widehat{C}_{\mathcal{D}_{N}}(X_{N+1}) \vert X_{N+1} = x) \geq 1 - \alpha\), but such coverage is unfortunately impossible to achieve in a distribution-free setting \citep{vovk2012conditionalvalidityinductiveconformal, barber2020limitsdistributionfreeconditionalpredictive}. Alternatively, we consider a relaxed version of conditional coverage \(p_{\mathcal{D}_{N}} = \mathbb{P}(Y_{N+1} \in \widehat{C}_{\mathcal{D}_{N}}(X_{N+1}) \vert \mathcal{D}_N, X_{N+1} \in \omega_{X})\), where we condition on \(X\) being in a small neighborhood \(\omega_{X} \in \mathcal{F}_{X}\) from the event space \(\mathcal{F}_{X}\) such that \(\mathbb{P}(X\in \omega_{X})\geq \delta\).
\citet{deutschmann2023adaptiveconformalregressionjackknife} recently proved that such coverage with split CP can be controlled with the mutual information between the inputs and the score function. Their bound was generalized in \citet{allain2025scalableandadaptivepredictionbandsusingkerenlsumofsquares} for normalized scores, with the Hilbert Schmidt Independence Criterion (HSIC, \citet{gretton2005measuring}) between the residuals and the width of the intervals. The dependence measure is now between one-dimensional variables and quantified with HSIC to improve numerical stability. Interestingly, we show below that it is possible to generalize further their result for asymmetric scores as in \Cref{eq:asymmetric score function}.
\begin{proposition}
\label{prop:hyperparameter tuning criterion}
    Let \(\widehat{C}_{\mathcal{D}_N}\) be prediction intervals built from a score function \(S(X, Y) = \max\bigl(l(X) - Y, Y - u(X) \bigr)\) through split CP with \(\mathcal{D}_N=\mathcal{D}_n\cup \mathcal{D}_m\). Then:
    \begin{align*}
        p_{\mathcal{D}_{N}} &\geq 1 - \alpha 
        -\frac{1}{\delta}\sqrt{1-\frac{\alpha_1}{1-\alpha_2\mathrm{HSIC}(\tilde{r}_{\mathcal{D}_n},W_{\mathcal{D}_n})}}
    \end{align*}
    where \(\tilde{r}_{\mathcal{D}_n} = \lvert Y - (\widehat{u}_{\mathcal{D}_n}(X)+\widehat{l}_{\mathcal{D}_n}(X))/2 \rvert\) are the absolute centered residuals and \(W_{\mathcal{D}_n}=\widehat{u}_{\mathcal{D}_n}(X)-\widehat{l}_{\mathcal{D}_n}(X)\) is the width of the prediction bands. \(\alpha_1\) is a constant and \(\alpha_2\) only depends on the kernel used for \(\mathrm{HSIC}\).
\end{proposition}
The proof can be found in \Cref{sec:apdx:local coverage}.
\Cref{prop:hyperparameter tuning criterion} provides an intuitive framework: a stronger dependence between the width of the intervals and the absolute centered residuals promotes neighborhood coverage, a relaxed version of conditional coverage. To target local coverage specifically, we thus propose to maximize \(\mathrm{HSIC}(\tilde{r}_{\mathcal{D}_n},W_{\mathcal{D}_n})\). This allows for the precise tuning of \((\theta_{\mathrm{low}}, \theta_{\mathrm{up}})\) according to this criterion, where HSIC is estimated with a cross-validation procedure.
To address potentially very small values of \(\mathrm{HSIC}\), \citet{allain2025scalableandadaptivepredictionbandsusingkerenlsumofsquares} advocated the use of a test of independence to determine if its value is significantly different from \(0\). If not, the simpler homoscedastic model with arbitrary large \((\theta_{\mathrm{low}}, \theta_{\mathrm{up}})\) is chosen.

The final hyperparameter \(\lambda_{\mathrm{pen}}\) determines the choice between symmetric and asymmetric prediction bands.
A natural approach is to select the level of symmetry that maximizes adaptivity by identifying \((\theta_{\mathrm{low}}, \theta_{\mathrm{up}})\) with highest \(\mathrm{HSIC}\) for each \(\lambda_{\mathrm{pen}}\). However, maximum \(\mathrm{HSIC}\) values may remain close across different \(\lambda_{\mathrm{pen}}\) values.
This raises a question similar to the homoscedastic case, where we ask if the variation in \(\mathrm{HSIC}\) is statistically significant across multiple values of \(\lambda_{\mathrm{pen}}\). To answer this, we perform a Kruskal-Wallis rank test over bootstrapped \(\mathrm{HSIC}\) values. If significant differences exist, we select \((\theta_{\mathrm{low}}, \theta_{\mathrm{up}}, \lambda_{\mathrm{pen}})\) that maximizes \(\mathrm{HSIC}\). Conversely, if the differences are not significant, we conclude that the asymmetric bands do not offer a clear advantage over symmetric ones. In that case, we default to the simpler, symmetric model: we show this phenomenon in \Cref{fig:contour plots for multiple noise distributions}. Finally, identifying the best penalty can be computationally expensive, as it requires evaluating multiple \(\lambda_{\mathrm{pen}}\). But we can take full advantage of our dual formulations, by using a warm-start approach which reduces computational cost by $\sim 65\%$ when optimizing over multiple \(\lambda_{\mathrm{pen}}\) values, making hyperparameter search practical even for $n \sim 1000$, see Appendix \ref{sec:implementation details}. \cref{alg:training} in Appendix \ref{sec:implementation details} details our full training pipeline.

\section{Experiments}
\label{sec:experiments}

We compare our method against established baselines representative of 
different approaches to CP: CQR as the standard 
asymmetric method using quantile regression \footnote{We implement CQR with random forests following the original 
paper's recommendations for tabular data. While other base learners 
(gradient boosting, neural networks) could be used, random forests provide 
a strong baseline for our dataset sizes.}, homoscedastic GP as a 
symmetric adaptive baseline, and heteroscedastic GP \citep{binois2018practical}  
as an adaptive baseline that can capture varying noise scales. We also consider the symmetric kSoS of \citet{allain2025scalableandadaptivepredictionbandsusingkerenlsumofsquares}. While additional recent methods exist (e.g., DCP, locally weighted CP), these baselines cover the key trade-offs: CQR provides asymmetric intervals but 
relies on quantile estimation which can be challenging in small samples, 
GPs provide probabilistic predictions with theoretical guarantees but 
typically assume Gaussian noise. Our method aims to combine the strengths 
of both while automatically detecting when asymmetry is beneficial.
For all experiments we set \(\alpha=0.1\) and homGP as the predictive model except for asymmetric datasets \(3-5\) where we use cubic splines.

\subsection{Synthetic datasets}
\label{sec:synthetic datasets}

\Cref{tab:synthetic datasets} details the \(4\) datasets considered here: datasets \(1\) and \(2\) (symmetric) and datasets \(3\) and \(4\) (asymmetric).

\begin{table}[htbp]
\caption{Synthetic datasets, \(n=100\), 20 repetitions. Target is $Y = \mu(X) + \sigma(X)\epsilon$, see Appendix \ref{sec:additional analytical test} for details.}
    \label{tab:synthetic datasets}
    \centering
    \resizebox{\columnwidth}{!}{
    \begin{tabular}{c l l l l}
    \toprule
    \textbf{Dataset} & \textbf{Input} $X$ & \textbf{Mean} $\mu(X)$ & \textbf{Scale} $\sigma(X)$ & \textbf{Noise} $\epsilon$ \\
    \midrule
    \(1\) & $\mathcal{U}(-1, 1)$ & $\omega(X)\cdot\mathbb{I}_{X \le c} + X\cdot\mathbb{I}_{X > c}$ & $\sqrt{0.1 + 2X^2}$ & $\mathcal{N}(0, 1)$ \\
    \(2\) & $\mathcal{N}(0, 1)$ & $0.5X$ & $|\sin(X)|$ & $\mathcal{N}(0, 1)$ \\
    \(3\) & $\mathcal{U}(-1, 1)$ & $\sin(5X)$ & $X$ & $\text{LogNormal}(1)$ \\
    \(4\) & $\mathcal{U}(-1, 1)$ & see \ref{sec:additional analytical test} & see \ref{sec:additional analytical test} & see \ref{sec:additional analytical test}\\
    \bottomrule
    \end{tabular}
    }
\end{table}

To assess the quality of prediction intervals, we rely on mean width and introduce two global measures of local coverage. The natural measure of adaptivity is local coverage, which can be estimated for synthetic datasets with conditional samples. We could then compare methods with the \emph{absolute coverage gap}, the distance to the target level \(\alpha\):  \(\mathrm{ACG} = \frac{1}{n_X} \sum_{i=1}^{n_X} \vert \hat{p}(X_i) - (1-\alpha)\vert\) where \(\hat{p}\) is an estimate of the local coverage obtained with conditional samples \(\{(X_i,\{Y_{ij}\}_{j=1}^{n_Y})\}_{i=1}^{n_X}\). However, in an asymmetric case, local coverage can be misleading as it does not account for the distribution tails: instead, we seek bands that satisfy lower and upper local coverage at \(1-\alpha/2\). This readily implies that local coverage will be at least \(1-\alpha\), but the converse does not hold. 

We thus consider low and up alternate versions of the absolute coverage gap: \(\mathrm{ACG}^{(\cdot)} = \frac{1}{n_X} \sum_{i=1}^{n_X} \vert \hat{p}^{(\cdot)}(X_i) - (1-\alpha/2)\vert\) where \(\hat{p}^{(\cdot)}\) are low/up local coverage estimations. These metrics allow to capture defects in both tails of the noise distribution, see Appendix \ref{sec:evaluation metrics} for an in-depth discussion. We denote by \(\mathrm{ACG}^{c}\) the combination of the low/up variants.

\begin{figure}[htbp]
    \centering
    \includegraphics[width=\linewidth]{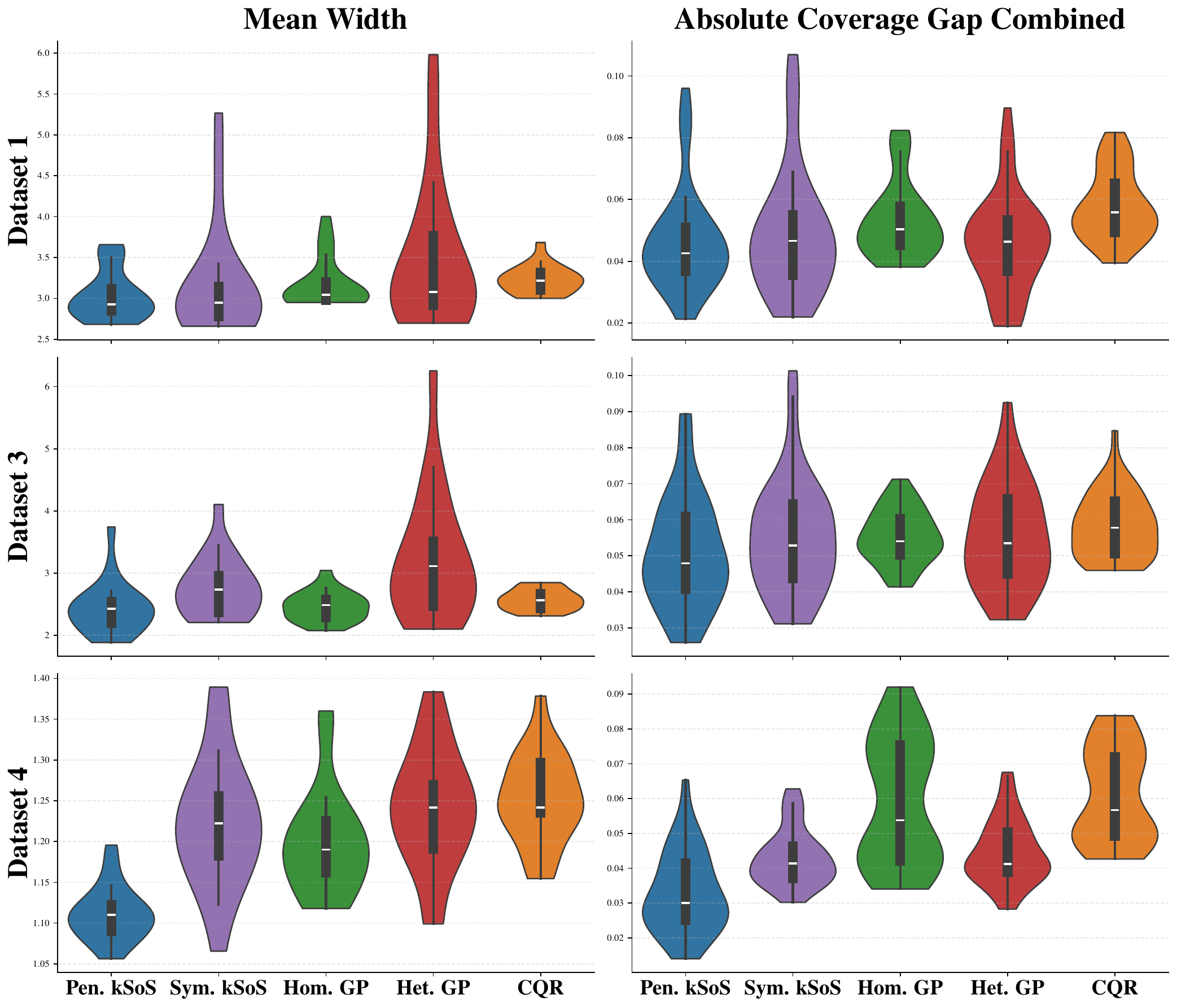}
    \caption{Mean width (left) and absolute coverage gap combined (right) for datasets \(1\), \(3\) and \(4\) with \(b=10\), \(n=100\), \(20\) repetitions.}
    \label{fig:synthetic results datasets 1 and 4}
\end{figure}

For dataset \(1\) with symmetric noise, \Cref{fig:synthetic results datasets 1 and 4} top row, shows that both CQR and homGP produce intervals with poor \(\mathrm{ACG}^{c}\). On the other hand, penalized kSoS achieves slightly lower \(\mathrm{ACG}^{c}\) than hetGP and symmetric kSoS, but with smaller mean width than hetGP and equivalent to symmetric kSoS. The superior performance of penalized kSoS over symmetric kSoS possibly comes from a small sample counterbalancing effect and the biased predictive model, see \Cref{fig:symmetric vs asymmetric choice dataset 1} for a discussion.

\begin{figure}[htbp]
    \centering
    \includegraphics[]{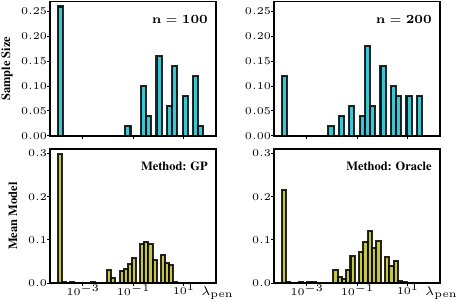}
    \caption{Histogram of selected \(\lambda_{\mathrm{pen}}\) among \(500\) repetitions. When the sample size increases from \(n=100\) to \(n=200\) (top row, dataset \(2\)), and when the predictive model changes from a Gaussian Process to the oracle (bottom row, dataset \(1\)), the purely asymmetric model is selected less often. Our hyperparameter tuning method favors a symmetric model when the sample size increases and when the learned predictive model is more accurate.}
    \label{fig:symmetric vs asymmetric choice dataset 1}
\end{figure}

\begin{figure*}[!b]
\centering
\begin{minipage}[t]{0.64\textwidth}
    \vspace{0pt}
    \centering
    \small
    \begin{tabular}{l|cccc}
        \toprule
        Dataset & CQR & Het. GP & Hom. GP & Pen. kSoS \\
        \midrule
        Concrete & $22.66 \pm 1.06$ & $21.53 \pm 1.58$ & $\mathbf{21.01} \pm 1.58$ & $21.74 \pm 1.19$ \\
        Bike & $217.98 \pm 6.54$ & $199.23 \pm 14.27$ & $169.57 \pm 7.83$ & $\mathbf{162.47} \pm 5.86$ \\
        Diabetes & $\mathbf{188.65} \pm 12.59$ & $196.25 \pm 16.01$ & $193.21 \pm 15.80$ & $\mathbf{190.46} \pm 14.63$ \\
        Housing log & $0.967 \pm 0.039$ & $0.855 \pm 0.041$ & $\mathbf{0.768} \pm 0.040$ & $0.836 \pm 0.037$ \\
        Housing & $1.77 \pm 0.055$ & $1.70 \pm 0.14$ & $\mathbf{1.60} \pm 0.082$ & $1.80 \pm 0.14$ \\
        MPG & $9.82 \pm 1.06$ & $9.76 \pm 1.31$ & $\mathbf{9.33} \pm 1.02$ & $\mathbf{9.32} \pm 1.14$ \\
        Boston & $12.16 \pm 1.18$ & $10.85 \pm 0.92$ & $\mathbf{9.76} \pm 0.80$ & $11.05 \pm 1.23$ \\
        Energy & $1.50 \pm 0.16$ & $1.46 \pm 0.098$ & $1.72 \pm 0.079$ & $\mathbf{1.39} \pm 0.12$ \\
        Miami & $31.4e^{4} \pm 1.17e^{4}$ & $36.4e^{4} \pm 8.23e^{4}$ & $29.5e^{4} \pm 2.29e^{4}$ & $\mathbf{26.9e^{4}} \pm 1.45e^{4}$ \\
        Sulfur & $\mathbf{0.051} \pm 0.003$ & $0.055 \pm 0.009$ & $\mathbf{0.050} \pm 0.003$ & $\mathbf{0.050} \pm 0.002$ \\
        Power & $13.29 \pm 0.30$ & $\mathbf{12.88} \pm 0.40$ & $13.05 \pm 0.36$ & $\mathbf{12.99} \pm 0.29$ \\
        Yacht & $0.645 \pm 0.085$ & $\mathbf{0.595} \pm 0.086$ & $0.609 \pm 0.082$ & $\mathbf{0.599} \pm 0.105$ \\
        Kin8nm & $0.607 \pm 0.020$ & $0.425 \pm 0.023$ & $\mathbf{0.386} \pm 0.009$ & $0.393 \pm 0.011$ \\
        \bottomrule
    \end{tabular}
\end{minipage}\hfill
\begin{minipage}[t]{0.32\textwidth}
    \vspace{0pt}
    \centering
    \includegraphics[scale=0.35]{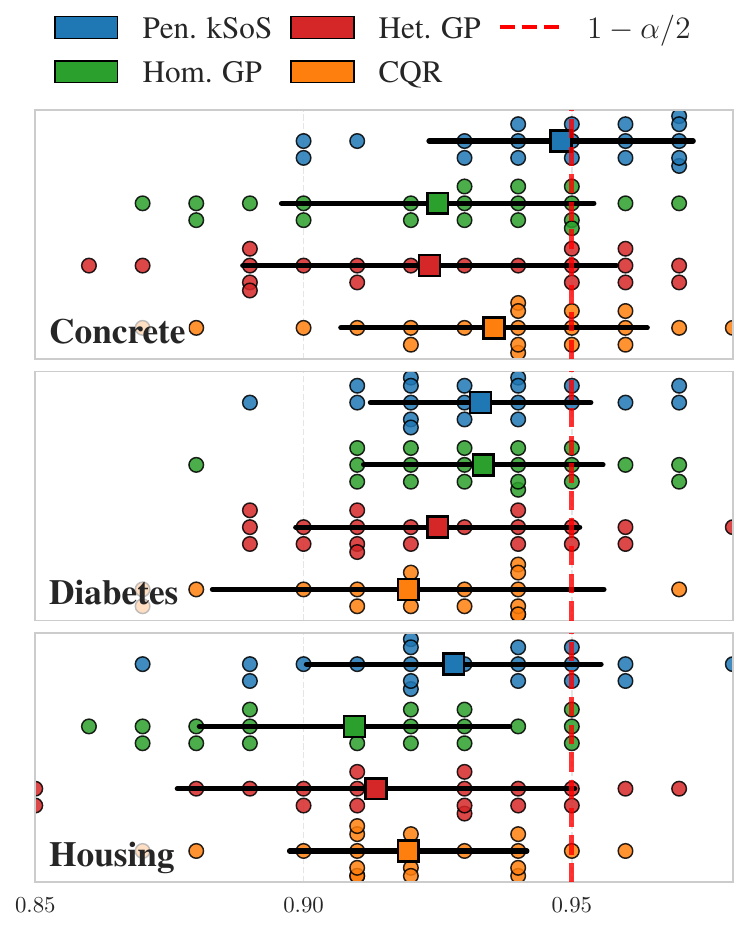}
\end{minipage}
\caption{Left: mean width of prediction intervals on the test set for thirteen real-world datasets (mean$\pm$sd on 10 repetitions, values within $1\%$ of the minimum in bold). Right: worst-set coverage low/up combined for three datasets.}\label{tab:meanwidth_realworld_main}
\end{figure*}

\newpage

For the asymmetric noise distribution in dataset \(3\), we observe in \Cref{fig:synthetic results datasets 1 and 4} middle row, that homGP and penalized kSoS are the only methods to produce tight intervals. However, homGP has much higher \(\mathrm{ACG}^{c}\) while penalized kSoS achieves the lowest \(\mathrm{ACG}^{c}\) among all methods.

Finally, for dataset \(4\), penalized kSoS has both mean width and \(\mathrm{ACG}^{c}\) much lower than all competitors. Penalized kSoS is the only method to achieve good adaptivity while maintaining small mean width in both types of noise distribution. Appendix \ref{sec:additional analytical test} contains results and discussion for additional test cases.

\subsection{Real-word datasets}
\label{sec:real world datasets}

Finally, we consider \(13\) real-word datasets commonly used for regression. To measure the performance of prediction bands we again consider mean width, but \(\mathrm{ACG}\) is now out of reach. To measure adaptivity in practice, we rely on the worst-set coverage introduced by \citet{thurin2025optimal}:  \(\min_{l=1,\ldots,L} \mathbb{P}(Y_{N+1}\in\widehat{C}_{\mathcal{D}_N}(X_{N+1}) \lvert X_{N+1}\in \mathcal{R}_l)\) where \(\{\mathcal{R}_l\}_{l=1,\ldots,L}\) is a partition of \(\mathcal{X}\). As before, we consider low/up variants and denote \(\mathrm{WSC}^{c}\) their combination.

\Cref{tab:meanwidth_realworld_main} reports the obtained mean widths and \(\mathrm{WSC}^{c}\). 
On Concrete, homGP has the smallest mean width, however only penalized kSoS achieves \(\mathrm{WSC}^{c}\) close to \(1-\alpha/2\). On Diabetes, homGP and penalized kSoS attain similar \(\mathrm{WSC}^{c}\), but penalized kSoS has much smaller mean width. On this dataset, CQR exhibits similar mean width as penalized kSoS, but underperforms for \(\mathrm{WSC}^{c}\). On Housing, penalized kSoS has a slightly larger mean width but is, by far, the closest one to  \(1-\alpha/2\) in terms of \(\mathrm{WSC}^{c}\), meaning that it has much better adaptivity. On the remaining datasets, penalized kSoS usually performs better in terms of \(\mathrm{WSC}^{c}\), and when ties occur, it reaches at least similar mean width, if not smaller. Detailed results on all datasets and comparisons with symmetric kSoS are available in Appendix \ref{sec:additional real-world datasets}. 

These findings demonstrate that penalized kSoS is always first or close second when measured by \(\mathrm{WSC}^{c}\) and mean width, a robustness never achieved by any other method.

\subsection{Low-rank approximation}
\label{sec:low rank approximation}

In our dual formulations from Proposition \ref{prop:dual formulation penalty}, the Fenchel conjugate function \(\Omega_{+}^{\star}\) requires to compute the positive part and hence the eigendecomposition of a \(n\times n\) matrix of the form \(\mathbf{V}_{(\cdot)}\mathrm{Diag}(\mathbf{\Gamma})\mathbf{V}_{(\cdot)}^{\top}\) for some dual variable \(\mathbf{\Gamma}\) depending on the penalty type and where \(\mathbf{V}_{(\cdot)}\) is the Cholesky decomposition of the kernel matrix \(\mathbf{K}_{(\cdot)}\). This means that each evaluation of the dual objective function and its gradient has complexity \(\mathcal{O}(n^3)\). To scale our method further, we now show how standard kernel low-rank approximations fit naturally into our dual framework, requiring no modification to the algorithm or theory. 

The full rank Cholesky decomposition of \(\mathbf{K}_{(\cdot)}\) can be replaced by a low-rank approximation \(\mathbf{K}_{(\cdot)} \approx \widetilde{\mathbf{V}}_{(\cdot)}^{\top}\widetilde{\mathbf{V}}_{(\cdot)}\) where \(\widetilde{\mathbf{V}}_{(\cdot)} \in \mathbb{R}^{r\times n}\) with rank \(r \ll n\). This reduces the computational complexity of \(\Omega_{+}^{\star}\) from \(\mathcal{O}(n^3)\) to \(\mathcal{O}(r^3)\). Several methods can be used to compute a low-rank kernel approximation, such as Nyström \citep{williams2000nystrom}, Random Fourier Features \citep{rahimi2007randomfeatures} or randomized SVD \citep{halko2011}.
In \cref{fig:benchmark lra}, we illustrate how our dual formulation scales when using such approximation, with a rank deduced to target a reconstruction error of the kernel matrix less than \(1\%\) (measured by the Frobenius norm). The dual formulation can be solved with \(n=10000\) samples in \(10\)s, whereas the full rank version requires more than \(100\)s for \(n=1000\).

\begin{figure}[htbp]
    \centering
    \includegraphics[scale=0.9]{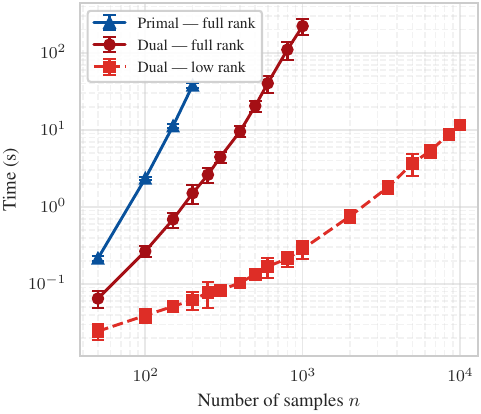}
    \caption{Dataset 1: computing time for solving primal, dual and low-rank dual  (reconstruction error \(\leq 1\%\)) formulations as a function of $n$ (mean$\pm$sd over 20 repetitions).}
    \label{fig:benchmark lra}
\end{figure}

\begin{table*}[!b]
    \centering
    \caption{\(\mathrm{WSC}^{c}\) and mean width of prediction intervals on the test set for five datasets using low-rank approximation with reconstruction error \(\leq 1\%\) (mean$\pm$sd on 10 repetitions, values within $1\%$ of best in bold).}
    \label{tab:my_label}
    \small
    {\setlength{\tabcolsep}{3pt}
    \begin{tabular}{l l cc cc cc cc}
        \toprule
        \multirow{2}{*}{\textbf{Dataset}} & \multirow{2}{*}{\textbf{Competitor}} & \multicolumn{2}{c}{$n = 1000$} & \multicolumn{2}{c}{$n = 2000$} & \multicolumn{2}{c}{$n = 5000$} & \multicolumn{2}{c}{$n = 10000$} \\
        \cmidrule(lr){3-4} \cmidrule(lr){5-6} \cmidrule(lr){7-8} \cmidrule(lr){9-10}
        & & \(\mathrm{WSC}^{c}\) & \(\mathrm{MW}\) & \(\mathrm{WSC}^{c}\) & \(\mathrm{MW}\) & \(\mathrm{WSC}^{c}\) & \(\mathrm{MW}\) & \(\mathrm{WSC}^{c}\) & \(\mathrm{MW}\) \\
        \midrule
        \multirow{2}{*}{Kin8nm} & Pen. kSoS & $\mathbf{0.92}{\scriptscriptstyle\,\pm 0.02}$ & $\mathbf{0.39}{\scriptscriptstyle\,\pm 0.01}$ & $\mathbf{0.92}{\scriptscriptstyle\,\pm 0.02}$ & $\mathbf{0.38}{\scriptscriptstyle\,\pm 0.01}$ & $\mathbf{0.93}{\scriptscriptstyle\,\pm 0.02}$ & $\mathbf{0.36}{\scriptscriptstyle\,\pm 0.01}$ & --- & --- \\
         & CQR & $0.91{\scriptscriptstyle\,\pm 0.02}$ & $0.61{\scriptscriptstyle\,\pm 0.02}$ & $0.89{\scriptscriptstyle\,\pm 0.03}$ & $0.57{\scriptscriptstyle\,\pm 0.02}$ & $0.90{\scriptscriptstyle\,\pm 0.02}$ & $0.52{\scriptscriptstyle\,\pm 0.01}$ & --- & --- \\
        \midrule
        \multirow{2}{*}{Power} & Pen. kSoS & $\mathbf{0.89}{\scriptscriptstyle\,\pm 0.03}$ & $\mathbf{13.05}{\scriptscriptstyle\,\pm 0.27}$ & $\mathbf{0.90}{\scriptscriptstyle\,\pm 0.03}$ & $\mathbf{12.66}{\scriptscriptstyle\,\pm 0.41}$ & $0.90{\scriptscriptstyle\,\pm 0.02}$ & $\mathbf{11.13}{\scriptscriptstyle\,\pm 0.58}$ & --- & --- \\
         & CQR & $0.88{\scriptscriptstyle\,\pm 0.03}$ & $13.30{\scriptscriptstyle\,\pm 0.31}$ & $0.89{\scriptscriptstyle\,\pm 0.04}$ & $\mathbf{12.63}{\scriptscriptstyle\,\pm 0.28}$ & $\mathbf{0.91}{\scriptscriptstyle\,\pm 0.02}$ & $11.60{\scriptscriptstyle\,\pm 0.45}$ & --- & --- \\
        \midrule
        \multirow{2}{*}{Dataset 3} & Pen. kSoS & $\mathbf{0.86}{\scriptscriptstyle\,\pm 0.03}$ & $\mathbf{2.04}{\scriptscriptstyle\,\pm 0.28}$ & $\mathbf{0.87}{\scriptscriptstyle\,\pm 0.03}$ & $\mathbf{2.00}{\scriptscriptstyle\,\pm 0.10}$ & $\mathbf{0.87}{\scriptscriptstyle\,\pm 0.03}$ & $\mathbf{1.97}{\scriptscriptstyle\,\pm 0.10}$ & $\mathbf{0.88}{\scriptscriptstyle\,\pm 0.02}$ & $\mathbf{1.96}{\scriptscriptstyle\,\pm 0.07}$ \\
         & CQR & $0.84{\scriptscriptstyle\,\pm 0.04}$ & $2.41{\scriptscriptstyle\,\pm 0.13}$ & $0.84{\scriptscriptstyle\,\pm 0.04}$ & $2.44{\scriptscriptstyle\,\pm 0.11}$ & $0.84{\scriptscriptstyle\,\pm 0.03}$ & $2.42{\scriptscriptstyle\,\pm 0.11}$ & $0.83{\scriptscriptstyle\,\pm 0.03}$ & $2.39{\scriptscriptstyle\,\pm 0.13}$ \\
        \midrule
        \multirow{2}{*}{Dataset 4} & Pen. kSoS & $\mathbf{0.91}{\scriptscriptstyle\,\pm 0.03}$ & $\mathbf{1.12}{\scriptscriptstyle\,\pm 0.04}$ & $\mathbf{0.91}{\scriptscriptstyle\,\pm 0.03}$ & $\mathbf{1.12}{\scriptscriptstyle\,\pm 0.03}$ & $\mathbf{0.91}{\scriptscriptstyle\,\pm 0.03}$ & $\mathbf{1.13}{\scriptscriptstyle\,\pm 0.03}$ & $\mathbf{0.91}{\scriptscriptstyle\,\pm 0.03}$ & $\mathbf{1.12}{\scriptscriptstyle\,\pm 0.03}$ \\
         & CQR & $0.88{\scriptscriptstyle\,\pm 0.08}$ & $1.25{\scriptscriptstyle\,\pm 0.06}$ & $0.89{\scriptscriptstyle\,\pm 0.07}$ & $1.22{\scriptscriptstyle\,\pm 0.06}$ & $0.90{\scriptscriptstyle\,\pm 0.08}$ & $1.24{\scriptscriptstyle\,\pm 0.04}$ & $0.89{\scriptscriptstyle\,\pm 0.07}$ & $1.24{\scriptscriptstyle\,\pm 0.06}$ \\
        \midrule
        \multirow{2}{*}{Dataset 5} & Pen. kSoS & $\mathbf{0.89}{\scriptscriptstyle\,\pm 0.05}$ & $\mathbf{3.12}{\scriptscriptstyle\,\pm 0.25}$ & $\mathbf{0.90}{\scriptscriptstyle\,\pm 0.04}$ & $\mathbf{3.38}{\scriptscriptstyle\,\pm 0.30}$ & $\mathbf{0.88}{\scriptscriptstyle\,\pm 0.04}$ & $\mathbf{3.45}{\scriptscriptstyle\,\pm 0.22}$ & $\mathbf{0.89}{\scriptscriptstyle\,\pm 0.04}$ & $\mathbf{3.39}{\scriptscriptstyle\,\pm 0.13}$ \\
         & CQR & $0.86{\scriptscriptstyle\,\pm 0.04}$ & $3.52{\scriptscriptstyle\,\pm 0.21}$ & $0.87{\scriptscriptstyle\,\pm 0.04}$ & $3.49{\scriptscriptstyle\,\pm 0.21}$ & $0.86{\scriptscriptstyle\,\pm 0.04}$ & $\mathbf{3.46}{\scriptscriptstyle\,\pm 0.26}$ & $0.86{\scriptscriptstyle\,\pm 0.04}$ & $\mathbf{3.42}{\scriptscriptstyle\,\pm 0.21}$ \\
        \bottomrule
    \end{tabular}
    }
\end{table*}

Importantly, the low-rank dual approach scales without degrading the performance of our method. \cref{tab:my_label} compares penalized kSoS with a low-rank approximation against CQR on some analytical datasets and two real-world ones where the available dataset is large enough to consider a training size \(n\) up to \(10000\). Our results show that penalized kSoS still achieves better or equivalent \(\mathrm{WSC}^{c}\) in all instances except one, and systematic better mean width.

\section{Conclusion}
\label{sec:conclusion}
In this work, we introduce a flexible framework for asymmetric conformal prediction bands using kernel sum-of-squares. By incorporating two novel symmetric penalties, our approach seamlessly transitions between symmetric and asymmetric bands. We establish representer theorems that reduce these infinite-dimensional problems to SDP formulations, and derive dual versions to ensure scalability to larger datasets. 

Beyond the core optimization, we contribute two significant tuning strategies: an HSIC-based approach for optimizing kernel lengthscales to enhance adaptivity, and a data-driven method for calibrating symmetric penalization levels. Our results demonstrate that this flexibility enhances resilience against small sample sizes, and compensates for bias in predictive models. Crucially, our experiments illustrate that while our framework can automatically detect symmetric noise distributions, allowing for controlled asymmetry can often yield superior adaptivity even when the noise itself is symmetric.

While effective, the current approach faces two primary constraints. First, consistent with standard kernel methods, performance is best suited for dimensions up to approximately 15. However, the framework remains compatible with specific kernels for structured objects such as time series or graphs, which would allow for higher-dimensional applications. Furthermore, the current approach scales efficiently to $n\sim 1000$ via dual formulations 
and warm-start optimization. For larger datasets, the framework remains 
compatible with kernel approximation techniques (Nyström, random features), but these extensions require additional theoretical analysis of their impact on bands, which is a promising direction for future work.

All our experiments can be reproduced using the code available at \href{https://gitlab.com/drti/ksos-bands-asym}{gitlab.com/drti/ksos-bands-asym}.

\section*{Impact statement}

This paper presents work whose goal is to advance the field
of Machine Learning. There are many potential societal
consequences of our work, none which we feel must be
specifically highlighted here.

\bibliography{main}
\bibliographystyle{icml2026}

\newpage
\appendix
\onecolumn

This appendix is organized as follows:
\begin{itemize}
\item \textbf{Section A: Theoretical Proofs}
\begin{itemize}
\item A.1: Asymmetric problem without penalty (\Cref{thm:representer theorem for asymmetric problem}, \Cref{prop:dual formulation for asymmetric problem})
\item  A.2: Asymmetric problem with operator penalty (\Cref{thm:representer theorem for penalty}, \Cref{prop:dual formulation penalty})
\item A.3: Asymmetric problem with training set penalty (\Cref{thm:representer theorem for penalty}, \Cref{prop:dual formulation penalty})
\item A.4: Error bounds for both penalties (\Cref{prop:error bounds})
\item A.5: HSIC-based coverage bounds (\Cref{prop:hyperparameter tuning criterion})
\end{itemize}

\item \textbf{Section B: Additional Experiments}
\begin{itemize}
\item B.1: Cross-validation and Kruskal-Wallis test details
\item B.2: Implementation details, computational complexity, warm-start
\item B.3: Extended experimental results
\begin{itemize}
\item Illustration of asymmetry and adaptivity of noise distribution
\item Evaluation metrics discussion and justification
\item Differences between operator and training set penalties
\item Comparison with symmetric prediction bands
\item Symmetric versus asymmetric calibration 
\item Additional synthetic test cases (datasets $1\text{–}5$)
\item Complete real-world dataset descriptions and results
  \end{itemize}
  \end{itemize}
    \end{itemize}

Key results we reference from the main text:
\begin{itemize}
\item For the discussion on adaptivity and asymmetry: Section \ref{sec:noise distribution}
\item For scalability: Section \ref{sec:implementation details} (warm-start), Figure \ref{fig:cold_vs_warm_start_iterations}
\item For algorithm description: Section \ref{sec:implementation details}
\item For adaptivity metric justification: Section \ref{sec:evaluation metrics} (Figure \ref{fig:mw_loccov_case_5})
\item For penalty selection patterns: Section \ref{sec:implementation details}  (Figure \ref{fig:contour_plots_evolution_with_b}), \ref{sec:additional numerical experiments} (Figure \ref{fig:hist_lambda_rw})
  \end{itemize}
  
\section{Proofs}

To begin, we first introduce some notation common to several proofs. Since we are working with kernel SoS functions depending on operators defined on different RKHSs, we use subscripts $(\cdot)$ to differentiate them. Associated to each kernel SoS function, we thus consider a RKHS $\mathcal{H}_{(\cdot)}$, a kernel $k_{(\cdot)}$, a kernel matrix $\mathbf{K}_{(\cdot)}$, a feature map $\phi_{(\cdot)}$ and a column vector
\begin{align*}
    \mathbf{k}_{(\cdot)}(X) = \left(k_{(\cdot)}(X_{1}, X), \ldots, k_{(\cdot)}(X_{n}, X)\right)^{\top}.
\end{align*}
defined for all $X \in \mathcal{X}$.

For kernel matrices $[\mathbf{K}_{(\cdot)}]_{ij}=k_{(\cdot)}(X_i,X_j)$, we further consider their Cholesky decomposition and empirical feature map
\begin{align*}
    \mathbf{K}_{(\cdot)} = \mathbf{V}^{\top}_{(\cdot)} \,\mathbf{V}_{(\cdot)} \quad \mathrm{and} \quad \bs{\Phi}_{(\cdot)}(X) = \mathbf{V}^{-T}_{(\cdot)}\mathbf{k}_{(\cdot)}(X).
\end{align*}
Depending on the theorems, $(\cdot)$ will refer to either $_\mathrm{low}$ and $_\mathrm{up}$, or will be used to enumerate a collection of kernel SoS functions $s=1,\ldots,p$.

Next, for a Hilbert space $\mathcal{H}$ we write $\mathcal{S}(\mathcal{H})$ the set of bounded Hermitian linear operators from $\mathcal{H}$ to $\mathcal{H}$ and $\mathcal{S}_{+}(\mathcal{H})$ those that are positive-definite.
We also write $\mathbb{S}^{n}= \mathbb{S}\left(\mathbb{R}^{n\times n}\right)$ the set of real, symmetric and square matrices of size $n$ and $\mathbb{S}_{+}^{n} = \mathbb{S}_{+}\left(\mathbb{R}^{n\times n}\right)$ the set of real, symmetric and positive-definite square matrices of size $n$. We also consider $m(\cdot)$ to be a fixed predictive model which has been trained separately.

In some of the following proofs, we will rely on a generalization of Theorem 1 from \citet{marteauferey2020nonparametricmodelsnonnegativefunctions} to different operators on $p$ different spaces, that we will particularize to our setting. We then start by proving this extension.

Let us consider a collection of $p$ operators
\begin{equation*}
    \mathcal{A} = \left(\mathcal{A}_1, \ldots, \mathcal{A}_p\right) \in \mathcal{K}^{C}\left(\mathcal{H}_1, \ldots, \mathcal{H}_p\right) \coloneq \mathcal{S}_{+}(\mathcal{H}_{1}) \otimes \ldots \otimes \mathcal{S}_{+}(\mathcal{H}_{p}),
\end{equation*}
and denote the following multivariate function
\begin{equation*}
    f_{\mathcal{A}}(X) = f_{\mathcal{A}_1, \ldots, \mathcal{A}_p}(X) = \left(f_{\mathcal{A}_1}(X), \ldots, f_{\mathcal{A}_p}(X)\right) \in \mathbb{R}_{+}^{p}, \; \forall X \in \mathcal{X}
\end{equation*}
where \(f_{\mathcal{A}_s}(X) = \langle\phi_{s}(X), \mathcal{A}_{s}\phi_{s}(X)\rangle_{\mathcal{H}_{s}}\), \(s=1,\ldots,p\). 
We also introduce below a specific class of regularizers.
\begin{assumption}
\label{assumption:regularizer}
    Let $\mathcal{H}$ be a Hilbert space, for any $\mathcal{A} \in \mathcal{S}(\mathcal{H})$, $\Omega(\mathcal{A})$ is of the form
    \begin{equation}
        \Omega(\mathcal{A}) = 
        \begin{cases}
            \mathrm{Tr}(q(\mathcal{A})) = \sum_k q(\sigma_k) \quad &\mathrm{if} \; \mathcal{A} = U\mathrm{Diag}(\sigma)U^{\top} \in \mathcal{S}_{\infty}(\mathcal{H}), \; \sum_k q(\sigma_k) < +\infty\\
            + \infty &\mathrm{otherwise},
        \end{cases}
    \end{equation}
    where $q\colon \mathbb{R} \rightarrow \mathbb{R}_{+}$ is:
    \begin{itemize}
        \item non decreasing on $\mathbb{R}_{+}$ with $q(0)=0$
        \item lower semi-continuous
        \item $q(\sigma) \underset{|\sigma|\rightarrow +\infty}{\longrightarrow} +\infty.$
    \end{itemize}
\end{assumption}
We can define for each operator $\mathcal{A}_s$ a penalty function $\Omega_{s}$, and consider the aggregated penalty function
\begin{equation}
\label{eq:penalty_functions_1}
    \Omega_{\mathrm{agg}}(\mathcal{A}) = \sum_{s=1}^{p}\Omega_{s}(\mathcal{A}_s).
\end{equation}

The following theorem is the main result and shows that \Cref{eq:inf_problem_p_operator} admits a solution that has a finite dimensional representation.

\begin{theorem}
\label{thm:generalization_mf_p_operators}
    Let $L$ be a lower semi-continuous and bounded below function and let $\Omega_{\mathrm{agg}}$ be defined as in \Cref{eq:penalty_functions_1} where each $\Omega_{s}, s=1,\ldots,p$ satisfies Assumption \ref{assumption:regularizer}. The problem
    \begin{equation}
    \label{eq:inf_problem_p_operator}
        \underset{
            \mathcal{A} \in \mathcal{K}^{C}\left(\mathcal{H}_1, \ldots, \mathcal{H}_p\right)
        }{\inf}
        L\left(f_{\mathcal{A}}(X_1), \ldots, f_{\mathcal{A}}(X_n)\right) + \Omega_{\mathrm{agg}}(\mathcal{A})
    \end{equation}
    admits a solution $\mathcal{A}^{\star} = \Bigl(\sum_{i,j=1}^{n}\mathbf{B}_{sij}\phi_{s}(X_i)\phi_{s}(X_j)^{\top}\Bigr)_{1\leq s\leq p}$, such that $f_{\mathcal{A}^{\star}}$ can be written as 
    \begin{equation}
    \label{eq:solution_f}
        f_{\mathcal{A}^{\star}}(X) = \Bigl(\sum_{i,j=1}^{n}\mathbf{B}_{sij}\phi_{s}(X_i)\phi_{s}(X_j)^{\top}\Bigr)_{1\leq s\leq p}
    \end{equation}
    for $p$ matrices $\mathbf{B}_{s} \in \mathbb{R}^{n\times n}$, $\mathbf{B}_{s} \succeq 0$.
\end{theorem}
\begin{proof}
    We follow the proof of \citet{marteauferey2020nonparametricmodelsnonnegativefunctions}, Section B.3. The first step is to prove the following lemma, which generalizes Lemma 2 of \citet{marteauferey2020nonparametricmodelsnonnegativefunctions}.

    \begin{lemma}\label{lem:Omega_properties}
        The function $\Omega_{\mathrm{agg}}$ in \Cref{eq:penalty_functions_1} satisfies the following properties:
        \begin{enumerate}
            \item For any collection of Hilbert spaces $\left(\mathcal{H}_{1}^{1}, \ldots, \mathcal{H}_{p}^{1}\right)$ and $\left(\mathcal{H}_{1}^{2}, \ldots, \mathcal{H}_{p}^{2}\right)$ and any linear isometries $O_{s}\colon \mathcal{H}_{s}^{1} \rightarrow \mathcal{H}_{s}^{2}$, $s=1,\ldots,p$, it holds that for all $\mathcal{A}_{1}, \ldots, \mathcal{A}_p \in \mathcal{K}^{C}\left(\mathcal{H}_{1}^{1}, \ldots, \mathcal{H}_{p}^{1}\right)$:
            \begin{equation*}
                \sum_{s=1}^{p}\Omega_{s}(O_{s}\mathcal{A}_{s}O_{s}^{*}) = \sum_{s=1}^{p}\Omega_{s}(\mathcal{A}_{s}).
            \end{equation*}
            \item For any collection of Hilbert spaces $\left(\mathcal{H}_{1}, \ldots, \mathcal{H}_{p}\right)$ and any orthogonal projections $\Pi_{s} \in \mathcal{S}(\mathcal{H}_{s})$, $s=1,\ldots,p$, it holds that for all $\mathcal{A}_{1}, \ldots, \mathcal{A}_p \in \mathcal{K}^{C}\left(\mathcal{H}_{1}, \ldots, \mathcal{H}_{p}\right)$:
            \begin{equation*}
                \sum_{s=1}^{p}\Omega_{s}(\Pi_{s}\mathcal{A}_{s}\Pi_{s}) \leq \sum_{s=1}^{p}\Omega_{s}(\mathcal{A}_{s})
            \end{equation*}
            \item For any collection of finite dimensional Hilbert spaces $\left(\mathcal{H}_{n1}, \ldots, \mathcal{H}_{np}\right)$, it holds that:
            \begin{equation}
                \Omega_{\mathrm{agg}} \ \mathrm{is\ lower\ semi-continuous} \quad \mathrm{and} \quad \Omega_{\mathrm{agg}}(\mathcal{A})= \sum_{s=1}^{p}\Omega_{s}(\mathcal{A}_{s}) \underset{\sup_s \|\mathcal{A}_{s}\|_{\mathrm{op}} \rightarrow +\infty}{\longrightarrow} + \infty
            \end{equation}
        \end{enumerate}
    \end{lemma}
    \begin{proof}
        $\Omega_{\mathrm{agg}}$ is defined as a separable sum of $\Omega_{s}, s=1,\ldots,p$ which all satisfy Lemma $2$ of \citet{marteauferey2020nonparametricmodelsnonnegativefunctions} under Assumption \ref{assumption:regularizer}. Properties 1, 2 and continuity in 3 are thus straightforward by applying equalities and inequalities element-wise on each $\Omega_{s}, s=1,\ldots,p$. The last part of Property 3 comes from Lemma $2$, (iii) from \citet{marteauferey2020nonparametricmodelsnonnegativefunctions} applied to $\Omega_{s^*}(\mathcal{A}_{s^*})$ for $s^*$ corresponding to $\sup_s \|\mathcal{A}_{s}\|_{\mathrm{op}}$.
    \end{proof}

    Now, the first part of \Cref{thm:generalization_mf_p_operators} comes from the following proposition, which shows that the solution of the infinite-dimensional problem of \Cref{eq:inf_problem_p_operator} with $p$ operators and penalty function in \Cref{eq:penalty_functions_1} lies in a space of finite dimension. For every Hilbert space $\mathcal{H}_s$, we write $\mathcal{H}_{ns}$ for the finite dimensional subspace of $\mathcal{H}_s$ generated by $\{(\phi_{s}(X_i))_{1\leq i\leq n}\}$, and define $\Pi_{ns}$ the orthogonal projection on $\mathcal{H}_{ns}$ such that:
    \begin{equation*}
        \Pi_{ns} \in \mathcal{S}(\mathcal{H}_s), \quad \Pi_{ns}^{2}=\Pi_{ns},\quad \mathrm{range}(\Pi_{ns}) = \mathcal{H}_{ns}.
    \end{equation*}
    Finally, we denote
    \begin{equation*}
        \mathcal{K}_{n}^{C}(\mathcal{H}_1, \ldots, \mathcal{H}_p) \coloneq \left\{\left(\Pi_{n1}\mathcal{A}_1\Pi_{n1}, \ldots, \Pi_{np}\mathcal{A}_p\Pi_{np}\right)\colon \; \left(\mathcal{A}_1, \ldots, \mathcal{A}_p\right) \in \mathcal{K}^{C}\left(\mathcal{H}_1, \ldots, \mathcal{H}_p\right) \right\}.
    \end{equation*}
    Since $\forall 1\leq s\leq p, \; \Pi_{ns}\mathcal{A}_{s}\Pi_{ns} \in \mathcal{S}(\mathcal{H}_{s})$ and $\Pi_{ns}\mathcal{A}_{s}\Pi_{ns} \geq 0$, we have the inclusion:
    \begin{equation}
    \label{eq:inclusion_Kn_in_K}
        \mathcal{K}_{n}^{C}(\mathcal{H}_1, \ldots, \mathcal{H}_p) \subset \mathcal{K}^{C}(\mathcal{H}_1, \ldots, \mathcal{H}_p).
    \end{equation}

    \begin{proposition}
    \label{prop:generalization_mf_p_operators}
        Let $L$ be a lower semi-continuous function, bounded below and let $\Omega_{\mathrm{agg}}$ be defined as in \Cref{eq:penalty_functions_1} where each $\Omega_{s}, s=1,\ldots,p$ satisfies Assumption \ref{assumption:regularizer}. The problem
        \begin{equation*}
            \underset{
                \mathcal{A} \in \mathcal{K}^{C}\left(\mathcal{H}_1, \ldots, \mathcal{H}_p\right)
            }{\inf}
            L\left(f_{\mathcal{A}}(X_1), \ldots, f_{\mathcal{A}}(X_n)\right) + \Omega_{\mathrm{agg}}(\mathcal{A})
        \end{equation*}
        admits a solution $\mathcal{A}^{\star} \in \mathcal{K}_{n}^{C}(\mathcal{H}_1, \ldots, \mathcal{H}_p)$.
    \end{proposition}
    \begin{proof}
        Let us first write for any $\mathcal{A} \in \bigotimes_{s=1}^{p}\mathcal{S}(\mathcal{H}_{s}), \; J(\mathcal{A})\coloneq L\left((f_{\mathcal{A}}(X_i))_{1\leq i\leq n}) + \Omega_{\mathrm{agg}}(\mathcal{A}\right)$.
        The goal here is to show that $\underset{\mathcal{A} \in \mathcal{K}^{C}\left(\mathcal{H}_1, \ldots, \mathcal{H}_p\right)}{\inf} J(\mathcal{A})$ admits a solution in $\mathcal{K}_{n}^{C}\left(\mathcal{H}_1, \ldots, \mathcal{H}_p\right)$. The proof consists of three main steps:
        \begin{itemize}
            \item we first show that 
            \begin{equation}
            \label{eq:representer theorem step 1}
                \underset{\mathcal{A} \in \mathcal{K}_{n}^{C}\left(\mathcal{H}_1, \ldots, \mathcal{H}_p\right)}{\inf} J(\mathcal{A}) = \underset{\mathcal{A} \in \mathcal{K}^{C}\left(\mathcal{H}_1, \ldots, \mathcal{H}_p\right)}{\inf}J(\mathcal{A})
            \end{equation}
            \item then, we show that if $\underset{\mathcal{A} \in \mathcal{K}_{n}^{C}\left(\mathcal{H}_1, \ldots, \mathcal{H}_p\right)}{\inf} J(\mathcal{A})$ exists, it is attained for operators with bounded Frobenius norm
            \item finally, we show that such minimum exists.
        \end{itemize}

        \paragraph{First step.}
        This is a direct generalization of the proof of Proposition $7$ in \citet{marteauferey2020nonparametricmodelsnonnegativefunctions}.
        Using the projections $\Pi_{ns}, s=1, \ldots, p$, it is easy to show that for all $1\leq i\leq n$, $f_{\mathcal{A}}(X_i) = f_{\left(\Pi_{ns}\mathcal{A}_{s}\Pi_{ns}\right)_{1\leq s\leq p}}(X_i)$ and that $\Omega_{\mathrm{agg}}(\mathcal{A}) = \Omega_{\mathrm{agg}}\left(\left(\Pi_{ns}\mathcal{A}_{s}\Pi_{ns}\right)_{1\leq s\leq p}\right)$ (using Property $2$ from \Cref{lem:Omega_properties}). Putting those results together and using the inclusion in \Cref{eq:inclusion_Kn_in_K}, the first step follows.

        \paragraph{Second step.}
        We rigorously mirror the proof of Proposition $7$ in \citet{marteauferey2020nonparametricmodelsnonnegativefunctions}. The key idea is to show that we can replace $\mathcal{K}_{n}^{C}\left(\mathcal{H}_1, \ldots, \mathcal{H}_p\right)$ with 
        \begin{equation}
        \label{eq:refined_search_space}
            \mathcal{K}_{R_0} \coloneq \left\{\tilde{\mathcal{A}} \in \mathcal{K}^{C}\left(\mathcal{H}_{n1}, \ldots, \mathcal{H}_{np}\right)\colon \, \forall 1\leq s\leq p, \|\tilde{\mathcal{A}_{s}}\|_{F} \leq R_{0}\right\}
        \end{equation}
        where $R_0$ is a constant. First, by defining for all $1\leq s\leq p$ the injections $V_{ns}\colon \mathcal{H}_{ns} \rightarrow \mathcal{H}_{s}$ (s.t. $V_{ns}V_{ns}^{*} = \Pi_{ns}$ and $V_{ns}^{*}V_{ns} = I_{H_{ns}}$), we have:  
        \begin{equation*}
            \underset{\mathcal{A} \in \mathcal{K}_{n}^{C}\left(\mathcal{H}_1, \ldots, \mathcal{H}_p\right)}{\inf} J(\mathcal{A}) = \underset{\tilde{\mathcal{A}_{1}}, \ldots, \tilde{\mathcal{A}_{p}} \in \mathcal{K}^{C}\left(\mathcal{H}_{n1}, \ldots, \mathcal{H}_{np}\right)}{\inf} J\left((V_{ns}\tilde{\mathcal{A}}_{s}V_{ns}^{*})_{1\leq s\leq p}\right).
        \end{equation*}
        Next, using Property $1$ from \Cref{lem:Omega_properties} it is immediate to show that for any $\tilde{\mathcal{A}_{1}}, \ldots, \tilde{\mathcal{A}_{p}} \in \mathcal{K}^{C}\left(\mathcal{H}_{n1}, \ldots, \mathcal{H}_{np}\right)$,
        \begin{equation}
        \label{eq:equality_of_omega_with_injection}
            J\left((V_{ns}\tilde{\mathcal{A}}_{s}V_{ns}^{*})_{1\leq s\leq p}\right) = L\left(f_{(V_{ns}\tilde{\mathcal{A}}_{s}V_{ns}^{*})_{1\leq s\leq p}}(X_i)\right) + \sum_{s=1}^{p}\Omega_{s}(\tilde{\mathcal{A}}_{s}).
        \end{equation}
        Now, let $(\tilde{\mathcal{A}}_{0s})_{1\leq s\leq p} \in \mathcal{K^{C}}(\mathcal{H}_{n1}, \ldots ,\mathcal{H}_{np})$ be a point such that $J_0 \coloneq J\left((\tilde{\mathcal{A}}_{0s})_{1\leq s\leq p}\right) < +\infty$, and let $c_0$ be a lower bound for $L$. By Property $3$ of \Cref{lem:Omega_properties}, there exists a radius $R_0$ such that $\forall (\tilde{\mathcal{A}}_{s})_{1\leq s\leq p} \in \mathcal{K^{C}}(\mathcal{H}_{n1}, \ldots ,\mathcal{H}_{np})$, $\exists s_0\in[p]$ satisfying:
        \begin{equation*}
        \label{eq:Omega_bigger_than_constant}
            \|\tilde{\mathcal{A}}_{s_0}\|_{F} > R_0 \Longrightarrow \sum_{s=1}^{p}\Omega_{s}(\tilde{\mathcal{A}}_{s}) > J_0 - c_0.
        \end{equation*}
        This means that the infimum of \Cref{eq:equality_of_omega_with_injection}, if it exists, lies in the space of operators with bounded Frobenius norm. Finally, $c_0$ being a lower bound for $L$, this implies that
        \begin{equation*}
            \underset{\tilde{\mathcal{A}_{1}}, \ldots, \tilde{\mathcal{A}_{p}} \in \mathcal{K}^{C}\left(\mathcal{H}_{n1}, \ldots, \mathcal{H}_{np}\right)}{\inf} J\left((V_{ns}\tilde{\mathcal{A}}_{s}V_{ns}^{*})_{1\leq s\leq p}\right) = \underset{
            \substack{
            \tilde{\mathcal{A}_{1}}, \ldots, \tilde{\mathcal{A}_{p}} \in \mathcal{K}^{C}\left(\mathcal{H}_{n1}, \ldots, \mathcal{H}_{np}\right)\\
            \forall s, \; \|\tilde{\mathcal{A}_{s}}\| \leq R_0
            }
            }{\inf} J\left((V_{ns}\tilde{\mathcal{A}}_{s}V_{ns}^{*})_{1\leq s\leq p}\right).
        \end{equation*}

        \paragraph{Third step.}
        We finally show that our problem admits a solution, because we minimize a lower semi-continuous function on a non empty compact space. Observe first that $\mathcal{K}_{R_0}$ is a non empty compact space. This follows directly from the fact that $\mathcal{H}_{n1}, \ldots, \mathcal{H}_{np}$ are finite dimensional: the space $\mathcal{K}_{R_0}$ is compact (closed and bounded) and non empty (it contains $(\tilde{\mathcal{A}}_{0s})_{1\leq s\leq p}$). Next, the function 
        \begin{equation}
        \label{eq:function_to_minimize}
            (\tilde{\mathcal{A}}_{s})_{1\leq s\leq p} \longrightarrow J\left((V_{ns}\tilde{\mathcal{A}}_{s}V_{ns}^{*})_{1\leq s\leq p}\right)
        \end{equation}
        is lower semi-continuous, as a composition of the linear semi-continuous function $L$ with the linear (thus continuous) function $(\tilde{\mathcal{A}}_{s})_{1\leq s\leq p} \longrightarrow (f_{(V_{ns}\tilde{\mathcal{A}}_{s}V_{ns}^{*})_{1\leq s\leq p}})_{1\leq i\leq n}$, plus the linear semi-continuous penalty function $\Omega$ (Assumption \ref{assumption:regularizer}).

        Thus, the function in \Cref{eq:function_to_minimize} admits a minimum on any non empty compact set, and in particular $\mathcal{K}_{R_0}$:
        \begin{equation*}
            \exists \tilde{\mathcal{A}}^{\star} \in \mathcal{K}_{R_0}\colon \; J\left((V_{ns}\tilde{\mathcal{A}}_{s}^{\star}V_{ns}^{*})_{1\leq s\leq p}\right) =  \underset{
            \substack{
            \tilde{\mathcal{A}_{1}}, \ldots, \tilde{\mathcal{A}_{p}} \in \mathcal{K}^{C}\left(\mathcal{H}_{n1}, \ldots, \mathcal{H}_{np}\right)\\
            \forall s, \; \|\tilde{\mathcal{A}_{s}}\| \leq R_0
            }
            }{\inf} J\left((V_{ns}\tilde{\mathcal{A}}_{s}V_{ns}^{*})_{1\leq s\leq p}\right).
        \end{equation*}
        By going back up to the previous equalities, we have that
        \begin{equation*}
            J(\mathcal{A}^{\star}) = \underset{\mathcal{A} \in \mathcal{K}^{C}\left(\mathcal{H}_1, \ldots, \mathcal{H}_p\right)}{\inf} J(\mathcal{A})
        \end{equation*}
        with $\mathcal{A}^{\star}=(V_{ns}\tilde{\mathcal{A}}_{s}V_{ns}^{*})_{1\leq s\leq p}$, which completes the proof of \Cref{prop:generalization_mf_p_operators}.
    \end{proof}

    The last part is to show that $\mathcal{A}^{\star} \in \mathcal{K}_{n}^{C}(\mathcal{H}_{1}, \ldots, \mathcal{H}_{p})$ the solution from \Cref{prop:generalization_mf_p_operators},  leads to a function $f$ as in \Cref{eq:solution_f}. This is achieved with the following lemma.

    \begin{lemma}
    \label{lem:finite_representation_as_matrices}
        The set $\mathcal{K}_{n}^{C}(\mathcal{H}_{1}, \ldots, \mathcal{H}_{p})$ can be represented as
        \begin{equation*}
            \mathcal{K}_{n}^{C}(\mathcal{H}_{1}, \ldots, \mathcal{H}_{p})  = \left\{ \Bigl(\sum_{i,j=1}^{n}\mathbf{B}_{sij}\phi_{s}(X_i)\phi_{s}(X_j)^{\top}\Bigr)_{1\leq s\leq p} \; \colon \; \forall 1\leq s\leq p, \; \mathbf{B}_{s} \in \mathbb{R}^{n\times n}, \mathbf{B}_{s}\succeq 0\right\}.
        \end{equation*}
        In particular, for any $\mathcal{A} \in \mathcal{K}_{n}^{C}(\mathcal{H}_{1}, \ldots, \mathcal{H}_{p})$, there exists $p$ matrices $\mathbf{B}_{s} \in \mathbb{R}^{n\times n}$, $\mathbf{B}_{s} \succeq 0$ such that
        \begin{equation*}
            \mathcal{A} = \Bigl(\sum_{i,j=1}^{n}\mathbf{B}_{sij}\phi_{s}(X_i)\phi_{s}(X_j)^{\top}\Bigr)_{1\leq s\leq p} \Longrightarrow \forall X \in \mathcal{X}, f_{\mathcal{A}}(X) = \Bigl(\sum_{i,j=1}^{n}\mathbf{B}_{sij}k_{s}(X_i, X)k_{s}(X_j, X)\Bigr)_{1\leq s\leq p}.
        \end{equation*}
    \end{lemma}
    \begin{proof}
        We apply Lemma 3 of \citet{marteauferey2020nonparametricmodelsnonnegativefunctions} to each individual operators $\mathcal{A}_{s} \in  \mathcal{S}_{+}(\mathcal{H}_s)$ to retrieve the result on the collection of operators $(A_{s})_{1\leq s\leq p}$.
    \end{proof}
    Applying \Cref{lem:finite_representation_as_matrices} to the solution $\mathcal{A}^{\star} \in \mathcal{K}_{n}^{C}(\mathcal{H}_{1}, \ldots, \mathcal{H}_{p})$ concludes the proof of \Cref{thm:generalization_mf_p_operators}.
\end{proof}

\Cref{thm:generalization_mf_p_operators} provides a finite-dimensional equivalent problem which involves unknown PSD matrices $\mathbf{B}_1,\ldots,\mathbf{B}_p$. \citet{marteauferey2020nonparametricmodelsnonnegativefunctions} and \citet{allain2025scalableandadaptivepredictionbandsusingkerenlsumofsquares} also propose an equivalent formulation with a different parameterization in terms of PSD matrices $\mathbf{A}_1,\ldots,\mathbf{A}_p$ based on the Cholesky decomposition of the kernels. More precisely, define 
\begin{equation}
\label{eq:f_tilde}
    \tilde{f}_{\mathbf{A}_{s}}(X)=\bs{\Phi}_s(x)^\top \mathbf{A}_s \, \bs{\Phi}_s(x)
\end{equation}
for $s=1,\ldots,p$. With these notations, the following proposition shows that we obtain the same solution if we optimize the PSD matrices $\mathbf{A}_1,\ldots,\mathbf{A}_p$ instead of $\mathbf{B}_1,\ldots,\mathbf{B}_p$.

\begin{proposition}
\label{prop:formulation A p operators}
    Under the assumptions of \Cref{thm:generalization_mf_p_operators}, the following problem has at least one solution, which is unique if for all $1\leq s \leq p$,  $\lambda_{s2} > 0$ and $L$ is convex:
    \begin{equation}
    \label{eq:thm rep formulation A}
        \inf_{\mathbf{A}_{1}, \ldots, \mathbf{A}_{p} \in \mathbb{S}_{+}^{n}}L\left(\tilde{f}_{(\mathbf{A}_{s})_{1\leq s\leq p}}(X_i)\right) + \sum_{s=1}^{p}\Omega_{s}(\mathbf{A}_{s}).
    \end{equation}
Moreover, for any given solution $\mathbf{A}_{1}^{\star}, \ldots, \mathbf{A}_{p}^{\star} \in \mathbb{S}_{+}^{n}$ of \Cref{eq:thm rep formulation A}, the function $\tilde{f}_{(\mathbf{A}_{s}^{\star})_{1\leq s\leq p}}$ is a minimizer of \Cref{eq:inf_problem_p_operator}.
\end{proposition}
\begin{proof}
For each \(s \in \{1, \dots, p\}\), define the operator
\begin{equation*}
    S_{ns} : \mathcal{H}_s \to \mathbb{R}^n, 
\quad 
S_{ns}(h) = \big( \langle h, \phi_s(x_i) \rangle_{\mathcal{H}_s} \big)_{1 \le i \le n}
\end{equation*}
and its adjoint \(S_{ns}^* : \mathbb{R}^n \to \mathcal{H}_s\) is then given by
\begin{equation*}
    S_{ns}^{*} \alpha = \sum_{i=1}^n \alpha_i \, \phi_s(x_i),
\quad \forall \alpha \in \mathbb{R}^n.
\end{equation*}
It follows that $\mathbf{K}_s = S_{ns} S_{ns}^{*}$.
Then, for each \(s \in \{1, \dots, p\}\), we define
\begin{equation*}
    O_{ns} : \mathbb{R}^{n} \to \mathcal{H}_s, \quad O_{ns} = S_{ns}^* \mathbf{V}_s^\top \big( \mathbf{V}_s \mathbf{V}_s^\top \big)^{-1},
\end{equation*}
where we assume that $\mathbf{K}_{s}$ is full rank (this is the case when using a universal kernel as the Matérn one, and if all training points $X_i$ are distinct). Further note that all $O_{ns}$ satisfy Lemma $4$ from \citet{marteauferey2020nonparametricmodelsnonnegativefunctions}.

Now, for each \(s \in \{1, \dots, p\}\), define:
\begin{equation*}
    \tilde{f}_{\mathbf{A}_s}(x) = \bs{\Phi}_s(x)^\top \mathbf{A}_s \, \bs{\Phi}_s(x),
\quad 
\mathbf{A}_s \in \mathbb{R}^{n \times n}, 
\quad 
\mathbf{A}_s \succeq 0,
\end{equation*}
where the feature map $\bs{\Phi}_s : \mathcal{X} \to \mathbb{R}^{n}$ is defined as $\bs{\Phi}_s(x) = O_{ns}^* \, \phi_s(x)$, such that 
\begin{equation}
\label{eq:equality f tilde f}
    f_{O_{ns} \mathbf{A}_s O_{ns}^*} = \tilde{f}_{\mathbf{A}_s}.
\end{equation}

Moreover, as in \citet{marteauferey2020nonparametricmodelsnonnegativefunctions} proof of Proposition $3$, Equation (b),
\begin{equation}
\label{eq:equality of space}
    \{\, O_{ns} \mathbf{A}_s O_{ns}^* : \mathbf{A}_s \in \mathbb{R}^{n \times n}, \; \mathbf{A}_s \succeq 0 \,\}
    = \mathcal{S}_{n}(\mathcal{H}_s)_+
\end{equation}
where $\mathcal{S}_{n}(\mathcal{H}_s)_+$ corresponds to $\mathcal{K}_{n}^{C}(\mathcal{H}_1, \ldots, \mathcal{H}_p) = \bigotimes_{s=1}^{p} \mathcal{S}_{n}(\mathcal{H}_s)_+$.

\noindent
Finally, since \(O_{ns}\) is an isometry, we have
\begin{equation}
\label{eq:equality omega with O}
    \Omega_s(O_{ns} \mathbf{A}_s O_{ns}^*) = \Omega_s(\mathbf{A}_s),
\end{equation}
and using Equations \ref{eq:equality f tilde f} and \ref{eq:equality omega with O} it holds that
\begin{equation*}
J\left((O_{ns}\mathbf{A}_{s}O_{ns}^{*})_{1\leq s\leq p}\right) 
= J\left((\mathbf{A}_{s})_{1\leq s\leq p}\right).
\end{equation*}

Finally, it follows from \Cref{eq:equality of space} that
\begin{equation*}
    \inf_{\mathbf{A}_{1}, \ldots, \mathbf{A}_{p} \in \mathbb{S}_{+}^{n}}J\left((\mathbf{A}_{s})_{1\leq s\leq p}\right)= \inf_{\mathcal{B}_1, \ldots, \mathcal{B}_{p}\in \mathcal{K}_{n}^{C}(\mathcal{H}_{1}, \ldots, \mathcal{H}_{p})} J\left((\mathcal{B}_{s})_{1\leq s\leq p}\right) =\underset{\mathcal{B} \in \mathcal{K}^{C}\left(\mathcal{H}_1, \ldots, \mathcal{H}_p\right)}{\inf}J(\mathcal{B})
\end{equation*}
where the last equality comes from \Cref{eq:representer theorem step 1}.
This shows that the solution of the left hand side problem (i.e. \Cref{eq:thm rep formulation A}) is a solution of the right hand side problem (i.e. \Cref{eq:inf_problem_p_operator}), which concludes the proof.
\end{proof}

\subsection{Asymmetric problem with different kernels}
\label{sec:apdx:asymmetric problem with different kernels}

For any linear operator $\mathcal{A}\in\mathcal{S}(\mathcal{H})$, let us first define the penalty function
\begin{equation}
\label{eq:regularizer function on one operator}
    \Omega_{(\cdot)}(\mathcal{A}) = \lambda_{(\cdot)1}\lVert\mathcal{A}\rVert_{\star}+\lambda_{(\cdot)2}\lVert\mathcal{A}\rVert_{F}^{2}
\end{equation}
with Fenchel conjugate $\Omega^{\star}_{(\cdot)}$, and the residuals
\begin{equation*}
    r_{(\cdot)}\left(X_i, Y_i\right) = 
    \begin{cases}
        m(X_i) - Y_i & \mathrm{if} \; {(\cdot)} \ \mathrm{is} \ {\mathrm{low}}\\
        Y_i - m(X_i) & \mathrm{if} \; {(\cdot)} \ \mathrm{is} \ {\mathrm{up}}.
    \end{cases}
\end{equation*}
Note that \(\Omega_{(\cdot)}(\cdot)\) satisfies Assumption \ref{assumption:regularizer} as long as \(\lambda_{(\cdot)1}>0\) or \(\lambda_{(\cdot)2}>0\).

For the asymmetric setting without penalty, our infinite dimensional problem writes:
\begin{align}
    \underset{
        \substack{
            \mathcal{A}_{\mathrm{low}} \in \mathcal{S}_{+}\left(\mathcal{H}_{\mathrm{low}}\right) \\
            \mathcal{A}_{\mathrm{up}} \in \mathcal{S}_{+}\left(\mathcal{H}_{\mathrm{up}}\right)
        }
    }{\inf}
    \quad& \frac{b}{n}\sum_{i=1}^{n} \left(f_{\mathcal{A_{\mathrm{low}}}}(X_{i}) + f_{\mathcal{A_{\mathrm{up}}}}(X_{i})\right) + \Omega_{\mathrm{low}}(\mathcal{A}_{\mathrm{low}}) + \Omega_{\mathrm{up}}(\mathcal{A_{\mathrm{up}}})\nonumber\\
    \mathrm{s.t.} \quad& r_{\mathrm{low}}\left(X_i, Y_i\right) - f_{\mathcal{A}_{\mathrm{low}}}(X_{i})\leq 0, \;i \in \left[n\right] \label{eq:eq:infdim_1}\\
    \quad& r_{\mathrm{up}}\left(X_i, Y_i\right) - f_{\mathcal{A_{\mathrm{up}}}}(X_{i})\leq 0, \;i \in \left[n\right].\nonumber
\end{align}

Since this problem is separable in both objectives and constraints, it is sufficient to prove a representer theorem for each problem independently:
\begin{align}
    \underset{
        \mathcal{A}_{(\cdot)} \in \mathcal{S}_{+}\left(\mathcal{H}_{(\cdot)}\right)
    }{\mathrm{inf}} \label{eq:infdim_separate_1}
    \quad& \frac{b}{n}\sum_{i=1}^{n} f_{\mathcal{A}_{(\cdot)}}(X_{i}) + \Omega_{(\cdot)}(\mathcal{A}_{(\cdot)})\\
    \mathrm{s.t.} \quad& r_{(\cdot)}\left(X_i, Y_i\right) - f_{\mathcal{A}_{(\cdot)}}(X_{i}) \leq 0, \;i \in \left[n\right]. \nonumber
\end{align}

\begin{theorem}[Representer theorem]
\label{thm:representer theorem with no penalty}
    Let \((b,\lambda_{(\cdot){1}})\in\mathbb{R}_{+}^{2}\) and \(\lambda_{(\cdot){2}}>0\).
    Then \Cref{{eq:infdim_separate_1}} admits a unique solution \(f_{\mathbf{A}_{(\cdot)}^{\star}}\) of the form \(f_{\mathbf{A}_{(\cdot)}^{\star}}(X) = \bs{\Phi}_{(\cdot)}(X)^{\top}\mathbf{A}_{(\cdot)}^{\star}\bs{\Phi}_{(\cdot)}(X)\) for some matrix \(\mathbf{A}_{(\cdot)}^{\star}\in \mathbb{S}_{+}^{n}\).
\end{theorem}

\begin{proof}
We follow the proof of Theorem $2$ from \citet{allain2025scalableandadaptivepredictionbandsusingkerenlsumofsquares} with a fixed mean function, by replacing their residuals $r^2$ with  $r_{(\cdot)}$. In the end, Problem (\ref{eq:infdim_separate_1}) admits a finite representation entirely characterized by a PSD matrix $\mathbf{A}_{(\cdot)} \in \mathbb{S}_{+}^{n}$ given by
\begin{align}
    \underset{\mathbf{A}_{(\cdot)}\in \mathbb{S}_{+}^{n}}{\inf} \quad& \frac{b}{n}\sum_{i=1}^{n} \tilde{f}_{\mathbf{A}_{(\cdot)}}(X_{i}) + \lambda_{1}\lVert \mathbf{A}_{(\cdot)}\rVert_{\star} + \lambda_{2}\lVert \mathbf{A}_{(\cdot)}\rVert_{F}^{2}\label{eq:finite_separate_2}\\
    \mathrm{s.t.} \quad& r_{(\cdot)}\left(X_i, Y_i\right)-\tilde{f}_{\mathbf{A}_{(\cdot)}}(X_{i}) \leq 0, \;i \in \left[n\right],\nonumber
\end{align}
with $\tilde{f}_{\mathbf{A}_{(\cdot)}}(X_{i}) = \bs{\Phi}_{(\cdot)}(X_{i})^{\top}\mathbf{A}_{(\cdot)}\bs{\Phi}_{(\cdot)}(X_{i})$.
\end{proof}

The optimal solution of Problem (\ref{eq:eq:infdim_1}) is thus recovered as \(f_{\mathbf{A}_{\mathrm{low}}^{\star}}(X) = \bs{\Phi}_{\mathrm{low}}(X)^{\top}\mathbf{A}_{\mathrm{low}}^{\star}\bs{\Phi}_{\mathrm{low}}(X)\) and \(f_{\mathbf{A}_{\mathrm{up}}^{\star}}(X) = \bs{\Phi}_{\mathrm{up}}(X)^{\top}\mathbf{A}_{\mathrm{up}}^{\star}\bs{\Phi}_{\mathrm{up}}(X)\) for matrices \(\mathbf{A}_{\mathrm{low}}^{\star}, \mathbf{A}_{\mathrm{up}}^{\star}\in \mathbb{S}_{+}^{n}\) solutions of \Cref{eq:finite_separate_2} with $(\cdot)$ equal to $_\mathrm{low}$ and $_\mathrm{up}$, respectively.

Now, for each independent problem $_\mathrm{low}$ and $_\mathrm{up}$, we exhibit a dual formulation. Let us first introduce, for any symmetric matrix $\mathbf{A}\in \mathbb{S}^{n}$, the penalty function $\Omega_{+,(\cdot)}$ defined as
\begin{equation}
        \Omega_{+,(\cdot)}(\mathbf{A}) = 
        \begin{cases}
            \Omega_{(\cdot)}(\mathbf{A}) \quad &\mathrm{if}\quad \mathbf{A} \succeq 0\\
            + \infty &\mathrm{otherwise}
        \end{cases}
    \end{equation}
where $\Omega_{(\cdot)}(\mathbf{A})$ is the matrix equivalent of \Cref{eq:regularizer function on one operator}.

\begin{proposition}[Dual formulation]
\label{prop:dual formulation}
    Let \((b,\lambda_{(\cdot){1}})\in\mathbb{R}_{+}^{2}\) and \(\lambda_{(\cdot){2}}>0\). \Cref{eq:finite_separate_2} admits a dual formulation of the form
    \begin{align}
    \label{eq:asymmetric_dual}
        \underset{
        \substack{
            \bs{\Gamma_{(\cdot)}} \in \mathbb{R}_+^{n} 
        }
        }{\sup}
        \bs{\Gamma_{(\cdot)}}\mathbf{r}_{(\cdot)}(\mathbf{X}, Y)^{\top}  - \Omega^{\star}_{+,(\cdot)}(\mathbf{V}_{(\cdot)}\mathrm{Diag}({\bs{\Gamma}_{(\cdot)}}_{-\mathbf{b}})\mathbf{V}_{(\cdot)}^{\top})
    \end{align}
    where \(\Omega_{+,(\cdot)}^{\star}(\mathbf{B}) = \frac{1}{4\lambda_{(\cdot)2}}\lVert \left[\mathbf{B}-\lambda_{(\cdot)1}\mathbf{I}_{n}\right]_{+}\rVert_{F}^{2}\) and \(\forall x \in \mathbb{R}, \;\mathrm{Diag}\left(\bs{\Gamma}_{(\cdot)x}\right) := \mathrm{Diag}\left(\bs{\Gamma}_{(\cdot)}\right) + \frac{x}{n}\mathbf{I}_{n}\).
    Moreover, if \(\widehat{\bs{\Gamma}}_{(\cdot)}\) is solution of \Cref{eq:asymmetric_dual}, a solution of \Cref{eq:finite_separate_2} can be retrieved as 
    \begin{align*}
        \widehat{\mathbf{A}}_{(\cdot)} = \frac{1}{2\lambda_{(\cdot)2}}\left[\mathbf{V}_{(\cdot)}\mathrm{Diag}(\widehat{\bs{\Gamma}}_{(\cdot){-\mathbf{b}}})\mathbf{V}_{(\cdot)}^{\top}-\lambda_{(\cdot){1}}\mathbf{I}_{n}\right]_{+}
    \end{align*}
    where \(\left[\mathbf{A}\right]_{+}\) denotes the positive part of \(\mathbf{A}\). For a symmetric matrix \(\mathbf{A}\) with eigendecomposition \(\mathbf{A} = \mathbf{U} \mathbf{D} \mathbf{U}^{\top}\), its positive part is defined as \(\left[\mathbf{A}\right]_{+} = \mathbf{U} \max(0,\mathbf{D}) \mathbf{U}^{\top}\).
\end{proposition}

\begin{proof}
Again, we apply the proof of Proposition $2$ from \citet{allain2025scalableandadaptivepredictionbandsusingkerenlsumofsquares} with a fixed mean function and by replacing $r^2$ with  $r_{(\cdot)}$.

Finally, to solve such dual formulation, we rely on explicit gradients given below.

\paragraph{Gradient computation.}
\begin{equation*}
\label{eq:grad_dual_1}
    \frac{\partial\left(\bs{\Gamma_{(\cdot)}}\mathbf{r}_{(\cdot)}(\mathbf{X}, Y)^{\top}\right)}{\partial\bs{\Gamma_{(\cdot)}}} = \mathbf{r}_{(\cdot)}(\mathbf{X}, Y).
\end{equation*}
\begin{equation*}
\label{eq:grad_dual_2}
    \frac{\partial\Omega^{\star}_{+,(\cdot)}(\mathbf{V}_{(\cdot)}\mathrm{Diag}({\bs{\Gamma}_{(\cdot)}}_{-\mathbf{b}})\mathbf{V}_{(\cdot)}^{\top})}{\partial\bs{\Gamma_{(\cdot)}}} = \mathrm{Diag}\left(\mathbf{V}_{(\cdot)}^{\top}\left[\nabla\Omega^{\star}_{+,(\cdot)}\left(\mathbf{V}_{(\cdot)}\mathrm{Diag}({\bs{\Gamma}_{(\cdot)}}_{-\mathbf{b}})\mathbf{V}_{(\cdot)}^{\top}\right)\right]^{\top}\mathbf{V}_{(\cdot)}\right).
\end{equation*}
More details can be found in \citet{allain2025scalableandadaptivepredictionbandsusingkerenlsumofsquares}.

\paragraph{Recovering the solution from optimal Lagrange multipliers.}
 Once the dual problem is solved (see Appendix \cref{sec:implementation details}), we need to recover the optimal solutions of the primal problem. Denoting $\widehat{\bs{\Gamma}}_{(\cdot)} \in \mathbb{R}_{+}^{n}$ the optimal Lagrange multipliers for both problems, to reconstruct the matrices $\mathbf{A}_{(\cdot)}$ we follow \citet{allain2025scalableandadaptivepredictionbandsusingkerenlsumofsquares}:
\begin{align*}
    \widehat{\mathbf{A}}_{(\cdot)} &= \nabla\Omega^{\star}_{+,(\cdot)}\left(\mathbf{V}_{(\cdot)}\mathrm{Diag}(\widehat{\bs{\Gamma}}_{(\cdot)-\mathbf{b}})\mathbf{V}^{\top}_{(\cdot)}\right)\\ 
    &= \frac{1}{2\lambda_{(\cdot)2}}\left[\mathbf{V}_{(\cdot)}\mathrm{Diag}(\widehat{\bs{\Gamma}}_{(\cdot)-\mathbf{b}})\mathbf{V}^{\top}_{(\cdot)}-\lambda_{(\cdot)1}\mathbf{I}_{n}\right]_{+}.
\end{align*}
\end{proof}

\newpage
\subsection{Asymmetric problem with operator penalty}
\label{sec:asymmetric problem with penalty 2}

For the asymmetric setting with the penalty on operators, our infinite dimensional problem writes:
\begin{align}
    \underset{
        \substack{
            \mathcal{A}_{\mathrm{low}} \in \mathcal{S}_{+}\left(\mathcal{H}\right) \\
            \mathcal{A}_{\mathrm{up}} \in \mathcal{S}_{+}\left(\mathcal{H}\right)
        }
    }{\inf}
    \quad& \frac{b}{n}\sum_{i=1}^{n} \left(f_{\mathcal{A_{\mathrm{low}}}}(X_{i}) + f_{\mathcal{A_{\mathrm{up}}}}(X_{i})\right) + \Psi(\mathcal{A}_{\mathrm{low}}, \mathcal{A}_{\mathrm{up}})\nonumber \\
    \mathrm{s.t.} \quad& r_{\mathrm{low}}\left(X_i, Y_i\right) - f_{\mathcal{A}_{\mathrm{low}}}(X_{i})\leq 0, \;i \in \left[n\right] \label{eq:infdim penalty 2} \\
    \quad& r_{\mathrm{up}}\left(X_i, Y_i\right) - f_{\mathcal{A_{\mathrm{up}}}}(X_{i})\leq 0, \;i \in \left[n\right]\nonumber
\end{align}
where $\Psi(\mathcal{A}_{\mathrm{low}}, \mathcal{A}_{\mathrm{up}}) =  \Omega_{\mathrm{low}}(\mathcal{A}_{\mathrm{low}}) + \Omega_{\mathrm{up}}(\mathcal{A_{\mathrm{up}}}) +\Omega_{\mathrm{pen}}(\mathcal{A}_{\mathrm{low}}-\mathcal{A}_{\mathrm{up}})$. This time, the problem is no longer separable because of the penalty term.

\begin{theorem}[Representer theorem with operator penalty]
\label{thm:representer theorem with penalty 2}
    Let \((b,\lambda_{\mathrm{low}{1}}, \lambda_{\mathrm{up}{1}}, \lambda_{\mathrm{pen}{1}})\in\mathbb{R}_{+}^{4}\) and \(\lambda_{\mathrm{low}{2}}, \lambda_{\mathrm{up}{2}}, \lambda_{\mathrm{pen}{2}}>0\). Then Problem \ref{eq:infdim penalty 2} admits a unique solution \((f_{\mathbf{B}_{\mathrm{low}}^{\star}}, f_{\mathbf{B}_{\mathrm{up}}^{\star}})\) of the form \((f_{\mathbf{B}_{\mathrm{low}}^{\star}}(X), f_{\mathbf{B}_{\mathrm{up}}^{\star}}(X)) = \bigl(\mathbf{k}(X)^{\top}\mathbf{B}_{\mathrm{low}}^{\star}\mathbf{k}(X),\ \mathbf{k}(X)^{\top}\mathbf{B}_{\mathrm{up}}^{\star}\mathbf{k}(X)\bigr)\) for some matrices \(\bigl(\mathbf{B}_{\mathrm{low}}^{\star}, \mathbf{B}_{\mathrm{up}}^{\star}\bigr)\in (\mathbb{S}_{+}^{n})^{2}\).
\end{theorem}
\begin{proof}
    We apply \Cref{thm:generalization_mf_p_operators} to the case $p=2$ with  $\mathcal{H}_{1} = \mathcal{H}_{2}=\mathcal{H}$ with kernel $k$ and loss function
    \begin{equation*}
        L((f_{{\mathcal{A}_{\mathrm{low}}}}(X_i), f_{{\mathcal{A}_{\mathrm{up}}}}(X_i))_{1\leq i\leq n}) = \begin{cases}
            \frac{b}{n}\sum_{i=1}^{n} \left(f_{\mathcal{A_{\mathrm{low}}}}(X_{i}) + f_{\mathcal{A_{\mathrm{up}}}}(X_{i})\right) 
            &\mathrm{if} \; (\mathcal{A}_{\mathrm{low}},\mathcal{A}_{\mathrm{up}})\in C\\
            +\infty & \mathrm{otherwise}
        \end{cases}
    \end{equation*}
    where $C$ is the set of all positive definite operators $(\mathcal{A}_{\mathrm{low}},\mathcal{A}_{\mathrm{up}})$ such that
    \begin{equation*}
         \begin{cases}
            r_{\mathrm{low}}\left(X_i, Y_i\right) - f_{\mathcal{A}_{\mathrm{low}}}(X_{i})\leq 0, \;i \in \left[n\right]\\
            r_{\mathrm{up}}\left(X_i, Y_i\right) - f_{\mathcal{A_{\mathrm{up}}}}(X_{i})\leq 0, \;i \in \left[n\right]
        \end{cases}
    \end{equation*}
     and the new penalty function $\Psi((\mathcal{A}_{\mathrm{low}}, \mathcal{A}_{\mathrm{up}})) = \Omega_{\mathrm{low}}(\mathcal{A}_\mathrm{low}) + \Omega_{\mathrm{up}}(\mathcal{A}_\mathrm{up}) + \Omega_{\mathrm{pen}}(\mathcal{A}_\mathrm{low}-\mathcal{A}_\mathrm{up})$, where $\Omega_{\mathrm{low}}$, $\Omega_{\mathrm{up}}$ and $\Omega_{\mathrm{pen}}$ are defined as in \Cref{eq:regularizer function on one operator}.
     
     First, notice that $\mathcal{A}_\mathrm{low}-\mathcal{A}_\mathrm{up}$ is not a positive semi-definite operator anymore but Remark $3$ in \citet{marteauferey2020nonparametricmodelsnonnegativefunctions} still applies since $\mathcal{A}_\mathrm{low}-\mathcal{A}_\mathrm{up} \in \mathcal{S}\left(\mathcal{H}\right)$. Furthermore, $\Omega_{\mathrm{pen}}$ verifies Assumption \ref{assumption:regularizer}. However, $\Psi$ no longer writes as a separable sum as in \Cref{eq:penalty_functions_1}, and we need to show that it satisfies the three properties of \Cref{lem:Omega_properties} to apply \Cref{thm:generalization_mf_p_operators}.
     
     For the first and second property, we can directly apply Lemma $2$ from \citet{marteauferey2020nonparametricmodelsnonnegativefunctions} individually to $\Omega_{\mathrm{low}}, \Omega_{\mathrm{up}}$ and $\Omega_{\mathrm{pen}}$, since it requires the operators to be symmetric only (not necessarily positive semi-definite), and thus also applies to $\Omega_{\mathrm{pen}}$. For the third property, $\Psi$ is continuous as a composition of the linear (hence continuous) function $(\mathcal{A}_\mathrm{low},\mathcal{A}_\mathrm{up}) \rightarrow (\mathcal{A}_\mathrm{low},\mathcal{A}_\mathrm{up}, \mathcal{A}_\mathrm{low}-\mathcal{A}_\mathrm{up})$ and lower-semi continuous functions ($\Omega_{\mathrm{low}}, \Omega_{\mathrm{up}}$ and $\Omega_{\mathrm{pen}}$). The last part of property $3$ holds because if at least one of $\lVert\mathcal{A}_\mathrm{low}\rVert_{\mathrm{op}},\ \lVert\mathcal{A}_\mathrm{up}\rVert_{\mathrm{op}}$ goes to infinity, we have:
     \begin{equation*}
         \Psi((\mathcal{A}_\mathrm{low},\mathcal{A}_\mathrm{up})) \underset{\max(\lVert\mathcal{A}_\mathrm{low}\rVert_{\mathrm{op}}, \lVert\mathcal{A}_\mathrm{up}\rVert_{\mathrm{op}})\rightarrow +\infty}{\longrightarrow} +\infty.
     \end{equation*}
     
     Finally, since $L\colon \mathbb{R}^{2n} \rightarrow \mathbb{R}$ is lower semi-continuous (notice that it is linear and bounded below by $0$) and $\Psi$ satisfies all properties from \Cref{lem:Omega_properties}, we can apply \Cref{thm:generalization_mf_p_operators} with $p=2$ to deduce that the solution is entirely characterized by two PSD matrices \(\bigl(\mathbf{B}_{\mathrm{low}}^{\star}, \mathbf{B}_{\mathrm{up}}^{\star}\bigr)\in (\mathbb{S}_{+}^{n})^{2}\) and
    \begin{equation*}
        (f_{\mathbf{B}_{\mathrm{low}}^{\star}}(X), f_{\mathbf{B}_{\mathrm{up}}^{\star}}(X)) = \bigl(\mathbf{k}(X)^{\top}\mathbf{B}_{\mathrm{low}}^{\star}\mathbf{k}(X),\ \mathbf{k}(X)^{\top}\mathbf{B}_{\mathrm{up}}^{\star}\mathbf{k}(X)\bigr).
    \end{equation*}

\end{proof}

\Cref{thm:representer theorem with penalty 2} also has an equivalent with matrices $\mathbf{A}_{\mathrm{low}}$, $\mathbf{A}_{\mathrm{up}}$ instead of $\mathbf{B}_{\mathrm{low}}$, $\mathbf{B}_{\mathrm{up}}$ involving $\tilde{f}_{\mathbf{A}_{(\cdot)}}$ as in \Cref{eq:f_tilde}.

\begin{proposition}[Representer theorem with operator penalty, formulation $\mathbf{A}$]
\label{prop:representer theorem with penalty 2 formulation A}
    Let \((b,\lambda_{\mathrm{low}{1}}, \lambda_{\mathrm{up}{1}}, \lambda_{\mathrm{pen}{1}})\in\mathbb{R}_{+}^{4}\) and \(\lambda_{\mathrm{low}{2}}, \lambda_{\mathrm{up}{2}}, \lambda_{\mathrm{pen}{2}}>0\). The following problem admits a unique solution if $L$ is convex:
    \begin{align}
        \underset{
            \substack{
                \mathbf{A}_{\mathrm{low}} \in \mathbb{S}_{+}^{n} \\
                \mathbf{A}_{\mathrm{up}} \in \mathbb{S}_{+}^{n}
            }
        }{\inf}
        \quad& \frac{b}{n}\sum_{i=1}^{n} \left(\tilde{f}_{\mathbf{A_{\mathrm{low}}}}(X_{i}) + \tilde{f}_{\mathbf{A_{\mathrm{up}}}}(X_{i})\right) + \Psi(\mathbf{A}_{\mathrm{low}}, \mathbf{A}_{\mathrm{up}})\nonumber \\
        \mathrm{s.t.} \quad& r_{\mathrm{low}}\left(X_i, Y_i\right) - \tilde{f}_{\mathbf{A}_{\mathrm{low}}}(X_{i})\leq 0, \;i \in \left[n\right] \label{eq:formulationA penalty 2} \\
        \quad& r_{\mathrm{up}}\left(X_i, Y_i\right) - \tilde{f}_{\mathbf{A_{\mathrm{up}}}}(X_{i})\leq 0, \;i \in \left[n\right]\nonumber.
    \end{align}
    Moreover, for any given solution $\mathbf{A}_{\mathrm{low}}^{\star}, \mathbf{A}_{\mathrm{up}}^{\star} \in \mathbb{S}_{+}^{n}$ of \Cref{eq:formulationA penalty 2}, the function $(\tilde{f}_{\mathbf{A}_{\mathrm{low}}^{\star}}, \tilde{f}_{\mathbf{A}_{\mathrm{up}}^{\star} })$ is a minimizer of \Cref{eq:infdim penalty 2}.
\end{proposition}
\begin{proof}
    Similarly to the proof of \Cref{thm:representer theorem with penalty 2}, we apply \Cref{prop:formulation A p operators} to the case $p=2$ with  $\mathcal{H}_{1} = \mathcal{H}_{2}=\mathcal{H}$, but $\Psi$ does not write as a separable sum: we thus must check if it satisfies the isometry invariance property to conclude. This is trivial, since
    \begin{align*}
        \Psi((O_{n} \mathbf{A}_\mathrm{low}O_{n}^*,O_{n}\mathbf{A}_\mathrm{up}O_{n}^*)) &= \Omega_\mathrm{low}(O_{n}\mathbf{A}_\mathrm{low}O_{n}^*) + \Omega_\mathrm{up}(O_{n}\mathbf{A}_\mathrm{up}O_{n}^*) + \Omega_\mathrm{pen}(O_{n}\mathbf{A}_\mathrm{low}O_{n}^*-O_{n}\mathbf{A}_\mathrm{up}O_{n}^*)\\
        &=\Omega_\mathrm{low}(\mathbf{A}_\mathrm{low}) + \Omega_\mathrm{up}(\mathbf{A}_\mathrm{up}) + 
        \Omega_\mathrm{pen}(O_{n}(\mathbf{A}_\mathrm{low}-\mathbf{A}_\mathrm{up})O_{n}^*)\\
        &= \Omega_\mathrm{low}(\mathbf{A}_\mathrm{low}) + \Omega_\mathrm{up}(\mathbf{A}_\mathrm{up}) + \Omega_\mathrm{pen}(\mathbf{A}_\mathrm{low}-\mathbf{A}_\mathrm{up})\\
        &= \Psi((\mathbf{A}_\mathrm{low},\mathbf{A}_\mathrm{up})).
    \end{align*}
\end{proof}

\begin{proposition}[Dual formulation with operator penalty]
\label{prop:dual formulation penalty 2}
Let \((b,\lambda_{\mathrm{low}{1}}, \lambda_{\mathrm{up}{1}}, \lambda_{\mathrm{pen}{1}})\in\mathbb{R}_{+}^{4}\) and \(\lambda_{\mathrm{low}{2}}, \lambda_{\mathrm{up}{2}}, \lambda_{\mathrm{pen}{2}}>0\).
Problem \ref{eq:formulationA penalty 2} admits a dual formulation of the form
    \begin{align}
    \label{eq:asymmetric_dual_penalty_2}
        \underset{
        \substack{
            \bs{\Gamma}_{\mathrm{low}} \in \mathbb{R}_+^{n},\bs{\Gamma}_{\mathrm{up}} \in \mathbb{R}_+^{n}\\
            \mathbf{W} \in \mathbb{S}^{n}
        }
        }{\sup}
        &(\bs{\Gamma}_{\mathrm{up}}-\bs{\Gamma}_{\mathrm{low}})\mathbf{r}^{\top} 
        - \Omega_{\mathrm{pen}}^{\star}(\mathbf{W})\\
        &-\Omega^{\star}_{+,\mathrm{low}}(\mathbf{V}_{\mathrm{low}}\mathrm{Diag}({\bs{\Gamma}_{\mathrm{low}}}_{-\mathbf{b}})\mathbf{V}_{\mathrm{low}}^{\top} - \mathbf{W})
        -\Omega^{\star}_{+,\mathrm{up}}(\mathbf{V}_{\mathrm{up}}\mathrm{Diag}({\bs{\Gamma}_{\mathrm{up}}}_{-\mathbf{b}})\mathbf{V}_{\mathrm{up}}^{\top} + \mathbf{W})\nonumber
    \end{align}
    where $\mathbf{r}$ is the vector of residuals $r_i=Y_i-m(X_i)$, \(\Omega_{+,(\cdot)}^{\star}(\mathbf{B}) = \frac{1}{4\lambda_{(\cdot)2}}\lVert \left[\mathbf{B}-\lambda_{(\cdot)1}\mathbf{I}_{n}\right]_{+}\rVert_{F}^{2}\), $\Omega_{\mathrm{pen}}^{\star}(\mathbf{B})=(1/4\lambda_{\mathrm{pen}2})\sum_{i=1}^{n}\max(0, \lvert \lambda_{i}(\mathbf{B}) \rvert-\lambda_{\mathrm{pen}1})^{2}$ and \(\forall x \in \mathbb{R}, \;\mathrm{Diag}\left({(\cdot)_x}\right) := \mathrm{Diag}\left({(\cdot)}\right) + \frac{x}{n}\mathbf{I}_{n}\).
    Moreover, if \((\widehat{\bs{\Gamma}}_{\mathrm{low}}, \widehat{\bs{\Gamma}}_{\mathrm{up}}, \widehat{\mathbf{W}})\) is a solution of \Cref{eq:asymmetric_dual_penalty_2}, a solution of Problem \ref{eq:formulationA penalty 2} can be retrieved as 
    \begin{align*}
        \widehat{\mathbf{A}}_{\mathrm{low}} &= \frac{1}{2\lambda_{\mathrm{low}2}}\left[\mathbf{V}_{\mathrm{low}}\mathrm{Diag}(\widehat{\bs{\Gamma}}_{\mathrm{low}})_{-\mathbf{b}}\mathbf{V}_{\mathrm{low}}^{\top}-\widehat{\mathbf{W}}-\lambda_{\mathrm{low}{1}}\mathbf{I}_{n}\right]_{+}\\
        \widehat{\mathbf{A}}_{\mathrm{up}} &= \frac{1}{2\lambda_{\mathrm{up}2}}\left[\mathbf{V}_{\mathrm{up}}\mathrm{Diag}(\widehat{\bs{\Gamma}}_{\mathrm{up}})_{-\mathbf{b}}\mathbf{V}_{\mathrm{up}}^{\top}+\widehat{\mathbf{W}}-\lambda_{\mathrm{up}{1}}\mathbf{I}_{n}\right]_{+}.
    \end{align*}
\end{proposition}
\begin{proof}
    The dual problem is defined as
    \begin{equation}
        d = \underset{
        \substack{
            \bs{\Gamma}_{\mathrm{low}} \in \mathbb{R}_+^{n}\\
            \bs{\Gamma}_{\mathrm{up}} \in \mathbb{R}_+^{n}
        }
        }{\sup}
        \underset{
        \substack{
            \mathbf{A}_{\mathrm{low}} \in \mathbb{S}^{n}\\
            \mathbf{A}_{\mathrm{up}} \in \mathbb{S}^{n}
        }
        }{\inf}
        \mathcal{L}(\bs{\Gamma}_{\mathrm{low}}, \bs{\Gamma}_{\mathrm{up}}, \mathbf{A}_{\mathrm{low}}, \mathbf{A}_{\mathrm{up}})
        = \underset{
        \substack{
            \bs{\Gamma}_{\mathrm{low}} \in \mathbb{R}_+^{n}\\
            \bs{\Gamma}_{\mathrm{up}} \in \mathbb{R}_+^{n}
        }
        }{\sup} (\bs{\Gamma}_{\mathrm{up}}-\bs{\Gamma}_{\mathrm{low}})\mathbf{r}^{\top} + D(\bs{\Gamma}_{\mathrm{low}}, \bs{\Gamma}_{\mathrm{up}}) \label{eq:primal problem penalty 2}
    \end{equation}
    where we denote 
    \begin{align}
    \label{eq:function D penalty 2}
    D(\bs{\Gamma}_{\mathrm{low}}, \bs{\Gamma}_{\mathrm{up}})\coloneq \underset{
        \substack{
            \mathbf{A}_{\mathrm{low}} \in \mathbb{S}^{n},\;
            \mathbf{A}_{\mathrm{up}} \in \mathbb{S}^{n}
        }
        }{\inf}
        \overline{\mathcal{L}}(\bs{\Gamma}_{\mathrm{low}}, \bs{\Gamma}_{\mathrm{up}}, \mathbf{A}_{\mathrm{low}}, \mathbf{A}_{\mathrm{up}}) 
    \end{align}
        with $\overline{\mathcal{L}}$ defined as
    \begin{align*}
        \overline{\mathcal{L}}(\bs{\Gamma}_{\mathrm{low}}, \bs{\Gamma}_{\mathrm{up}}, \mathbf{A}_{\mathrm{low}}, \mathbf{A}_{\mathrm{up}}) = 
        &\sum_{i=1}^{n} (\frac{b}{n} - \Gamma_{\mathrm{low},i})\tilde{f}_{\mathbf{A}_{\mathrm{low}}}(X_{i}) + \sum_{i=1}^{n} (\frac{b}{n} - \Gamma_{\mathrm{up},i})\tilde{f}_{\mathbf{A}_{\mathrm{up}}}(X_{i})\\
        & + \Omega_{+,\mathrm{low}}(\mathbf{A}_{\mathrm{low}}) + \Omega_{+,\mathrm{up}}(\mathbf{A}_{\mathrm{up}}) + \Omega_{\mathrm{pen}}(\mathbf{A}_{\mathrm{low}}-\mathbf{A}_{\mathrm{up}}).
    \end{align*}

    This time, to derive the optimality conditions for $(\mathbf{A}_{\mathrm{low}}, \mathbf{A}_{\mathrm{up}})$, we follow \citet{allain2025scalableandadaptivepredictionbandsusingkerenlsumofsquares} Appendix A.2, Equation (15). First, observe that
    \begin{equation*}
        \overline{\mathcal{L}} = \Psi_+(\mathbf{A}_{\mathrm{low}}, \mathbf{A}_{\mathrm{up}}) - \langle(\mathbf{A}_{\mathrm{low}}, \mathbf{A}_{\mathrm{up}}), (\mathbf{V}_{\mathrm{low}}\mathrm{Diag}(\widehat{\bs{\Gamma}}_{\mathrm{low}})_{-\mathbf{b}}\mathbf{V}_{\mathrm{low}}^{\top}, \mathbf{V}_{\mathrm{up}}\mathrm{Diag}(\widehat{\bs{\Gamma}}_{\mathrm{up}})_{-\mathbf{b}}\mathbf{V}_{\mathrm{up}}^{\top})\rangle_{\mathbb{S}^{n}\times \mathbb{S}^{n}}
    \end{equation*}
    where $\Psi_+(\mathbf{A}_{\mathrm{low}}, \mathbf{A}_{\mathrm{up}})=  \Omega_{+,\mathrm{low}}(\mathbf{A}_{\mathrm{low}}) + \Omega_{+,\mathrm{up}}(\mathbf{A}_{\mathrm{up}}) +\Omega_{\mathrm{pen}}(\mathbf{A}_{\mathrm{low}}-\mathbf{A}_{\mathrm{up}})$.
    Then, by definition, $D(\bs{\Gamma}_{\mathrm{low}}, \bs{\Gamma}_{\mathrm{up}})$ writes as the Fenchel conjugate function of $\Psi_+$ evaluated at specific matrices
    \begin{equation}
    \label{eq:function D as dual psi penalty 2}
        D(\bs{\Gamma}_{\mathrm{low}}, \bs{\Gamma}_{\mathrm{up}}) = -\Psi_{+}^{\star}\Bigl(\mathbf{V}_{\mathrm{low}}\mathrm{Diag}(\bs{\Gamma}_{\mathrm{low}})_{-\mathbf{b}}\mathbf{V}_{\mathrm{low}}^{\top}, \mathbf{V}_{\mathrm{up}}\mathrm{Diag}(\bs{\Gamma}_{\mathrm{up}})_{-\mathbf{b}}\mathbf{V}_{\mathrm{up}}^{\top}\Bigr).
    \end{equation}

    The explicit formulation of $\Psi_{+}^{\star}$ is given in the following lemma.

    \begin{lemma}
    \label{lem:psi dual function penalty 2}
        The Fenchel conjugate of $\Psi_+(\mathbf{A}_{1}, \mathbf{A}_{2}) =  \Omega_{+,\mathrm{low}}(\mathbf{A}_{1}) + \Omega_{+,\mathrm{up}}(\mathbf{A}_{2}) +\Omega_{\mathrm{pen}}(\mathbf{A}_{1}-\mathbf{A}_{2})$ writes
        \begin{equation*}
            \Psi_{+}^{\star}(\mathbf{B}_{1}, \mathbf{B}_{2}) = -\sup_{\mathbf{W}\in \mathbb{S}^{n}} 
            - \Omega_{+, \mathrm{low}}^{\star}(\mathbf{B}_{1} - \mathbf{W}) 
            - \Omega_{+, \mathrm{up}}^{\star}(\mathbf{B}_{2} + \mathbf{W}) 
            - \Omega_{\mathrm{pen}}^{\star}(\mathbf{W}).
        \end{equation*}
    \end{lemma}
    \begin{proof}
        The proof relies on the fact that $\Psi_{+}$ can be written as a composition $\Psi(\mathbf{A}_{1}, \mathbf{A}_{2}) = f\circ L (\mathbf{A}_{1}, \mathbf{A}_{2})$ of the linear operator $L\colon (\mathbf{A}_{1}, \mathbf{A}_{2}) \in (\mathbb{S}^{n})^{2} \mapsto (\mathbf{A}_{1}, \mathbf{A}_{2}, \mathbf{A}_{1}- \mathbf{A}_{2}) \in (\mathbb{S}^{n})^{3}$ and the function $f\colon (\mathbf{U}, \mathbf{V}, \mathbf{W})\in (\mathbb{S}^{n})^{3} \mapsto \Omega_{+, \mathrm{low}}(\mathbf{U}) + \Omega_{+, \mathrm{up}}(\mathbf{V}) +\Omega_{\mathrm{pen}}(\mathbf{W}) \in  \mathbb{R}$.
        The Fenchel conjugate of $\Psi_{+}$ then writes (see, e.g., Theorem 16.3 in \citet{Rockafellar+1970}):
        \begin{equation}
        \label{eq:psi conjugate def}
            \Psi_{+}^{\star}(\mathbf{B}_1, \mathbf{B}_2) =
            \inf_{\substack{
                \mathbf{U}, \mathbf{V}, \mathbf{W} \\
                \text{s.t. } L^{\star}(\mathbf{U}, \mathbf{V}, \mathbf{W}) = (\mathbf{B}_1, \mathbf{B}_2)
            }}
                \Omega_{+, \mathrm{low}}^{\star}(\mathbf{U})
                + \Omega_{+, \mathrm{up}}^{\star}(\mathbf{V})
                + \Omega_{\mathrm{pen}}^{\star}(\mathbf{W})
        \end{equation}
        where $L^{\star}$, the adjoint of $L$, can be deduced from
        \begin{align*}
            \langle L(\mathbf{A}_{1}, \mathbf{A}_{2}), (\mathbf{U}, \mathbf{V}, \mathbf{W})\rangle_{(\mathbb{S}^{n})^{3}} 
            &= \langle \mathbf{A}_1, \mathbf{U}\rangle + \langle \mathbf{A}_2, \mathbf{V}\rangle + \langle \mathbf{A}_1 - \mathbf{A}_2, \mathbf{W}\rangle \\
            &= \langle \mathbf{A}_1, \mathbf{U}+\mathbf{W}\rangle + \langle \mathbf{A}_2, \mathbf{V}-\mathbf{W}\rangle\\
            &= \langle (\mathbf{A}_{1}, \mathbf{A}_{2}), L^{\star}(\mathbf{U}, \mathbf{V}, \mathbf{W})\rangle_{(\mathbb{S}^{n})^{2}}
        \end{align*}
        such that $L^{\star}(\mathbf{U}, \mathbf{V}, \mathbf{W}) = (\mathbf{U} + \mathbf{W}, \mathbf{V}-\mathbf{W}) \in (\mathbb{S}^{n})^{2}$.
        Now \Cref{eq:psi conjugate def} writes
        \begin{equation*}
            \Psi_{+}^{\star}(\mathbf{B}_1, \mathbf{B}_2) =
            \inf_{\substack{
                \mathbf{U} + \mathbf{W} =\mathbf{B}_1\\
                \mathbf{V} - \mathbf{W} =\mathbf{B}_2\\
                \mathbf{W}
            }}
                \Omega_{+, \mathrm{low}}^{\star}(\mathbf{U})
                + \Omega_{+, \mathrm{up}}^{\star}(\mathbf{V})
                + \Omega_{\mathrm{pen}}^{\star}(\mathbf{W})
        \end{equation*}
        and the result follows by replacing $\mathbf{U}$ and $\mathbf{V}$ by $\mathbf{B}_{1} - \mathbf{W}$ and $\mathbf{B}_{2} + \mathbf{W}$, respectively.
    \end{proof}

    Before concluding the proof, we need to derive the dual function $\Omega_{\mathrm{pen}}^{\star}$ (this is not the same as $\Omega_{(\cdot)}^{\star}$, because it is defined on symmetric matrices and not positive semi-definite ones), which is addressed in the following lemma.

    \begin{lemma}
        \label{lem:conjugate_derivative_norms}
        For any symmetric matrix $\mathbf{X} \in \mathbb{S}^{n}$ such that $\mathbf{X}=\mathbf{U}^{\top}\mathrm{Diag}(\lambda(\mathbf{X}))\mathbf{U}$ where $\lambda(\mathbf{X})=(\lambda_1(\mathbf{X}),\ldots,\lambda_n(\mathbf{X}))$ denotes the eigenvalues of $X$, the function$f(\mathbf{X}) = \lambda_1\lVert\mathbf{X}\rVert_{\star} + \lambda_2\lVert\mathbf{X}\rVert_{F}^{2}$
        admits a Fenchel conjugate given by
        \begin{equation*}
            f^{\star}(\mathbf{Y}) = \frac{1}{4\lambda_2}\sum_{i=1}^{n}\max(0, \lvert\lambda_{i}(\mathbf{Y})\rvert-\lambda_1)^2
        \end{equation*}
        with gradient 
        \begin{equation*}
            \nabla f^{\star}(\mathbf{Y}) = \frac{1}{2\lambda_2}\mathbf{U}^{\top}\mathrm{Diag}(\mathrm{sign(\lambda(\mathbf{Y}))\times\max(0, \lvert\lambda(\mathbf{Y})\rvert-\lambda_1)})\mathbf{U}.
        \end{equation*}
        \end{lemma}
    \begin{proof}
        We first write $f$ as a function $\varphi$ of the singular values of $\mathbf{X}$. The nuclear and Frobenius norms are the $1$- and $2$-Schatten norms, respectively, which implies that
        \begin{align*}
            f(\mathbf{X}) = \lambda_1 \sum_{i=1}^{n}\sigma_{i}(\mathbf{X}) + \lambda_2 \sum_{i=1}^{n}\sigma_{i}(\mathbf{X})^2 = \varphi(\sigma(\mathbf{X}))
        \end{align*}
        where $\varphi((s_{1}, \ldots, s_{n})) = \lambda_1 \sum_{i=1}^{n} s_i + \lambda_2 \sum_{i=1}^{n} s_i^2$ and $\sigma(\mathbf{X})$ are the singular values of $\mathbf{X}$.
        Because $f$ is unitary invariant, from Theorem $2.4$ in \citet{lewis1995convex}, we know that \begin{equation*}
            f^{\star}(\mathbf{Y}) = \varphi^{\star}(\sigma(\mathbf{Y})).
        \end{equation*}

        Now, let us compute the dual function $\varphi^{\star}$. By definition, it is given by
        \begin{align*}
            \varphi^{\star}(t) &= \sup_{s\geq0}\ \langle t, s\rangle - \lambda_1 \sum_{i=1}^{n} s_i - \lambda_2 \sum_{i=1}^{n} s_i^2\\
            &= \sup_{(s_1,\ldots, s_n)\geq0}\ \sum_{i=1}^{n}(t_i s_i - \lambda_1 s_i - \lambda_2 s_{i}^2).
        \end{align*}
        Since this optimization problem is separable, we can maximize each term individually. Let $g(s) = ts-\lambda_{1}s - \lambda_{2}s^2$, which attains its maximum for $s=(t-\lambda_1)/2\lambda_2$. We now have two cases: either $t-\lambda_1 \leq 0$ and the supremum is reached for $s=0$ and its value is $0$, or $t-\lambda_1 > 0$, in which case the supremum is reached for $s=(t-\lambda_1)/2\lambda_2$ and its value is $(t-\lambda_1)^2/4\lambda_2$.
        All in one, $\sup_{s\geq0} g(s) = \max(0, (t-\lambda_1)/4\lambda_2)^2$ and $\varphi^{\star}(t) = \frac{1}{4\lambda_2}\sum_{i=1}^{n}\max(0, (t_i - \lambda_1))^{2}$.

        This means that the Fenchel conjugate of $f$ is $f^{\star}(\mathbf{Y}) = \frac{1}{4\lambda_2}\sum_{i=1}^{n}\max(0, \sigma_{i}(\mathbf{Y})-\lambda_1)^2$,
        and since $\mathbf{X}$ and $\mathbf{Y}$ are symmetric matrices, we have $\sigma_{i}(\mathbf{Y})=\lvert\lambda_{i}(\mathbf{Y})\rvert$ from which we deduce
        \begin{equation*}
            f^{\star}(\mathbf{Y}) = \frac{1}{4\lambda_2}\sum_{i=1}^{n}\max(0, \lvert\lambda_{i}(\mathbf{Y})\rvert-\lambda_1)^2.
        \end{equation*}

        Now, to compute its gradient, observe that $f^{\star}(\mathbf{Y})$ writes as some function $h$ applied to the eigenvalues of $\mathbf{Y}$ with $h(u) = \frac{1}{4\lambda_2}\sum_{i=1}^{n}\max(0, \vert u_i\vert -\lambda_1)^2$.
        Because $h$ is permutation invariant, we can apply Theorem $1.1$ from \citet{lewis1996derivatives} to $f^{\star}$ to get
        \begin{equation}
        \label{eq:grad_conjugate_derivative_with_h}
            \nabla f^{\star}(\mathbf{Y}) = \mathbf{U}^{\top}\mathrm{Diag}(h^{\prime}(\lambda(\mathbf{Y})))\mathbf{U}.
        \end{equation}

        The last step is to compute the gradient of $h$. For all $1\leq i\leq n$, we have
        \begin{equation*}
            \frac{\partial h(u)}{\partial u_i} = \frac{1}{4\lambda_2}\frac{d}{du_i} \max(0, \lvert u_i\rvert-\lambda_1)^{2}
        \end{equation*}
        and observe that:
        \begin{itemize}
            \item If $\lvert u_i \rvert \leq \lambda_1$, then $\max(0, \lvert u_i\rvert-\lambda_1)^{2}=0$ and the derivative is $0$.
            \item If $u_i > \lambda_1$, then $\max(0, \lvert u_i\rvert-\lambda_1)^{2}=(u_i-\lambda_1)^2$ and the derivative is $2(u_i-\lambda_1) \underset{u_i \rightarrow \lambda_1^+}{\rightarrow}0$.
            \item If $u_i < -\lambda_1$, then $\max(0, \lvert u_i\rvert-\lambda_1)^{2}=(u_i+\lambda_1)^2$ and the derivative is $2(u_i+\lambda_1) \underset{u_i \rightarrow -\lambda_1^-}{\rightarrow} 0$.
        \end{itemize}

        Consequently, $\max(0, \lvert u_i\rvert-\lambda_1)^{2}$ is differentiable with derivative $2\,\mathrm{sign}(u_i)\times \max(0, \lvert u_i\rvert-\lambda_1)$.
        Plugging this into \Cref{eq:grad_conjugate_derivative_with_h} gives
        \begin{equation*}
            \nabla f^{\star}(\mathbf{Y}) = \frac{1}{2\lambda_2}\mathbf{U}^{\top}\mathrm{Diag}(\mathrm{sign(\lambda(\mathbf{Y}))\times\max(0, \lvert\lambda(\mathbf{Y})\rvert-\lambda_1)})\mathbf{U}
        \end{equation*}
        which concludes the proof of \Cref{lem:conjugate_derivative_norms}.
    \end{proof}

    Finally, applying \Cref{lem:psi dual function penalty 2} together with the explicit formulation of $\Omega_{\mathrm{pen}}^{\star}$ from \Cref{lem:conjugate_derivative_norms} in \Cref{eq:primal problem penalty 2}, our main dual problem writes
    \begin{align*}
        d = 
        \underset{
        \substack{
            \bs{\Gamma}_{\mathrm{low}} \in \mathbb{R}_+^{n},\bs{\Gamma}_{\mathrm{up}} \in \mathbb{R}_+^{n}\\
            \mathbf{W} \in \mathbb{S}^{n}
        }
        }{\sup}
        &(\bs{\Gamma}_{\mathrm{up}}-\bs{\Gamma}_{\mathrm{low}})\mathbf{r}^{\top} 
        - \Omega_{\mathrm{pen}}^{\star}(\mathbf{W})\\
        &-\Omega^{\star}_{+,\mathrm{low}}(\mathbf{V}_{\mathrm{low}}\mathrm{Diag}({\bs{\Gamma}_{\mathrm{low}}}_{-\mathbf{b}})\mathbf{V}_{\mathrm{low}}^{\top} - \mathbf{W})
        -\Omega^{\star}_{+,\mathrm{up}}(\mathbf{V}_{\mathrm{up}}\mathrm{Diag}({\bs{\Gamma}_{\mathrm{up}}}_{-\mathbf{b}})\mathbf{V}_{\mathrm{up}}^{\top} + \mathbf{W})\nonumber\\
        = \underset{
        \substack{
            \bs{\Gamma}_{\mathrm{low}} \in \mathbb{R}_+^{n},\bs{\Gamma}_{\mathrm{up}} \in \mathbb{R}_+^{n}\\
            \mathbf{W} \in \mathbb{S}^{n}
        }
        }{\sup} &g(\bs{\Gamma}_{\mathrm{low}}, \bs{\Gamma}_{\mathrm{up}}, \mathbf{W})
    \end{align*}

    \paragraph{Gradient computation.}
    \begin{equation*}
        \frac{\partial g}{\partial \bs{\Gamma}_{\mathrm{low}}} = -\mathbf{r} - \mathrm{Diag}\left(\mathbf{V}_{\mathrm{low}}^{\top}\left[\nabla\Omega^{\star}_{+,\mathrm{low}}\left(\mathbf{V}_{\mathrm{low}}\mathrm{Diag}{(\bs{\Gamma}_{\mathrm{low}}}_{-\mathbf{b}})\mathbf{V}_{\mathrm{low}}^{\top} - \mathbf{W}\right)\right]^{\top}\mathbf{V}_{\mathrm{low}}\right)
    \end{equation*}
    \begin{equation*}
        \frac{\partial g}{\partial\bs{\Gamma}_{\mathrm{up}}} = 
        \mathbf{r} 
        - \mathrm{Diag}\left(\mathbf{V}_{\mathrm{up}}^{\top}\left[\nabla\Omega^{\star}_{+,\mathrm{up}}\left(\mathbf{V}_{\mathrm{up}}\mathrm{Diag}{(\bs{\Gamma}_{\mathrm{up}}}_{-\mathbf{b}})\mathbf{V}_{\mathrm{up}}^{\top} + \mathbf{W}\right)\right]^{\top}\mathbf{V}_{\mathrm{up}}\right)
    \end{equation*}
    \begin{align*}
        \frac{\partial g}{\partial \mathbf{W}} 
        = - \nabla\Omega^{\star}_{\mathrm{pen}}(\mathbf{W}) 
        &+ \nabla\Omega^{\star}_{+,\mathrm{low}}\left(\mathbf{V}_{\mathrm{low}}\mathrm{Diag}{(\bs{\Gamma}_{\mathrm{low}}}_{-\mathbf{b}} )\mathbf{V}_{\mathrm{low}}^{\top} +\mathbf{W}\right) \\
        &- \nabla\Omega^{\star}_{+,\mathrm{up}}\left(\mathbf{V}_{\mathrm{up}}\mathrm{Diag}{(\bs{\Gamma}_{\mathrm{up}}}_{-\mathbf{b}} )\mathbf{V}_{\mathrm{up}}^{\top} +\mathbf{W}\right)\\
    \end{align*}
    These gradients come from elementary computations as in previous sections, and $\nabla\Omega^{\star}_{\mathrm{pen}}$ is given in \Cref{lem:conjugate_derivative_norms}.

    \paragraph{Recovering the solution from optimal Lagrange multipliers.}
    Similarly to \Cref{prop:dual formulation}, by denoting $(\widehat{\bs{\Gamma}}_{\mathrm{low}}, \widehat{\bs{\Gamma}}_{\mathrm{up}}, \widehat{\mathbf{W}}) \in \mathbb{R}_{+}^{n}\times \mathbb{R}_{+}^{n}\times \mathbb{S}^{n}$ the optimal variables of the dual problem, to reconstruct the matrices $\widehat{\mathbf{A}}_{\mathrm{low}}, \widehat{\mathbf{A}}_{\mathrm{up}}$ we have:
    \begin{align*}
        \widehat{\mathbf{A}}_{\mathrm{low}} 
        &=\nabla\Omega^{\star}_{+,\mathrm{low}}\left(\mathbf{V}_{\mathrm{low}}\mathrm{Diag}({\widehat{\bs{\Gamma}}_{\mathrm{low}-\mathbf{b}}})\mathbf{V}_{\mathrm{low}}^{\top} - \widehat{\mathbf{W}}\right)\\
        &= \frac{1}{2\lambda_{\mathrm{low}2}}\left[\mathbf{V}_{\mathrm{low}}\mathrm{Diag}({\widehat{\bs{\Gamma}}_{\mathrm{low}-\mathbf{b}}})\mathbf{V}_{\mathrm{low}}^{\top}-\widehat{\mathbf{W}}-\lambda_{\mathrm{low}1}\mathbf{I}_{n}\right]_{+}\\
        \widehat{\mathbf{A}}_{\mathrm{up}} &= \nabla\Omega^{\star}_{+,\mathrm{up}}\left(\mathbf{V}_{\mathrm{up}}\mathrm{Diag}(\widehat{\bs{\Gamma}}_{\mathrm{up}-\mathbf{b}})\mathbf{V}_{\mathrm{up}}^{\top}+\widehat{\mathbf{W}}\right)\\
        &= \frac{1}{2\lambda_{\mathrm{up}2}}\left[\mathbf{V}_{\mathrm{up}}\mathrm{Diag}({\widehat{\bs{\Gamma}}_{\mathrm{up}-\mathbf{b}}})\mathbf{V}_{\mathrm{up}}^{\top}+\widehat{\mathbf{W}}-\lambda_{\mathrm{up}1}\mathbf{I}_{n}\right]_{+}
    \end{align*}    
    
\end{proof}

\newpage
\subsection{Asymmetric problem with training set penalty}
\label{sec:asymmetric problem with penalty 1}

For the asymmetric setting with training set penalty, our infinite dimensional problem writes:
\begin{align}
    \underset{
        \substack{
            \mathcal{A}_{\mathrm{low}} \in \mathcal{S}_{+}\left(\mathcal{H}_{\mathrm{low}}\right) \\
            \mathcal{A}_{\mathrm{up}} \in \mathcal{S}_{+}\left(\mathcal{H}_{\mathrm{up}}\right)
        }
    }{\inf}
    \quad& \frac{b}{n}\sum_{i=1}^{n} \left(f_{\mathcal{A_{\mathrm{low}}}}(X_{i}) + f_{\mathcal{A_{\mathrm{up}}}}(X_{i})\right) + \Omega_{\mathrm{low}}(\mathcal{A}_{\mathrm{low}}) + \Omega_{\mathrm{up}}(\mathcal{A_{\mathrm{up}}}) +\lambda_{\mathrm{pen}} \sum_{i=1}^{n}\left(f_{\mathcal{A_{\mathrm{low}}}}(X_{i}) - f_{\mathcal{A_{\mathrm{up}}}}(X_{i})\right)^2\nonumber \\
    \mathrm{s.t.} \quad& r_{\mathrm{low}}\left(X_i, Y_i\right) - f_{\mathcal{A}_{\mathrm{low}}}(X_{i})\leq 0, \;i \in \left[n\right] \label{eq:infdim2} \\
    \quad& r_{\mathrm{up}}\left(X_i, Y_i\right) - f_{\mathcal{A_{\mathrm{up}}}}(X_{i})\leq 0, \;i \in \left[n\right]\nonumber.
\end{align}
Once again, the problem is not separable because of the penalty term. To derive a representer theorem, we use \Cref{thm:generalization_mf_p_operators}, from which the following theorem is a special case.

\begin{theorem}[Representer theorem with training set penalty]
\label{thm:representer theorem with penalty 1}
    Let \((b,\lambda_{\mathrm{low}{1}}, \lambda_{\mathrm{up}{1}})\in\mathbb{R}_{+}^{3}\) and \(\lambda_{\mathrm{low}{2}}, \lambda_{\mathrm{up}{2}}, \lambda_{\mathrm{pen}}>0\).
    Then Problem (\ref{eq:infdim2}) admits a unique solution \((f_{\mathbf{B}_{\mathrm{low}}^{\star}}, f_{\mathbf{B}_{\mathrm{up}}^{\star}})\) of the form \((f_{\mathbf{B}_{\mathrm{low}}^{\star}}(X), f_{\mathbf{B}_{\mathrm{up}}^{\star}}(X)) = \bigl(\mathbf{k}_{\mathrm{low}}(X)^{\top}\mathbf{B}_{\mathrm{low}}^{\star}\mathbf{k}_{\mathrm{low}}(X),\ \mathbf{k}_{\mathrm{up}}(X)^{\top}\mathbf{B}_{\mathrm{up}}^{\star}\mathbf{k}_{\mathrm{up}}(X)\bigr)\) for some matrices \(\bigl(\mathbf{B}_{\mathrm{low}}^{\star}, \mathbf{B}_{\mathrm{up}}^{\star}\bigr)\in (\mathbb{S}_{+}^{n})^{2}\).
\end{theorem}
\begin{proof}
    We apply \Cref{thm:generalization_mf_p_operators} to the case $p=2$, with the loss function
    \begin{equation*}
        L((f_{{\mathcal{A}_{\mathrm{low}}}}(X_i), f_{{\mathcal{A}_{\mathrm{up}}}}(X_i))_{1\leq i\leq n}) = \begin{cases}
            \frac{b}{n}\sum_{i=1}^{n} \left(f_{\mathcal{A_{\mathrm{low}}}}(X_{i}) + f_{\mathcal{A_{\mathrm{up}}}}(X_{i})\right) \\
            + \lambda_{\mathrm{pen}} \sum_{i=1}^{n}\left(f_{\mathcal{A_{\mathrm{low}}}}(X_{i}) - f_{\mathcal{A_{\mathrm{up}}}}(X_{i})\right)^2 &\mathrm{if} \; (\mathcal{A}_{\mathrm{low}},\mathcal{A}_{\mathrm{up}})\in C\\
            +\infty & \mathrm{otherwise}
        \end{cases}
    \end{equation*}
    where $C$ is the set of all positive definite operators $(\mathcal{A}_{\mathrm{low}},\mathcal{A}_{\mathrm{up}})$ such that
    \begin{equation*}
         \begin{cases}
            r_{\mathrm{low}}\left(X_i, Y_i\right) - f_{\mathcal{A}_{\mathrm{low}}}(X_{i})\leq 0, \;i \in \left[n\right]\\
            r_{\mathrm{up}}\left(X_i, Y_i\right) - f_{\mathcal{A_{\mathrm{up}}}}(X_{i})\leq 0, \;i \in \left[n\right]
        \end{cases}
    \end{equation*}
     and with the penalty function $\Omega((\mathcal{A}_{\mathrm{low}}, \mathcal{A}_{\mathrm{up}})) = \Omega_{\mathrm{low}}(\mathcal{A}_\mathrm{low}) + \Omega_{\mathrm{up}}(\mathcal{A}_\mathrm{up})$, where $\Omega_{\mathrm{low}}$ and $\Omega_{\mathrm{up}}$ are defined as in \Cref{eq:regularizer function on one operator}.
     
     First, from Remark $3$ in \citet{marteauferey2020nonparametricmodelsnonnegativefunctions}, $\Omega_{\mathrm{low}}$ and $\Omega_{\mathrm{up}}$ verify Assumption \ref{assumption:regularizer}. Then, since $L\colon \mathbb{R}^{2n} \rightarrow \mathbb{R}$ is lower semi-continuous (notice that it is linear and bounded below by $0$) and $\Omega$ writes as a sum as in \Cref{eq:penalty_functions_1}, we can apply \Cref{thm:generalization_mf_p_operators} with $p=2$ to deduce that the solution is entirely characterized by two PSD matrices \(\bigl(\mathbf{B}_{\mathrm{low}}^{\star}, \mathbf{B}_{\mathrm{up}}^{\star}\bigr)\in (\mathbb{S}_{+}^{n})^{2}\) and
    \begin{equation*}
        (f_{\mathbf{B}_{\mathrm{low}}^{\star}}(X), f_{\mathbf{B}_{\mathrm{up}}^{\star}}(X)) = \bigl(\mathbf{k}_{\mathrm{low}}(X)^{\top}\mathbf{B}_{\mathrm{low}}^{\star}\mathbf{k}_{\mathrm{low}}(X),\ \mathbf{k}_{\mathrm{up}}(X)^{\top}\mathbf{B}_{\mathrm{up}}^{\star}\mathbf{k}_{\mathrm{up}}(X)\bigr).
    \end{equation*}
\end{proof}

Once again, a solution of the previous representer theorem can be recovered with matrices $\mathbf{A}_{\mathrm{low}}$, $\mathbf{A}_{\mathrm{up}}$ instead of $\mathbf{B}_{\mathrm{low}}$, $\mathbf{B}_{\mathrm{up}}$.

\begin{proposition}[Representer theorem with penalty $1$, formulation $\mathbf{A}$]
\label{thm:representer theorem with penalty 1 formulation A}
    Let \((b,\lambda_{\mathrm{low}{1}}, \lambda_{\mathrm{up}{1}})\in\mathbb{R}_{+}^{3}\) and \(\lambda_{\mathrm{low}{2}}, \lambda_{\mathrm{up}{2}}, \lambda_{\mathrm{pen}}>0\).
    The following problem admits a unique solution if $L$ is convex:
    \begin{align}
        \underset{
            \substack{
                \mathbf{A}_{\mathrm{low}} \in \mathbb{S}_{+}^{n} \\
                \mathbf{A}_{\mathrm{up}} \in \mathbb{S}_{+}^{n}
            }
        }{\inf}
        \quad& \frac{b}{n}\sum_{i=1}^{n} \left(\tilde{f}_{\mathbf{A_{\mathrm{low}}}}(X_{i}) + \tilde{f}_{\mathbf{A_{\mathrm{up}}}}(X_{i})\right) + \Omega_{\mathrm{low}}(\mathbf{A}_{\mathrm{low}}) + \Omega_{\mathrm{up}}(\mathbf{A_{\mathrm{up}}}) +\lambda_{\mathrm{pen}} \sum_{i=1}^{n}\left(\tilde{f}_{\mathbf{A_{\mathrm{low}}}}(X_{i}) - \tilde{f}_{\mathbf{A_{\mathrm{up}}}}(X_{i})\right)^2\nonumber \\
        \mathrm{s.t.} \quad& r_{\mathrm{low}}\left(X_i, Y_i\right) - \tilde{f}_{\mathbf{A}_{\mathrm{low}}}(X_{i})\leq 0, \;i \in \left[n\right] \label{eq:formulationA penalty 1} \\
        \quad& r_{\mathrm{up}}\left(X_i, Y_i\right) - \tilde{f}_{\mathbf{A_{\mathrm{up}}}}(X_{i})\leq 0, \;i \in \left[n\right]\nonumber.
    \end{align}
    Moreover, for any given solution $\mathbf{A}_{\mathrm{low}}^{\star}, \mathbf{A}_{\mathrm{up}}^{\star} \in \mathbb{S}_{+}^{n}$ of \Cref{eq:formulationA penalty 1}, the function $(\tilde{f}_{\mathbf{A}_{\mathrm{low}}^{\star}}, \tilde{f}_{\mathbf{A}_{\mathrm{up}}^{\star} })$ is a minimizer of \Cref{eq:infdim2}.
    
\end{proposition}
\begin{proof}
    This is a direct application of \Cref{prop:formulation A p operators} with $p=2$, and the penalty and loss functions as in the proof of \Cref{thm:representer theorem with penalty 1}.
\end{proof}

We now exhibit a dual formulation for Problem (\ref{eq:formulationA penalty 1}).

\begin{proposition}[Dual formulation with training set penalty]
\label{prop:dual formulation penalty 1}
    Let \((b,\lambda_{\mathrm{low}{1}}, \lambda_{\mathrm{up}{1}})\in\mathbb{R}_{+}^{3}\) and \(\lambda_{\mathrm{low}{2}}, \lambda_{\mathrm{up}{2}}, \lambda_{\mathrm{pen}}>0\). Problem (\ref{eq:formulationA penalty 1}) admits a dual formulation of the form
    \begin{align}
    \label{eq:asymmetric_dual_penalty_1}
        \underset{
        \substack{
            \bs{\Gamma}_{\mathrm{low}} \in \mathbb{R}_+^{n},\bs{\Gamma}_{\mathrm{up}} \in \mathbb{R}_+^{n}\\
            \bs{\alpha}_0 \in \mathbb{R}^{n}
        }
        }{\sup}
        &(\bs{\Gamma}_{\mathrm{up}}-\bs{\Gamma}_{\mathrm{low}})\mathbf{r}^{\top}
        - \frac{1}{4\lambda_{\mathrm{pen}}} \bs{\alpha}_0\bs{\alpha}_0^{\top}\\
        &-\Omega^{\star}_{+,\mathrm{low}}(\mathbf{V}_{\mathrm{low}}\mathrm{Diag}(({\bs{\Gamma}_{\mathrm{low}}}+\bs{\alpha}_0)_{-\mathbf{b}})\mathbf{V}_{\mathrm{low}}^{\top})
        -\Omega^{\star}_{+,\mathrm{up}}(\mathbf{V}_{\mathrm{up}}\mathrm{Diag}(({\bs{\Gamma}_{\mathrm{up}}}-\bs{\alpha}_0)_{-\mathbf{b}})\mathbf{V}_{\mathrm{up}}^{\top}) \nonumber
    \end{align}
    where $\mathbf{r}$ is the vector of residuals $r_i=Y_i-m(X_i)$, \(\Omega_{+,(\cdot)}^{\star}(\mathbf{B}) = \frac{1}{4\lambda_{(\cdot)2}}\lVert \left[\mathbf{B}-\lambda_{(\cdot)1}\mathbf{I}_{n}\right]_{+}\rVert_{F}^{2}\) and \(\forall x \in \mathbb{R}, \;\mathrm{Diag}\left({(\cdot)_x}\right) := \mathrm{Diag}\left({(\cdot)}\right) + \frac{x}{n}\mathbf{I}_{n}\).
    Moreover, if \((\widehat{\bs{\Gamma}}_{\mathrm{low}}, \widehat{\bs{\Gamma}}_{\mathrm{up}}, \widehat{\bs{\alpha}}_0)\) is a solution of \Cref{eq:asymmetric_dual_penalty_1}, a solution of Problem \ref{thm:representer theorem with penalty 1 formulation A} can be retrieved as 
    \begin{align*}
        \widehat{\mathbf{A}}_{\mathrm{low}} &= \frac{1}{2\lambda_{\mathrm{low}2}}\left[\mathbf{V}_{\mathrm{low}}\mathrm{Diag}((\widehat{\bs{\Gamma}}_{\mathrm{low}}+\widehat{\bs{\alpha}}_0)_{-\mathbf{b}})\mathbf{V}_{\mathrm{low}}^{\top}-\lambda_{\mathrm{low}{1}}\mathbf{I}_{n}\right]_{+}\\
        \widehat{\mathbf{A}}_{\mathrm{up}} &= \frac{1}{2\lambda_{\mathrm{up}2}}\left[\mathbf{V}_{\mathrm{up}}\mathrm{Diag}((\widehat{\bs{\Gamma}}_{\mathrm{up}}-\widehat{\bs{\alpha}}_0)_{-\mathbf{b}})\mathbf{V}_{\mathrm{up}}^{\top}-\lambda_{\mathrm{up}{1}}\mathbf{I}_{n}\right]_{+}.
    \end{align*}
\end{proposition}

\begin{proof}
    As opposed to the separable case in \Cref{prop:dual formulation}, the interaction between $\mathbf{A}_{\mathrm{low}}$ and $\mathbf{A}_{\mathrm{up}}$ in the penalty term requires more attention. The dual problem is defined as
    \begin{equation}
        d = \underset{
        \substack{
            \bs{\Gamma}_{\mathrm{low}} \in \mathbb{R}_+^{n}\\
            \bs{\Gamma}_{\mathrm{up}} \in \mathbb{R}_+^{n}
        }
        }{\sup}
        \underset{
        \substack{
            \mathbf{A}_{\mathrm{low}} \in \mathbb{S}^{n}\\
            \mathbf{A}_{\mathrm{up}} \in \mathbb{S}^{n}
        }
        }{\inf}
        \mathcal{L}(\bs{\Gamma}_{\mathrm{low}}, \bs{\Gamma}_{\mathrm{up}}, \mathbf{A}_{\mathrm{low}}, \mathbf{A}_{\mathrm{up}})
        = \underset{
        \substack{
            \bs{\Gamma}_{\mathrm{low}} \in \mathbb{R}_+^{n}\\
            \bs{\Gamma}_{\mathrm{up}} \in \mathbb{R}_+^{n}
        }
        }{\sup} (\bs{\Gamma}_{\mathrm{up}}-\bs{\Gamma}_{\mathrm{low}})\mathbf{r}^{\top} + D(\bs{\Gamma}_{\mathrm{low}}, \bs{\Gamma}_{\mathrm{up}}) \label{eq:primal problem penalty 0}
    \end{equation}
    where we denote 
    \begin{align}
    D(\bs{\Gamma}_{\mathrm{low}}, \bs{\Gamma}_{\mathrm{up}})\coloneq \underset{
        \substack{
            \mathbf{A}_{\mathrm{low}} \in \mathbb{S}^{n},\;
            \mathbf{A}_{\mathrm{up}} \in \mathbb{S}^{n}
        }
        }{\inf}
        \overline{\mathcal{L}}(\bs{\Gamma}_{\mathrm{low}}, \bs{\Gamma}_{\mathrm{up}}, \mathbf{A}_{\mathrm{low}}, \mathbf{A}_{\mathrm{up}}) \label{eq:primal problem penalty 1}
    \end{align}
        with $\overline{\mathcal{L}}$ defined as
    \begin{align*}
        \overline{\mathcal{L}}(\bs{\Gamma}_{\mathrm{low}}, \bs{\Gamma}_{\mathrm{up}}, \mathbf{A}_{\mathrm{low}}, \mathbf{A}_{\mathrm{up}}) = 
        &\sum_{i=1}^{n} (\frac{b}{n} - \Gamma_{\mathrm{low},i})\tilde{f}_{\mathbf{A}_{\mathrm{low}}}(X_{i}) + \sum_{i=1}^{n} (\frac{b}{n} - \Gamma_{\mathrm{up},i})\tilde{f}_{\mathbf{A}_{\mathrm{up}}}(X_{i})\\
        &+ \lambda_{\mathrm{pen}}\sum_{i=1}^{n}(\tilde{f}_{\mathbf{A}_{\mathrm{low}}}(X_{i}) - \tilde{f}_{\mathbf{A}_{\mathrm{up}}}(X_{i}))^2\\
        & + \Omega_{+,\mathrm{low}}(\mathbf{A}_{\mathrm{low}}) + \Omega_{+,\mathrm{up}}(\mathbf{A}_{\mathrm{up}}).
    \end{align*}

    Contrary to \citet{allain2025scalableandadaptivepredictionbandsusingkerenlsumofsquares} where $D()$ has an explicit expression, here we follow \citet{marteauferey2020nonparametricmodelsnonnegativefunctions, allain2025scalableandadaptivepredictionbandsusingkerenlsumofsquares} and use Theorem $3.3.5$ from \citet{borwein2006convex} to get the following lemma.
    \begin{lemma}
    \label{lem:dual problem penalty 1}
        The dual problem associated to \Cref{eq:primal problem penalty 1} writes
        \begin{align}
        \label{eq:dual problem penalty 1}
            D(\bs{\Gamma}_{\mathrm{low}}, \bs{\Gamma}_{\mathrm{up}}) = \underset{
            \bs{\alpha}_{0} \in \mathbb{R}^{n}
            }{\sup}
            &-\frac{1}{4\lambda_{\mathrm{pen}}}\bs{\alpha}_0\bs{\alpha}_0^{\top}\\
            &- \Omega_{+,\mathrm{low}}^{\star}(\mathbf{V}_{\mathrm{low}}\mathrm{Diag}((\bs{\Gamma}_{\mathrm{low}}+\bs{\alpha}_0)_{-\mathbf{b}})\mathbf{V}_{\mathrm{low}}^{\top})\\ 
            &- \Omega_{+,\mathrm{up}}^{\star}(\mathbf{V}_{\mathrm{up}}\mathrm{Diag}(({\bs{\Gamma}_{\mathrm{up}}}-\bs{\alpha}_0)_{-\mathbf{b}}))\mathbf{V}_{\mathrm{up}}^{\top})\nonumber
        \end{align}.
    \end{lemma}
    \begin{proof}
    First notice that $D$ can be decomposed as:
    \begin{equation*}
        D(\bs{\Gamma}_{\mathrm{low}}, \bs{\Gamma}_{\mathrm{up}}) = \underset{
        \mathbf{x} \in \mathbb{S}^{n}\times \mathbb{S}^{n}
        }{\inf}
        L(R\mathbf{x}) + f(\mathbf{x})
    \end{equation*}
    where $\mathbf{x} = (\mathbf{A}_{\mathrm{low}}, \mathbf{A}_{\mathrm{up}})$, $R:\  \mathbf{x} \in  \mathbb{S}^{n}\times \mathbb{S}^{n} \mapsto (\ldots, \tilde{f}_{\mathbf{A}_{\mathrm{low}}}(X_{i}), \ldots, \tilde{f}_{\mathbf{A}_{\mathrm{up}}}(X_{i}), \ldots)$, $f(\mathbf{x})= \Omega_{+,\mathrm{low}}(\mathbf{A}_{\mathrm{low}}) + \Omega_{+,\mathrm{up}}(\mathbf{A}_{\mathrm{up}})$ and 
    \begin{equation*}
        L(R\mathbf{x}) = \sum_{i=1}^{n} (\frac{b}{n} - \Gamma_{\mathrm{low},i})\tilde{f}_{\mathbf{A}_{\mathrm{low}}}(X_{i}) + \sum_{i=1}^{n} (\frac{b}{n} - \Gamma_{\mathrm{up},i})\tilde{f}_{\mathbf{A}_{\mathrm{up}}}(X_{i})\\
        + \lambda_{\mathrm{pen}}\sum_{i=1}^{n}(\tilde{f}_{\mathbf{A}_{\mathrm{low}}}(X_{i}) - \tilde{f}_{\mathbf{A}_{\mathrm{up}}}(X_{i}))^2.
    \end{equation*}

        Applying Theorem $3.3.5$ from \citet{borwein2006convex} to \Cref{eq:primal problem penalty 1} gives
        \begin{equation}
        \label{eq:dual function penalty 1}
            D(\bs{\Gamma}_{\mathrm{low}}, \bs{\Gamma}_{\mathrm{up}}) = \underset{
            \bs{\alpha}^{\star} = (\bs{\alpha}_{\mathrm{low}}^{\star}, \bs{\alpha}_{\mathrm{up}}^{\star}) \in \mathbb{R}^{2n}
            }{\sup}
            -L^{\star}(-\bs{\alpha}^{\star}) - f^{\star}(R^{\star}\bs{\alpha}^{\star}).
        \end{equation}
        We first derive the expression for $L^{\star}$, by expressing $L$ in quadratic form.
        Note that $L^{\star}(-\bs{\alpha}^{\star}) = [\tilde{L}(\mathbf{u})]^{\star}(\bs{\alpha}^{\star})$, where $\tilde{L}(\mathbf{u}) = L(-\mathbf{u})$, such that:
        \begin{align*}
            \tilde{L}(\mathbf{u}) &= L(-\mathbf{u})\\
            &= \sum_{i=1}^{n}(\Gamma_{\mathrm{low}, i}-\frac{b}{n})u_{\mathrm{low}, i} + \sum_{i=1}^{n}(\Gamma_{\mathrm{up}, i}-\frac{b}{n})u_{\mathrm{up}, i} + \lambda_{\mathrm{pen}}\sum_{i=1}^{n}(u_{\mathrm{up}, i}-u_{\mathrm{low}, i})^2\\
            &= \mathbf{a}^{\top}\mathbf{u} + \frac{1}{2}\mathbf{u}^{\top}\mathbf{M}\mathbf{u}
        \end{align*}
        where $\mathbf{a} = [\ldots, \Gamma_{\mathrm{low}, i}-\frac{b}{n}, \ldots, \Gamma_{\mathrm{up}, i}-\frac{b}{n}, \ldots]$ and $
            \mathbf{M} = 2\lambda_{\mathrm{pen}}
            \begin{bmatrix}
                I_n & -I_n \\
                -I_n & I_n
            \end{bmatrix}
        $. It follows that
        \begin{equation}
        \label{eq:dual_formulation_L_tilde_penalty_1}
            \tilde{L}^{\star}(\bs{\alpha}^{\star}) = \frac{1}{2}(\bs{\alpha}^{\star}-\mathbf{a})^{\top}\mathbf{M}^{\dagger}(\bs{\alpha}^{\star}-\mathbf{a})
        \end{equation}
        for $\bs{\alpha}^{\star} \in \mathrm{Range}(\mathbf{M}) + \mathbf{a}$ \citep{becker2019quasi} and where $\mathbf{M}^{\dagger}$ is the pseudo-inverse of $\mathbf{M}$. It remains to compute $\mathbf{M}^{\dagger}$ and $\mathrm{Range}(\mathbf{M}) + \mathbf{a}$.

        First, observe that the pseudo-inverse of the matrix $\begin{bmatrix}
                I_n & -I_n \\
                -I_n & I_n
            \end{bmatrix}$ is given by $\frac{1}{4}\begin{bmatrix}
                I_n & -I_n \\
                -I_n & I_n
            \end{bmatrix}$, which leads to:
        \begin{equation*}
            \mathbf{M}^{\dagger} = \left(2\lambda_{\mathrm{pen}}
            \begin{bmatrix}
                I_n & -I_n \\
                -I_n & I_n
            \end{bmatrix}\right)^{\dagger}  = \frac{1}{8\lambda_{\mathrm{pen}}}\begin{bmatrix}
                I_n & -I_n \\
                -I_n & I_n
            \end{bmatrix} = \frac{1}{16\lambda_{\mathrm{pen}}^2}\mathbf{M}.
        \end{equation*}

        For identifying $\mathrm{Range}(\mathbf{M}) + \mathbf{a}$, let $\mathbf{v} \in \mathbb{R}^{2n}$. Then,
        \begin{equation*}
            \mathbf{M}\mathbf{v} 
            = 2\lambda_{\mathrm{pen}}
            \begin{bmatrix}
                I_n & -I_n \\
                -I_n & I_n
            \end{bmatrix}
            \begin{bmatrix}
                v_{1:n} \\
                v_{(n+1):2n}
            \end{bmatrix}
            = 2\lambda_{\mathrm{pen}}
            \begin{bmatrix}
                v_{1:n} - v_{(n+1):2n} \\
                v_{(n+1):2n} - v_{1:n}
            \end{bmatrix}
        \end{equation*}
        and it follows that
        \begin{equation*}
            \mathrm{Range}(\mathbf{M}) = \left\{\mathbf{v}\in \mathbb{R}^{2n}\colon\; v_{1:n}=-v_{(n+1):2n}\right\}.
        \end{equation*}
        Denoting $\bs{\alpha}_0 \in \mathbb{R}^{n}$ the vector which consists of the first $n$ components of $\mathbf{v}$, then $\bs{\alpha}^{\star}$ in \Cref{eq:dual_formulation_L_tilde_penalty_1} can be parameterized as
        \begin{equation}
        \label{eq:definition alpha star}
            \bs{\alpha}^{\star} = (\bs{\alpha}_{\mathrm{low}}^{\star}, \bs{\alpha}_{\mathrm{up}}^{\star}) = [\bs{\alpha}_0 + \mathbf{a}_{1:n}, \; -\bs{\alpha}_0 + \mathbf{a}_{(n+1):2n}]\in \mathrm{Range}(\mathbf{M}) + \mathbf{a},
        \end{equation}
        such that 
        \Cref{eq:dual_formulation_L_tilde_penalty_1} becomes
        \begin{equation*}
            L^{\star}(-\bs{\alpha}^{\star}) = \frac{1}{4\lambda_{\mathrm{pen}}}\bs{\alpha}_0\bs{\alpha}_0^{\top}.
        \end{equation*}

        Next, since $f$ is defined as a separable sum, its conjugate is the sum of the conjugates: $f^{\star}(\mathbf{A}_{\mathrm{low}}^{\star}, \mathbf{A}_{\mathrm{low}}^{\star}) = \Omega_{+,\mathrm{low}}^{\star}(\mathbf{A}_{\mathrm{low}}^{\star}) + \Omega_{+,\mathrm{up}}^{\star}(\mathbf{A}_{\mathrm{up}}^{\star})$. Finally, to compute $R^{\star}\bs{\alpha}^{\star}$ notice that
        \begin{align*}
            R\mathbf{x} &= (\ldots, \tilde{f}_{\mathbf{A}_{\mathrm{low}}}(X_{i}), \ldots, \tilde{f}_{\mathbf{A}_{\mathrm{up}}}(X_{i}), \ldots)\\
            &= \Bigl(\mathrm{Diag}(\mathbf{V}_{\mathrm{low}}\mathbf{A}_{\mathrm{low}}\mathbf{V}_{\mathrm{low}}^{\top}), \mathrm{Diag}(\mathbf{V}_{\mathrm{up}}\mathbf{A}_{\mathrm{up}}\mathbf{V}_{\mathrm{up}}^{\top})\Bigr).
        \end{align*}
        By considering the Frobenius scalar product on $(\mathbb{S}^{n})^2$, we thus have
        \begin{align*}
            R^{\star}\bs{\alpha}^{\star} 
            &= \Bigl(\mathbf{V}_{\mathrm{low}}\mathrm{Diag}(\bs{\alpha}_{\mathrm{low}}^{\star})\mathbf{V}_{\mathrm{low}}^{\top}, \mathbf{V}_{\mathrm{up}}\mathrm{Diag}(\bs{\alpha}_{\mathrm{up}}^{\star})\mathbf{V}_{\mathrm{up}}^{\top}\Bigr)\\
            &= (R_{\mathrm{low}}^{\star}\bs{\alpha}_{\mathrm{low}}^{\star}, R_{\mathrm{up}}^{\star}\bs{\alpha}_{\mathrm{up}}^{\star}).
        \end{align*}
        where $R_{(\cdot)}^{\star}\bs{\alpha}_{(\cdot)}^{\star}=\mathbf{V}_{(\cdot)}\mathrm{Diag}(\bs{\alpha}_{(\cdot)}^{\star})\mathbf{V}_{(\cdot)}^{\top}$.
        The value of $f^{\star}$ at $R^{\star}\bs{\alpha}^{\star}$ is finally given by
        \begin{align*}
            f^{\star}(R^{\star}\bs{\alpha}^{\star}) 
            &= \Omega_{+,\mathrm{low}}^{\star}(\mathbf{V}_{\mathrm{low}}\mathrm{Diag}(\bs{\alpha}_{\mathrm{low}}^{\star})\mathbf{V}_{\mathrm{low}}^{\top}) + \Omega_{+,\mathrm{up}}^{\star}(\mathbf{V}_{\mathrm{up}}\mathrm{Diag}(\bs{\alpha}_{\mathrm{up}}^{\star})\mathbf{V}_{\mathrm{up}}^{\top})\\
            &= \Omega_{+,\mathrm{low}}^{\star}(\mathbf{V}_{\mathrm{low}}\mathrm{Diag}(\bs{\alpha}_0+\mathbf{a})\mathbf{V}_{\mathrm{low}}^{\top}) + \Omega_{+,\mathrm{up}}^{\star}(\mathbf{V}_{\mathrm{up}}\mathrm{Diag}(-\bs{\alpha}_0+\mathbf{a})\mathbf{V}_{\mathrm{up}}^{\top}),
        \end{align*}
        which concludes the proof of \Cref{lem:dual problem penalty 1}.
    \end{proof}
    Putting things together, using \Cref{lem:dual problem penalty 1} and \Cref{eq:primal problem penalty 0}, our main dual problem writes
    \begin{align*}
        d = \underset{
        \substack{
            \bs{\Gamma}_{\mathrm{low}} \in \mathbb{R}_+^{n},\bs{\Gamma}_{\mathrm{up}} \in \mathbb{R}_+^{n}\\
            \bs{\alpha}_0 \in \mathbb{R}^{n}
        }
        }{\sup}
        &(\bs{\Gamma}_{\mathrm{up}}-\bs{\Gamma}_{\mathrm{low}})\mathbf{r}^{\top}
        - \frac{1}{4\lambda_{\mathrm{pen}}} \bs{\alpha}_0\bs{\alpha}_0^{\top}\\
        &-\Omega^{\star}_{+,\mathrm{low}}(\mathbf{V}_{\mathrm{low}}\mathrm{Diag}(({\bs{\Gamma}_{\mathrm{low}}}+\bs{\alpha}_0)_{-\mathbf{b}})\mathbf{V}_{\mathrm{low}}^{\top})
        -\Omega^{\star}_{+,\mathrm{up}}(\mathbf{V}_{\mathrm{up}}\mathrm{Diag}(({\bs{\Gamma}_{\mathrm{up}}}-\bs{\alpha}_0)_{-\mathbf{b}})\mathbf{V}_{\mathrm{up}}^{\top}) \nonumber\\
        =\underset{
        \substack{
            \bs{\Gamma}_{\mathrm{low}} \in \mathbb{R}_+^{n},\bs{\Gamma}_{\mathrm{up}} \in \mathbb{R}_+^{n}\\
            \bs{\alpha}_0 \in \mathbb{R}^{n}
        }
        }{\sup}& g(\bs{\Gamma}_{\mathrm{low}},\bs{\Gamma}_{\mathrm{up}},\bs{\alpha}_0).
    \end{align*}
    \paragraph{Gradient computation.}
    \begin{equation*}
        \frac{\partial g}{\partial\bs{\Gamma}_{\mathrm{low}}} = 
        -\mathbf{r} 
        - \mathrm{Diag}\left(\mathbf{V}_{\mathrm{low}}^{\top}\left[\nabla\Omega^{\star}_{+,\mathrm{low}}\left(\mathbf{V}_{\mathrm{low}}\mathrm{Diag}({(\bs{\Gamma}_{\mathrm{low}}}+\bs{\alpha}_0)_{-\mathbf{b}})\mathbf{V}_{\mathrm{low}}^{\top}\right)\right]^{\top}\mathbf{V}_{\mathrm{low}}\right)
    \end{equation*}
    \begin{equation*}
        \frac{\partial g}{\partial\bs{\Gamma}_{\mathrm{up}}} = 
        \mathbf{r} 
        - \mathrm{Diag}\left(\mathbf{V}_{\mathrm{up}}^{\top}\left[\nabla\Omega^{\star}_{+,\mathrm{up}}\left(\mathbf{V}_{\mathrm{up}}\mathrm{Diag}({(\bs{\Gamma}_{\mathrm{up}}}-\bs{\alpha}_0)_{-\mathbf{b}})\mathbf{V}_{\mathrm{up}}^{\top}\right)\right]^{\top}\mathbf{V}_{\mathrm{up}}\right)
    \end{equation*}
    \begin{align*}
        \frac{\partial g}{\partial\bs{\alpha}_{0}} = 
        -\frac{1}{2\lambda_{\mathrm{pen}}}\bs{\alpha}_0
        &- \mathrm{Diag}\left(\mathbf{V}_{\mathrm{low}}^{\top}\left[\nabla\Omega^{\star}_{+,\mathrm{low}}\left(\mathbf{V}_{\mathrm{low}}\mathrm{Diag}({(\bs{\Gamma}_{\mathrm{low}}}+\bs{\alpha}_0)_{-\mathbf{b}})\mathbf{V}_{\mathrm{low}}^{\top}\right)\right]^{\top}\mathbf{V}_{\mathrm{low}}\right)\\
        &+ \mathrm{Diag}\left(\mathbf{V}_{\mathrm{up}}^{\top}\left[\nabla\Omega^{\star}_{+,\mathrm{up}}\left(\mathbf{V}_{\mathrm{up}}\mathrm{Diag}({(\bs{\Gamma}_{\mathrm{up}}}-\bs{\alpha}_0)_{-\mathbf{b}})\mathbf{V}_{\mathrm{up}}^{\top}\right)\right]^{\top}\mathbf{V}_{\mathrm{up}}\right)
    \end{align*}
    These gradients come from elementary computations following  \citet{allain2025scalableandadaptivepredictionbandsusingkerenlsumofsquares}.
    
    \paragraph{Recovering the solution from optimal Lagrange multipliers.}
    Similarly to \Cref{prop:dual formulation}, by denoting $(\widehat{\bs{\Gamma}}_{\mathrm{low}}, \widehat{ \bs{\Gamma}}_{\mathrm{up}}, \widehat{\bs{\alpha}}_0) \in \mathbb{R}_{+}^{n}\times \mathbb{R}_{+}^{n}\times \mathbb{R}^{n}$ the optimal variables of the dual problem, to reconstruct the matrices $\widehat{\mathbf{A}}_{\mathrm{low}}, \widehat{\mathbf{A}}_{\mathrm{up}}$ we have:
    \begin{align*}
        \widehat{\mathbf{A}}_{\mathrm{low}} &= \nabla\Omega^{\star}_{+,\mathrm{low}}\left(\mathbf{V}_{\mathrm{low}}\mathrm{Diag}({(\widehat{\bs{\Gamma}}_{\mathrm{low}}}+\widehat{\bs{\alpha}}_0)_{-\mathbf{b}})\mathbf{V}_{\mathrm{low}}^{\top}\right)\\
        &= \frac{1}{2\lambda_{\mathrm{low}2}}\left[\mathbf{V}_{\mathrm{low}}\mathrm{Diag}((\widehat{\bs{\Gamma}}_{\mathrm{low}}+\widehat{\bs{\alpha}}_0)_{-\mathbf{b}})\mathbf{V}_{\mathrm{low}}^{\top}-\lambda_{\mathrm{low}{1}}\mathbf{I}_{n}\right]_{+}\\
        \widehat{\mathbf{A}}_{\mathrm{up}} &= \nabla\Omega^{\star}_{+,\mathrm{up}}\left(\mathbf{V}_{\mathrm{up}}\mathrm{Diag}({(\widehat{\bs{\Gamma}}_{\mathrm{up}}}-\widehat{\bs{\alpha}}_0)_{-\mathbf{b}})\mathbf{V}_{\mathrm{up}}^{\top}\right)\\
        &= \frac{1}{2\lambda_{\mathrm{up}2}}\left[\mathbf{V}_{\mathrm{up}}\mathrm{Diag}((\widehat{\bs{\Gamma}}_{\mathrm{up}}-\widehat{\bs{\alpha}}_0)_{-\mathbf{b}})\mathbf{V}_{\mathrm{up}}^{\top}-\lambda_{\mathrm{up}{1}}\mathbf{I}_{n}\right]_{+}
    \end{align*}
    see Theorem 8 in \citet{marteauferey2020nonparametricmodelsnonnegativefunctions} and Appendix A.2 in \citet{allain2025scalableandadaptivepredictionbandsusingkerenlsumofsquares}.
\end{proof}

\newpage
\subsection{Error bounds}
\label{sec:apdx:error bounds}

To derive the error bounds for both penalties, we make the following assumption on the space $\mathcal{H}$ (Assumption 2(a) in \citet{rudi2025finding}).
\begin{assumption}
    \label{assumption:rkhs_1}
    For a bounded open set $\mathcal{X}\in\mathbb{R}^d$, the RKHS $\mathcal{H}$ of functions on $\mathcal{X}$ with norm $\Vert \cdot \Vert_{\mathcal{H}}$ satisfies $f\vert\mathcal{X}\in\mathcal{H}$, $\forall f \in C^{\infty}(\mathbb{R}^d)$. Moreover $\forall u,v\in\mathcal{H}$, $u\cdot v \in\mathcal{H}$ and $\exists M\geq 1$ such that
    \[\Vert u\cdot v \Vert_{\mathcal{H}} \leq \mathrm{M} \Vert u \Vert_{\mathcal{H}} \Vert v \Vert_{\mathcal{H}}.\]
\end{assumption}
This assumption has an important implication: if it holds, kernel SoS functions will live in the RKHS $\mathcal{H}$, and their norm can be controlled by the nuclear norm of their operator, as elaborated in the following lemma.
\begin{lemma}[Lemma 9 in \citet{rudi2025finding}.]
    \label{lemma:rudi}
    Let $\mathcal{X}$ and $\mathcal{H}$ satisfy Assumption \ref{assumption:rkhs_1}. If $\mathcal{A}$ is a trace-class operator, then $f_{\mathcal{A}}(X)=\langle \phi(X),\mathcal{A}\phi(X)\rangle \in \mathcal{H}$ and
    \[\Vert f_{\mathcal{A}}\Vert_{\mathcal{H}} \leq \mathrm{M} \Vert \mathcal{A} \Vert_{\star}.\]
\end{lemma}

\subsubsection{Operator penalty.}
We can now state the error bound with the operator penalty, which controls the difference between two kernel SoS functions with the nuclear norm of their operator difference.

\begin{proposition}[Error bound with operator penalty]
\label{prop:error bound penalty 2}
Let $\mathcal{H}$ be a reproducing kernel Hilbert space with associated kernel $k$ and feature map $\phi$ which satisfies Assumption \ref{assumption:rkhs_1}. If $k$ is bounded such that $\Vert k\Vert_{\infty}:= \sup_{x\in\mathcal{X}} \sqrt{k(X,X)} < \infty$, then for any two PSD operators $\mathcal{A}_1, \mathcal{A}_2 \in \mathbb{S}_{+}(\mathcal{H})$, we have
\begin{equation*}
    \sup_{x\in\mathcal{X}}\, \lvert f_{\mathcal{A}_1}(x) - f_{\mathcal{A}_2}(x) \rvert \leq \mathrm{M} \Vert k\Vert_{\infty} \,\lVert \mathcal{A}_1 - \mathcal{A}_2\rVert_{\star},
\end{equation*}
where $f_{\mathcal{A}_1}(X)=\langle \phi(X),\mathcal{A}_1\phi(X)\rangle \in \mathcal{H}$ and $f_{\mathcal{A}_2}(X)=\langle \phi(X),\mathcal{A}_2\phi(X)\rangle \in \mathcal{H}$ are the kernel sum-of-squares functions associated to $\mathcal{A}_1$ and $\mathcal{A}_2$, respectively.
\end{proposition}

\begin{proof}
Writing $\Delta f=f_1-f_2$, if $\mathcal{H}$ satisfies Assumption \ref{assumption:rkhs_1} then $\Delta f \in \mathcal{H}$, such that
\begin{align*}
\Delta f(X) = \langle \Delta f, k(X,\cdot)\rangle_{\mathcal{H}} \leq \Vert \Delta f \Vert_{\mathcal{H}} \Vert k\Vert_{\infty},
\end{align*}
see \citet{steinwart2008support} Lemma 4.23. We conclude the proof by applying \Cref{lemma:rudi} to $\Delta f = \langle \phi(X),(\mathcal{A}_1-\mathcal{A}_2)\phi(X)\rangle$ with $\mathcal{A}=\mathcal{A}_1-\mathcal{A}_2$.
\end{proof}

\subsubsection{Training set penalty.}

While the operator penalty controls the difference of kSoS functions through the difference of their associated operators, the training set penalty relies instead on the difference of the functions on the training points.\\
For this penalty, we first need a new assumption on the space $\mathcal{H}$ (Assumption 2(d) in \citet{rudi2025finding} with $m=1$).
\begin{assumption}
    \label{assumption:rkhs_2}
    For a bounded open set $\mathcal{X}\in\mathbb{R}^d$, the RKHS $\mathcal{H}$ of functions on $\mathcal{X}$ with associated kernel $k$ satisfies
    \[\max_{\vert \alpha\vert=1} \sup_{x,y\in\mathcal{X}} \vert \partial^{\alpha}_x \partial^{\alpha}_y k(x,y)\vert \leq \mathrm{D}^2 < \infty\]
    for some $\mathrm{D}\geq 1$.
\end{assumption}
Assumption \ref{assumption:rkhs_2} implies that $\mathcal{H}\subseteq C^1(\mathcal{X})$ and that $\forall u\in\mathcal{H}$,
\begin{equation}
    \label{eq:control_derivative_rkhs}
    \max_{\vert \alpha\vert=1} \sup_{x\in\mathcal{X}} \vert \partial^{\alpha}_x u(x)\vert \leq \mathrm{D}
\end{equation}
see Remark 2 in \citet{rudi2025finding}.\\
We also require the following geometric property on the domain $\mathcal{X}$:
\begin{assumption}
\label{assumption:geometric of the domain}
    $\mathcal{X}$ writes as $\mathcal{X} = \cup_{x\in S}\mathrm{B}_{r}(x)$, where $S$ is a bounded subset of $\mathbb{R}^d$ and $\mathrm{B}_{r}(x)$ is the ball of center $x$ and radius $r$.
\end{assumption}
In order to show that controlling the difference of functions on a finite subset $\widehat{X} = \{X_1, \ldots, X_n\}$ allows to control the difference of functions on their whole domain $\mathcal{X}$, we rely on traditional scattered data approximation techniques. In particular, they bring into play the distance of any data point $x$ to $\widehat{X}$, known as the \emph{fill-in distance} defined by:
\begin{equation*}
    \rho_{\widehat{X}, \mathcal{X}} = \sup_{x\in\mathcal{X}}\min_{X_i\in\widehat{X}}\lVert x-X_i\rVert.
\end{equation*}

\begin{proposition}[Error bound with training set penalty]
\label{prop:error bound penalty 1}
Let $\mathcal{X}$ satisfying Assumption \ref{assumption:geometric of the domain} and $\widehat{X}= \{X_1, \ldots, X_n\}$ be a finite subset of $\mathcal{X}$. Let $\mathcal{H}$ be a RKHS of functions defined on $\mathcal{X}$ with associated kernel $k$ satisfying Assumptions \ref{assumption:rkhs_1} and \ref{assumption:rkhs_2}. Then for any two PSD operators $\mathcal{A}_1, \mathcal{A}_2 \in \mathcal{S}_{+}(\mathcal{H})$, it holds
\begin{equation*}
    \sup_{x\in\mathcal{X}}\, \lvert f_{\mathcal{A}_1}(x)-f_{\mathcal{A}_2}(x) \rvert \leq 2\,C_{f_1,f_2} \rho_{\widehat{X}, \mathcal{X}} + \sqrt{\sum_{i=1}^{n}(f_1(X_i) - f_2(X_i))^{2}}
\end{equation*}
where $f_{\mathcal{A}_1}(X)=\langle \phi(X),\mathcal{A}_1\phi(X)\rangle \in \mathcal{H}$ and $f_{\mathcal{A}_2}(X)=\langle \phi(X),\mathcal{A}_2\phi(X)\rangle \in \mathcal{H}$ are the kernel sum-of-squares functions associated to $\mathcal{A}_1$ and $\mathcal{A}_2$, respectively, and $C_{f_1,f_2}=2d\mathrm{D}\mathrm{M}\lVert \mathcal{A}_1-\mathcal{A}_2 \rVert_{\star}$ with $\mathrm{D}$ and $\mathrm{M}$ constants depending on kernel $k$.
\end{proposition}
\begin{proof}
The proof relies on the following general result (\citet{wendland2005approximate}, \citet{rudi2025finding} Theorem $13$ with $k=m=0$) which states that for any function $f\colon \mathcal{X}\rightarrow\mathbb{R}$, if $\mathcal{X}$ satisfies Assumption \ref{assumption:geometric of the domain} and $f$ is at least $C^{1}(\mathcal{X})$ then
\begin{equation*}
    \sup_{x\in \mathcal{X}}\, \lvert f(x) \rvert \leq 2\,C_{f}\rho_{\widehat{X}, \mathcal{X}} + \max_{X_i \in\widehat{X}}\lvert f(X_i) \rvert
\end{equation*}
with $C_{f} = \sum_{\lvert \alpha \rvert = 1} \frac{1}{\alpha !}\sup_{x\in\mathcal{X}}\lvert \partial^{\alpha} f\rvert$.\\
The core idea of the proof is to apply this result to the difference of kSoS functions. First, notice that $\Delta f= f_1 - f_2 \in C^{1}$ as a difference of $C^{1}$ functions. Then, we need to explicit the bound $C_{f}$. By Assumption \ref{assumption:rkhs_2}, applying \Cref{eq:control_derivative_rkhs} to $\Delta f$ we have:
\begin{equation*}
    \max_{\lvert \alpha \rvert = 1}\sup_{x\in \mathcal{X}}\lvert \partial^{\alpha}\Delta f(x) \rvert \leq \mathrm{D} \lVert \Delta f \rVert_{\mathcal{H}}.
\end{equation*}
Next, applying \Cref{lemma:rudi} to $\Delta f$ we obtain $\lVert \Delta f \rVert_{\mathcal{H}} \leq \mathrm{M}\lVert \mathcal{A}_1 - \mathcal{A}_2 \rVert_{\star}$
where $\mathrm{M}$ is the constant from Assumption \ref{assumption:rkhs_1}.
Putting things together, we have:
\begin{align*}
    C_{\Delta f} &= \sum_{\lvert \alpha \rvert = 1} \frac{1}{\alpha !}\sup_{x\in\mathcal{X}}\lvert \partial^{\alpha} \Delta f(x)\rvert
    \leq d\mathrm{D} \mathrm{M}\lVert \mathcal{A}_1 - \mathcal{A}_2 \rVert_{\star}.
\end{align*}

To conclude the proof, we use
\begin{equation*}
    \max_{X_i \in\widehat{X}}\lvert \Delta f(X_i) \rvert \leq \sqrt{ \sum_{i=1}^{n} ( \Delta f(X_i))^{2}}.
\end{equation*}
\end{proof}

In our work, we focus on the Matérn $5/2$ kernel for its superior empirical performance. Crucially, the following remark shows that this kernel actually satisfies all the assumptions needed for \Cref{prop:error bound penalty 1} and \Cref{prop:error bound penalty 2}.
\begin{remark}
\label{remark:apdx:matern constant}
    Following Proposition $1$ in \citet{rudi2025finding}, the Matérn $5/2$ kernel with lengthscale \(\theta\) is a Sobolev kernel with $\nu = s-d/2$ which satisfies Assumptions \ref{assumption:rkhs_1} and \ref{assumption:rkhs_2} with constants
    \begin{equation*}
        \mathrm{M} = 2^{3+d/2}(2\pi)^{d/2} \quad \mathrm{and} \quad\mathrm{D} = \frac{1}{\theta}(2\pi)^{d/4}\sqrt{d/3}.
    \end{equation*}
\end{remark}

\subsubsection{Post-optimization error bounds.}

\Cref{prop:error bound penalty 2} and \Cref{prop:error bound penalty 1} give error bounds for the infinite dimensional operators. We need to show that the finite representation of those operators still controls the difference between the finite representation of the functions. This is done using the following lemma.
\begin{lemma}
    \label{lem:partial_isometries}
    Let $\mathcal{H}$ be a reproducing kernel Hilbert space with associated kernel $k$ and feature map $\phi$.
    For any matrix $\mathbf{A} \in \mathbb{S}^{n}$, any partial isometry $\mathrm{U}\colon \mathcal{H}\rightarrow \mathbb{R}^{n}$ and any $p\geq1$, we have
    \begin{equation*}
        \lVert \mathbf{A} \rVert_{p} = \lVert \mathrm{U}^{\star}\mathbf{A}\mathrm{U} \rVert_{p},
    \end{equation*}
    where $\lVert\cdot\rVert_{p}$ is the $p$-Schatten norm.
\end{lemma}
\begin{proof}
    Let $\mathrm{U}\colon \mathcal{H}\rightarrow \mathbb{R}^{n}$ be a partial isometry (i.e. $\mathrm{U}\mathrm{U}^{\star}$ is the identity over $\mathbb{R}^{n}$).

    By first noticing that $\left(\mathrm{U}^{\star}\mathbf{A}\mathrm{U}\right)^{\star}\left(\mathrm{U}^{\star}\mathbf{A}\mathrm{U}\right) = \mathrm{U}^{\star}\mathbf{A}^{\star}\mathbf{A}\mathrm{U}$, non-zero singular values of $\mathrm{U}^{\star}\mathbf{A}\mathrm{U}$ are given by
    \begin{align*}
        \sigma(\mathrm{U}^{\star} \mathbf{A} \mathrm{U}) &= \sqrt{\lambda(\left(\mathrm{U}^{\star}\mathbf{A}\mathrm{U}\right)^{\star}\left(\mathrm{U}^{\star}\mathbf{A}\mathrm{U}\right))}\\
        &=\sqrt{\lambda(\mathrm{U}^{\star}\mathbf{A}^{\star}\mathbf{A}\mathrm{U})}\\
        &=\sqrt{\lambda(\mathbf{A}^{\star}\mathbf{A})}\\
        &=\sigma(\mathbf{A}),
    \end{align*}
    where the third equality comes from \citet{pedersen2012analysis} Exercise 4.1.3, which states that $\lambda(\mathcal{A}\mathcal{B})\setminus \{0\} = \lambda(\mathcal{B}\mathcal{A})\setminus \{0\}$ for any two bounded operators, applied to \(\mathrm{U}^{\star}\mathbf{A}^{\star}\) and \(\mathbf{A}\mathrm{U}\) (which are bounded because \(\mathrm{U}^{\star}\mathbf{A}^{\star}\) is defined on \(\mathbb{R}^{n}\) and \(\mathbf{A}\mathrm{U}\) is its adjoint).
    This concludes the proof because $p$-Schatten norms are solely defined by the singular values.
\end{proof}

\begin{proposition}
\label{prop:error bound penalty 2 matrices}
    Under the same hypothesis as \Cref{prop:error bound penalty 2}, if \((\mathbf{A}_{\mathrm{low}}, \mathbf{A}_{\mathrm{up}})\) are solutions of \Cref{eq:formulationA penalty 2} then
    \begin{equation*}
        \sup_{x\in\mathcal{X}}\, \lvert \tilde{f}_{\mathbf{A}_{\mathrm{low}}}(x) - \tilde{f}_{\mathbf{A}_{\mathrm{up}}}(x) \rvert \leq \mathrm{M} \Vert k\Vert_{\infty} \,\lVert \mathbf{A}_{\mathrm{low}} - \mathbf{A}_{\mathrm{low}}\rVert_{\star}
    \end{equation*}
    where $\tilde{f}_{\mathbf{A}_{\mathrm{low}}}(X)=\bs{\Phi}(X)^{\top}\mathbf{A}_{\mathrm{low}}\bs{\Phi}(X) \in \mathcal{H}$ and $\tilde{f}_{\mathbf{A}_{\mathrm{up}}}(X)=\bs{\Phi}(X)^{\top} \mathbf{A}_{\mathrm{up}}\bs{\Phi}(X)  \in \mathcal{H}$ are the kernel sum-of-squares functions associated to $\mathbf{A}_{\mathrm{low}}$ and $\mathbf{A}_{\mathrm{up}}$, respectively.
\end{proposition}

\begin{proof}
    Let \((\mathbf{A}_{\mathrm{low}}, \mathbf{A}_{\mathrm{up}})\) be solutions of \Cref{eq:formulationA penalty 2} and write \((\mathcal{A}_{\mathrm{low}}, \mathcal{A}_{\mathrm{up}})\) their associated operators.
    Let us consider $\mathrm{U}\colon \mathcal{H}\rightarrow \mathbb{R}^{n}$, the following partial isometry $f \longmapsto \mathbf{V}^{-\top}\left(\langle\phi(X_{1}), f\rangle, \ldots, \langle\phi(X_{n}), f\rangle\right)$.

    First, we show that \(f_{\mathcal{A}_{\mathrm{low}}}(X) = \tilde{f}_{\mathbf{A}_{\mathrm{low}}}(X)\) for any \(X\in \mathcal{X}\):
    \begin{align*}
        \tilde{f}_{\mathbf{A}_{\mathrm{low}}}(X) &= \langle \bs{\Phi}(X), \mathbf{A}_{\mathrm{low}}\bs{\Phi}(X)\rangle \\
        &= \langle \mathrm{U}\phi(X), \mathbf{A}_{\mathrm{low}}\mathrm{U}\phi(X)\rangle \\
        &= \langle \phi(X), \mathrm{U}^{\star}\mathbf{A}_{\mathrm{low}}\mathrm{U}\phi(X)\rangle \\
        &= \langle \phi(X), \mathcal{A}_{\mathrm{low}}\phi(X)\rangle \\
        &= f_{\mathcal{A}_{\mathrm{low}}}(X).
    \end{align*}
    The same results holds for \(_{\mathrm{up}}\): \(f_{\mathcal{A}_{\mathrm{up}}}(X) = \tilde{f}_{\mathbf{A}_{\mathrm{up}}}(X)\) for any \(X\in \mathcal{X}\).

    By applying \Cref{prop:error bound penalty 2} to \((\mathcal{A}_{\mathrm{low}}, \mathcal{A}_{\mathrm{up}})\), we have 
    \begin{equation*}
        \sup_{x\in\mathcal{X}}\, \lvert f_{\mathcal{A}_{\mathrm{low}}}(x) - f_{\mathcal{A}_{\mathrm{up}}}(x) \rvert \leq \mathrm{M} \Vert k\Vert_{\infty} \,\lVert \mathcal{A}_{\mathrm{low}} - \mathcal{A}_{\mathrm{up}}\rVert_{\star},
    \end{equation*}
    and
    \begin{equation*}
        \sup_{x\in\mathcal{X}}\, \lvert \tilde{f}_{\mathbf{A}_{\mathrm{low}}}(x) - \tilde{f}_{\mathbf{A}_{\mathrm{up}}}(x) \rvert \leq \mathrm{M} \Vert k\Vert_{\infty} \,\lVert \mathcal{A}_{\mathrm{low}} - \mathcal{A}_{\mathrm{up}}\rVert_{\star}.
    \end{equation*}

    Finally, by applying \Cref{lem:partial_isometries} with \(p=1\) to \(\mathbf{A}_{\mathrm{low}} - \mathbf{A}_{\mathrm{up}}\) and \(\mathrm{U}\), we have that \(\lVert \mathbf{A}_{\mathrm{low}} - \mathbf{A}_{\mathrm{up}} \rVert_{\star} = \lVert \mathrm{U}^{\star}(\mathbf{A}_{\mathrm{low}} - \mathbf{A}_{\mathrm{up}})\mathrm{U} \rVert_{\star} = \lVert \mathcal{A}_{\mathrm{low}} - \mathcal{A}_{\mathrm{up}} \rVert_{\star}\), which concludes the proof.
\end{proof}

\begin{proposition}
    Under the same hypothesis as \Cref{prop:error bound penalty 1}, if \((\mathbf{A}_{\mathrm{low}}, \mathbf{A}_{\mathrm{up}})\) are solutions of \Cref{eq:formulationA penalty 1} then
    \begin{equation*}
        \sup_{x\in\mathcal{X}}\, \lvert \tilde{f}_{\mathbf{A}_{\mathrm{low}}}(x) - \tilde{f}_{\mathbf{A}_{\mathrm{up}}}(x) \rvert \leq 2\,C_{\tilde{f}_{\mathbf{A}_{\mathrm{low}}},\tilde{f}_{\mathbf{A}_{\mathrm{up}}}} \rho_{\widehat{X}, \mathcal{X}} + \sqrt{\sum_{i=1}^{n}(\tilde{f}_{\mathbf{A}_{\mathrm{low}}}(X_i) - \tilde{f}_{\mathbf{A}_{\mathrm{up}}}(X_i))^{2}}
    \end{equation*}
    where $\tilde{f}_{\mathbf{A}_{\mathrm{low}}}(X)=\bs{\Phi}(X)^{\top}\mathbf{A}_{\mathrm{low}}\bs{\Phi}(X) \in \mathcal{H}$ and $\tilde{f}_{\mathbf{A}_{\mathrm{up}}}(X)=\bs{\Phi}(X)^{\top} \mathbf{A}_{\mathrm{up}}\bs{\Phi}(X)  \in \mathcal{H}$ are the kernel sum-of-squares functions associated to $\mathbf{A}_{\mathrm{low}}$ and $\mathbf{A}_{\mathrm{up}}$, respectively, and with $C_{\tilde{f}_{\mathbf{A}_{\mathrm{low}}},\tilde{f}_{\mathbf{A}_{\mathrm{up}}}}=2d\mathrm{D}\mathrm{M}\lVert \mathbf{A}_{\mathrm{low}}-\mathbf{A}_{\mathrm{up}} \rVert_{\star}$.
\end{proposition}

\begin{proof}
    Let \((\mathbf{A}_{\mathrm{low}}, \mathbf{A}_{\mathrm{up}})\) be solutions of \Cref{eq:formulationA penalty 1} and write \((\mathcal{A}_{\mathrm{low}}, \mathcal{A}_{\mathrm{up}})\) their associated operators. As in the proof of \Cref{prop:error bound penalty 2 matrices}, we have \(f_{\mathcal{A}_{\mathrm{low}}}(X) = \tilde{f}_{\mathbf{A}_{\mathrm{low}}}(X)\) and \(f_{\mathcal{A}_{\mathrm{up}}}(X) = \tilde{f}_{\mathbf{A}_{\mathrm{up}}}(X)\) for any \(X\in \mathcal{X}\).

    We now apply \Cref{prop:error bound penalty 1} to \((\mathcal{A}_{\mathrm{low}}, \mathcal{A}_{\mathrm{up}})\) to get 
    \begin{equation*}
        \sup_{x\in\mathcal{X}}\, \lvert f_{\mathcal{A}_{\mathrm{low}}}(x) - f_{\mathcal{A}_{\mathrm{up}}}(x) \rvert \leq 2\,C_{f_{\mathcal{A}_{\mathrm{low}}},f_{\mathcal{A}_{\mathrm{up}}}} \rho_{\widehat{X}, \mathcal{X}} + \sqrt{\sum_{i=1}^{n}(f_{\mathcal{A}_{\mathrm{low}}}(X_i) - f_{\mathcal{A}_{\mathrm{up}}}(X_i))^{2}}
    \end{equation*}
    where \(C_{f_{\mathcal{A}_{\mathrm{low}}},f_{\mathcal{A}_{\mathrm{up}}}}=2d\mathrm{D}\mathrm{M}\lVert \mathcal{A}_{\mathrm{low}}-\mathcal{A}_{\mathrm{up}} \rVert_{\star}\) and
    \begin{equation*}
        \sup_{x\in\mathcal{X}}\, \lvert \tilde{f}_{\mathbf{A}_{\mathrm{low}}}(x) - \tilde{f}_{\mathbf{A}_{\mathrm{up}}}(x) \rvert \leq 2\,C_{f_{\mathcal{A}_{\mathrm{low}}},f_{\mathcal{A}_{\mathrm{up}}}} \rho_{\widehat{X}, \mathcal{X}} + \sqrt{\sum_{i=1}^{n}(\tilde{f}_{\mathbf{A}_{\mathrm{low}}}(X_i) - \tilde{f}_{\mathbf{A}_{\mathrm{up}}}(X_i))^{2}}
    \end{equation*}

    Finally, by applying \Cref{lem:partial_isometries} with \(p=1\) to \(\mathbf{A}_{\mathrm{low}} - \mathbf{A}_{\mathrm{up}}\) and \(\mathrm{U}\), we have that \(\lVert \mathbf{A}_{\mathrm{low}} - \mathbf{A}_{\mathrm{up}} \rVert_{\star} = \lVert \mathrm{U}^{\star}(\mathbf{A}_{\mathrm{low}} - \mathbf{A}_{\mathrm{up}})\mathrm{U} \rVert_{\star} = \lVert \mathcal{A}_{\mathrm{low}} - \mathcal{A}_{\mathrm{up}} \rVert_{\star}\), which gives \(C_{f_{\mathcal{A}_{\mathrm{low}}},f_{\mathcal{A}_{\mathrm{up}}}} = C_{\tilde{f}_{\mathbf{A}_{\mathrm{low}}},\tilde{f}_{\mathbf{A}_{\mathrm{up}}}}\) and concludes the proof.
\end{proof}

\subsection{Bounds on local coverage}
\label{sec:apdx:local coverage}
Before giving a detailed proof of our bounds, we first recall the definitions of the maximum mean discrepancy and the Hilbert-Schmidt independence criterion.
\begin{definition}[Maximum Mean Discrepancy \citep{smola2007hilbert}]
Let $X$ and $Y$ be random vectors defined on a topological space $\mathcal{Z}$, with respective Borel probability measures $P_X$ and $P_Y$. 
Let $k : \mathcal{Z} \times \mathcal{Z} \rightarrow \mathbb{R}$ be a kernel function and let $\mathcal{H}$ be the associated reproducing kernel Hilbert space. The maximum mean discrepancy between $P_X$ and $P_Y$ is defined as
\begin{align*}
    \mathrm{MMD}(P_X, P_Y) &= \sup_{\|f\|_{\mathcal{H}} \leq 1} |\mathbb{E}_{X \sim P_X}[f(X)] - \mathbb{E}_{Y\sim P_Y}[f(Y)]|\\
    &= \Vert \mu_{P_X} - \mu_{P_Y}\Vert_{\mathcal{H}}
\end{align*}
where \(\mu_{P_X}=\int k(x,\cdot)dP_X(x)\) and \(\mu_{P_Y}=\int k(x,\cdot)dP_Y(x)\) are the \emph{kernel mean embeddings} of $P_X$ and $P_Y$, respectively.
\end{definition}
The squared MMD admits the following closed-form expression:
\begin{align*}
    \mathrm{MMD}^2(P_X, P_Y) &= \mathbb{E}_{X\sim P_X,X' \sim P_X}[k(X,X')] + \mathbb{E}_{Y \sim P_Y,Y' \sim P_Y}[k(Y,Y')] \\
    &- 2\mathbb{E}_{X \sim P_X,Y \sim P_Y}[k(X,Y)]\,,
\end{align*}
which can be estimated thanks to U- or V-statistics.

Now given a pair of random vectors $(U,V)\in\mathcal{X}\times\mathcal{Y}$ with probability distribution $P_{UV}$, we define the product RKHS $\mathcal{H}=\mathcal{F}\times\mathcal{G}$ with kernel $k((u,v),(u',v'))=k_\mathcal{X}(u,u')k_\mathcal{Y}(v,v')$. A measure of the dependence between $U$ and $V$ can then be defined as the distance between the mean embedding \(\mu_{P_{UV}}\) of $P_{UV}$ and \(\mu_{P_{U}}\otimes \mu_{P_{V}}\) of $P_{U}\otimes P_{V}$, the joint distribution with independent marginals $P_{U}$ and $P_{V}$:
\begin{equation*}
\mathrm{MMD}^2(P_{UV},P_{U}\otimes P_{V}) = \Vert \mu_{P_{UV}} - \mu_{P_{U}}\otimes \mu_{P_{V}}\Vert_{\mathcal{H}}^2.
\end{equation*}
This measure is the so-called \textit{Hilbert-Schmidt independence criterion} (HSIC, see \citet{gretton2005measuring}) and can be expanded as
\begin{align*}
\mathrm{HSIC}(U,V) &= \mathrm{MMD}^2(P_{UV},P_{U}\otimes P_{V})\nonumber \\
 &= \mathbb{E}_{U,U',V,V'} k_\mathcal{X}(U,U')k_\mathcal{Y}(V,V')\nonumber \\
&+ \mathbb{E}_{U,U'} k_\mathcal{X}(U,U')\mathbb{E}_{V,V'} k_\mathcal{Y}(V,V')\nonumber \\
&- 2 \mathbb{E}_{U,V}\left[\mathbb{E}_{U'} k_\mathcal{X}(U,U')\mathbb{E}_{V'} k_\mathcal{Y}(V,V')\right] \label{eq:hsic}
\end{align*}
where $(U',V')$ is an independent copy of $(U,V)$. Once again, the reproducing property implies that HSIC can be expressed as expectations of kernels, which facilitates its estimation when compared to other dependence measures such as the mutual information.
Let us now state our proposition again, before giving a detailed proof.
\begin{proposition}
\label{apdx:prop:hyperparameter tuning criterion}
    Let \(\widehat{C}_{\mathcal{D}_N}\) be the prediction intervals built from a score function \(S(X, Y) = \max\bigl(l(X) - Y, Y - u(X) \bigr)\) through split CP with \(\mathcal{D}_N=\mathcal{D}_n\cup \mathcal{D}_m\). Then for any \(\omega_X\) in \(\mathcal{F}_X\) such that \(\mathbb{P}(X\in \omega_{X})\geq \delta\), denoting \(p_{\mathcal{D}_N}=\mathbb{P}(Y_{N+1}\in\widehat{C}_{\mathcal{D}_N}(X_{N+1}) \lvert \mathcal{D}_N,X_{N+1}\in \omega_{X})\) we have
    \begin{equation}
    \label{apdx:eq:neighborhood coverage mi}
        p_{\mathcal{D}_N} \geq 1 - \alpha -\frac{1}{\delta}\sqrt{1-\alpha_1\exp(\mathrm{MI}(\tilde{r}_{\mathcal{D}_n}(X,Y),W_{\mathcal{D}_n}(X)))}
    \end{equation}
    and
    \begin{equation}
    \label{apdx:eq:neighborhood coverage hsic}
        p_{\mathcal{D}_N} \geq 1 - \alpha -\frac{1}{\delta}\sqrt{1-\frac{\alpha_1}{1-\alpha_2\mathrm{HSIC}(\tilde{r}_{\mathcal{D}_n}(X,Y),W_{\mathcal{D}_n}(X))}},
    \end{equation}
    where \(\tilde{r}_{\mathcal{D}_n}(X,Y) = \lvert Y - (\widehat{u}_{\mathcal{D}_n}(X)+\widehat{l}_{\mathcal{D}_n}(X))/2 \rvert\) are the absolute centered residuals and \(W_{\mathcal{D}_n}(X)=(\widehat{u}_{\mathcal{D}_n}(X)-\widehat{l}_{\mathcal{D}_n}(X))/2\) is the width of the prediction bands. \(\alpha_1\) is a constant and \(\alpha_2\) only depends on the kernel used for \(\mathrm{HSIC}\), which must be characteristic.
\end{proposition}
\begin{proof}
    First, let us rewrite the score as a function of the width and the absolute centered residuals.
    \begin{align*}
        S(X, Y) &= \max\bigl(l(X) - Y, Y - u(X) \bigr)\\
        &= - \frac{u(X)-l(X)}{2} + \max\left(l(X) - Y + \frac{u(X)-l(X)}{2}, Y - u(X) + \frac{u(X)-l(X)}{2} \right)\\
        &= - \frac{u(X)-l(X)}{2} + \max\left(\frac{u(X)+l(X)}{2} - Y, Y - \frac{u(X)+l(X)}{2}\right)\\
        & = - \frac{u(X)-l(X)}{2} + \left\vert Y - \frac{u(X)+l(X)}{2} \right\vert\\
        & = - \frac{W(X)}{2} + \lvert Y - \widetilde{m}(X) \rvert\\
        & = - \frac{W(X)}{2} + R(X, Y)\\
        & = f(W(X), R(X, Y)).
    \end{align*}
    Then, placing ourselves in the context of split CP, we are working conditionally on \(\mathcal{D}_N\), thus \(\widehat{u}_{\mathcal{D}_n}(\cdot)\) and \(\widehat{l}_{\mathcal{D}_n}(\cdot)\) are deterministic functions and to lighten notations we will write \(W = W(X)\) and \(R = R(X,Y)\). We now have
    \begin{align*}
        \mathrm{MI}(X,S) &= \mathrm{MI}(X,f(W, R)) & \\
        & \leq \mathrm{MI}(X,(R,W)) &\quad \forall g,\; \mathrm{MI}(g(X),Y) \leq \mathrm{MI}(X,Y)\\
        & \leq \mathrm{MI}((X,Y),(R,W)) &\quad \mathrm{MI}((X_1,X_2),Y) \geq \mathrm{MI}(X_1,Y)\\
        &= H(R,W) - H(R,W \vert X,Y) &\quad \mathrm{MI}(X,Y) = H(X) - H(X \vert Y)\\
        &= H(R,W) & \quad H(Y \vert X)=0 \Leftrightarrow \exists g,\, Y=g(X)\\
        &= H(R) + H(W) - \mathrm{MI}(W,R) &\quad \mathrm{MI}(X,Y) = H(X) + H(Y) - H(X,Y)\\
        &\leq H(X,Y) + H(X) - \mathrm{MI}(W,R) &\quad \forall g,\; H(g(X))\leq H(X)\\
    \end{align*}
    and we can observe that only $\mathrm{MI}(W,R)$ depends on \(\widehat{u}_{\mathcal{D}_n}(\cdot)\) and \(\widehat{l}_{\mathcal{D}_n}(\cdot)\). We thus deduce that
    \begin{align*}
        1-\exp(-\mathrm{MI}(X,S)) \leq 1-\alpha_1\exp(\mathrm{MI}(W,R))
    \end{align*}
    where $\alpha_1=\exp(-H(X,Y)-H(X))$ is independent from \(\widehat{u}_{\mathcal{D}_n}(\cdot)\) and \(\widehat{l}_{\mathcal{D}_n}(\cdot)\). \citet{deutschmann2023adaptiveconformalregressionjackknife} showed that
    \begin{equation*}
        p_{\mathcal{D}_N} \geq 1 - \alpha -\frac{1}{\delta}\sqrt{1-\exp(-\mathrm{MI}(X,S))}
    \end{equation*}
    and we obtain \Cref{apdx:eq:neighborhood coverage mi} by using the previous bound.
    
    For the second part of the proposition, from \citet[Equation~15]{wang2023seminonparametricestimationdistributiondivergence}, we have the bound
    \begin{align*}
        \mathrm{TV}(\mathbb{P},\mathbb{Q}) \geq \frac{1}{2\sqrt{M_k}} \mathrm{MMD}_k(\mathbb{P},\mathbb{Q})
    \end{align*}
    where $\mathrm{TV}(\mathbb{P},\mathbb{Q})=\underset{A\in\mathcal{F}}{\sup} \vert\mathbb{P}(A)-\mathbb{Q}(A)\vert$ for $\mathbb{P}$, $\mathbb{Q}$ defined on a measurable space $(\Omega,\mathcal{F})$ and $\mathrm{MMD}_k(\mathbb{P},\mathbb{Q})$ are the total variation and the maximum mean discrepancy between probability distributions $\mathbb{P}$ and $\mathbb{Q}$, respectively. Here, the MMD depends on the choice of a kernel $k$, which is bounded by $M_k = \underset{x\in\mathcal{X}}{\sup}\; k(x,x)$, and must be characteristic for the inequality to hold. We then apply this inequality to $\mathbb{P}=P_{WR}$ the joint distribution of $(W,R)$ and $\mathbb{Q}=P_{W} \otimes P_{R}$ the joint distribution with independent marginals $P_{W}$ and $P_{R}$, to get
    \begin{align*}
        1 - \exp(-\mathrm{MI}(W,R)) \geq \mathrm{TV}^2(P_{WR},P_{W} \otimes P_{R}) \geq \alpha_2 \mathrm{HSIC}(W,R),
    \end{align*}
    where the inequality on the left is the Bretagnolle-Huber inequality, the inequality on the right comes from the HSIC definition $\mathrm{HSIC}(X,Y)=\mathrm{MMD}^2(P_{XY},P_X \otimes P_Y)$ and we denote $\alpha_2 = 1/(4M_k)$ with $k$ the kernel used in HSIC. We finally have
    \begin{align*}
        1 - \alpha_1 \exp(\mathrm{MI}(W,R)) \leq 1 - \frac{\alpha_1}{1-\alpha_2 \mathrm{HSIC}(W,R)}
    \end{align*}
    and \Cref{apdx:eq:neighborhood coverage hsic} follows.
    

\end{proof}

\newpage
\section{Additional experiments and details}

\subsection{Cross-validation for kernel hyperparameter estimation and Kruskal-Wallis rank test}
\label{sec:cross_val}

\paragraph{Cross-validation.} Let \(K\) be the number of folds. For \(k\in [K]\), we write \(\mathcal{D}_{k}\) the fold dataset \(k\) and \(\mathcal{D}_{-k} = \mathcal{D}_n \setminus \mathcal{D}_{k}\). We denote by \(\hat{f}_{\mathrm{low},-k}, \hat{f}_{\mathrm{up},-k}\) the lower and upper bands trained on \(\mathcal{D}_{-k}\). Define two sets,
\begin{equation*}
    R_{K} = \bigcup_{k=1}^{K} \left\{\left\vert Y_i - \widehat{m}_n(X_i) -  \frac{\hat{f}_{\mathrm{up},-k}(X_i)-\hat{f}_{\mathrm{low},-k}(X_i)}{2}\right\vert\right\}_{i\in\mathcal{D}_{k}} \quad \text{and}\quad  W_{K} = \bigcup_{k=1}^{K} \{\hat{f}_{\mathrm{up},-k}(X_i)+\hat{f}_{\mathrm{low},-k}(X_i)\}_{i\in\mathcal{D}_{k}}.
\end{equation*}
We seek
\begin{equation}
\label{equation:hsic optimization problem}
    \max_{(\theta_{\mathrm{low}},\theta_{\mathrm{up}})\,\in\, \mathbb{R}^{d}\times \mathbb{R}^{d}} \quad \widehat{\textrm{HSIC}}\left(W, R\right),
\end{equation}
where \(\widehat{\textrm{HSIC}}(W,R)\) is estimated with samples \(W_K\) and \(R_K\).
In all our experiments, we use the energy distance kernel \(k(x,x')=\vert x \vert + \vert x'\vert - \vert x-x'\vert\), which has been shown to be characteristic by \citet{sejdinovic2013equivalence}.

\paragraph{Kruskal-Wallis rank test.} To assess whether HSIC varies significantly across different $\lambda_{\mathrm{pen}}$, we treat each value of $\lambda_{\mathrm{pen}}$ as a “group” and the HSIC values computed from each cross-validation replicate as repeated observations. However, because cross-validation replicates are not independent and identically distributed across groups (they overlap, share data, and violate the independent-samples assumption of classical parametric tests), we use a permutation-based nonparametric framework rather than standard ANOVA or rank-based tests with asymptotic null distributions. Specifically, we first compute the classical Kruskal–Wallis test statistic $H_{\mathrm{obs}}$ on the pooled data (after rank-transforming all HSIC values), but derive its null distribution via permuting the $\lambda_{\mathrm{pen}}$ labels across all replicates. We generate $B=2000$ random label permutations and compute the permuted $H$-statistics, producing an empirical null distribution. The p-value is the fraction of permutations whose permuted $H$ exceeds or equals the observed $H_{\mathrm{obs}}$. This yields a valid significance test under the null hypothesis that the distributions of HSIC are identical across $\lambda_{\mathrm{pen}}$, without requiring independence of replicates or parametric assumptions.

\subsection{Implementation details}
\label{sec:implementation details}

\paragraph{Optimization.}

For the primal problems with the SDP formulation, we use the SCS algorithm \cite{odonoghue2021operatorsplittinghomogeneousembeddinglinearcomplementaryproblem,odonoghue2023software} available in the convex optimization software \texttt{CVXPY} \cite{diamond2016cvxpy,agrawal2018rewriting}, with a maximum number of iterations equal to $10000$. An example of script that implements the SDP problems defined by \Cref{eq:finite_separate_2,eq:formulationA penalty 2,eq:formulationA penalty 1} is shown below. Here, \texttt{matrix\_variable\_low} and \texttt{matrix\_variable\_upp} denote the unknowns $\mathbf{A}_{\mathrm{low}} \in \mathbb{S}_{+}^{n}$ and $\mathbf{A}_{\mathrm{up}} \in \mathbb{S}_{+}^{n}$. The remaining variables, such as the feature vectors \texttt{phi\_matrix\_low} ($\bs{\Phi}_{\mathrm{low}}(\mathbf{X})$) and \texttt{phi\_matrix\_upp} ($\bs{\Phi}_{\mathrm{up}}(\mathbf{X})$) have to be computed by the user using the training set and the chosen kernels.

\begin{lstlisting}[caption={Solving the SDP problem with CVXPY and SCS.}]
import numpy as np
import cvxpy as cp
import scs

penalty = "Asymmetric"/"Operator"/"Training set"

# Variables
matrix_variable_low = cp.Variable((n, n), symmetric=True)
matrix_variable_upp = cp.Variable((n, n), symmetric=True)

# Expressions
f_low = cp.reshape(
    cp.diag(phi_matrix_low.T @ matrix_variable_low @ phi_matrix_low),
    (-1, 1),
    order="C",
)
f_upp = cp.reshape(
    cp.diag(phi_matrix_upp.T @ matrix_variable_upp @ phi_matrix_upp),
    (-1, 1),
    order="C",
)

# Objective
match penalty:
    case "Asymmetric":
        objective = cp.Minimize(
            + (b/n) * cp.sum(f_low+f_upp)
            + lambda1_low * cp.trace(matrix_variable_low) + lambda2_low * cp.square(cp.norm(matrix_variable_low, "fro"))
            + lambda1_upp * cp.trace(matrix_variable_upp) + lambda2_upp * cp.square(cp.norm(matrix_variable_upp, "fro"))
        )
    case "Training set":
        objective = cp.Minimize(
            + (b/n) * cp.sum(f_low+f_upp)
            + lambda1_low * cp.trace(matrix_variable_low) + lambda2_low * cp.square(cp.norm(matrix_variable_low, "fro"))
            + lambda1_upp * cp.trace(matrix_variable_upp) + lambda2_upp * cp.square(cp.norm(matrix_variable_upp, "fro"))
            + lambda1_pen * cp.sum(cp.square(f_low-f_upp))
        )
    case "Operator":
        objective = cp.Minimize(
            + (b/n) * cp.sum(f_low+f_upp)
            + lambda1_low * cp.trace(matrix_variable_low) + lambda2_low * cp.square(cp.norm(matrix_variable_low, "fro"))
            + lambda1_upp * cp.trace(matrix_variable_upp) + lambda2_upp * cp.square(cp.norm(matrix_variable_upp, "fro"))
            + lambda1_pen * cp.norm((matrix_variable_low-matrix_variable_upp), "nuc")
            + lambda2_pen * cp.square(cp.norm((matrix_variable_low-matrix_variable_upp), "fro"))
        )

# Positive semi-definite constraints
constraints = [matrix_variable_low >> 0]
constraints += [matrix_variable_upp >> 0]

# 100% coverage constraints
constraints += [y_train <= mean_predictions + f_upp]
constraints += [y_train >= mean_predictions - f_low]

# Problem definition
prob = cp.Problem(objective, constraints)

# Solve
prob.solve(solver=cp.SCS)

# Retrieve final matrices
Ahat_low = matrix_variable_low.value
Ahat_upp = matrix_variable_upp.value

\end{lstlisting}

For dual problems, unlike \citet{allain2025scalableandadaptivepredictionbandsusingkerenlsumofsquares} who developed a projected gradient method with Nesterov acceleration, we prefer to use the L-BFGS algorithm \citep{liu1989limited} for its robustness and faster convergence, since we only have at most a few thousands optimization variables in our experiments. We use \texttt{SciPy}'s minimize algorithm \citep{2020SciPy-NMeth}, with a maximum number of iterations equal to $10000$, tolerance equal to $10^{-2}$ and initialize all Lagrange multipliers to $0$ (except when using the warm-start strategy detailed below).

\paragraph{Hyperparameter and training pipeline.}

In \cref{alg:training} we give details about the steps to follow in order to train the penalized kSoS model for both penalties with the Kruskal-Wallis rank test and the HSIC independence test. The conformal step is performed after the final stage when $(\widehat{\mathbf{A}}_{\mathrm{low}}, \widehat{\mathbf{A}}_{\mathrm{up}})$ are computed.

\begin{algorithm}[!t]
\caption{Hyperparameter selection and training for the penalized kSoS model}
\label{alg:training}
\begin{algorithmic}[1]
\REQUIRE Training set $\mathcal{D}_n$; mean model $\widehat{m}_{n}$; candidate sets $\Theta$ (pairs $(\theta_{\mathrm{low}}, \theta_{\mathrm{up}})$) and $\Lambda$ (values of $\lambda_{\text{pen}}$); number of CV folds $K$; bootstrap HSIC repetitions $B$; validation HSIC test repetitions $R$; p-value thresholds $\alpha_{\textrm{KW}}, \alpha_{\textrm{HSIC}}$; default lengthscales $\theta_{\text{hom}}$ for homoscedastic noise (typically \(20\)) and default penalization $\lambda_{\text{sym}}$ for symmetric noise (typically \(100\)).
\ENSURE Estimator $(\widehat{\mathbf{A}}_{\mathrm{low}}, \widehat{\mathbf{A}}_{\mathrm{up}})$.

\STATEx
\STATEx \textit{// Stage 1 --- Cross-validation grid search}
\STATEx \textit{// Can be coupled with warm-start}
\STATE Partition $\mathcal{D}_n$ into $K$ folds $\{\mathcal{D}_k\}_{k=1}^K$
\FOR{each $(\theta, \lambda) \in \Theta \times \Lambda$}
    \FOR{$k = 1, \ldots, K$}
        \STATE Fit Pen. kSoS on $\mathcal{D}_{-k}$ with $(\theta, \lambda)$ and predict $(\widehat{f}_{\mathrm{low}}, \widehat{f}_{\mathrm{up}})$ on $\mathcal{D}_k$
    \ENDFOR
    \STATE Aggregate out-of-fold residuals $R_K$ and widths $W_K$
    \STATE Compute observed $\mathrm{HSIC}(\theta, \lambda)$ and bootstrap replicates $\{\mathrm{HSIC}^{(b)}(\theta, \lambda)\}_{b=1}^B$
\ENDFOR

\STATEx
\STATEx \textit{// Stage 2 --- Best $\theta^\star$ per $\lambda$}
\FOR{each $\lambda \in \Lambda$}
    \STATE $\theta^\star(\lambda) \gets \arg\max_{\theta \in \Theta} \mathrm{HSIC}(\theta, \lambda)$
    \STATE $\widetilde{\mathrm{HSIC}}^{(b)}(\lambda) \gets \max_{\theta \in \Theta} \mathrm{HSIC}^{(b)}(\theta, \lambda)$, for each $b = 1, \ldots, B$
\ENDFOR

\STATEx
\STATEx \textit{// Stage 3 --- Kruskal--Wallis rank test with bootstrap}
\STATE $p_{\text{KW}} \gets \textsc{TEST}_{\text{KW}}\big(\{\widetilde{\mathrm{HSIC}}^{(b)}(\lambda)\}_{b\in[B],\lambda\in\Lambda}\big)$
\IF{$p_{\text{KW}} > \alpha_{\text{KW}}$}
    \STATE $\lambda^\star \gets \lambda_{\text{sym}}$ \COMMENT{default to symmetric model}
\ELSE
    \STATE $\lambda^\star \gets \arg\max_{\lambda \in \Lambda} \mathrm{HSIC}(\theta^\star(\lambda), \lambda)$
\ENDIF
\STATE $\theta^\star \gets \theta^\star(\lambda^\star)$

\STATEx
\STATEx \textit{// Stage 4 --- HSIC test of independence}
\FOR{$j = 1, \ldots, R$}
    \STATE Partition $\mathcal{D}_n$ into $K$ folds and run CV with $(\theta^\star, \lambda^\star)$ to obtain $R_{K}^{(j)}$ and $W_{K}^{(j)}$
    \STATE $p_j \gets \textsc{TEST}_{\text{HSIC}}\left(R_{K}^{(j)}, W_{K}^{(j)}\right)$
\ENDFOR
\IF{$\operatorname{median}(p_1, \ldots, p_R) \geq \alpha_{\textrm{HSIC}}$}
    \STATE $\theta^\star \gets \theta_{\text{hom}},\; \lambda^\star \gets \lambda_{\text{sym}}$ \COMMENT{default to homoscedastic model}
\ENDIF

\STATEx
\STATEx \textit{// Stage 5 --- Final fit}
\STATE $(\widehat{\mathbf{A}}_{\mathrm{low}}, \widehat{\mathbf{A}}_{\mathrm{up}}) \gets$ Pen. kSoS trained on full $\mathcal{D}_n$ with $(\theta^\star, \lambda^\star)$
\STATE \textbf{return} $(\widehat{\mathbf{A}}_{\mathrm{low}}, \widehat{\mathbf{A}}_{\mathrm{up}})$
\end{algorithmic}
\end{algorithm}

\clearpage

\paragraph{Scaling with respect to \(n\).}
To illustrate the advantage of our dual formulation over the primal one, in Figure \ref{fig:computation_time_primal_dual} we compare the computation time for both of them when the number of samples $n$ increases. The primal formulation can only handle up to $n=200$ samples, while the dual solver easily scales to $n=1000$. 

\begin{figure}[htbp]
    \centering
    \includegraphics[scale=0.5]{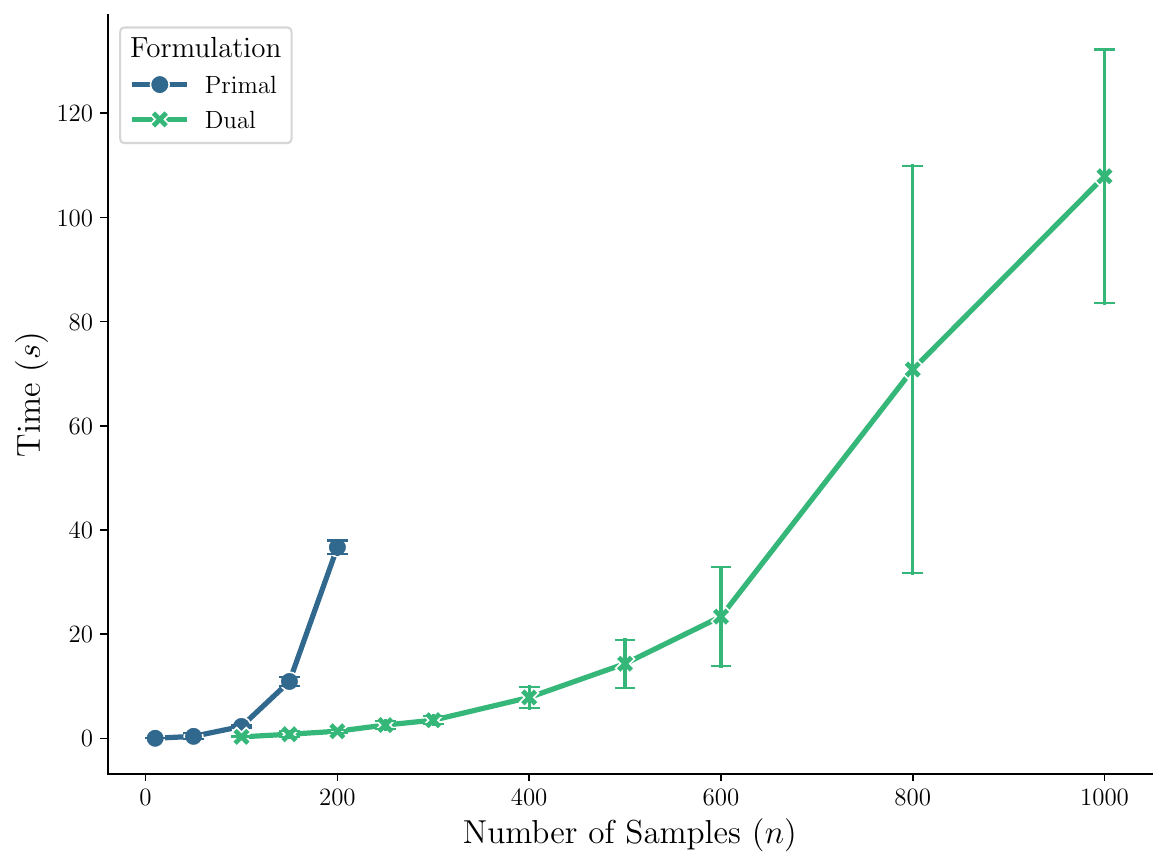}
    \caption{Dataset 1: computing time for solving SDP and dual formulation as a function of $n$ (penalty 1, $b=0$, \(\theta_{\mathrm{low}}=\theta_{\mathrm{up}}=0.3\), \(\lambda_{\mathrm{pen}}=1\), max iter = $10^4$), mean$\pm$sd over 20 repetitions.}
    \label{fig:computation_time_primal_dual}
\end{figure}

The simulations were performed on an AMD Ryzen 7 9700X 8-Core Processor (3.80 GHz) with four max threads.

\clearpage
\paragraph{Interaction between $b$ and \((\theta_{\mathrm{low}},\theta_{\mathrm{up}})\).}

In \citet{allain2025scalableandadaptivepredictionbandsusingkerenlsumofsquares}, it was observed that the penalty intensity \(b\) on the one hand, and the kernel lengthscales \(\theta\) (for the symmetric kSoS model) on the other hand, compensate each other through an interaction. Indeed, both influence the optimal kSoS function complexity: for a given \((b,\theta)\), similar bands with equivalent adaptivity can be obtained for higher values of both hyperparameters, as long as \(b\) is sufficiently large (typically $10$ or larger). 

In order to corroborate this assertion for our new penalized kSoS, we focus on one of our analytical test cases and compute the HSIC criterion for a grid of \((\lambda_{\mathrm{pen}},\theta)\) values (we fix \(\theta_{\mathrm{low}}=\theta_{\mathrm{up}}=\theta\)) and four different values of \(b=1,\, 10,\, 100,\, 1000\). Figure \ref{fig:contour_plots_evolution_with_b} displays these contour plots for three different random seeds (one row corresponds to one seed).

\begin{figure}[htbp]
    \centering
    \includegraphics[width=\linewidth]{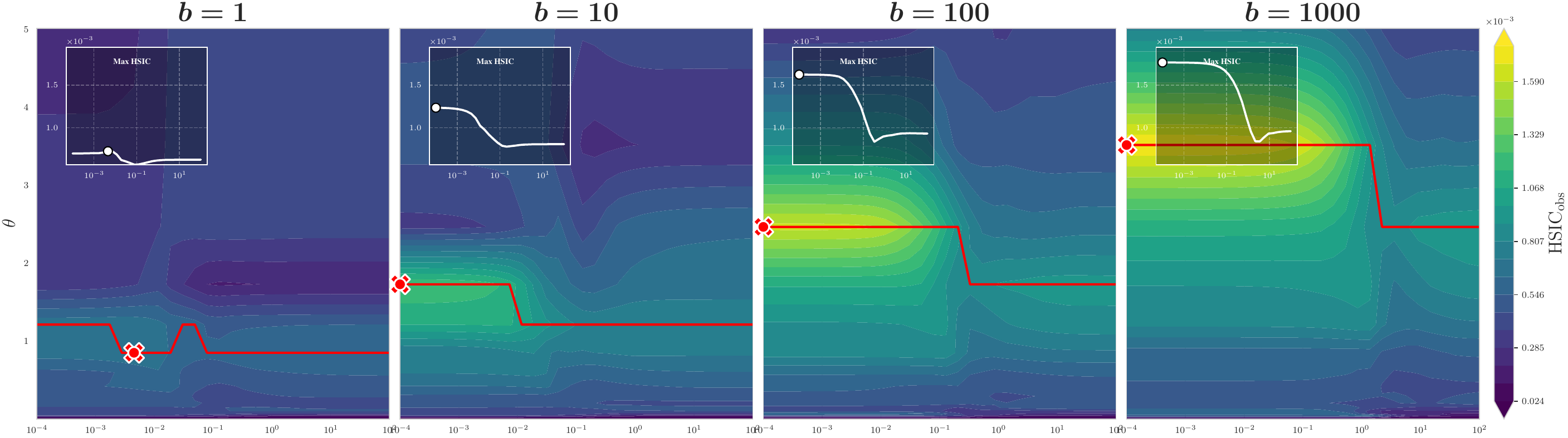}
    \includegraphics[width=\linewidth]{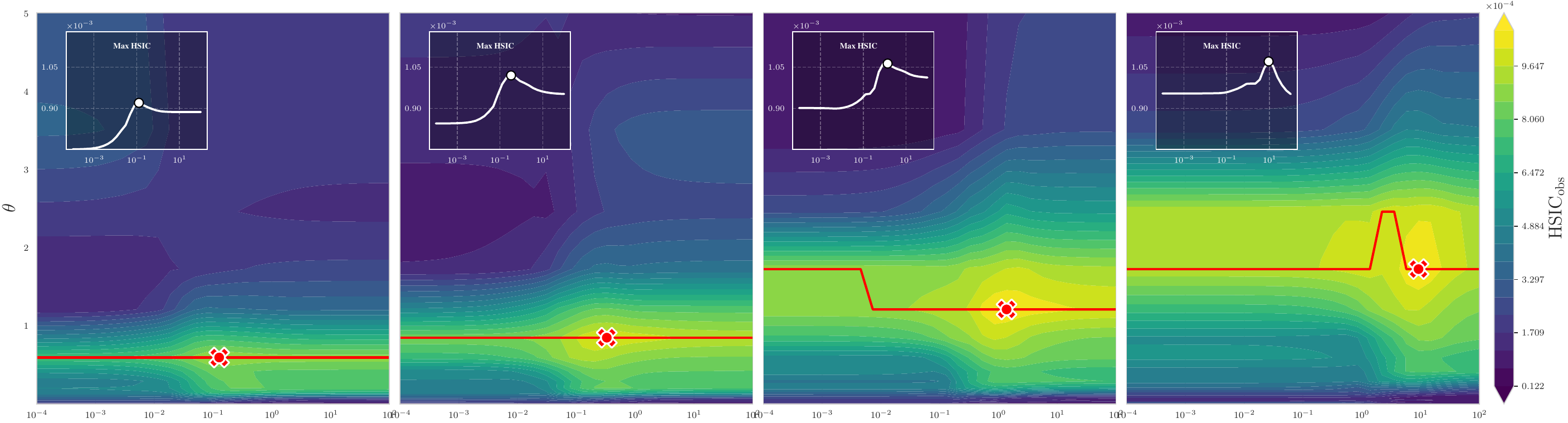}
    \includegraphics[width=\linewidth]{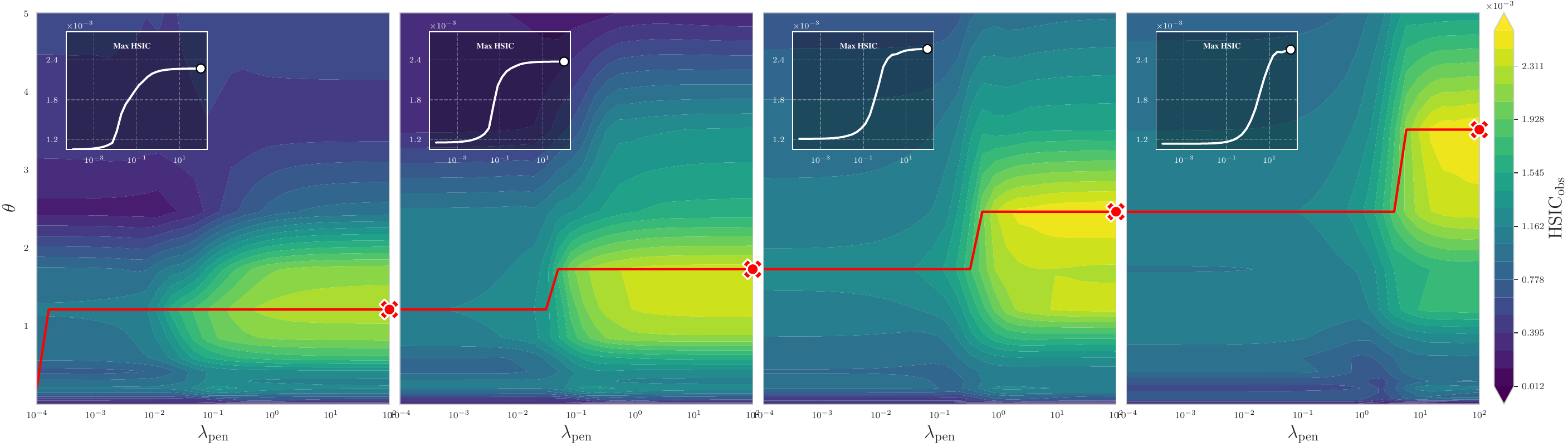}
    \caption{Contour plots of the HSIC criterion for a grid of \((\lambda_{\mathrm{pen}},\theta)\) and \(b=1,\, 10,\, 100,\, 1000\), case 2, \(n=100\).}
    \label{fig:contour_plots_evolution_with_b}
\end{figure}

For all seeds, we recover the expected behavior: when \(b\) increases, the optimal HSIC is reached for larger values of \(\theta\), and we observe that the value of the optimal HSIC is attained as soon as \(b=10\) or \(b=100\) in average. Interestingly, we can deepen the analysis by having a closer look at the differences between the seeds. At the top, we see that our model systematically chooses the smallest value for \(\lambda_{\mathrm{pen}}\), thus indicating that asymmetry is strongly favored for this seed, and only \(\theta\) increases to reach the optimal HSIC. Similarly, we notice the exact same phenomenon, although mirrored, in the bottom contour plots. This time, a strong symmetry is identified for the random samples, with only large selected values of \(\lambda_{\mathrm{pen}}\). In between, the seed points towards including a small amount of symmetry for better adaptivity, where intermediate penalty values are preferred. But the optimal \(\lambda_{\mathrm{pen}}\) increases with \(b\) and \(\theta\): to preserve the same amount of beneficial symmetry, the model necessitates a stronger constraint to compensate for the effect of increasing both \(b\) and \(\theta\).

\paragraph{Warm-start optimization strategy.}

To search for the best \((\theta_{\mathrm{low}}, \theta_{\mathrm{up}}, \lambda_{\mathrm{pen}})\) that maximizes \(\mathrm{HSIC}\), a brute-force and naive approach consists in solving the dual problem separately for all hyperparameter values to test. However, for fixed kernel lengthscales  \((\theta_{\mathrm{low}}, \theta_{\mathrm{up}})\), it is reasonable and intuitive to expect that the solution of the dual will only slightly vary between two close values of \(\lambda_{\mathrm{pen}}\). This means that once a penalized kSoS is trained for a given value of \(\lambda_{\mathrm{pen}}\), the corresponding optimal Lagrange multipliers can be used as initialization for training a subsequent penalized kSoS with a new value of \(\lambda_{\mathrm{pen}}\). If both are close, we anticipate a drastic reduction in the number of iterations required for convergence.

In order to validate this intuition, we consider one of our analytical test case and record the number of iterations at convergence for two strategies: a) the brute-force approach where all dual problems are solved separately and are initialized at the same value (the \emph{cold-start approach}) and b) an iterative approach where we first solve the dual for \(\lambda_{\mathrm{pen}}=10^{-4}\) and gradually increase its value while using the previous optimum as the initial point for the next problem (the \emph{warm-start approach}). Results are reported in Figure \ref{fig:cold_vs_warm_start_iterations}.

\begin{figure}[htbp]
    \centering
    \includegraphics[scale=0.7]{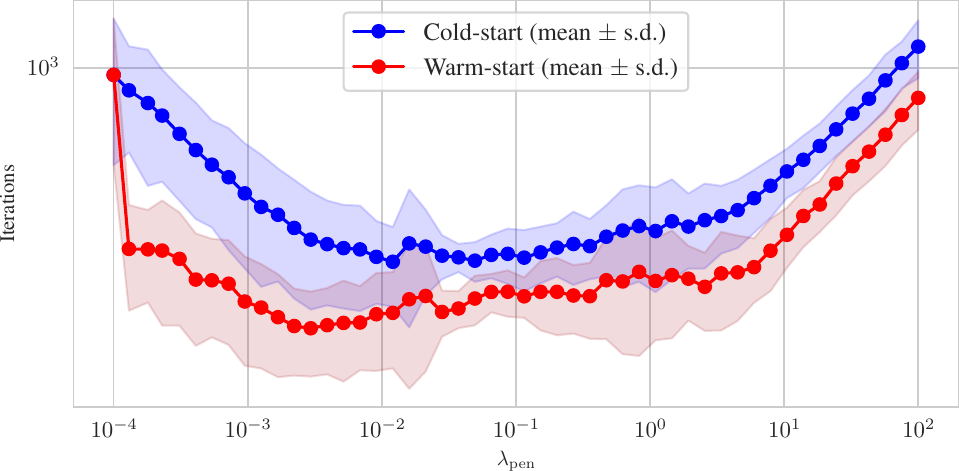}
    \caption{Dataset 1. Number of iterations at convergence for the cold-start (blue) and warm-start (red) approach with the dual formulation and penalty on the training set, \(n=100\), \(b=10\) (mean$\pm$sd on $10$ repetitions).}
    \label{fig:cold_vs_warm_start_iterations}
\end{figure}

We observe first that for both approaches, the number of iterations can substantially vary with the value of \(\lambda_{\mathrm{pen}}\): this means that some dual problems are harder to solve than others, typically those associated to very low or very high values of \(\lambda_{\mathrm{pen}}\). But crucially, we also clearly see that the warm-start approach yields considerable computational savings, with a reduction factor that can reach $3$ or $4$. In average, warm-start is $65\%$ cheaper than cold-start: this means that in practice, we can investigate $10$ different values of \(\lambda_{\mathrm{pen}}\) for maximizing HSIC at the cost of $3$ single kSoS trainings only.

\clearpage
\subsection{Additional numerical experiments}
\label{sec:additional numerical experiments}

\subsubsection{Illustration of asymmetry and adaptivity of noise distribution}
\label{sec:noise distribution}

The most commonly pursued feature in prediction bands is adaptivity, also known as conditional coverage, and is a needed property when the noise exhibits heteroscedastic behavior, as showcased in \cref{fig:noise taxonomy}, bottom row, where the noise is a function of the input $X$. On the other hand, both homoscedastic (top row) and heteroscedastic noise (bottom row) can exhibit either symmetric (left column) or asymmetric noise (right column). \textbf{Adaptivity} refers to whether the noise distribution is independent of $X$ or not, while \textbf{asymmetry} refers to whether the noise distribution is symmetric around the mean or not.

\begin{figure}[htbp]
        \centering
        \includegraphics[width=\linewidth]{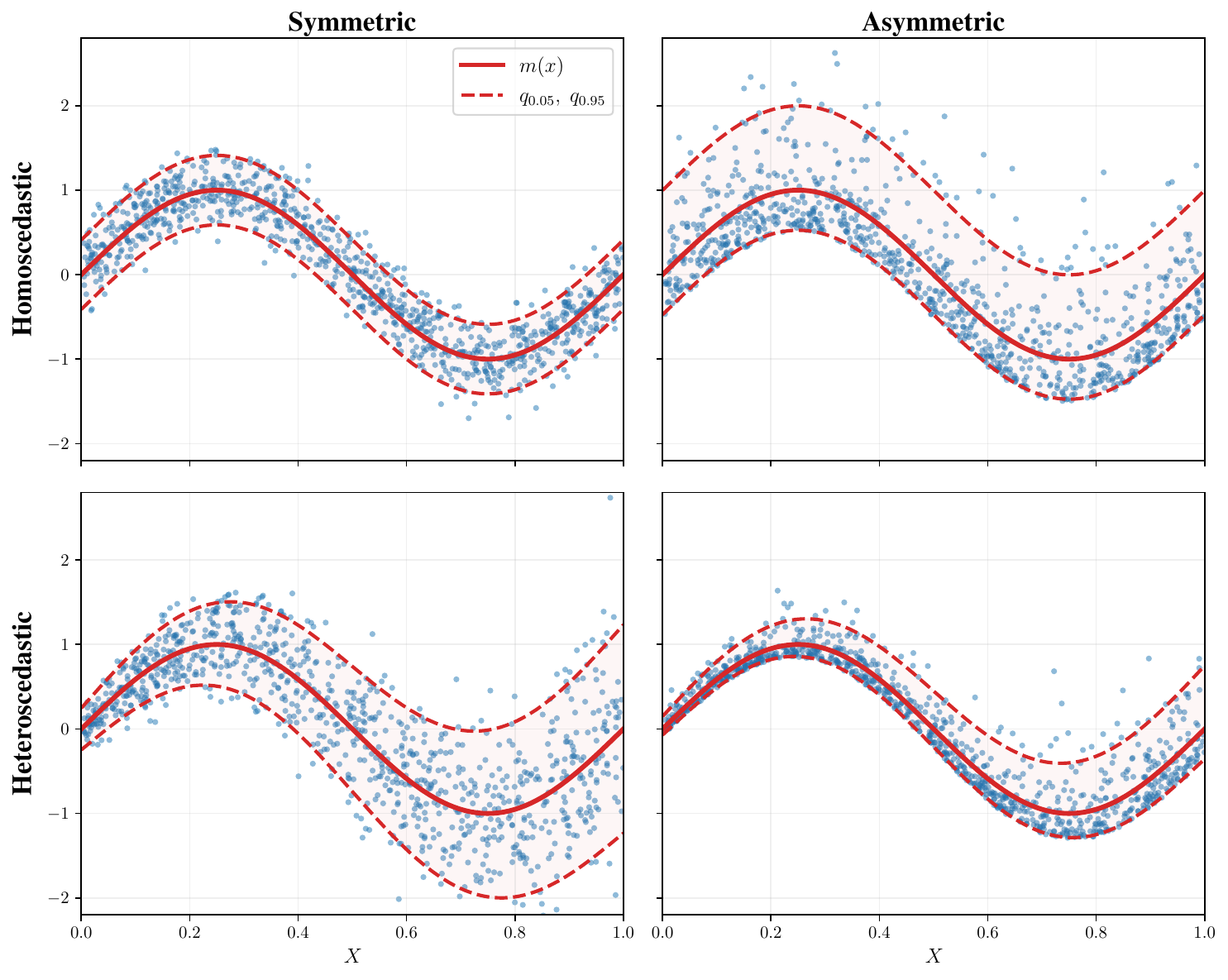}
    \caption{Noise taxonomy between two properties: adaptivity (homoscedastic versus heteroscedastic, rows) and asymmetry (columns).}
    \label{fig:noise taxonomy}
\end{figure}

In our numerical experiments on synthetic datasets, we thus consider several of these variants: datasets \(1\text{–}5\) are heteroscedastic while dataset \(6\) is homoscedastic, and asymmetric noise appears in datasets \(3\text{–}5\). 
Dataset $4$ is special because it has homoscedastic lower noise and heteroscedastic upper noise, resulting in asymmetric noise.


\subsubsection{Discussion on evaluation metrics}
\label{sec:evaluation metrics}

As is usually done in the literature to compare interval adaptivity, we compute the mean width of the prediction intervals on the test set. Unfortunately, when considered alone, this performance metric is of limited relevance. Indeed it is very common to have intervals with similar mean width but with highly different local coverage quality (we illustrate this phenomenon below). As such, to better measure adaptivity, we also consider additional metrics. 

For our experiments on analytical test cases, we consider the local coverage, obtained by approximating  \(\mathbb{P}(Y_{N+1}\in \widehat{C}_{\mathcal{D}_N}(X_{N+1}) \mid X_{N+1}=x )\) by its empirical counterpart with samples from \(Y\) (of size \(n_Y\)) at different random locations \(X_{i}\) (of size \(n_X\)), denoted $\hat{p}(X_i)$. We then compute a global measure of local coverage quality, the \emph{absolute coverage gap}, by considering the average distance to the target level $1-\alpha$:
\begin{equation*}
    \mathrm{ACG} = \frac{1}{n_X} \sum_{i=1}^{n_X} \vert \hat{p}(X_i) - (1-\alpha)\vert.
\end{equation*}
For real-world datasets, since such conditional samples are not available, we rely instead on the worst-set coverage introduced by \citet{thurin2025optimal}. Starting from a partition \(\{\mathcal{R}_l\}_{l=1,\ldots,L}\) of the input space, we compute the marginal coverage in each region \(\mathbb{P}(Y_{N+1}\in\widehat{C}_{\mathcal{D}_N}(X_{N+1}) \lvert X_{N+1}\in \mathcal{R}_l)\). The worst-set coverage $\mathrm{WSC}$ is defined as the minimum of all these coverages: the closer it is to the target $(1-\alpha)$, the more adaptive the intervals. In practice, we follow the ideas of \citet{thurin2025optimal} to define the regions, with a procedure that may not yield a partition: we randomly select \(L=10\) samples from the test set, and for each of them we identity the \(100\)-th closest neighbors in the feature space to estimate the marginal coverage.

These two measures indicate how well prediction intervals capture the noise distribution and are powerful in practice to compare CP procedures. However, they only focus on the central part of the intervals: for asymmetric noise distributions, they may fail to detect if a prediction interval is better than another one. Instead, we can consider their asymmetric variants, where we evaluate local coverage below or above the prediction interval, i.e. \(p^{\mathrm{low}}(X_{N+1})=\mathbb{P}(Y_{N+1}>l(X_{N+1}) \mid X_{N+1}=x )\) and \(p^{\mathrm{upp}}(X_{N+1})=\mathbb{P}(Y_{N+1}<u(X_{N+1}) \mid X_{N+1}=x )\) for an interval \(\widehat{C}_{\mathcal{D}_N}(X_{N+1})=[l(X_{N+1}),u(X_{N+1})]\), see \citet{linussonSignedErrorConformalRegression2014} and \citet{romano2019conformalizedquantileregression}. Ideally, we would like \(p^{\mathrm{low}}(X_{N+1})\) and \(p^{\mathrm{upp}}(X_{N+1})\) to be greater than \(1-\alpha/2\) for all \(X_{N+1}\), which would mean the prediction interval also captures well the lower and upper tails of the distribution. If this is the case, traditional local coverage also holds since
\begin{align*}
    \mathbb{P}(Y_{N+1}\notin \widehat{C}_{\mathcal{D}_N}(X_{N+1}) \mid X_{N+1}=x ) &= \mathbb{P}(Y_{N+1}<l(X_{N+1}) \; \cup \;  Y_{N+1}>u(X_{N+1} \mid X_{N+1}=x ) \\
    & \leq \mathbb{P}(Y_{N+1}<l(X_{N+1})\mid X_{N+1}=x ) + \mathbb{P}(  Y_{N+1}>u(X_{N+1}) \mid X_{N+1}=x ) \\
    &= (1-p^{\mathrm{low}}(X_{N+1})) + (1-p^{\mathrm{upp}}(X_{N+1})) \\
    &\leq \alpha/2 + \alpha/2 = \alpha.
\end{align*}
Following this idea, we can define the lower and upper absolute coverage gaps
\begin{align*}
    \mathrm{ACG}^{\mathrm{low}} &= \frac{1}{n_X} \sum_{i=1}^{n_X} \vert \hat{p}^{\mathrm{low}}(X_i) - (1-\alpha/2)\vert \\
    \mathrm{ACG}^{\mathrm{upp}} &= \frac{1}{n_X} \sum_{i=1}^{n_X} \vert \hat{p}^{\mathrm{upp}}(X_i) - (1-\alpha/2)\vert.
\end{align*}
Similarly, we consider the lower and upper worst-set coverage \(\mathrm{WSC}^{\mathrm{low}}\) and \(\mathrm{WSC}^{\mathrm{upp}}\). When aggregating such lower and upper adaptivity measures over several experimental replications, we will concatenate lower and upper indicators to draw boxplots or compute averages: we will thus refer to them as "combined", with notations \(\mathrm{ACG}^{c}\) and \(\mathrm{WSC}^{c}\).

Now, we illustrate numerically two important facts which motivate the evaluation metrics discussed so far:
\begin{enumerate}
    \item Mean width alone is not sufficient to measure adaptivity, since two intervals can have similar mean width but one can be locally adaptive while the other one is not.
    \item For asymmetric noise distribution, local coverage is not sufficient because an interval can have equivalent or better local coverage than another one, while failing at capturing lower and upper tails.
\end{enumerate}
We focus on the analytical test case 4 (described later on), and compare the intervals produced by heteroscedastic GP and our penalized kSoS (we consider both methods after split CP calibration which guarantees marginal coverage). For two replications, Figure \ref{fig:mw_loccov_case_5} shows the prediction intervals from the two procedures as well as their mean width and \(\mathrm{ACG}\) variants. Visually, it is clear that heteroscedastic GP is not adaptive, since it does not capture the noise distribution with an interval composed of "holes" on the right. On the contrary, kSoS learns the shape of the noise and is much more adaptive. But when looking at the metrics, we observe that heteroscedastic GP has equivalent or better mean width and \(\mathrm{ACG}\) than kSoS, which clearly shows that they are not sufficient to properly evaluate adaptivity. On the other side, both \(\mathrm{ACG}^{\mathrm{low}}\) and \(\mathrm{ACG}^{\mathrm{upp}}\) indicate that kSoS certainly outperforms heteroscedastic GP.

\begin{figure}[htbp]
    \centering
    \includegraphics[width=\linewidth]{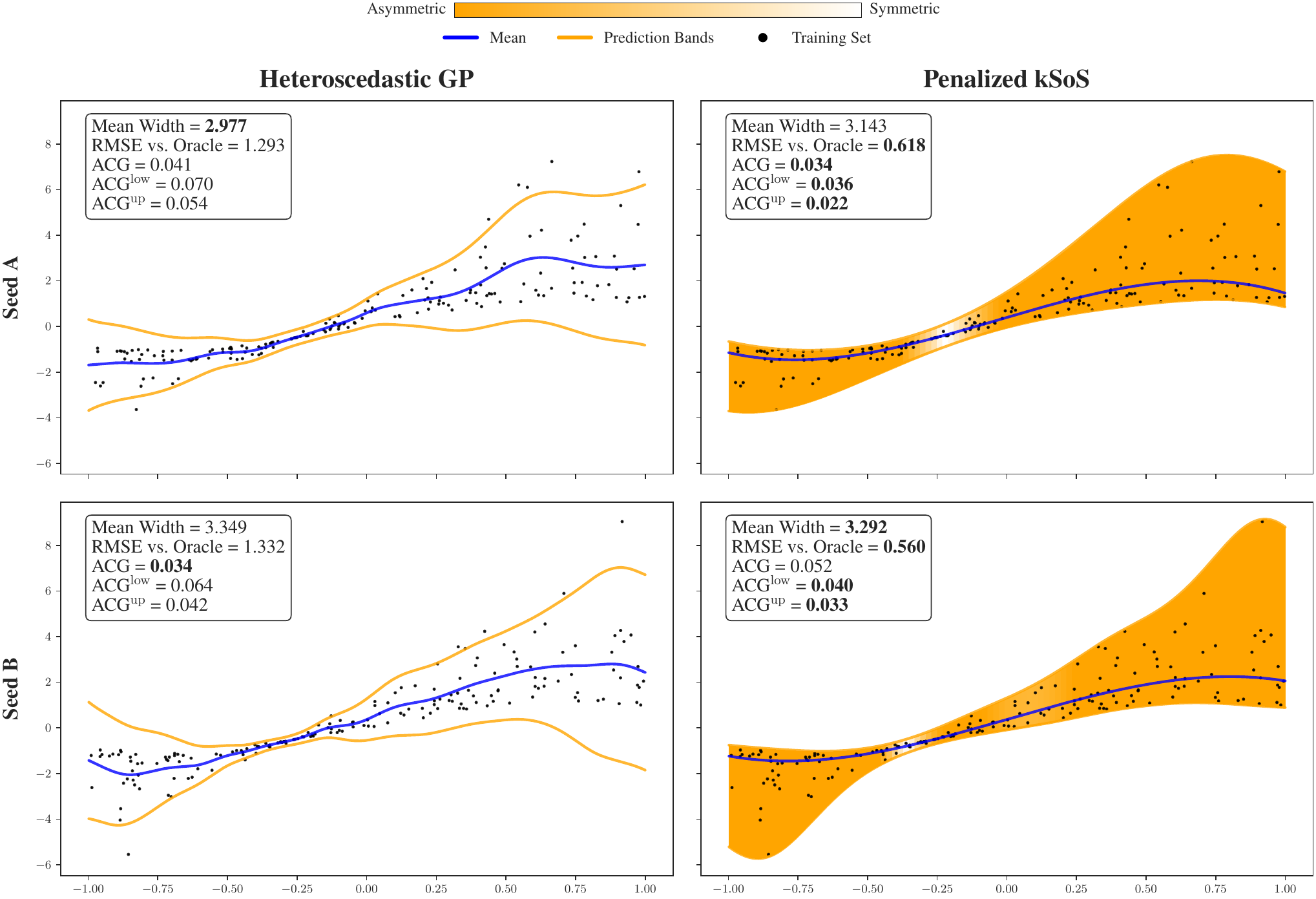}
    \caption{Prediction intervals for heteroscedastic GP and penalized kSoS for two random seeds on test case 5.}
    \label{fig:mw_loccov_case_5}
\end{figure}

All in one, we advocate the comparison of \(\mathrm{ACG}^{\mathrm{c}}\) or \(\mathrm{WSC}^{c}\) first (since in practice we do not know beforehand if the noise distribution is symmetric or not, it may be misleading to use \(\mathrm{ACG}\) or \(\mathrm{WSC}\)), and for methods that have similar lower and upper adaptivity, to compare as a second step their mean width to potentially break the ties. This is similar in spirit to the assertion "maximize the sharpness of the predictive distributions subject to calibration" of \citet{Gneiting01032007}, which was later emphasized again in \citet{chung2021beyond}.

\clearpage
\subsubsection{Operator penalty versus training set penalty}

Both penalties achieve similar goals but differ in implementation and theoretical properties. The operator penalty provides a continuous functional view of the problem, and consequently inherits tighter bound independent of the fill-in distance. But in terms of computational complexity, the operator penalty involves \(\mathcal{O}(n^2)\) dual variables, as opposed to \(\mathcal{O}(n)\) for the training set one, and is also less flexible since the exact same kernels (and lengthscales) must be used for both bands. In this particular case, we illustrate in Figure \ref{fig:comparison_penalties} that the operator penalty does not yield improvement in either mean width or coverage over the training set penalty for two test cases, but we observe this phenomenon for all the datasets we investigated. 

\begin{figure}[htbp]
  \centering
  \begin{subfigure}{0.44\textwidth}
    \centering
    \includegraphics[width=\linewidth]{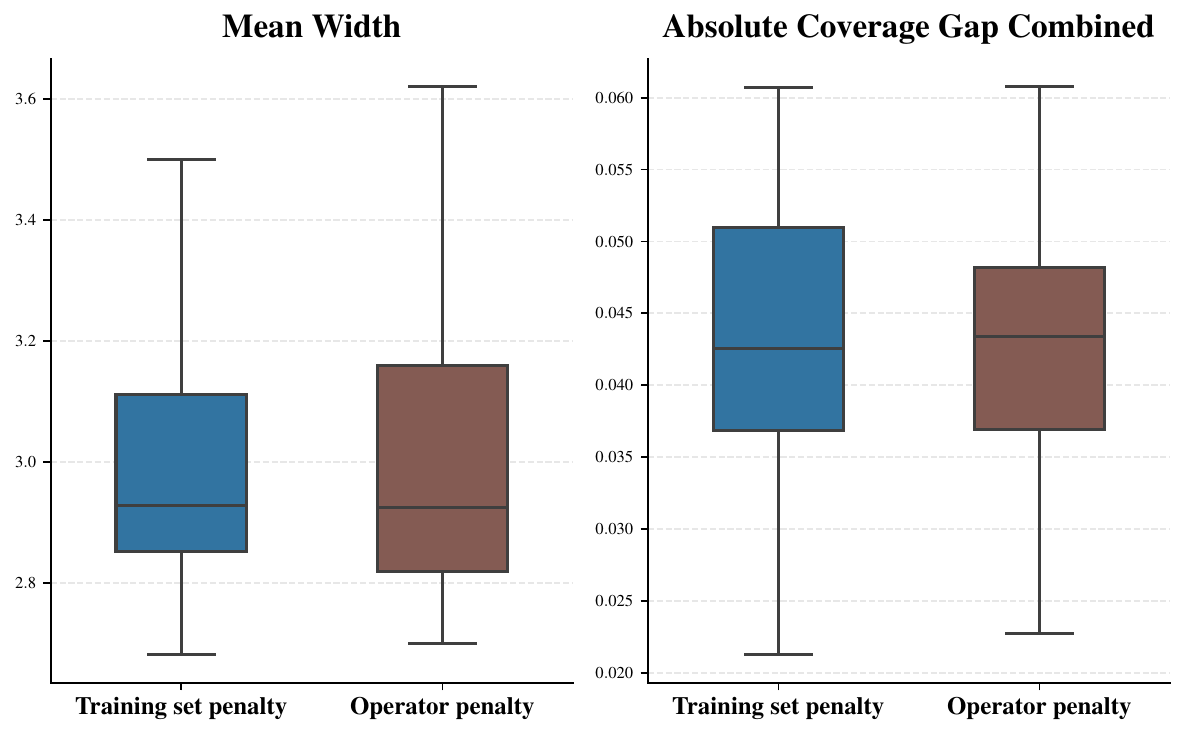}
    \caption{Mean width and absolute coverage gap combined for both penalties on dataset 1.}
    \label{fig:comparison_penalties_case_1}
  \end{subfigure}
  \hfill
  \begin{subfigure}{0.44\textwidth}
    \centering
    \includegraphics[width=0.49\linewidth]{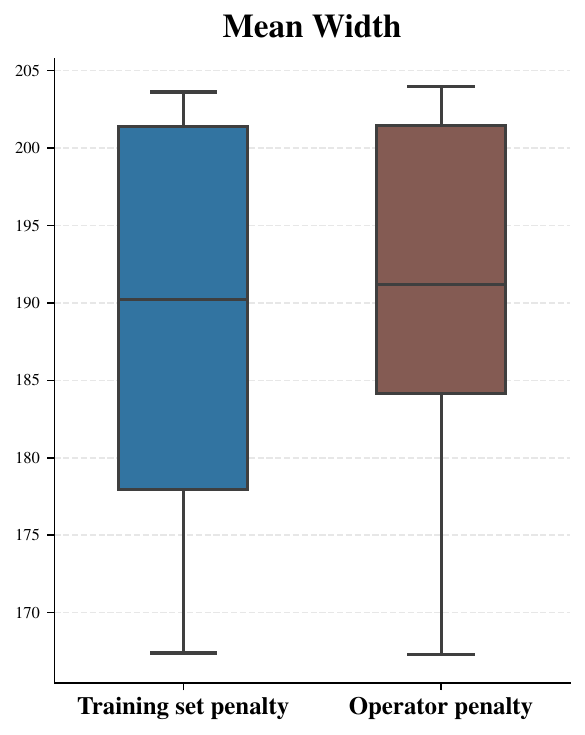}
    \includegraphics[scale=0.32]{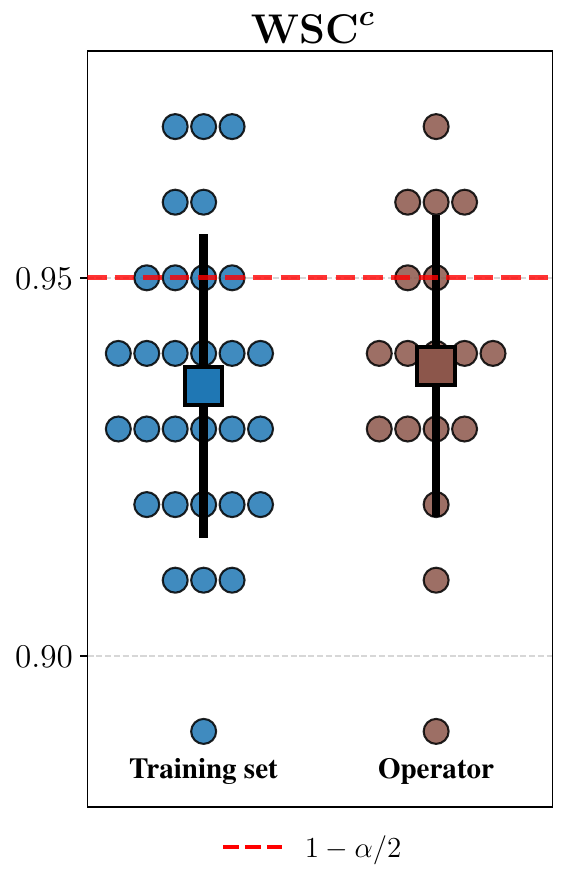}
    \caption{Mean width and worst-set coverage combined for both penalties on dataset Diabetes.}
    \label{fig:comparison_penalties_case_diabtes}
  \end{subfigure}
  \caption{Comparison of operator and training set penalties on an analytical test case (left) and a real-world dataset (right).}
  \label{fig:comparison_penalties}
\end{figure}

We thus recommend to use the operator penalty for small datasets where strongest theoretical control is desired and computation is not limiting, but use the training set one in all other instances.

\subsubsection{Difference with symmetric kSoS}
\label{sec:comparison with symmetric and asymmetric prediction bands}

When the predictive model \(\widehat{m}_n(\cdot)\) is fixed, the symmetric kSoS problem introduced in \citet{allain2025scalableandadaptivepredictionbandsusingkerenlsumofsquares} writes
\begin{align*}
    \underset{
            \mathcal{A} \in \mathcal{S}_{+}\left(\mathcal{H}\right)
    }{\inf}
    \quad& \frac{b}{n}\sum_{i=1}^{n} f_{\mathcal{A}}(X_{i}) +\Omega(\mathcal{A}) \\ 
    \mathrm{s.t.} \quad &\left(Y_i-\widehat{m}_n(X_i)\right)^2 \leq  f_{\mathcal{A}}(X_{i}), \;i \in \left[n\right]
\end{align*}
while our penalty operator formulation is
\begin{align*}
    \underset{
        \substack{
            \mathcal{A}_{\mathrm{low}} \in \mathcal{S}_{+}\left(\mathcal{H}\right) \\
            \mathcal{A}_{\mathrm{up}} \in \mathcal{S}_{+}\left(\mathcal{H}\right)
        }
    }{\inf}
    \quad& \frac{b}{n}\sum_{i=1}^{n} \left(f_{\mathcal{A_{\mathrm{low}}}}(X_{i}) + f_{\mathcal{A_{\mathrm{up}}}}(X_{i})\right) +\Omega_{\mathrm{low}}(\mathcal{A}_{\mathrm{low}}) + \Omega_{\mathrm{up}}(\mathcal{A_{\mathrm{up}}}) +\Omega_{\mathrm{pen}}(\mathcal{A}_{\mathrm{low}}-\mathcal{A}_{\mathrm{up}})\\
    \mathrm{s.t.} \quad &\widehat{m}_n(X_i) - Y_i - f_{\mathcal{A_{\mathrm{low}}}}(X_{i})\leq 0, \;i \in \left[n\right] \\
    \quad& Y_i - \widehat{m}_n(X_i) - f_{\mathcal{A_{\mathrm{up}}}}\leq 0, \;i \in \left[n\right]
\end{align*}
where \(\Omega_{\mathrm{pen}}(\mathcal{A}_{\mathrm{low}}-\mathcal{A}_{\mathrm{up}}) = \lambda_{\mathrm{pen}1}\lVert\mathcal{A}_\mathrm{low}-\mathcal{A}_\mathrm{up}\rVert_{\star}+\lambda_{\mathrm{pen}2}\lVert\mathcal{A}_\mathrm{low}-\mathcal{A}_\mathrm{up}\rVert_{F}^{2}\).\\
When \(\lambda_{\mathrm{pen}1}, \lambda_{\mathrm{pen}2} \rightarrow +\infty\), \(\Omega_{\mathrm{pen}}(\mathcal{A}_{\mathrm{low}}-\mathcal{A}_{\mathrm{up}}) \rightarrow \iota_{\mathcal{A}_{\mathrm{low}}=\mathcal{A}_{\mathrm{up}}}\) with \(\iota_C\) the indicator function of the set \(C\). The penalized formulation thus tends to
\begin{align*}
    \underset{
        \substack{
            \mathcal{A}_{\mathrm{low}} \in \mathcal{S}_{+}\left(\mathcal{H}\right) \\            \mathcal{A}_{\mathrm{up}} \in \mathcal{S}_{+}\left(\mathcal{H}\right) \\
\mathcal{A}_{\mathrm{low}} = \mathcal{A}_{\mathrm{up}}
        }
    }{\inf}
    \quad& \frac{b}{n}\sum_{i=1}^{n} \left(f_{\mathcal{A_{\mathrm{low}}}}(X_{i}) + f_{\mathcal{A_{\mathrm{up}}}}(X_{i})\right) + \Omega_{\mathrm{low}}(\mathcal{A}_{\mathrm{low}}) + \Omega_{\mathrm{up}}(\mathcal{A_{\mathrm{up}}}) \\
    \mathrm{s.t.} \quad &\widehat{m}_n(X_i) - Y_i - f_{\mathcal{A_{\mathrm{low}}}}(X_{i})\leq 0, \;i \in \left[n\right] \\
    \quad& Y_i - \widehat{m}_n(X_i) - f_{\mathcal{A_{\mathrm{up}}}}(X_{i})\leq 0, \;i \in \left[n\right].
\end{align*}
Now, since by definition \(\mathcal{A}_{\mathrm{low}}=\mathcal{A}_{\mathrm{up}}\) implies \(\forall X,\, f_{\mathcal{A_{\mathrm{low}}}}(X)=f_{\mathcal{A_{\mathrm{up}}}}(X)\), the last problem is equivalent to 
\begin{align*}
    \underset{
            \mathcal{A} \in \mathcal{S}_{+}\left(\mathcal{H}\right)
    }{\inf}
    \quad& \frac{b}{n}\sum_{i=1}^{n} \left(f_{\mathcal{A}}(X_{i}) + f_{\mathcal{A}}(X_{i})\right) +\Omega_{\mathrm{low}}(\mathcal{A}) + \Omega_{\mathrm{up}}(\mathcal{A})\nonumber \\ 
    \mathrm{s.t.} \quad &\widehat{m}_n(X_i) - Y_i - f_{\mathcal{A}}(X_{i})\leq 0, \;i \in \left[n\right] \nonumber\\
    \quad& Y_i - \widehat{m}_n(X_i) - f_{\mathcal{A}}(X_{i})\leq 0, \;i \in \left[n\right].
\end{align*}
If we also take \(\lambda_{\mathrm{low}1}=\lambda_{\mathrm{up}1}\) and \(\lambda_{\mathrm{low}2}=\lambda_{\mathrm{up}2}\) such that \(\Omega_{\mathrm{low}}(\mathcal{A}) = \Omega_{\mathrm{up}}(\mathcal{A})=\Omega(\mathcal{A})\), this boils down to
\begin{align*}
    \underset{
            \mathcal{A} \in \mathcal{S}_{+}\left(\mathcal{H}\right)
    }{\inf}
    \quad& 2 \left(\frac{b}{n}\sum_{i=1}^{n} f_{\mathcal{A}}(X_{i}) +\Omega(\mathcal{A})\right)\\ 
    \mathrm{s.t.} \quad &\lvert Y_i-\widehat{m}_n(X_i)\rvert \leq  f_{\mathcal{A}}(X_{i}), \;i \in \left[n\right]
\end{align*}
with a score function which becomes \(S(X,Y)=\max\left(\widehat{m}_n(X) - Y - f_{\mathcal{A}}(X),Y - \widehat{m}_n(X) - f_{\mathcal{A}}(X)\right)=\lvert Y - \widehat{m}_n(X)\rvert - f_{\mathcal{A}}(X)\). At first sight, this closely resembles the symmetric kSoS of \citet{allain2025scalableandadaptivepredictionbandsusingkerenlsumofsquares}, where the constraints write \(\left( Y_i-\widehat{m}_n(X_i)\right)^2 \leq  f_{\mathcal{A}}(X_{i}) \Leftrightarrow \lvert Y_i-\widehat{m}_n(X_i) \rvert \leq \sqrt{f_{\mathcal{A}}(X_{i})}\).
From this perspective, \(f_{\mathcal{A}}(X)\) in our work plays the role of a standard deviation function, while in \citet{allain2025scalableandadaptivepredictionbandsusingkerenlsumofsquares} it relates to a variance function. However, there are important differences:
\begin{itemize}
    \item To get a strict equivalence of the problems, we should use \(\sum_{i=1}^n \sqrt{f_{\mathcal{A}}(X_{i})}\) instead of \(\sum_{i=1}^n f_{\mathcal{A}}(X_{i})\). This means that the optimal solutions are not the same
    \item More importantly, the score functions are fundamentally different and the final intervals differ: 
    \begin{align*}
& C^{\mathrm{Allain}}(X) = \widehat{m}_n(X) \pm q_{\alpha}\left(\frac{\lvert Y_j - \widehat{m}_n(X_j)\rvert}{\sqrt{f_{\mathcal{A}}(X_j)}}\right)\sqrt{f_{\mathcal{A}}(X)}\\
& C^{\mathrm{ours}}(X) = \widehat{m}_n(X) \pm \left(q_{\alpha}\left(\lvert Y_j - \widehat{m}_n(X_j)\rvert - f_{\mathcal{A}}(X_j)\right) + f_{\mathcal{A}}(X)\right)
\end{align*}
\end{itemize}
Ultimately, the method by \citet{allain2025scalableandadaptivepredictionbandsusingkerenlsumofsquares} is not a special case of our penalty formulation, because of both the objective function and the CP score functions. However, we insist that when the penalty is large, our problem still produces symmetric prediction bands, although different from \citet{allain2025scalableandadaptivepredictionbandsusingkerenlsumofsquares}.

\clearpage
\subsubsection{Symmetric versus asymmetric calibration}

As discussed in \citet{romano2019conformalizedquantileregression}, score functions of the form \(S(X, Y) = \max\bigl(l(X) - Y, Y - u(X) \bigr)\), as in CQR or penalized kSoS, can easily be used in an asymmetric calibration procedure. Denoting \(\widehat{q}_{\alpha_{\mathrm{low}}}\) and \(\widehat{q}_{\alpha_{\mathrm{upp}}}\) the adjusted quantiles of the sets \(\{l(X_i)-Y_i, \; i\in \mathcal{D}_m\}\) and \(\{Y_i-l(X_i), \; i\in \mathcal{D}_m\}\), respectively, then 
\begin{equation*}
\widehat{C}_{\mathcal{D}_N}(X_{N+1}) = \left[l(X_{N+1}) - \widehat{q}_{\alpha_{\mathrm{low}}},u(X_{N+1}) + \widehat{q}_{\alpha_{\mathrm{upp}}}\right]
\end{equation*}
has \(1-\alpha\) marginal coverage as long as \(\alpha_{\mathrm{low}}+\alpha_{\mathrm{upp}}=\alpha\). But \citet{romano2019conformalizedquantileregression} also mentions that the stronger coverage guarantee (marginal coverage below and above the interval) comes at the cost of slightly longer intervals. We also investigate this behavior on twelve real-world datasets (detailed later on). For both CQR and penalized kSoS, we examine how often asymmetric calibration produces longer intervals, and how close to local coverage intervals after either calibration are, when measured with \(\mathrm{WSC}^{\mathrm{c}}\). Results are given in Table \ref{tab:percent_times_calib_realworld}, averaged over 10 repetitions.

\begin{table}[ht]
\centering
\begin{minipage}[t]{0.60\textwidth}
\centering
\footnotesize
\setlength{\tabcolsep}{3pt}
\caption{Percentage of times asymmetric calibration yields larger intervals than symmetric calibration, and mean distance between worst-set coverage (lower and upper) and target level $1-\alpha/2$, over 10 repetitions and for both CQR and penalized kSoS.}\label{tab:percent_times_calib_realworld}
\vspace{0.08cm}
\begin{tabular}{l|ccc}
\toprule
Dataset & $ \mathrm{MW}_{\mathrm{asym}} > \mathrm{MW}_{\mathrm{sym}}$ & $\vert \mathrm{WSC}^{\mathrm{c}}_{\mathrm{asym}}-(1-\alpha/2)\vert $ & $\vert \mathrm{WSC}^{\mathrm{c}}_{\mathrm{sym}}-(1-\alpha/2)\vert $ \\
\midrule
\textbf{All datasets}   & $69.17\%$ & $0.0322$ & $0.0315$ \\
\midrule
Concrete   & $65\%$ & $0.01700$ & $0.01425$ \\
Bike       & $75\%$ & $0.05400$ & $0.05775$ \\
Diabetes   & $70\%$ & $0.02100$ & $0.02350$ \\
Housing log& $65\%$ & $0.02975$ & $0.02725$ \\
Housing    & $55\%$ & $0.02900$ & $0.02675$ \\
MPG        & $95\%$ & $0.02375$ & $0.02375$ \\
Boston     & $60\%$ & $0.01700$ & $0.01025$  \\
Energy     & $85\%$ & $0.02500$ & $0.02375$  \\
Miami      & $40\%$ & $0.04500$ & $0.04225$  \\
Sulfur     & $75\%$ & $0.04175$ & $0.04425$  \\
Power      & $75\%$ & $0.06600$ & $0.06525$  \\
Yacht      & $70\%$ & $0.01725$ & $0.01950$  \\
\bottomrule
\end{tabular}
\end{minipage}
\hfill
\begin{minipage}[t]{0.30\textwidth}
\centering
\footnotesize
\setlength{\tabcolsep}{3pt}
\caption{Mean width relative increase for intervals with asymmetric calibration vs intervals with symmetric calibration, averaged over 10 repetitions.}\label{tab:percent_increase_calib_realworld} 
\begin{tabular}{l|c}
\toprule
Dataset & $\frac{\mathrm{MW}_{\mathrm{asym}}-\mathrm{MW}_{\mathrm{sym}}}{\mathrm{MW}_{\mathrm{sym}}}$ \\
\midrule
\textbf{All datasets}   & $1.48\%$ \\
\midrule
Concrete CQR  & $1.47\%$ \\
Bike CQR      & $1.20\%$ \\
Bike kSoS      & $1.30\%$ \\
Boston CQR      & $2.96\%$ \\
Diabetes CQR  & $2.02\%$ \\
Energy kSoS & $3.06\%$ \\
Housing kSoS & $1.45\%$ \\
MPG CQR & $8.32\%$ \\
MPG kSoS & $5.68\%$\\
Yacht CQR & $2.46\%$ \\
Yacht kSoS & $3.52\%$\\
\bottomrule
\end{tabular}
\end{minipage}
\end{table}

We observe first that both calibration methods have equivalent lower and upper worst-set coverage, meaning that asymmetric calibration does not improve local coverage in these examples. However, except for Miami, asymmetric calibration predominantly produces longer intervals, in $69.17\%$ of the cases in average. To go further, we study in Table \ref{tab:percent_increase_calib_realworld} the relative increase of mean width induced by asymmetric calibration on some of these datasets. We observe that in average the increase is limited, but it can be very large in specific instances. Since asymmetric calibration was observed to come with no benefits on lower and upper local coverage in these experiments, we only consider symmetric calibration in all the following experiments.

\clearpage
\subsubsection{Additional analytical test cases and results}
\label{sec:additional analytical test}

For all experiments related to adaptivity metrics, we perform $20$ replications with different random seeds, and local coverage is estimated with $n_X=100$ independent random locations $X_{N+1}$ for which we generate $n_Y=1000$ independent samples from $Y_{N+1}$. Mean width is estimated with a test set of size \(n_{\textrm{test}}=1000\). We consider CQR with random forests and both homoscedastic and heteroscedastic GPs with Mat\'ern $5/2$ kernel. As for kSoS, we also use a Mat\'ern $5/2$ kernel with a predictive model given by a homoscedastic GP for symmetric cases and a cubic spline for asymmetric ones. We evaluate the symmetric variant of \citet{allain2025scalableandadaptivepredictionbandsusingkerenlsumofsquares} and our penalized version with the penalty on the training set only, since we did not observe differences with the operator penalty in our experiments. For both kSoS methods, we train first an initial model with $\theta_{\mathrm{low}}=\theta_{\mathrm{up}}$ equal to the median of the feature distances (a usual rule-of-thumb for kernel methods) and extract the values of the mean-width and the norms, which serve as a normalization before setting the hyperparameter values \(\lambda_{(\cdot)1}=\lambda_{(\cdot)2}=1\) and \(b\) (depending on the test case). For all methods, we use a calibration set of size $m=2000$, and also compute the root mean-squared error (RMSE) with respect to the oracle prediction bands, as well as plot the density of local coverage.

Case 1. Inspired from \citet{gramacy2009adaptivedesignanalysissupercomputerexperiments}.
    \begin{align*}
        &d=1,\quad X\sim \mathcal{U}[-1,1],\quad Y=m(X)+\sigma(X)\epsilon, \quad\epsilon\sim\mathcal{N}(0,1)\\
        &m(X) = 
        \begin{cases}
            \sin(\pi (2X+1/5)) + 0.2\cos(4\pi(2X+1/5))\quad \text{if}\; 10X+1 \leq 9.6 \\
            X - 9/10 \quad \text{otherwise}
        \end{cases}\\
        &\sigma(X) = \sqrt{0.1+2X^2}
    \end{align*}
    
We display the adaptivity metrics in Figure \ref{fig:comparisons_case_1}. CQR and homoscedastic GP tend to produce intervals that overcover and are too large. Heteroscedastic GP, symmetric and penalized kSoS yield better \(\mathrm{ACG}^{c}\) (with a slight advantage for penalized kSoS), but both kSoS models have smaller mean width among all competitors. This is confirmed by the density plot of local coverage, where heteroscedastic GP, symmetric and penalized kSoS have local coverage density closer to the target $1-\alpha/2$ than homoscedastic GP and CQR. This illustrates that the automatic choice of the penalty allows to recover the best performing results of symmetric kSoS, which is also confirmed with both of them reaching the smallest RMSE. 

\begin{figure}[htbp]
        \centering
        \includegraphics[width=0.8\linewidth]{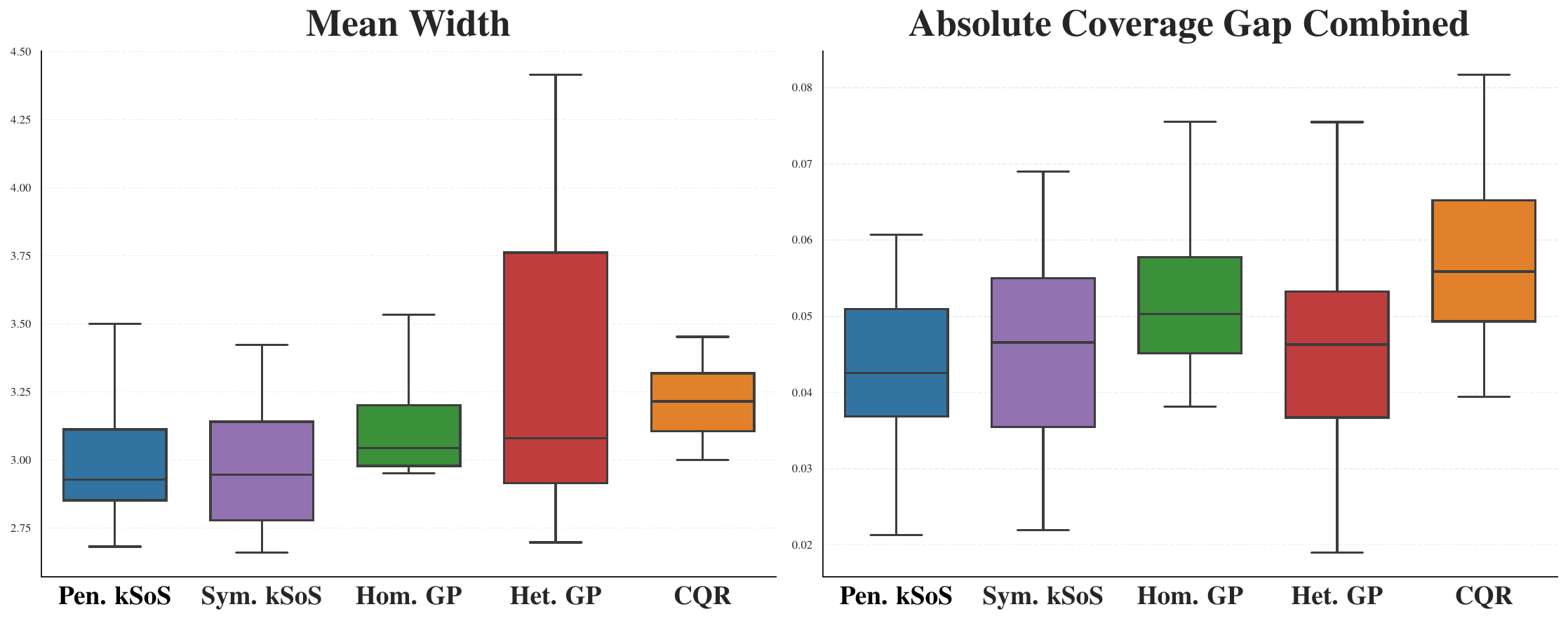}
        \includegraphics[width=0.4\linewidth]{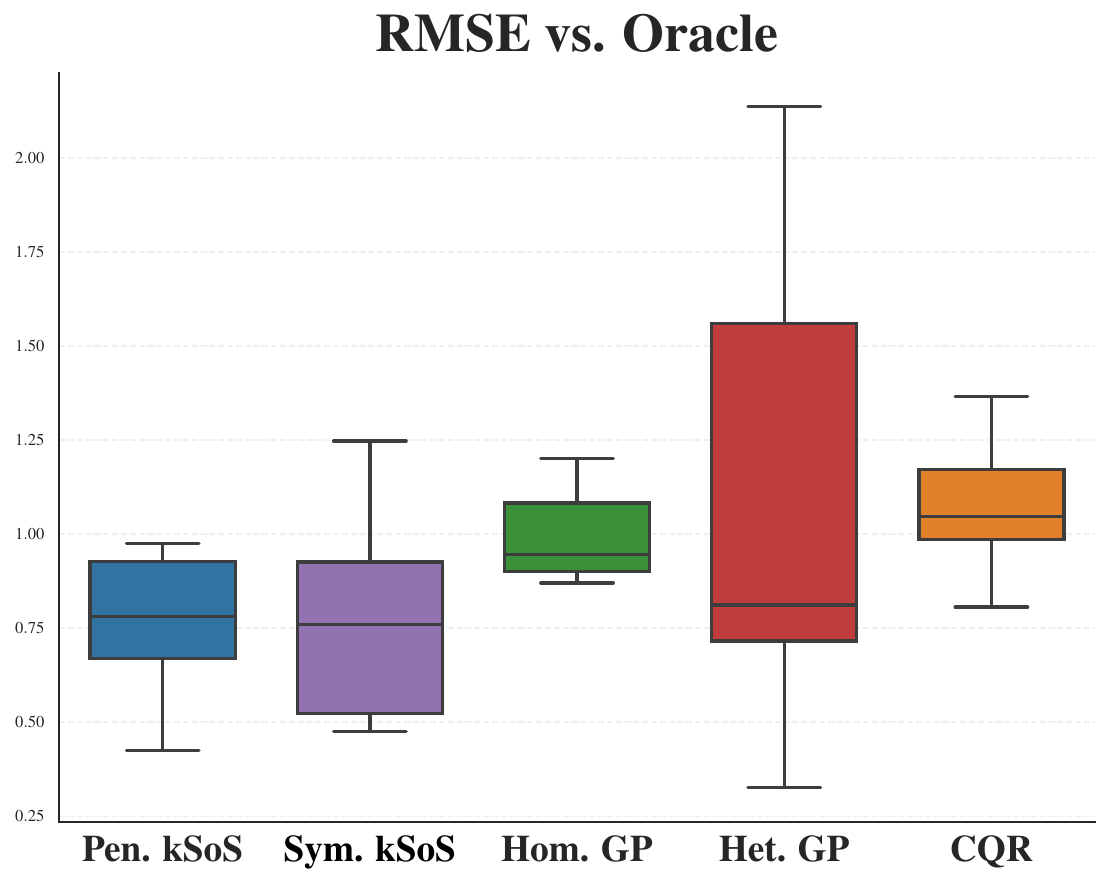}
        \includegraphics[width=0.4\linewidth]{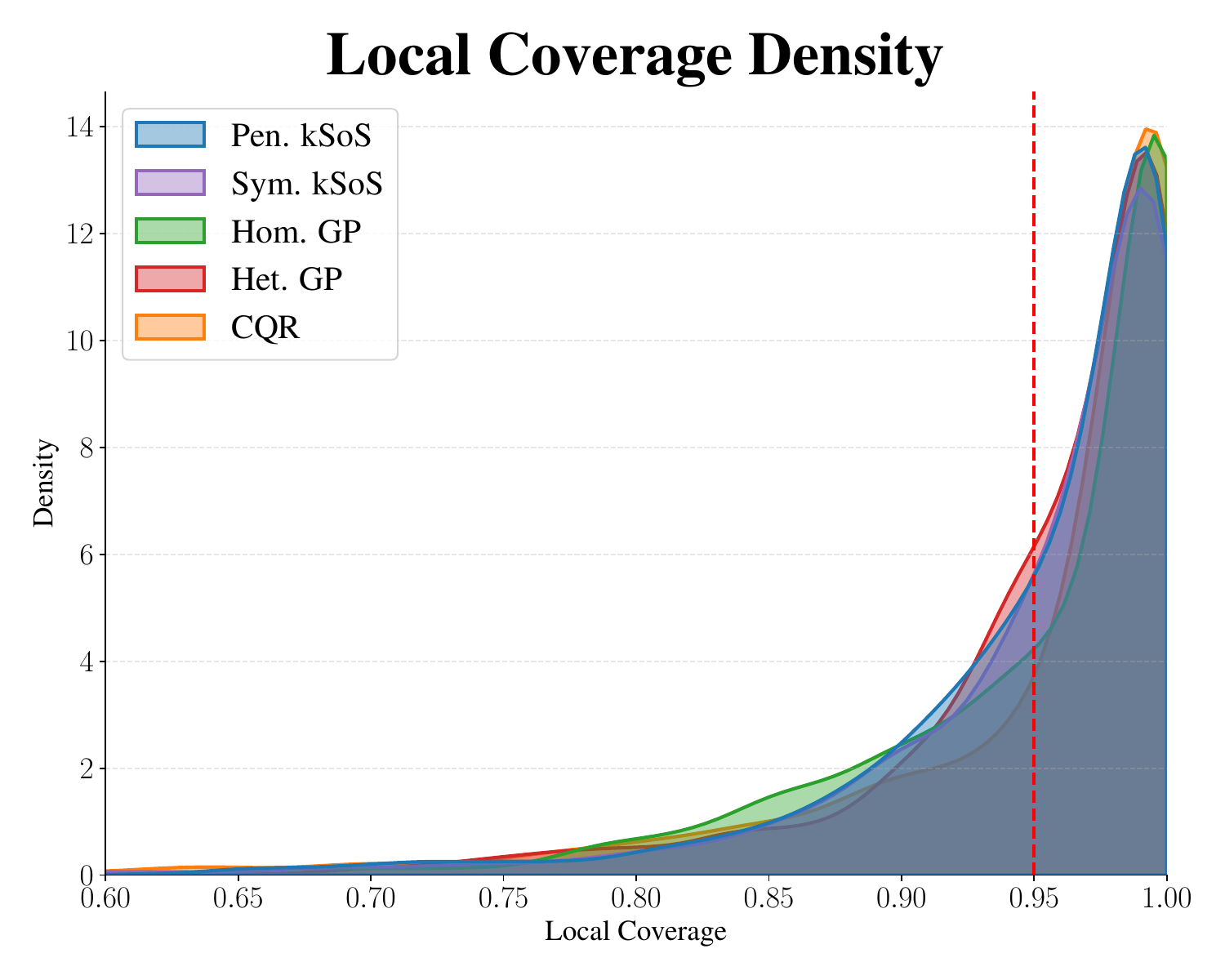}
        \caption{Test case 1 with \(d=1\) and \(n=100\). Mean width, Absolute Coverage Gap combined, RMSE versus oracle and density of local coverage lower and upper combined, \(b=10\).}
        \label{fig:comparisons_case_1}
\end{figure}

Figure \ref{fig:case_1_n_1000} gives the optimal solution of our dual formulation for $n=1000$.

\begin{figure}[htbp]
        \centering
        \includegraphics[width=0.6\linewidth]{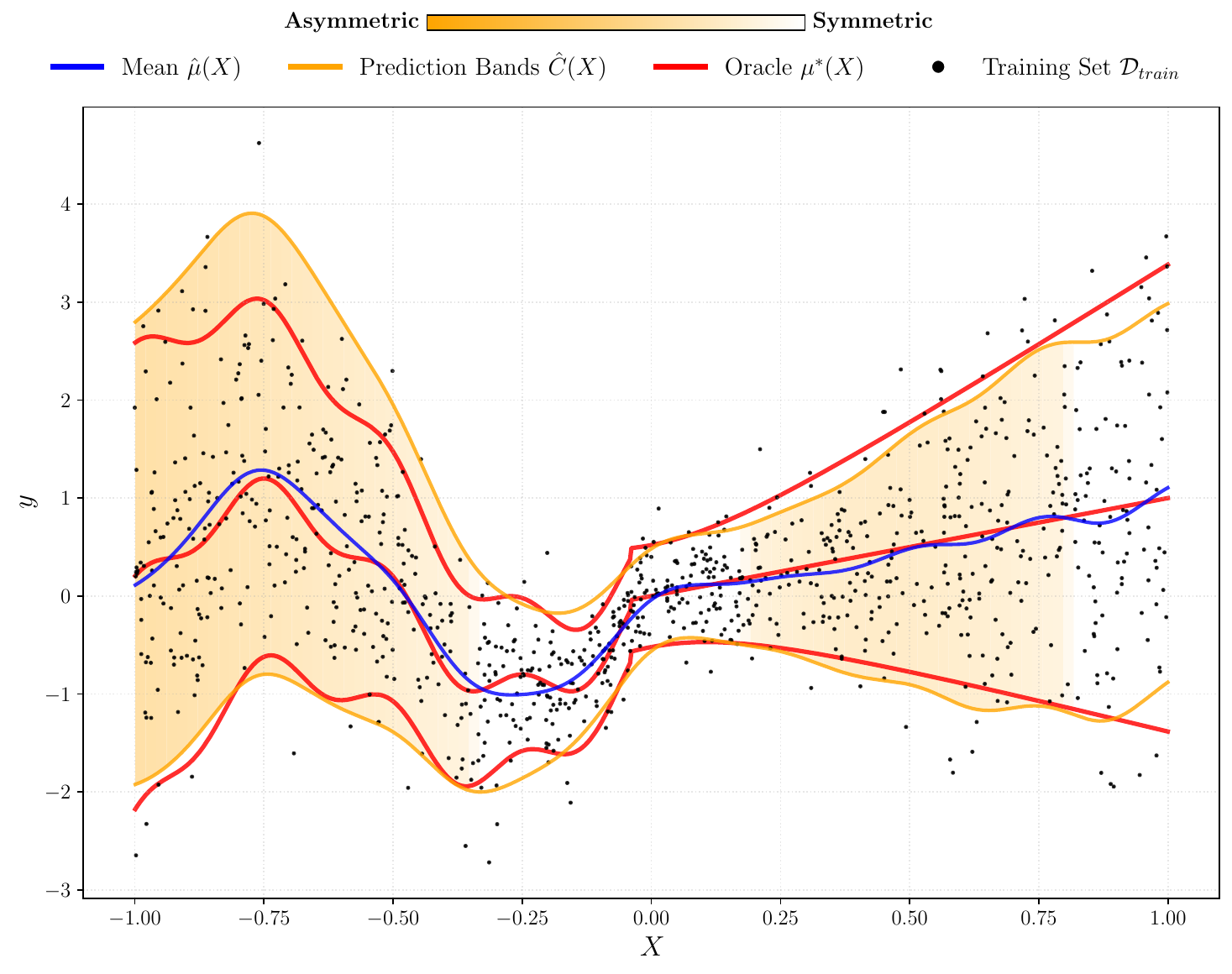}
        \caption{Test case 1 with \(d=1\) and \(n=1000\). Optimal solution of dual formulation with penalty 1.}
        \label{fig:case_1_n_1000}
\end{figure}
    
Case 2. Corresponds to setting 1 in \citet{hore2024conformalpredictionlocalweights}.
    \begin{align*}
        &X\sim \mathcal{N}_{d}(0,I_{d}),\quad Y=m(X)+\sigma(X)\epsilon, \quad\epsilon\sim\mathcal{N}(0,1)\\
        &m(X) = 0.5\sum_{i=1}^d X^{(i)}\\
        &\sigma(X) = \sum_{i=1}^d \vert\sin(X^{(i)})\vert
    \end{align*}
    
Similarly to the previous case, penalized kSoS and heteroscedastic GP have better \(\mathrm{ACG}^{c}\) than other competitors and equivalent to symmetric kSoS, as can be seen in Figure \ref{fig:comparisons_case_2}. This is again confirmed by the local coverage densities: they show that homoscedastic GP tends to undercover and overcover more than other methods, which can be explained by the constant width of the produced bands. But penalized kSoS has much smaller mean width and RMSE, although it does not reach the mean width or RMSE of symmetric kSoS.
    
\begin{figure}[htbp]
    \centering
        \includegraphics[width=0.8\linewidth]{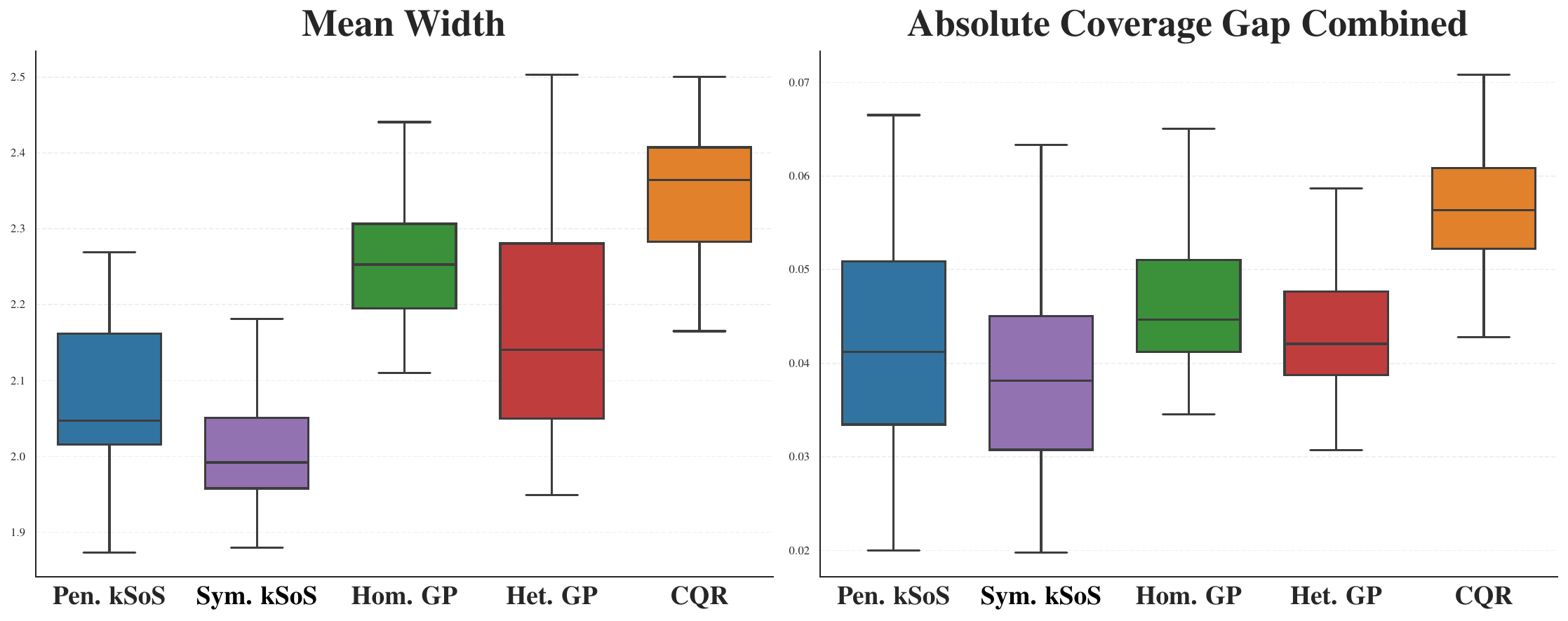}
        \includegraphics[width=0.4\linewidth]{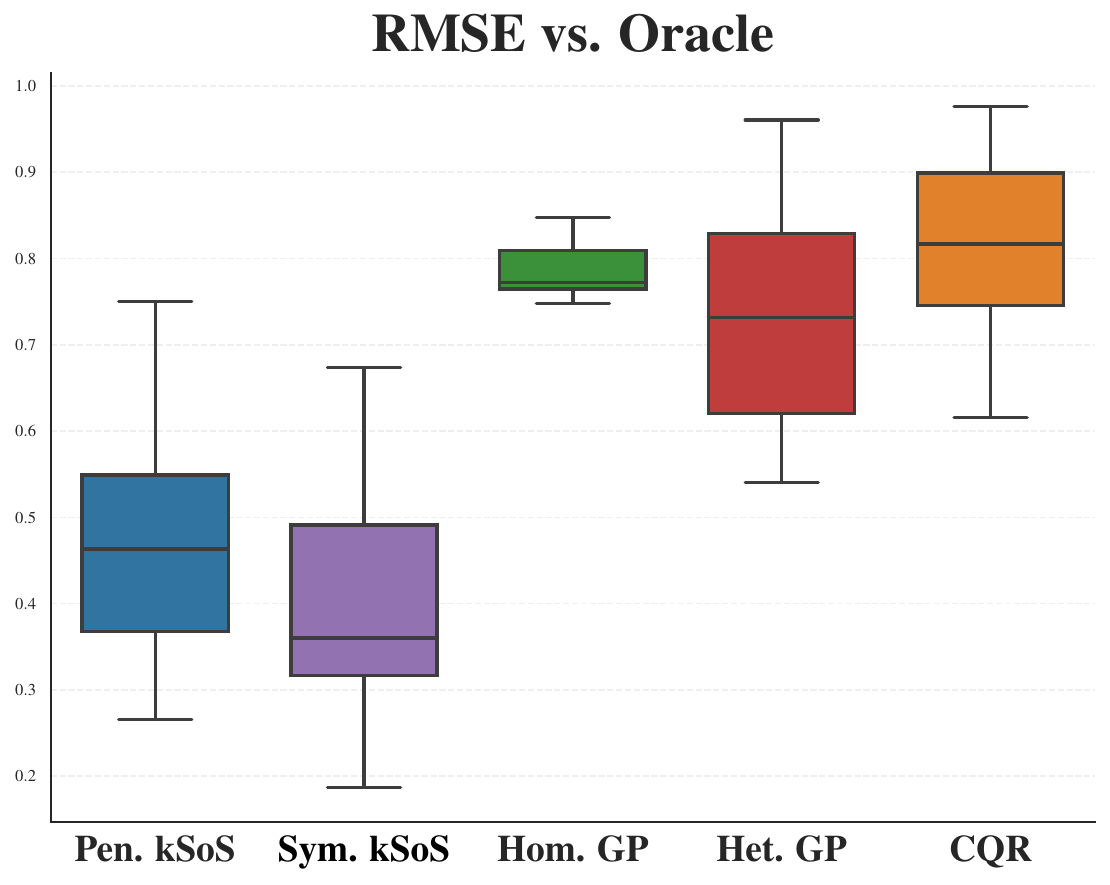}
        \includegraphics[width=0.4\linewidth]{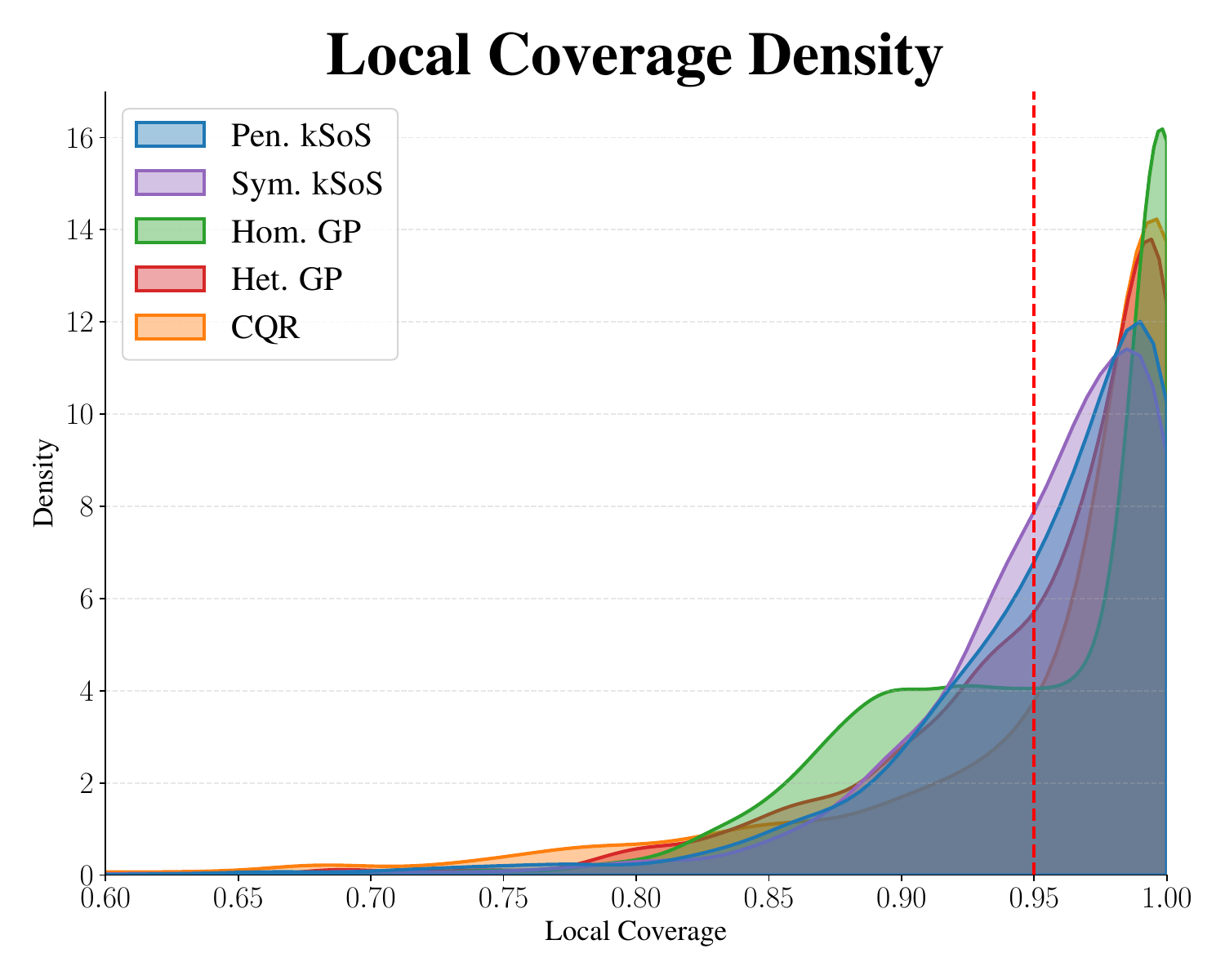}
    \caption{Test case 2 with \(d=1\) and \(n=100\). Mean width, Absolute Coverage Gap combined, RMSE versus oracle and density of local coverage lower and upper combined, \(b=100\).}
    \label{fig:comparisons_case_2}
\end{figure}

With $n=1000$, we obtain in Figure \ref{fig:case_2_n_1000} the following optimal solution of the dual formulation.

\begin{figure}[htbp]
        \centering
        \includegraphics[width=0.6\linewidth]{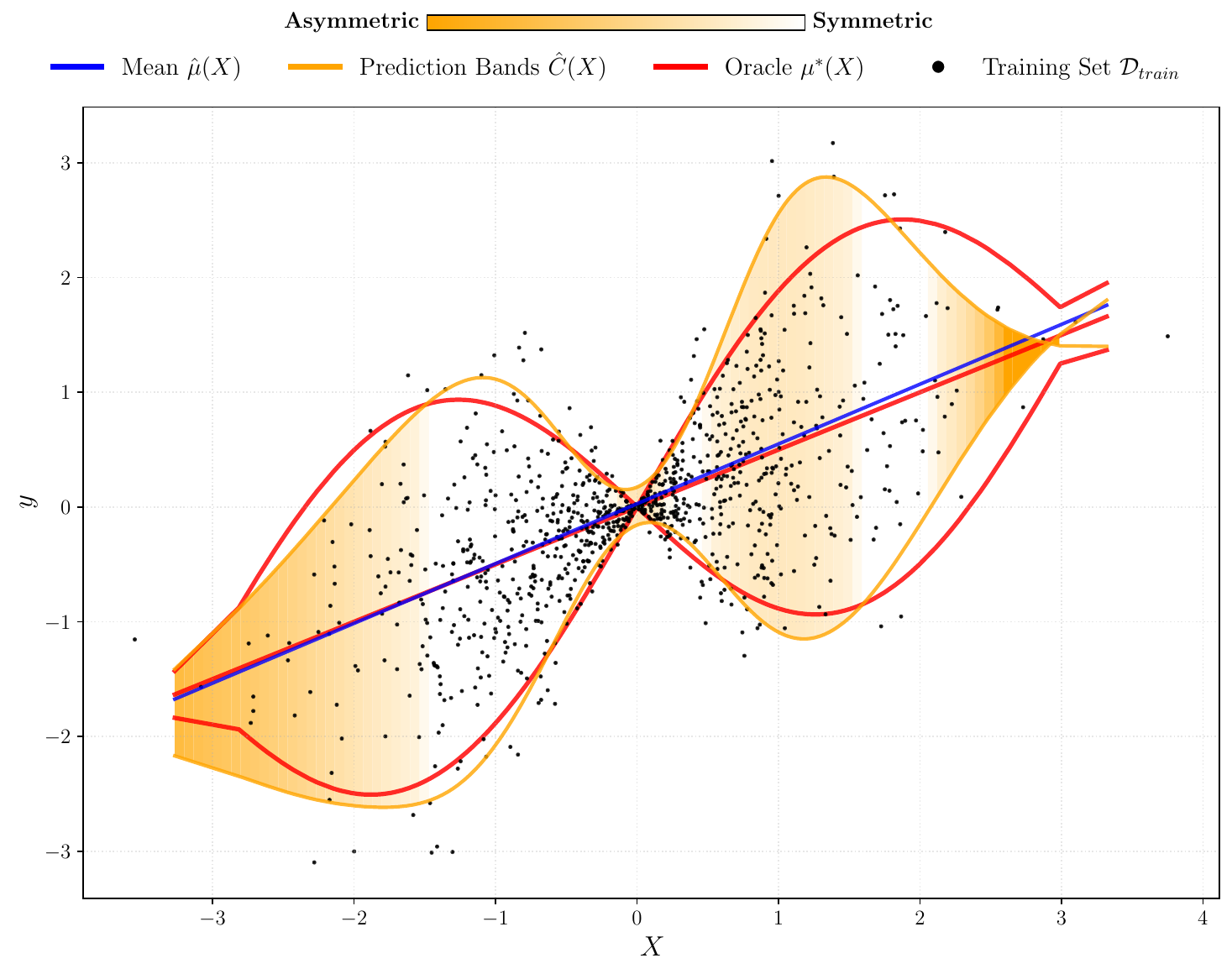}
        \caption{Test case 2 with \(d=1\) and \(n=1000\). Optimal solution of dual formulation with penalty 1.}
        \label{fig:case_2_n_1000}
\end{figure}
    
Case 3. Inspired from \citet{braun2025minimumvolumeconformalsets}, with a lognormal noise: 
\begin{align*}
    &d=1,\quad X\sim \mathcal{U}[-1,1],\quad Y=m(X)+\sigma(X)\epsilon, \quad\epsilon\sim\mathrm{Lognormal}(0,1)\\
    &m(X) = \sin(5X)\\
    &\sigma(X) = X
\end{align*}

Figure \ref{fig:comparisons_case_3} demonstrates once again that on such asymmetric case, penalized kSoS produces intervals with much better \(\mathrm{ACG}^{c}\) and smaller mean width than all competitors, except for homoscedastic GP in terms of mean width. But as already discussed, using this indicator only can be misleading. As in the previous test case, penalized kSoS also approximates the oracle prediction intervals with a much higher accuracy, with a very small RMSE. The local coverage density plot clearly shows that penalized kSoS has the best local coverage. It also showcases a similar behavior for all symmetric methods (symmetric kSoS, homoscedastic and heteroscedastic GP) with a tendency to overcover in some input regions and undercover in others.

\begin{figure}[htbp]
    \centering
        \includegraphics[width=0.8\linewidth]{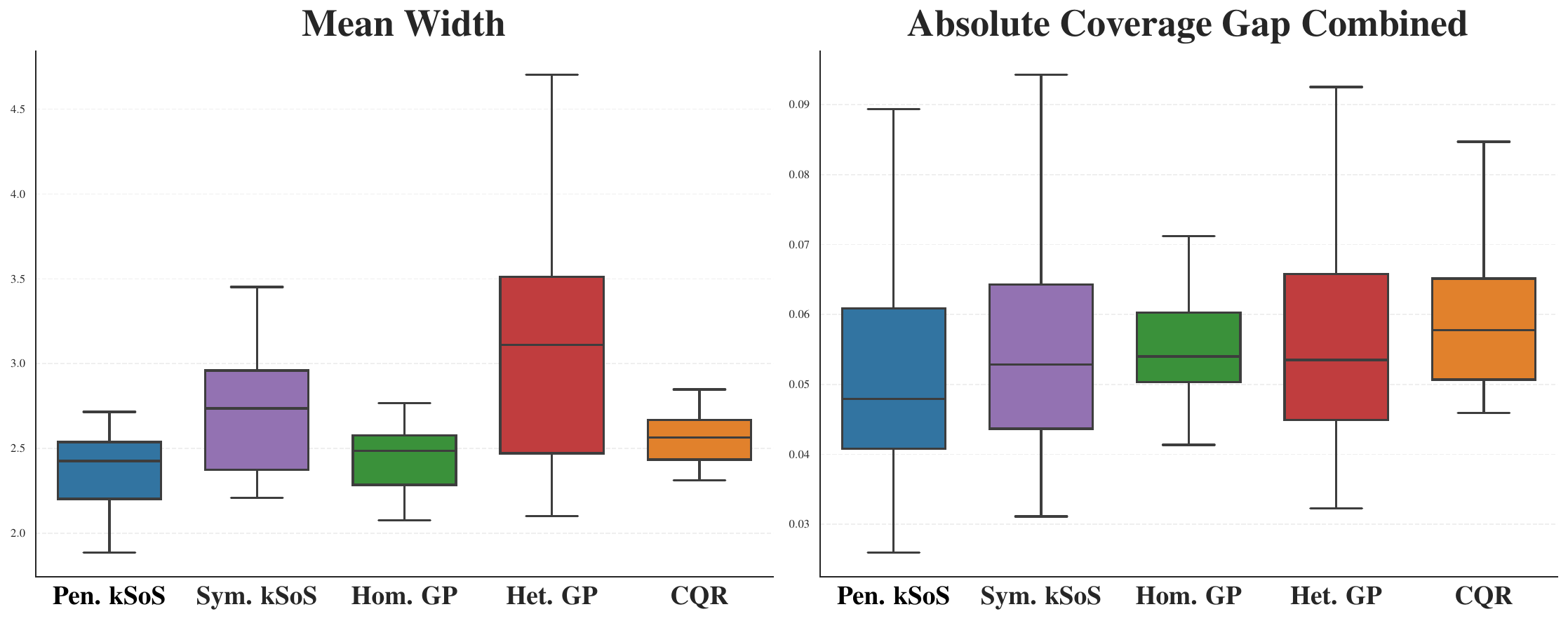}
        \includegraphics[width=0.4\linewidth]{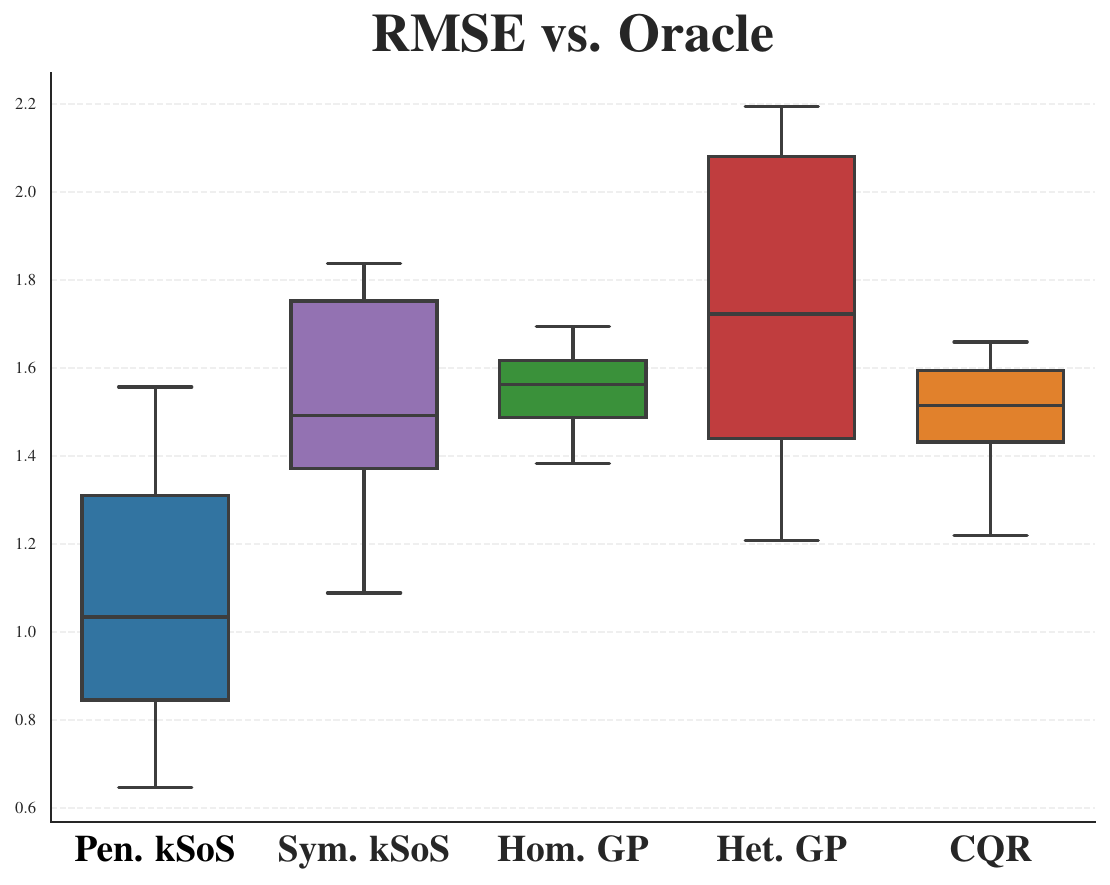}
        \includegraphics[width=0.4\linewidth]{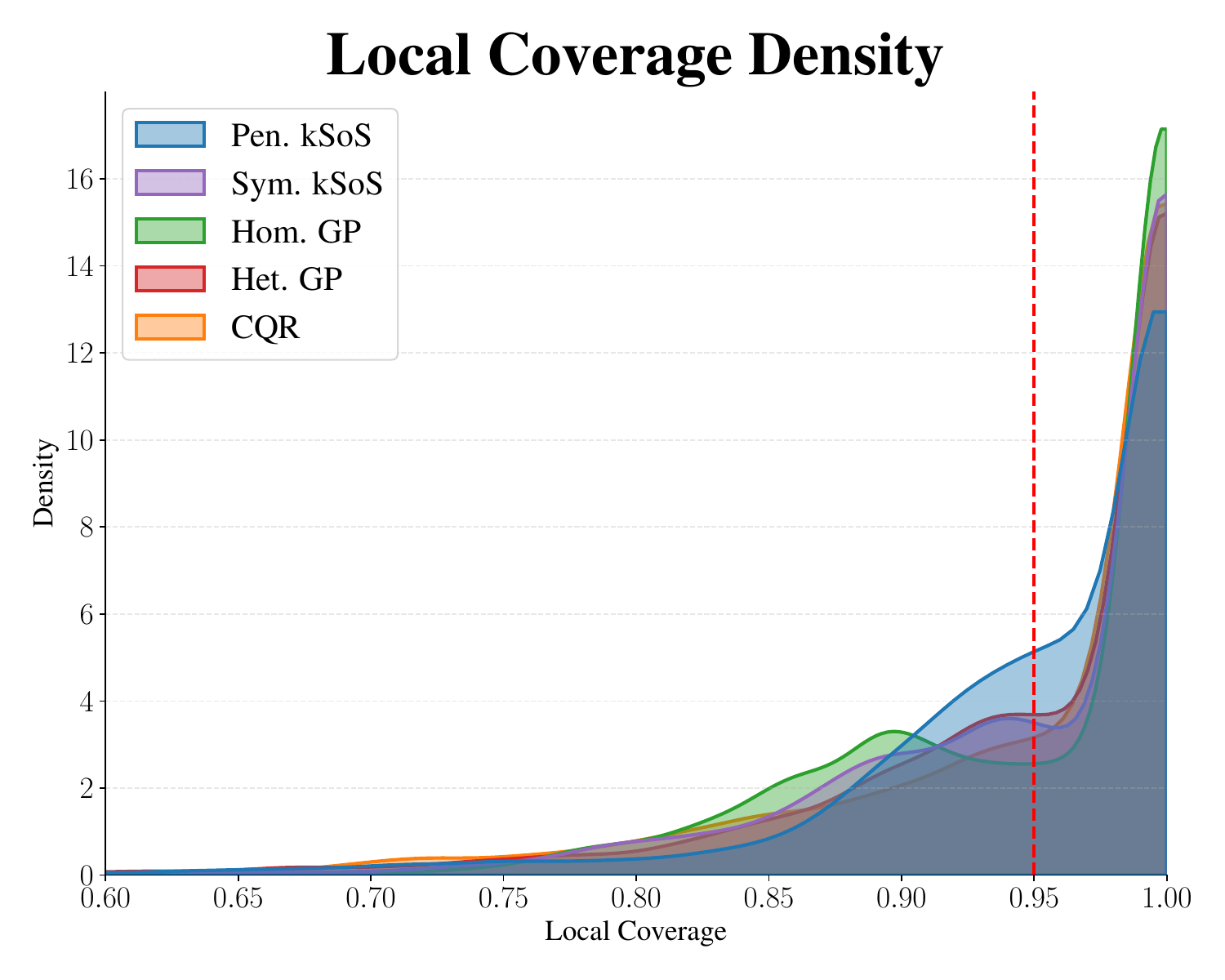}
    \caption{Test case 3 with \(d=1\) and \(n=100\). Mean width, Absolute Coverage Gap combined, RMSE versus oracle and density of local coverage lower and upper combined, \(b=10\).}
    \label{fig:comparisons_case_3}
\end{figure}

In Figure \ref{fig:case_3_n_1000} we display the optimal solution of the dual formulation for $n=1000$.

\begin{figure}[htbp]
        \centering
        \includegraphics[width=0.6\linewidth]{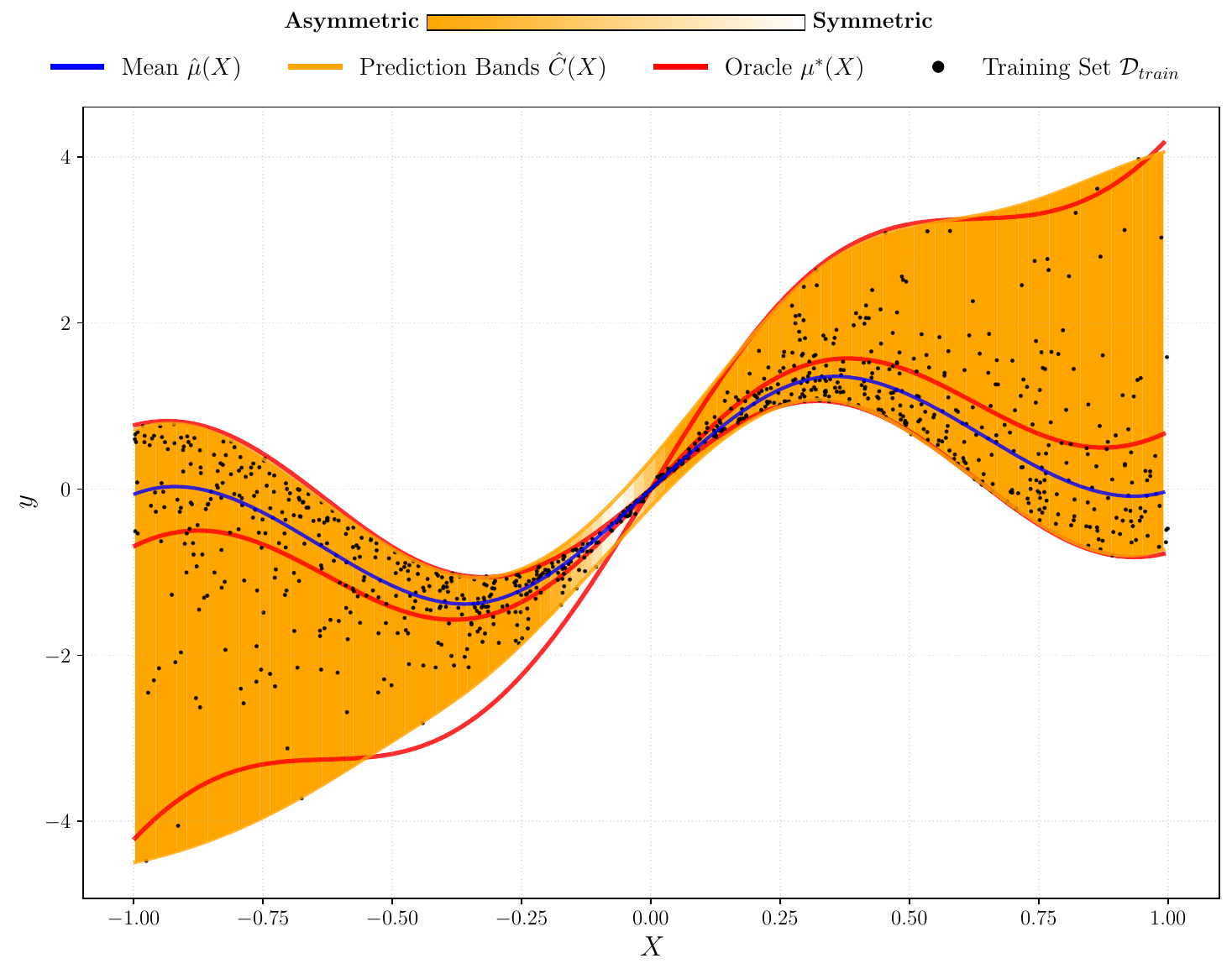}
        \caption{Test case 3 with \(d=1\) and \(n=1000\). Optimal solution of dual formulation with penalty 1.}
        \label{fig:case_3_n_1000}
\end{figure}

Case 4. Asymmetric test case with split-normal noise.
\begin{align*}
    &d=1,\quad X\sim \mathcal{U}[0,4\pi],\quad Y=m(X)+\epsilon\Bigl[\sigma_-(X)\,\mathbf{1}_{\{\epsilon<0\}} + \sigma_+(X)\,\mathbf{1}_{\{\epsilon\ge 0\}}\Bigr], \quad\epsilon\sim\mathcal{N}(0,1)\\
    &m(X) = \sin(X)\\
    &\sigma_-(X) = 0.2\\
    & \sigma_+(X) = 0.4(\sin(X) + 1) + 0.1
\end{align*}

Case 4 involves asymmetric bands with a heteroscedastic upper band and a homoscedastic lower one. The flexibility of our framework, which allows to use different kernels for the upper and lower bands, is particularly appropriate here. Figure \ref{fig:comparisons_case_4} demonstrates that penalized kSoS is much more adaptive than any other method, when measured with the \(\mathrm{ACG}^{c}\) or when looking at the densities of local coverage. Penalized kSoS also produces much narrower prediction bands and achieves much higher accuracy, with a very small RMSE.

\begin{figure}[htbp]
    \centering
        \includegraphics[width=0.8\linewidth]{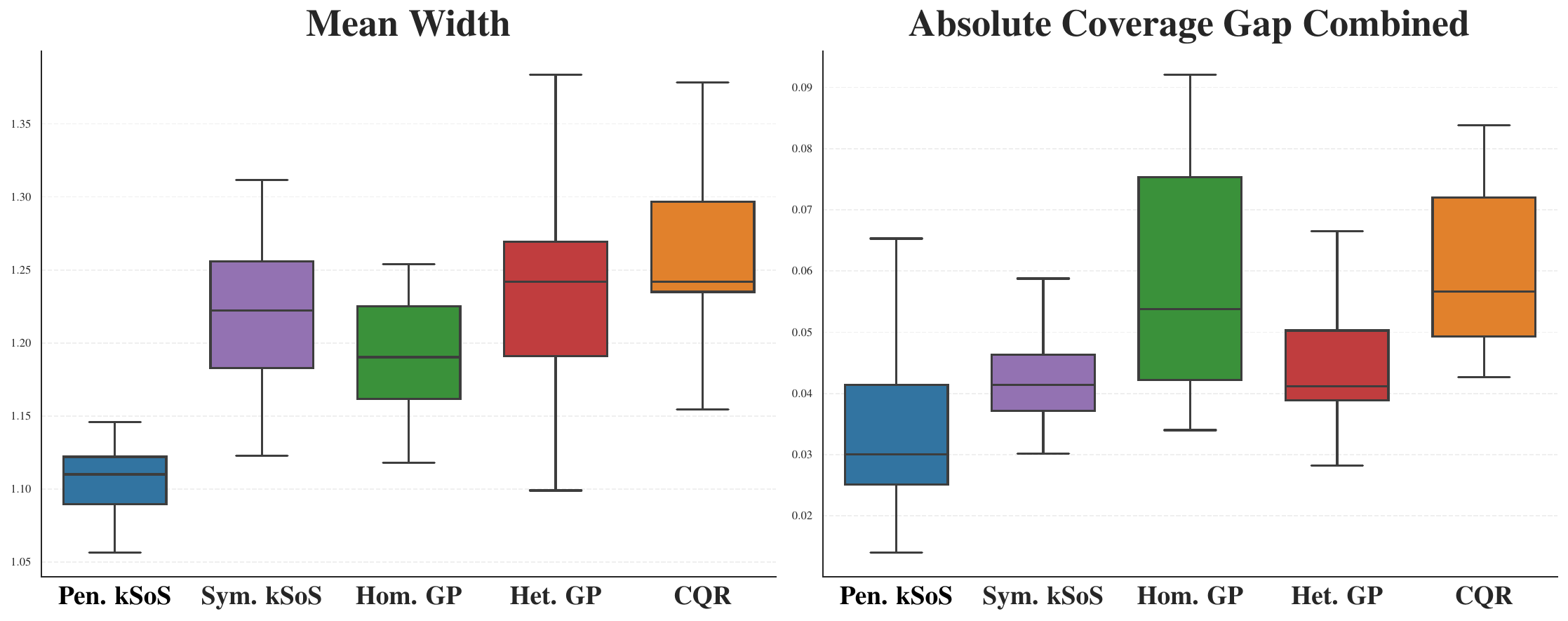}
        \includegraphics[width=0.4\linewidth]{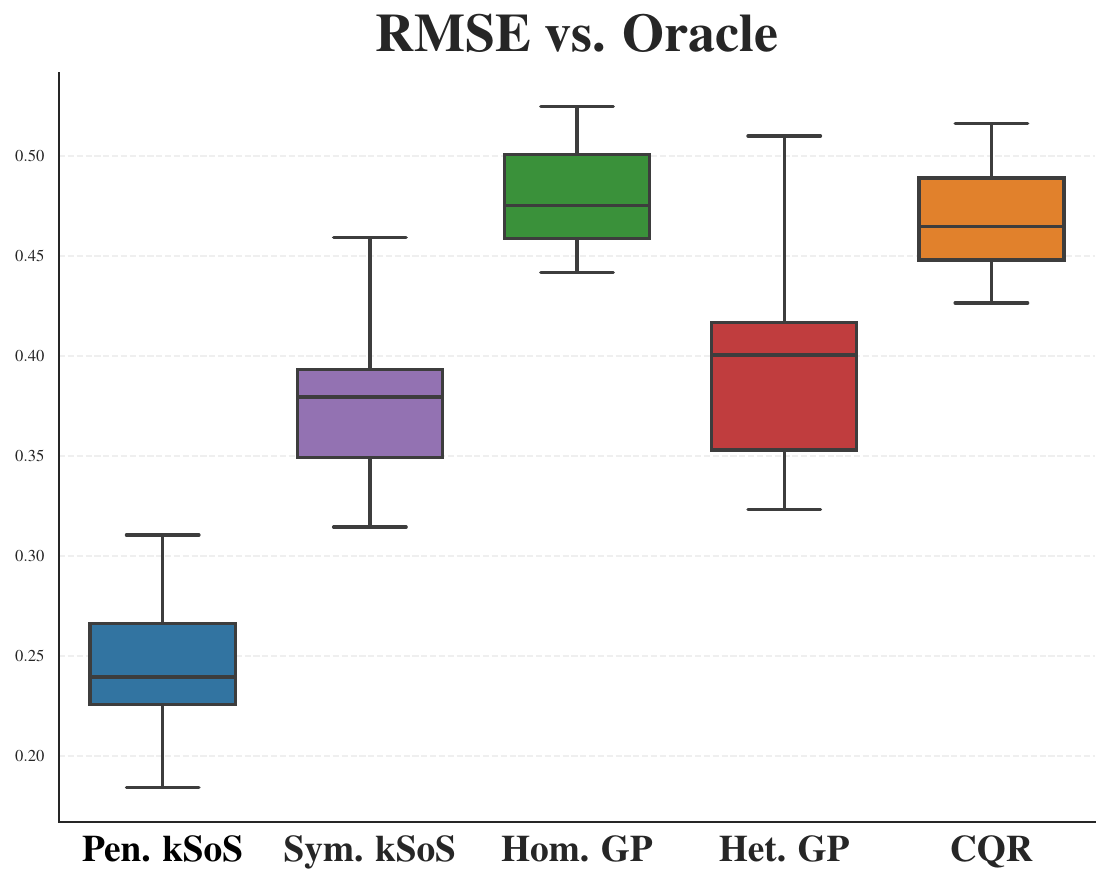}
        \includegraphics[width=0.4\linewidth]{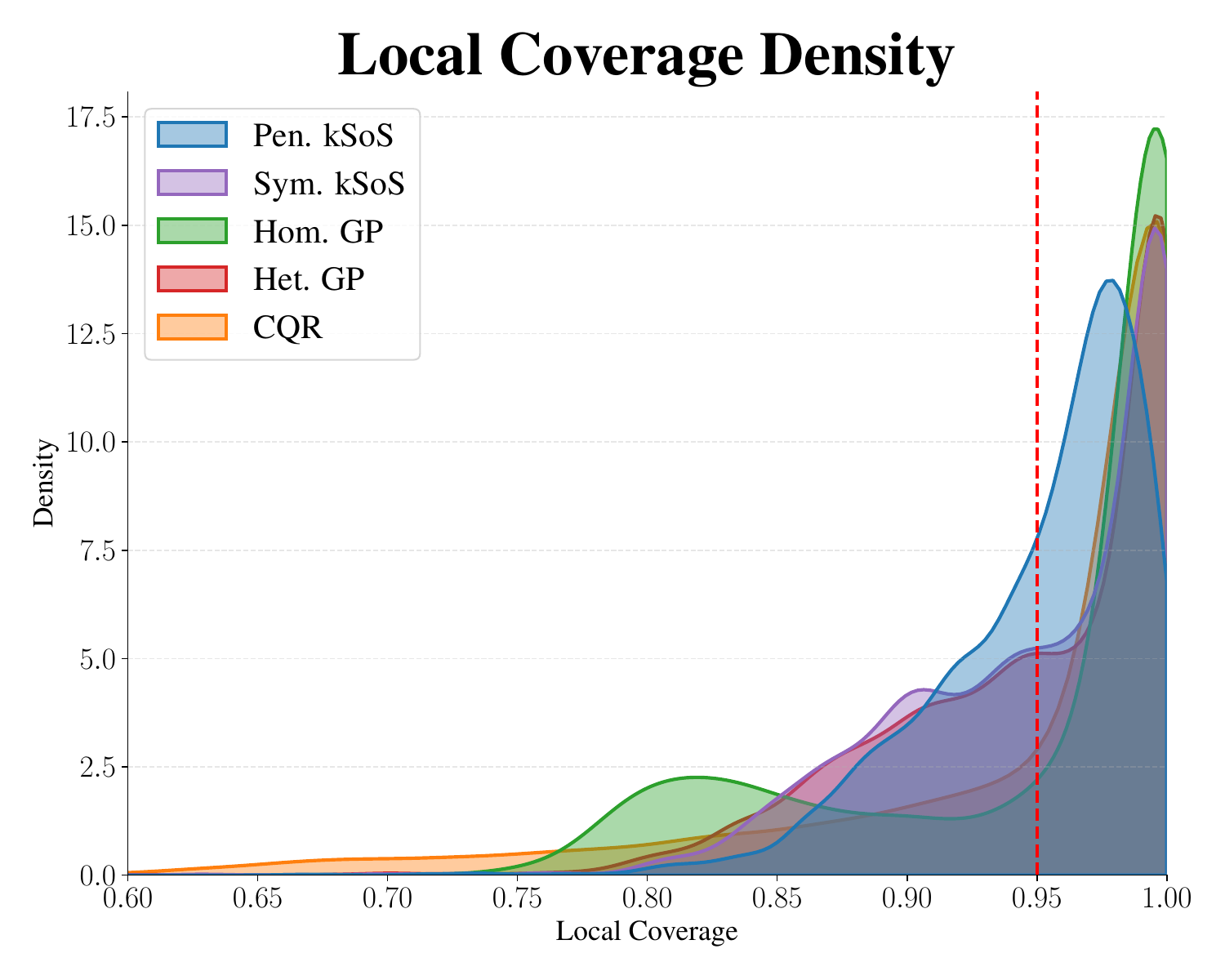}
    \caption{Test case 4 with \(d=1\) and \(n=500\). Mean width, Absolute Coverage Gap combined, RMSE versus oracle and density of local coverage lower and upper combined, \(b=10\).}
    \label{fig:comparisons_case_4}
\end{figure}

We give in Figure \ref{fig:case_4_n_1000} the optimal solution of the dual formulation obtained with $n=1000$.

\begin{figure}[htbp]
        \centering
        \includegraphics[width=0.6\linewidth]{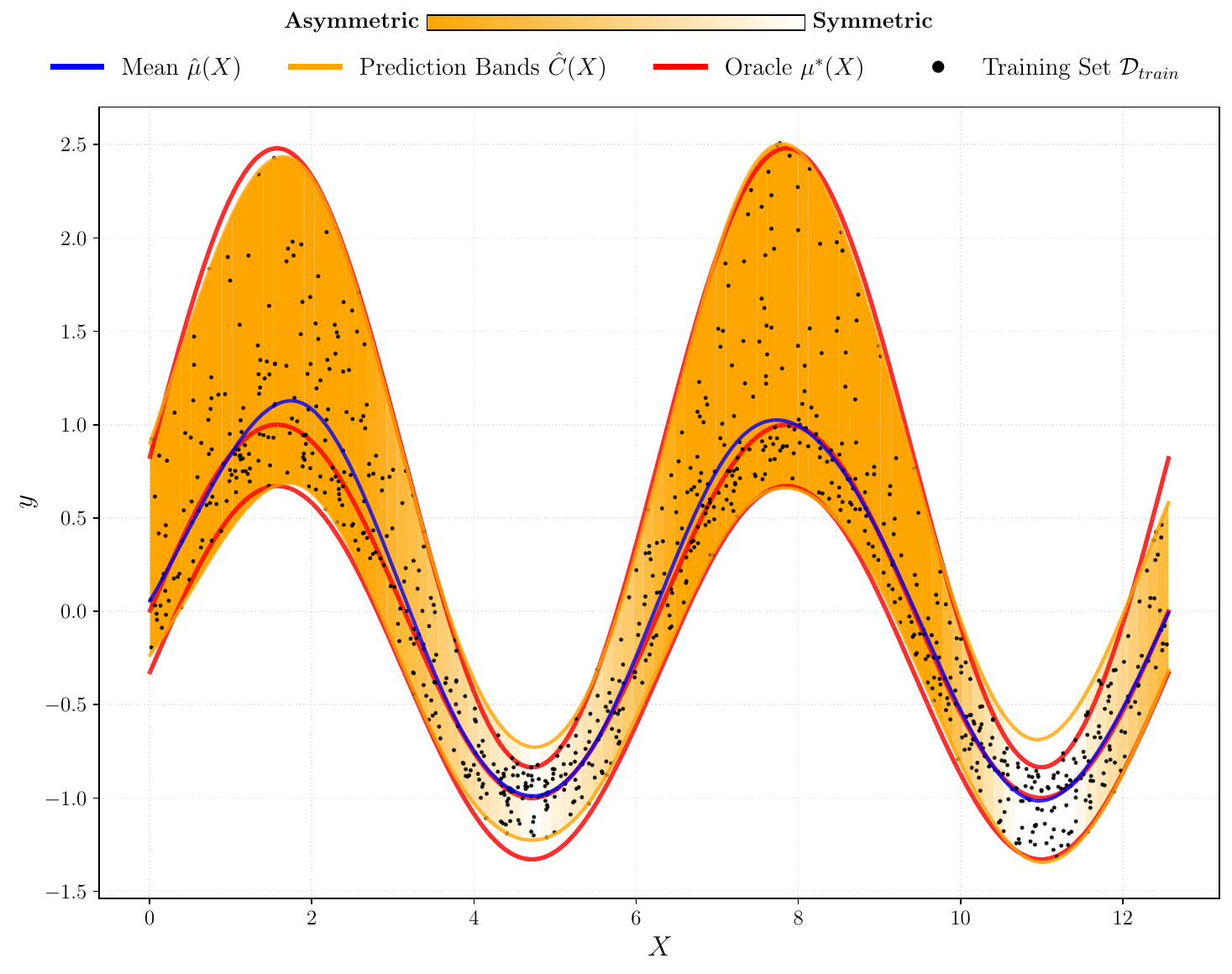}
        \caption{Test case 4 with \(d=1\) and \(n=1000\). Optimal solution of dual formulation with penalty 1.}
        \label{fig:case_4_n_1000}
\end{figure}

Case 5. Test case from \citet{braun2025minimumvolumeconformalsets}, which involves an exponentially distributed noise: 
\begin{align*}
    &d=1,\quad X\sim \mathcal{U}[-1,1],\quad Y=m(X)+\sigma(X)\epsilon, \quad\epsilon\sim\mathcal{E}(1)\\
    &m(X) = \sin(2X)\\
    &\sigma(X) = 0.5+2X
\end{align*}

For this asymmetric dataset, Figure \ref{fig:comparisons_case_5} shows that penalized kSoS has much better local coverage (as measured by the \(\mathrm{ACG}^{c}\) and in the density of local coverage plot) than all competitors, where all symmetric procedures have poor adaptivity. In addition, penalized kSoS also yields the smallest mean width, with a much smaller RMSE with respect to oracle prediction bands.

\begin{figure}[htbp]
    \centering
        \includegraphics[width=0.8\linewidth]{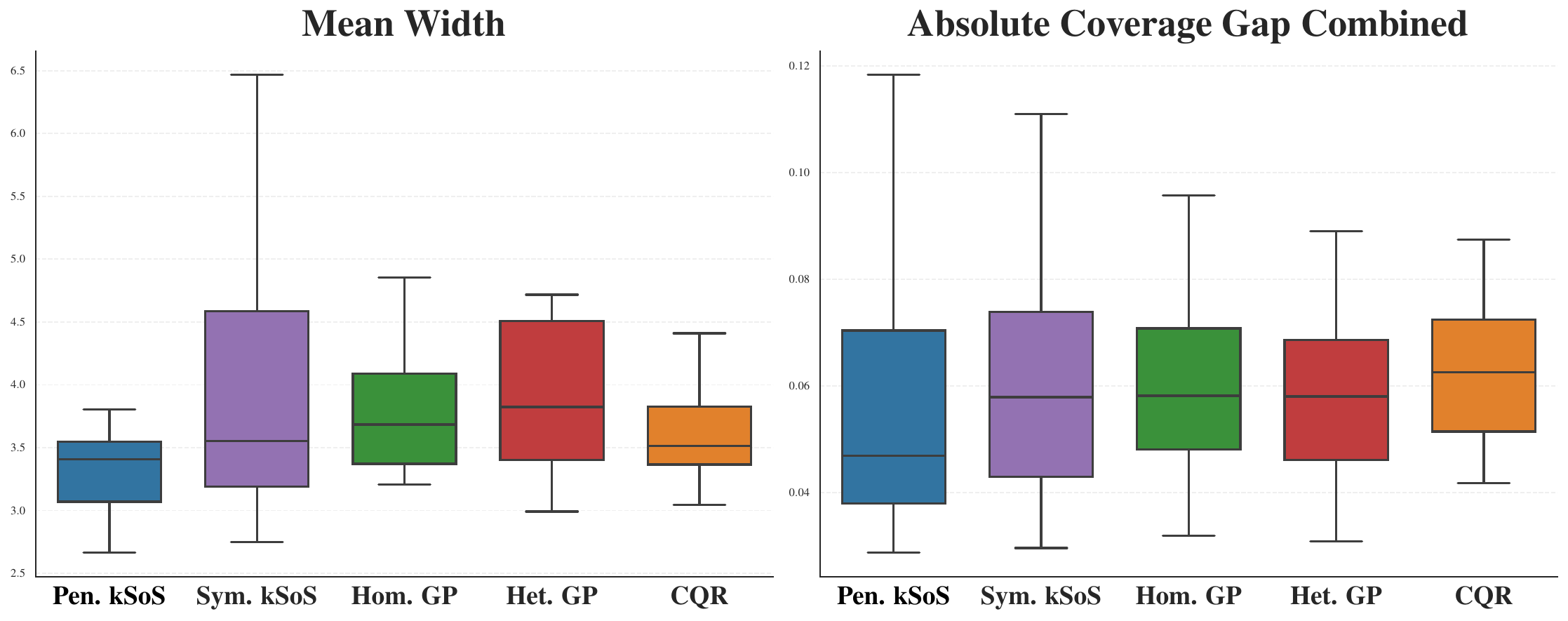}
        \includegraphics[width=0.4\linewidth]{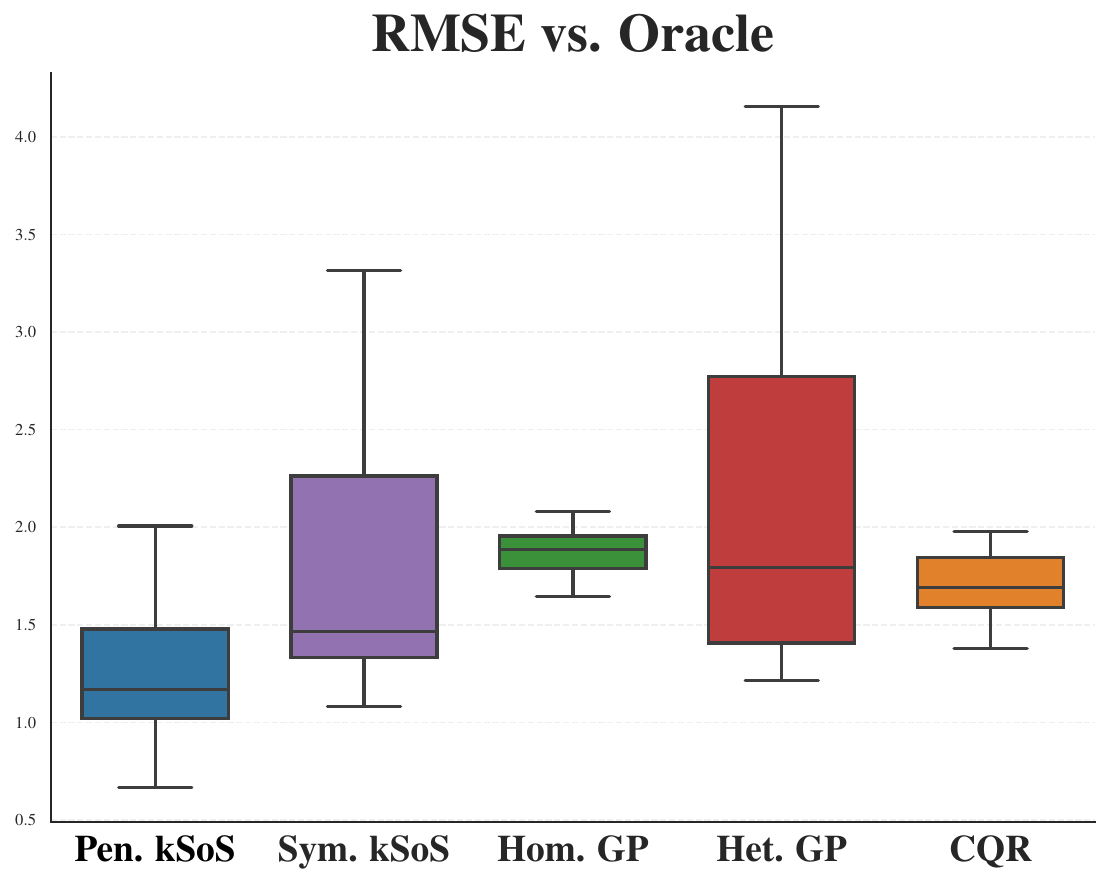}
        \includegraphics[width=0.4\linewidth]{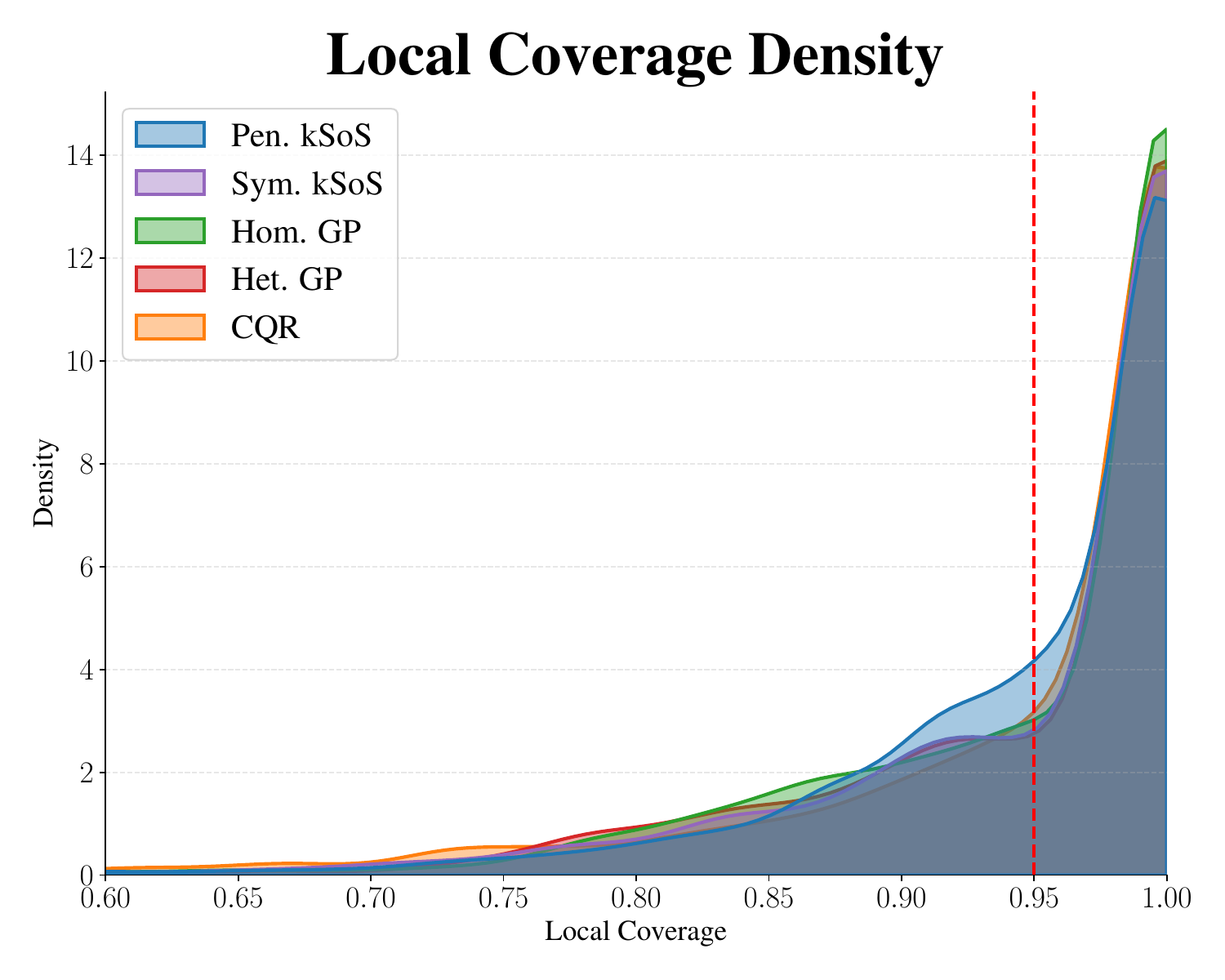}
    \caption{Test case 5 with \(d=1\) and \(n=100\). Mean width, Absolute Coverage Gap combined, RMSE versus oracle and density of local coverage lower and upper combined, \(b=10\).}
    \label{fig:comparisons_case_5}
\end{figure}

The optimal solution of the dual formulation obtained for $n=1000$ is given in Figure \ref{fig:case_5_n_1000}.

\begin{figure}[htbp]
        \centering
        \includegraphics[width=0.6\linewidth]{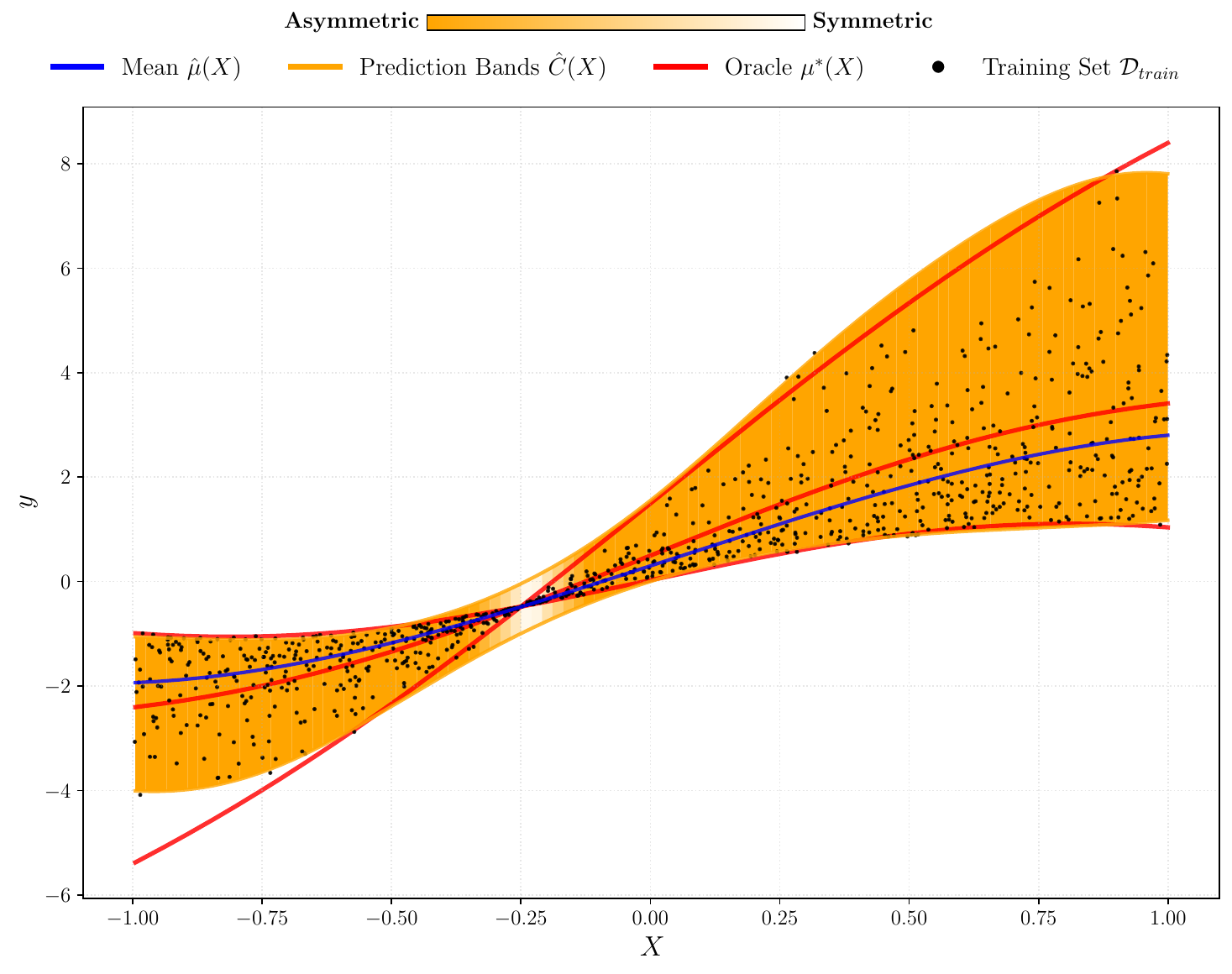}
        \caption{Test case 5 with \(d=1\) and \(n=1000\). Optimal solution of dual formulation with penalty 1.}
        \label{fig:case_5_n_1000}
\end{figure}

Case 6. Inspired from \citet{kivaranovic2020adaptive}.
    \begin{align*}
        &X\sim \mathcal{U}[0,1]^d,\quad Y=m(X)+\sigma(X)\epsilon, \quad\epsilon\sim\mathcal{N}(0,1)\\
        &m(X) = 2 \sin(\pi \beta^{\top}X) + \pi \beta^{\top}X \\
        &\sigma(X) = \sqrt{1+(\beta^{\top}X)^2}
    \end{align*}

The oracle prediction bands are close to be constant for this test case: we expect the automatic HSIC independence test to activate. In dimension $d=1$ we set $\beta=1$ and obtain the remaining adaptivity metrics given in Figure \ref{fig:comparisons_case_6}.

\begin{figure}[htbp]
    \centering
        \includegraphics[width=0.8\linewidth]{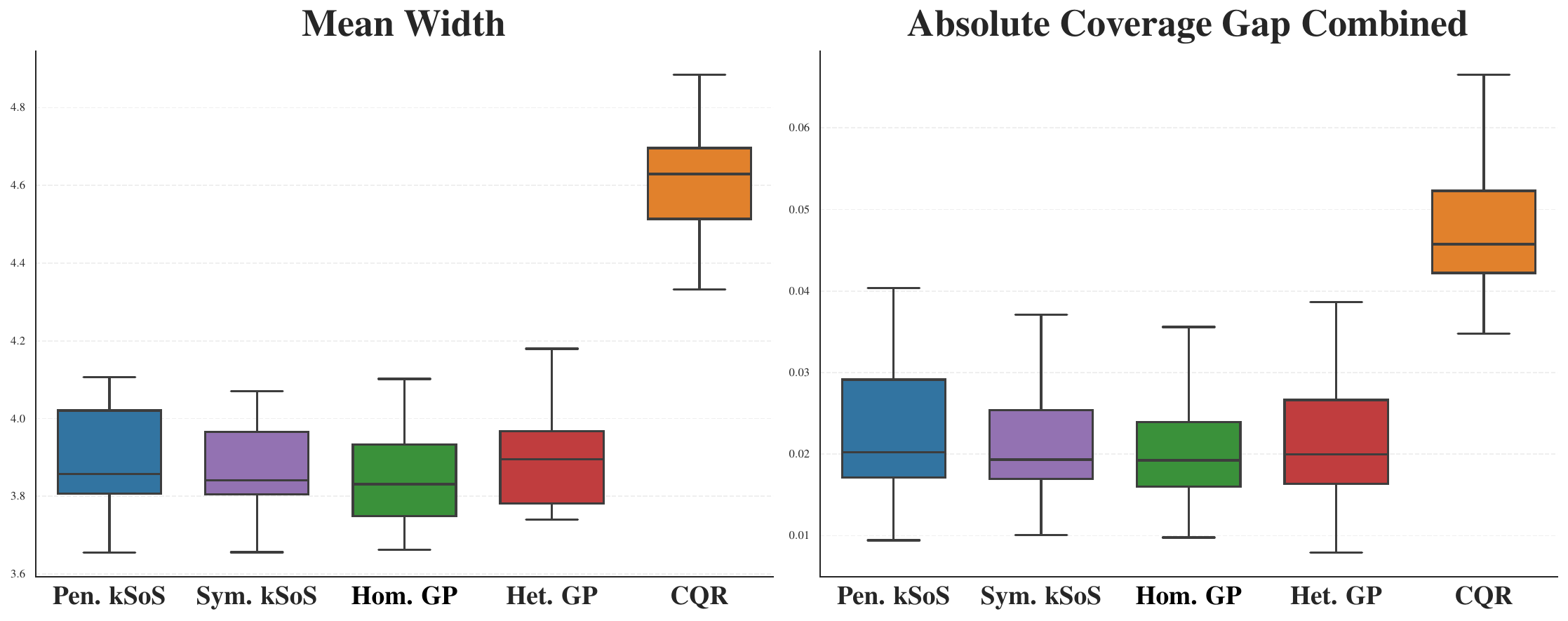}
        \includegraphics[width=0.4\linewidth]{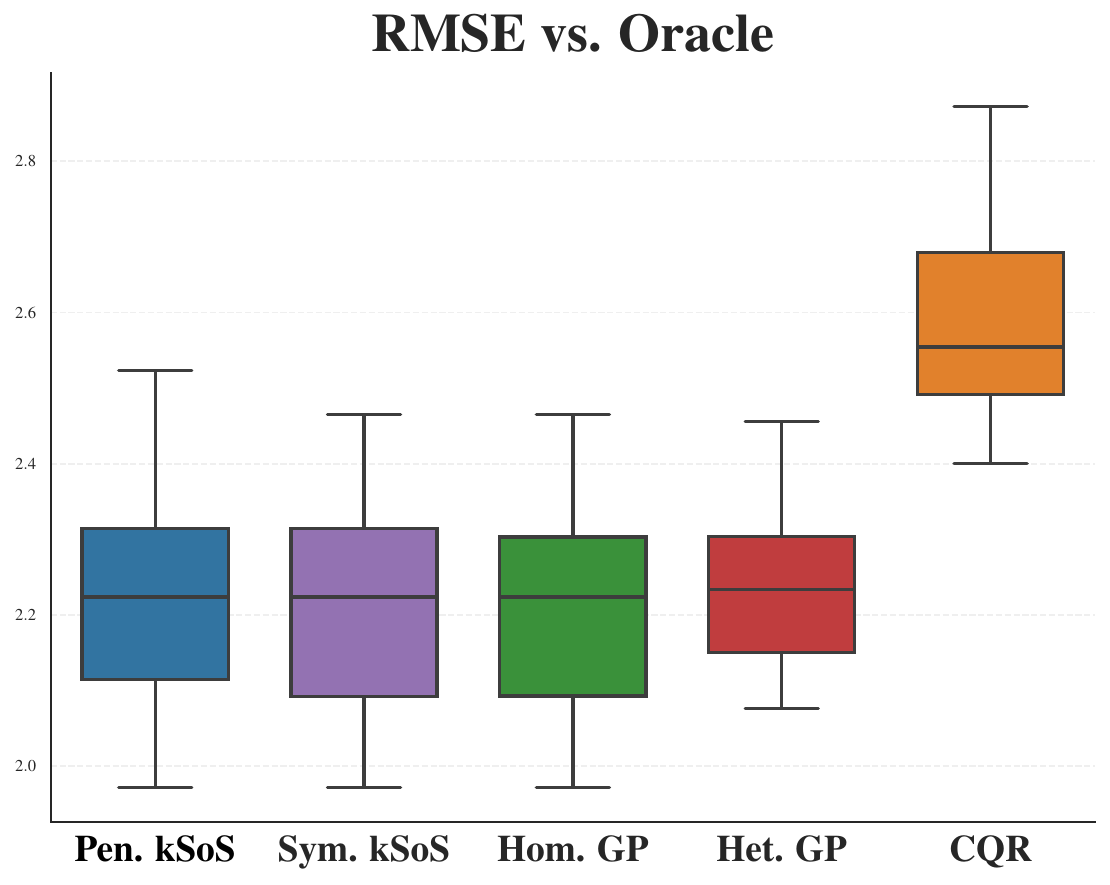}
        \includegraphics[width=0.4\linewidth]{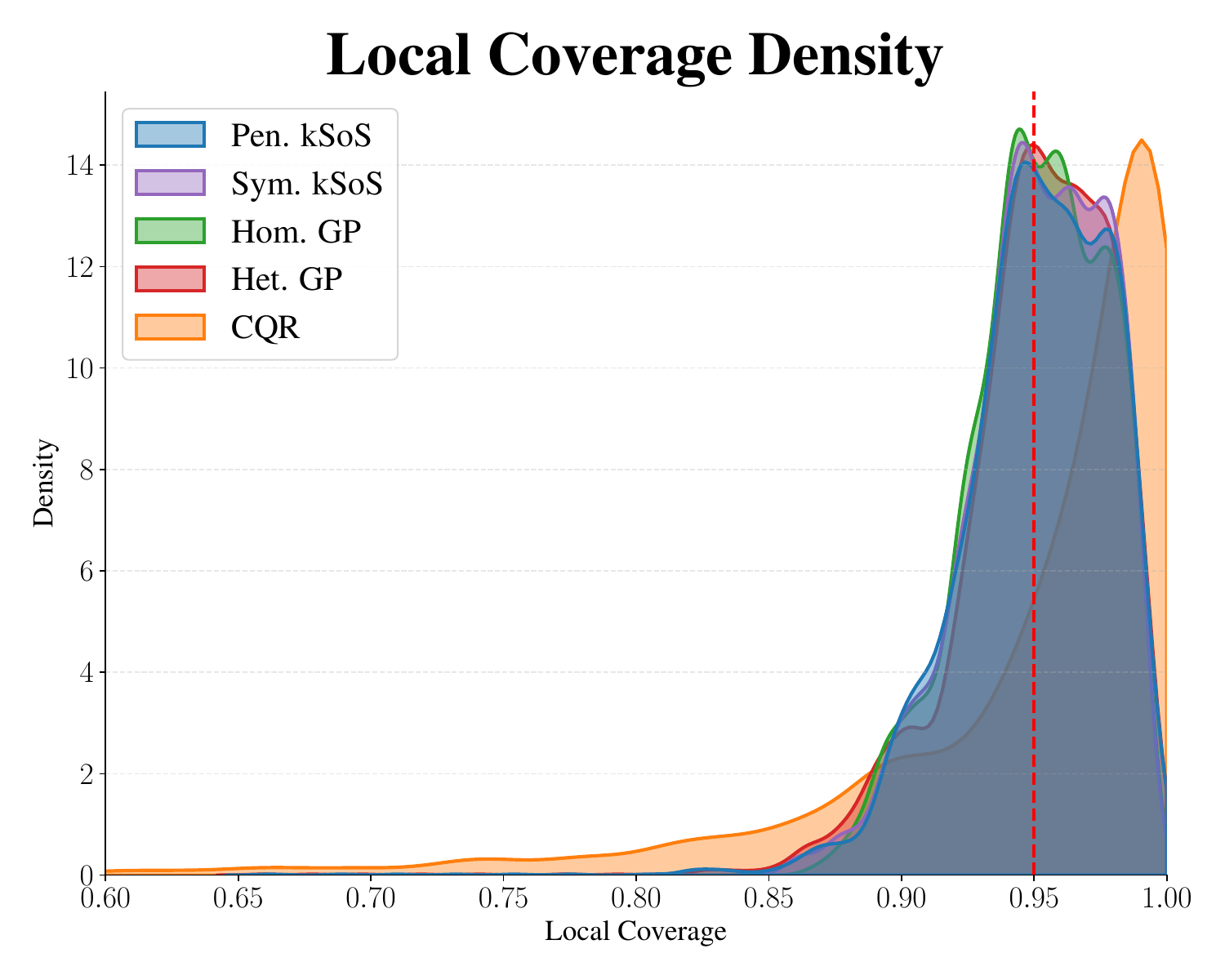}
    \caption{Test case 6 with \(d=1\) and \(n=100\). Mean width, Absolute Coverage Gap combined, RMSE versus oracle and density of local coverage lower and upper combined, \(b=0\).}
    \label{fig:comparisons_case_6}
\end{figure}

As expected, homoscedastic GP, which produces almost constant intervals, performs the best in this setting. Symmetric kSoS activates the HSIC test of independence on most random seeds, and thus yields similar performance. Interestingly, our penalized kSoS almost always selects a symmetric model, thus allowing to get metrics close to the best ones.

Figure \ref{fig:case_6_n_1000} shows the optimal solution of the dual formulation obtained for $n=1000$.

\begin{figure}[htbp]
        \centering
        \includegraphics[width=0.6\linewidth]{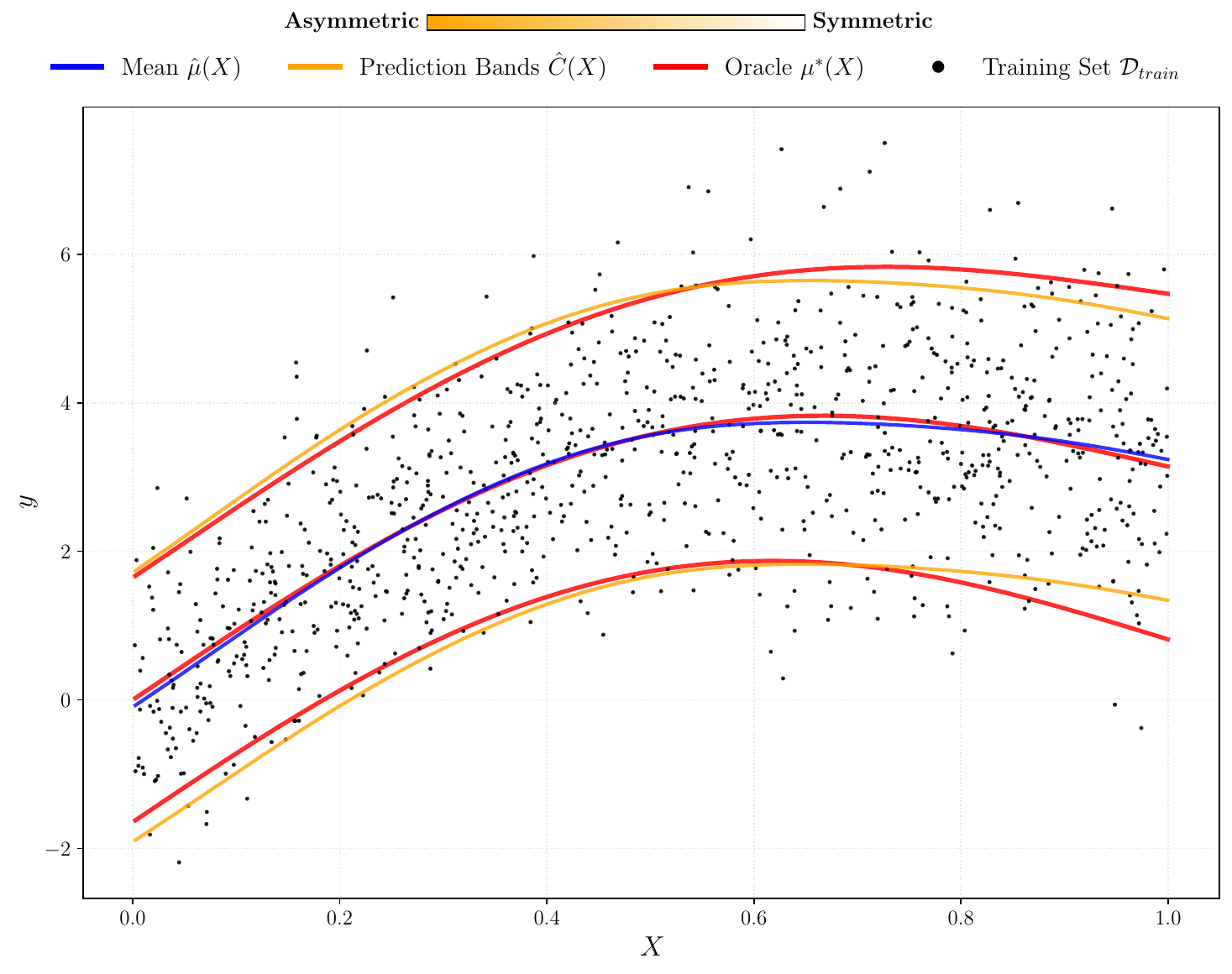}
        \caption{Test case 6 with \(d=1\) and \(n=1000\). Optimal solution of dual formulation with penalty 1.}
        \label{fig:case_6_n_1000}
\end{figure}

\clearpage
\subsubsection{Additional real-world datasets and results}
\label{sec:additional real-world datasets}

The 13 real-world datasets we consider are the following:
\begin{enumerate}
    \item Concrete \citep{concrete_compressive_strength_165} and Bike \citep{bike_sharing_275} are taken from \citet{romano2019conformalizedquantileregression}.
    \item  Diabetes \citep{efron2004least}, Housing \citep{pace1997sparse}, MPG \citep{auto_mpg_9}, Boston \citep{harrison1978hedonic}, Energy \citep{energy_efficiency_242}, Miami \citep{mayer2022machine}, Sulfur \citep{fortuna2007soft}, Power \citep{combined_cycle_power_plant_294}, Yacht \citep{yacht_hydrodynamics_243} and Kin8nm \citep{ghahramani1996kin} are standard regression datasets.
\end{enumerate}

Table \ref{tab:descr_realworld} provides a detailed description of each of them in terms of number of features and total sample size. Note that for some of them, we perform a preliminary preprocessing step, by removing outliers (with a homoscedastic GP model) and removing inactive features (following previous literature where they were investigated for CP or regression tasks). For Housing specifically, we remove censored data (target $\geq 5$) and consider two variants with or without logarithm transformation of the target. We also apply such transformation for Yacht.\\

\begin{table}[ht]
\centering
\caption{Description of thirteen real-world datasets: number of features (before and after filtering, if applicable), total sample size, training set sample size, calibration set sample size, test set sample size, percentage of removed outliers in training set.}
\begin{tabular}{l|cccccc}
\toprule
Dataset & Nb features & Total sample size & $n_\textrm{train}$ & $n_\textrm{cal}$ & $n_\textrm{test}$ & Outliers \\
\midrule
Concrete   & 8 & 1030 & 412 & 412 & 206 & $30\%$\\
\midrule
Bike       & 13 (18) & 10886 & 1000 & 1000 & 1000 & $20\%$\\
\midrule
Diabetes   & 10 & 442 & 101 & 170 & 171 & NA\\
\midrule
Housing log    & 8 & 20640 & 1000 & 1000 & 1000 & NA \\
\midrule
Housing    & 8 & 20640 & 1000 & 1000 & 1000 & NA \\
\midrule
MPG        & 7 & 398 & 100 & 146 & 146 & NA\\
\midrule
Boston        & 10 (13) & 490 & 100 & 200 & 190 & NA\\
\midrule
Energy        & 6 (8) & 768 & 300 & 234 & 234 & NA\\
\midrule
Miami        & 7 (16) & 13932 & 1000 & 1000 & 1000 & NA\\
\midrule
Sulfur        & 6 & 10081 & 1000 & 1000 & 1000 & NA\\
\midrule
Power        & 4 & 9568 & 1000 & 1000 & 1000 & $30\%$\\
\midrule
Yacht        & 1 (6) & 308 & 100 & 108 & 100 & $20\%$\\
\midrule
Kin8nm      & 6 (8) & 8192 & 1000 & 1000 & 1000 & NA \\
\bottomrule
\end{tabular}
\label{tab:descr_realworld}
\end{table}

Each experiment is repeated 10 times, where we randomly sample the training, calibration and test datasets. We compare our penalized kSoS (predictive model obtained with a homoscedastic GP and hyperparameters estimated with our HSIC criterion, isotropic Mat\'ern 5/2 kernel) with CQR, homoscedastic GP and heteroscedastic GP (both also with Mat\'ern 5/2 kernel, but anisotropic) and symmetric kSoS from \citet{allain2025scalableandadaptivepredictionbandsusingkerenlsumofsquares}. For kSoS, we apply the same normalizing preprocessing step as for analytical test cases, and only consider the training set penalty which scales linearly with the number of training samples. For completeness, we first provide in Table \ref{tab:cov_realworld} the estimated marginal coverage on the test set: as expected from theory, all methods achieve the target coverage ($\alpha=0.1$ here).

\begin{table}[ht]
\centering
\caption{Estimated marginal coverage on the test test for thirteen real-world datasets (mean$\pm$sd on 10 repetitions).}
\begin{tabular}{l|ccccc}
    \toprule
    Dataset & CQR & Het. GP & Hom. GP & Sym. kSoS & Pen. kSoS \\
    \midrule
    Concrete & $0.906 \pm 0.023$ & $0.900 \pm 0.027$ & $0.903 \pm 0.027$ & $0.904 \pm 0.018$ & $0.916 \pm 0.026$ \\
    Bike & $0.897 \pm 0.017$ & $0.897 \pm 0.020$ & $0.902 \pm 0.018$ & $0.895 \pm 0.016$ & $0.894 \pm 0.014$ \\
    Diabetes & $0.881 \pm 0.044$ & $0.888 \pm 0.024$ & $0.907 \pm 0.023$ & $0.902 \pm 0.027$ & $0.899 \pm 0.023$ \\
    Housing log & $0.900 \pm 0.009$ & $0.900 \pm 0.016$ & $0.897 \pm 0.016$ & $0.895 \pm 0.013$ & $0.899 \pm 0.011$ \\
    Housing & $0.902 \pm 0.007$ & $0.903 \pm 0.012$ & $0.899 \pm 0.014$ & $0.892 \pm 0.011$ & $0.888 \pm 0.011$ \\
    MPG & $0.901 \pm 0.048$ & $0.905 \pm 0.025$ & $0.889 \pm 0.048$ & $0.886 \pm 0.042$ & $0.894 \pm 0.049$ \\
    Boston & $0.910 \pm 0.021$ & $0.900 \pm 0.032$ & $0.902 \pm 0.022$ & $0.904 \pm 0.017$ & $0.912 \pm 0.016$ \\
    Energy & $0.902 \pm 0.024$ & $0.896 \pm 0.021$ & $0.897 \pm 0.026$ & $0.904 \pm 0.015$ & $0.903 \pm 0.016$ \\
    Miami & $0.899 \pm 0.016$ & $0.899 \pm 0.016$ & $0.903 \pm 0.016$ & $0.900 \pm 0.017$ & $0.900 \pm 0.011$ \\
    Sulfur & $0.897 \pm 0.015$ & $0.900 \pm 0.014$ & $0.891 \pm 0.014$ & $0.901 \pm 0.014$ & $0.902 \pm 0.011$ \\
    Power & $0.895 \pm 0.007$ & $0.894 \pm 0.010$ & $0.898 \pm 0.010$ & $0.896 \pm 0.007$ & $0.897 \pm 0.008$ \\
    Yacht & $0.895 \pm 0.055$ & $0.901 \pm 0.034$ & $0.901 \pm 0.036$ & $0.909 \pm 0.040$ & $0.907 \pm 0.041$ \\
    Kin8nm & $0.902 \pm 0.017$ & $0.900 \pm 0.013$ & $0.895 \pm 0.014$ & $0.898 \pm 0.011$ & $0.894 \pm 0.014$ \\
    \bottomrule
\end{tabular}
\label{tab:cov_realworld}
\end{table}

The mean width achieved by each method is given in Table \ref{tab:meanwidth_realworld}. First observe that homoscedastic GP very often yields the smaller intervals among all methods, in particular for Concrete, Housing, Boston and Kin8nm. But we will see below that precisely for these datasets, it comes at the price of a poorer local coverage, thus confirming our assertion that mean width alone is not sufficient to differentiate competitors. Our penalized kSoS model also exhibits the smaller mean width in most instances, but with much better local coverage, as elaborated in what follows.

\begin{table}[ht]
\centering
\caption{Mean width of prediction intervals on the test test for thirteen real-world datasets (\textbf{mean}$\pm$sd on 10 repetitions). Mean width values within $1\%$ of the minimum are displayed in bold.}
\begin{tabular}{l|ccccc}
    \toprule
    Dataset & CQR & Het. GP & Hom. GP & Sym. kSoS & Pen. kSoS \\
    \midrule
    Concrete & $22.66 \pm 1.06$ & $21.53 \pm 1.58$ & $\mathbf{21.01} \pm 1.58$ & $28.00 \pm 2.78$ & $21.74 \pm 1.19$ \\
    Bike & $217.98 \pm 6.54$ & $199.23 \pm 14.27$ & $169.57 \pm 7.83$ & $177.53 \pm 11.82$ & $\mathbf{162.47} \pm 5.86$ \\
    Diabetes & $\mathbf{188.65} \pm 12.59$ & $196.25 \pm 16.01$ & $193.21 \pm 15.80$ & $195.80 \pm 15.81$ & $\mathbf{190.46} \pm 14.63$ \\
    Housing log & $0.967 \pm 0.039$ & $0.855 \pm 0.041$ & $\mathbf{0.768} \pm 0.040$ & $0.967 \pm 0.140$ & $0.836 \pm 0.037$ \\
    Housing & $1.77 \pm 0.055$ & $1.70 \pm 0.14$ & $\mathbf{1.60} \pm 0.082$ & $2.02 \pm 0.38$ & $1.80 \pm 0.14$ \\
    MPG & $9.82 \pm 1.06$ & $9.76 \pm 1.31$ & $\mathbf{9.33} \pm 1.02$ & $9.68 \pm 1.35$ & $\mathbf{9.32} \pm 1.14$ \\
    Boston & $12.16 \pm 1.18$ & $10.85 \pm 0.92$ & $\mathbf{9.76} \pm 0.80$ & $17.93 \pm 18.06$ & $11.05 \pm 1.23$ \\
    Energy & $1.50 \pm 0.16$ & $1.46 \pm 0.098$ & $1.72 \pm 0.079$ & $1.79 \pm 0.42$ & $\mathbf{1.39} \pm 0.12$ \\
    Miami & $31.4e^{4} \pm 1.17e^{4}$ & $36.4e^{4} \pm 8.23e^{4}$ & $29.5e^{4} \pm 2.29e^{4}$ & $28.3e^{4} \pm 1.53e^{4}$ & $\mathbf{26.9e^{4}} \pm 1.45e^{4}$ \\
    Sulfur & $\mathbf{0.051} \pm 0.003$ & $0.055 \pm 0.009$ & $\mathbf{0.050} \pm 0.003$ & $0.064 \pm 0.010$ & $\mathbf{0.050} \pm 0.002$ \\
    Power & $13.29 \pm 0.30$ & $\mathbf{12.88} \pm 0.40$ & $13.05 \pm 0.36$ & $13.62 \pm 0.45$ & $\mathbf{12.99} \pm 0.29$ \\
    Yacht & $0.645 \pm 0.085$ & $\mathbf{0.595} \pm 0.086$ & $0.609 \pm 0.082$ & $0.616 \pm 0.097$ & $\mathbf{0.599} \pm 0.105$ \\
    Kin8nm & $0.607 \pm 0.020$ & $0.425 \pm 0.023$ & $\mathbf{0.386} \pm 0.009$ & $0.453 \pm 0.043$ & $0.393 \pm 0.011$ \\
    \bottomrule
\end{tabular}
\label{tab:meanwidth_realworld}
\end{table}

To investigate adaptivity, we thus also analyze the worst-set coverage (where lower and upper variants are combined) in Figure \ref{fig:minRC_all}. In average, CQR and kSoS (penalized and symmetric) have better worst-set coverage than both GPs: this confirms the initial findings of \citet{allain2025scalableandadaptivepredictionbandsusingkerenlsumofsquares}, and shows that our penalized kSoS has the desired adaptive property. Now, if we compare exhaustively penalized kSoS against competitors, we notice that:
\begin{enumerate}
    \item Penalized kSoS almost always outperforms CQR in terms of worst-set coverage, while the latter systematically yields intervals with much larger mean width
    \item Penalized and symmetric kSoS have equivalent worst-set coverage in all cases, but symmetric kSoS always produces bands with larger mean width
    \item For the datasets where homoscedastic GP had smaller mean width than penalized kSoS (Concrete, Housing, Boston and Kin8nm), it is significantly behind in terms of worst-set coverage in most instances. For the other datasets, where penalized kSoS and GPs have similar worst-set coverage:
    \begin{itemize}
    \item Bike: homoscedastic GP and penalized kSoS also have similar mean width, which hints towards a symmetric dataset. In Figure \ref{fig:hist_lambda_rw}, selected \(\lambda_{\mathrm{pen}}\) values for kSoS indicate that pure asymmetry is never considered for our model.
    \item Diabetes: both kSoS have mean width equivalent to homoscedastic GP, once again suggesting this is a symmetric dataset. This intuition is confirmed by Figure \ref{fig:hist_lambda_rw}, where we automatically select a symmetric model half of the time.
    \item Energy: penalized kSoS clearly outperforms both GPs in terms of mean width.
    \item Miami: same behavior as for Energy.
    \item Sulfur: mean widths for penalized kSoS and homoscedastic GP are similar, from which we can infer underlying symmetry. Figure \ref{fig:hist_lambda_rw} shows that our model indeed heavily favors high values of \(\lambda_{\mathrm{pen}}\).
    \item Power: heteroscedastic GP produces intervals with the same mean width as penalized kSoS. As can be seen in Figure \ref{fig:hist_lambda_rw}, our model also selects larger penalty values in average, which corroborates the hypothesis of a symmetric noise distribution.
    \item Yacht: same behavior as for Power.
\end{itemize}
\end{enumerate}

\begin{figure}[ht]
    \centering
    \includegraphics[width=\linewidth]{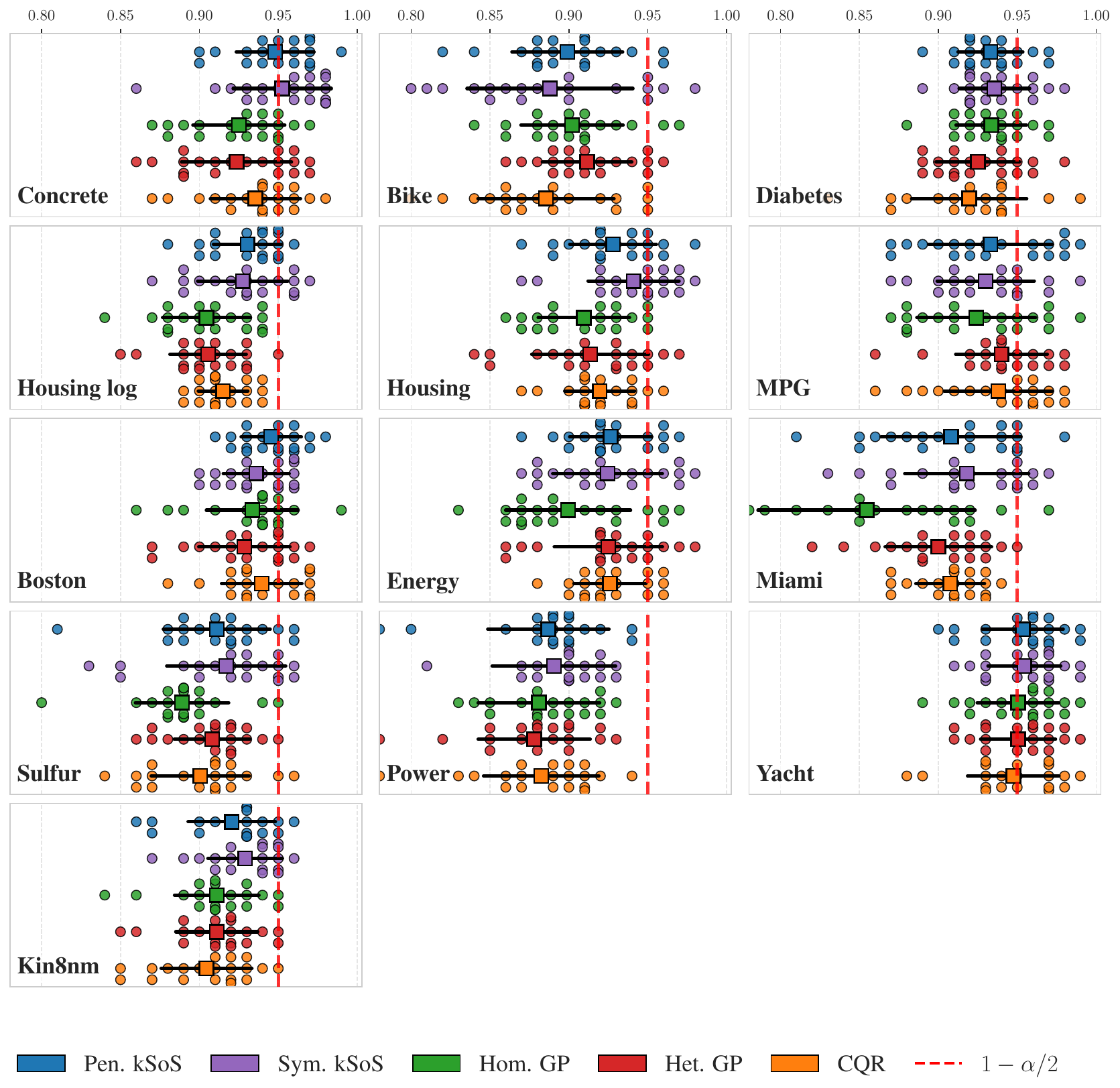}
    \caption{Mean and standard deviation of worst-set coverage lower and upper combined, 10 repetitions.}
    \label{fig:minRC_all}
\end{figure}

\begin{figure}[ht]
    \centering
    \includegraphics[width=\linewidth]{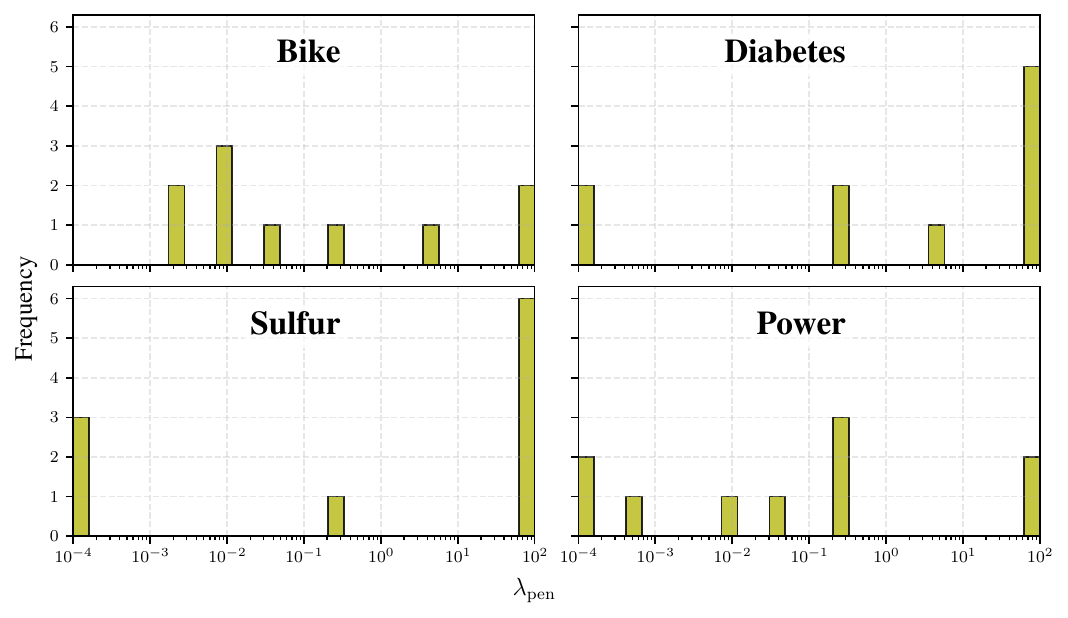}
    \caption{Histogram of automatically selected value of \(\lambda_{\mathrm{pen}}\) with HSIC and Kruskal-Wallis test, 10 repetitions.}
    \label{fig:hist_lambda_rw}
\end{figure}


\end{document}